\documentclass{hwthesis}

\usepackage{amsfonts,amsmath,amsthm,fancybox,amssymb,amstext,amscd,array,bbold,bm,bbm,slashed,graphicx,mathrsfs,subfigure}
\usepackage{mathrsfs}
\usepackage{color}
\usepackage{cite}
\usepackage{url}
\usepackage{latexsym,amssymb}
\usepackage{tikz}
\usepackage{verbatim}
\DeclareGraphicsExtensions{.png}
\usepackage[all]{xy}

\newtheorem{theorem}{Theorem}[section]
\newtheorem{proposition}[theorem]{Proposition}
\newtheorem{thm}[theorem]{Theorem}
\newtheorem{ex}[theorem]{Example}
\newtheorem{defin}[theorem]{Definition}
\newtheorem{corollary}[theorem]{Corollary}
\newtheorem{example}[theorem]{Example}

\newtheorem{lemma}[theorem]{Lemma}
\newtheorem{remark}[theorem]{Remark}
\newtheorem{conjecture}[theorem]{Conjecture}

\newcommand{\dom}{\mathbf{d}}
\newcommand{\ran}{\mathbf{r}}
\newcommand{\rank}{\hbox{rank}}
\newcommand{\supp}{\hbox{Supp}}
\newcommand{\im}{\hbox{im}}

\newcommand{\Obj}{\hbox{Obj}}
\newcommand{\Arr}{\hbox{Arr}}
\newcommand{\Idem}{\hbox{Idem}}
\newcommand{\Aut}{\hbox{Aut}}
\newcommand{\Vect}{\hbox{Vec}}

\newcommand{\bbox}{\hfill\rule[.25mm]{1.9mm}{1.9mm}}

\newcommand{\dm}{\mathbf{d}}
\newcommand{\rn}{\mathbf{r}}

\DeclareMathOperator{\C}{\bf{C}}

\DeclareMathOperator{\Set}{\bf{Set}}
\DeclareMathOperator{\Mod}{\bf{Mod}}
\DeclareMathOperator{\Premod}{\bf{Premod}}

\DeclareMathOperator{\IdemProj}{\bf{IdemProj}}
\DeclareMathOperator{\Etale}{\bf{Etale}}
\DeclareMathOperator{\Proj}{\bf{Proj}}
\DeclareMathOperator{\Proje}{\hbox{Proj}}

\input xypic

\title{Semigroup and Category-Theoretic Approaches to Partial Symmetry}
\author{Alistair R. Wallis}
\date{July 2013}
\begin{document}
\pagenumbering{roman}
\maketitle

\begin{abstract}
This thesis is about trying to understand various aspects of partial symmetry using ideas from semigroup and category theory. In Chapter 2 it is shown that the left Rees monoids underlying self-similar group actions are precisely monoid HNN-extensions. In particular it is shown that every group HNN-extension arises from a self-similar group action. Examples of these monoids are constructed from fractals. These ideas are generalised in Chapter 3 to a correspondence between left Rees categories, self-similar groupoid actions and category HNN-extensions of groupoids, leading to a deeper relationship with Bass-Serre theory. In Chapter 4 of this thesis a functor $K$ between the category of orthogonally complete inverse semigroups and the category of abelian groups is constructed in two ways, one in terms of idempotent matrices and the other in terms of modules over inverse semigroups, and these are shown to be equivalent. It is found that the $K$-group of a Cuntz-Krieger semigroup of a directed graph $G$ is isomorphic to the operator $K^{0}$-group of the Cuntz-Krieger algebra of $G$ and the $K$-group of a Boolean algebra is isomorphic to the topological $K^{0}$-group of the corresponding Boolean space under Stone duality.  
\end{abstract}

\renewcommand{\abstractname}{Acknowledgements}
\begin{abstract}
I would like to thank all who have encouraged and aided me in my PhD studies. Above all I would like to express sincere gratitude to my supervisor Mark V. Lawson for all his support and guidance. I am also grateful for the additional support of my second supervisor Jim Howie and others in the mathematics department at Heriot-Watt University with whom I have had discussions on mathematics, and to those with whom I have shared an office for their succor and stimulation. Finally, I would like to express my thanks to my family and friends who have kept me going through all my research, in particular Alex Bailey, James Hyde and Nick Loughlin for all the exciting ideas we have had about mathematics and the world of business. The human element is always important in keeping one on one's toes.
\end{abstract}

\newpage 

%\setcounter{page}{1}
%\listoffigures{}
\tableofcontents{}
\pagenumbering{arabic}
\setcounter{page}{0}

\chapter{Introduction}

\section{Overview}

A semigroup is a set with an associative binary operation and a monoid is a semigroup with identity. Introductions to semigroup theory include \cite{CliffordPreston1},\cite{CliffordPreston2}, \cite{HowieBook}. Some of the earliest work on semigroups was done by Suschkewitsch and Rees, and in fact one of the fundamental objects of study in Chapter 2 of this thesis are left Rees monoids, named in honour of David Rees who studied a particular class of such semigroups in his paper \cite{Rees}. A monoid $M$ is \emph{right rigid} if $aM \cap bM \neq \emptyset$ implies that $aM \subseteq bM$ or $bM \subseteq aM$;
this terminology is derived from Cohn \cite{Cohn71}. A left Rees monoid is a right rigid left cancellative monoid which satisfies an ascending chain condition on the chains of principal right ideals. Right Rees monoids are defined to be the right-hand dual, i.e. left rigid right cancellative monoids with an ascending chain condition on chains of principal left ideals. A Rees monoid is one which is both left and right Rees. It is proved in \cite{LawLRM1} that every right cancellative left Rees monoid is in fact a Rees monoid. Semigroups and monoids can often be thought of in a similar manner to rings, the idea being that the multiplicative structure of a ring has the structure of a semigroup and if the ring has an identity then this semigroup is a monoid. This thought will be pursued further later, suffice to say for the moment that one of the most important initial considerations of a semigroup is its ideal structure, and this gives an indication as to why Rees was already thinking about similar structures to those which we will be studying so early on in the history of semigroup theory. Self-similar group actions were introduced by Nivat and Perrot (\cite{NivatPerrot},\cite{Perrot1},\cite{Perrot2}) in the study of certain $0$-bisimple inverse monoids, though this is not the terminology they used. We will see how their ideas came about in Section 3.9. The concept of self-similar group actions re-emerged with the work of Grigorchuk, Bartholdi, Nekrashevych and others in the study of groups generated by automata. We will study such automata with examples in Section 2.7. Cain (\cite{cain}) has generalised these ideas to the notion of a self-similar semigroup. Lawson in \cite{LawLRM1} showed that there is in fact a one-one correspondence between left Rees monoids and self-similar group actions. The idea here is that given a self-similar group action of a group $G$ on a free monoid $X^{\ast}$, when one takes their Zappa-Sz\'{e}p product (essentially a two-sided semidirect product) the resulting structure is a left Rees monoid, and all left Rees monoids can be constructed in this manner. We will summarise the details of this in Section 2.2. In Section 2.4 we will consider when one can extend a self-similar group action of a group $G$ on a free monoid $X^{\ast}$ to self-similar action of the group $G$ on the free group $FG(X)$; this turns out to be precisely when the left Rees monoid is \emph{symmetric}.
We will briefly consider the representation theory of left Rees monoids in Section 2.8.

One way of viewing self-similar group actions is in terms of homomorphisms into the automorphism groups of regular rooted trees, giving rise to the so-called wreath recursion, details of which are summarised in Section 2.2. The salient point to note is that we have a group acting on a tree. Groups acting on trees give rise to graphs of groups (\cite{LyndonSchupp},\cite{Serre}). A number of interesting groups arise as the fundamental groups of graphs of groups. Perhaps the simplest situation is where one has a single vertex with associated group $G$, and $|I|$ loops from the vertex to itself, each labelled by an injective homomorphism $\alpha_{i}$ from a subgroup $H_{i}$ of $G$ into $G$. The fundamental group of such a graph of groups is called an \emph{HNN-extension}. The resulting group $\Gamma$ has group presentation
$$\Gamma = \langle G, t_{i}:i\in I| \mathcal{R}(G), ht_{i} = t_{i}\alpha_{i}(h) \quad h\in H_{i}, i\in I\rangle,$$
where $\mathcal{R}(G)$ denotes the relations of $G$. Note that several authors have the condition reversed, i.e. they adopt the convention $t_{i}h = \alpha_{i}(h)t_{i}$. One of the main results of this thesis is that if one takes such a presentation, and let us not assume that the maps are injective, then if we instead take a monoid presentation the resulting monoid is a left Rees monoid, and every left Rees monoid is such a monoid HNN-extension (see Section 2.3). It then follows as a corollary that if we in fact have a Rees monoid then its group of fractions is a group HNN-extension, and every group HNN-extension arises in this way. In addition one finds that if the Rees monoid is symmetric then this group HNN-extension is a Zappa-Sz\'{e}p product of a free group and a group. Part of the inspiration for this result is a theorem by Cohn on the embeddability of cancellative right rigid monoids into groups, whose proof in \cite{Cohn71} has the flavour of taking the fundamental group of a graph of groups. The author suspects that in fact this proof doesn't quite work, and this is further evidenced by the fact that Cohn utilises a different proof in the second edition of the same book (\cite{Cohn85}) (see more details in Section 2.3). These results were then to some extent generalised to the situation of categories embedding in groupoids by von Karger (\cite{VKarger}). The notion of HNN-extension has previously been generalised to the situation of semigroups in \cite{GilbertHNN}, \cite{Howie63} and \cite{Yamamura}. Gilbert and Yamamura consider the case where the semigroup is inverse and Howie considers the situation where $tt^{\prime}$ and $t^{\prime}t$ are idempotents for some $t^{\prime}$.

The term \emph{fractal} was coined by Mandelbrot in the 1970's to describe a number of geometric structures which were very jagged in structure (here \emph{fractal} is derived from the Latin word \emph{fractum} meaning \emph{broken}). One of the key properties many fractals share is that of \emph{self-similarity}. The idea here is that if we zoom in on a fractal we find a structure similar to the one with which we started. Self-similarity can be seen as one form of partial symmetry, a notion to be considered again later. Fractal-like structures appear in a variety of contexts in the natural world, for example in modelling coastlines and certain ferns (\cite{PrusinkiewiczLindenmayer}). They have also been used in the modelling of electrical resistance networks (\cite{Bajorin}, \cite{Boyle}). Another application is in optimising reception while at the same time minimising surface area in mobile telephone aerials (\cite{Puenta}).
It has been realised for some time that there exist connections between self-similar group actions and fractals and other interesting geometric structures (see for example \cite{BartGrigNekra}, \cite{BartholdiHenriquesNekrashevych}, \cite{GrigNekSus}, \cite{KelleherSteinhurstWong}, \cite{NekrashevychHyperbolic}, \cite{NekrashevychBook}). In many of the examples considered the fractal is obtained as a kind of limit space of a self-similar group action; that is, the regular rooted tree modulo the action in a specified way gives rise to a geometric structure with fractal-like properties. In this thesis it will in fact be shown that the monoid of similarity transformations of the attractor of an iterated function system is often a Rees monoid. This will be proved in Section 2.5 and a number of examples will be considered. This fact is actually used implicitly in the calculations of \cite{Bajorin}. The idea is that given an iterated function system $f_{1},\ldots,f_{n}:F\rightarrow F$, one finds in a number of examples that the semigroup generated by these maps is free and thus letting $X = \left\{f_{1},\ldots,f_{n}\right\}$ we can sometimes get a self-similar action of the group of isometries $G$ of $F$ on $X^{\ast}$. Moreover, the group $G$ is in several cases finite and so one has associated an automaton which gives rise to this self-similar action. One can then view the automaton as describing a computer programme with two recursively defined functions, one for the action and the other for the restriction, together with a number of base cases. A Scala programme is given in Appendix A which models this situation.
The algebraic properties of transformations on fractals have previously been studied in \cite{FalconerOConnor}, \cite{Strichartz} and \cite{Strichartz2}. Bandt and Retta (\cite{BandtRetta}) have discovered a number of fractal-like structures whose properties depend only up to homeomorphism, and such that every into-homeomorphism is in fact a similarity transformation. We describe some corollaries of their work in Section 2.6.

Another mathematical structure which will be important in this thesis is a \emph{category}. Categories were introduced by Samuel Eilenberg and Saunders MacLane in the 1940's in the study of the many functors arising in algebraic topology. Since then categories have found their way into many other areas of mathematics. For example, Lawvere and Rosebrugh have shown that much of axiomatic set theory can be reformulated in the language of category theory \cite{LawvereRosebrugh}. All undefined terms from category theory can be found in \cite{Awodey},\cite{Borceux},\cite{MacLane},\cite{Mitchell}. We will say more about categories in the following section.

\emph{Left Krieger semigroups} were introduced by Lawson in \cite{LawsonSubshifts} as a generalisation of left Rees monoids, these being precisely left Krieger semigroups which do not have a zero element. It was found that there existed a similar description of arbitary left Krieger semigroups in terms of Zappa-Sz\'{e}p products. The underlying category of a left Krieger semigroup categorical at zero was termed a \emph{left Rees category}. These left Rees categories were then further investigated in \cite{JonesLawsonGraph} in the study of graph inverse semigroups. By adapting slightly the notions and results of \cite{LawsonSubshifts} one is led to a correspondence between left Rees categories and self-similar groupoids (here the free monoid on a set $X$ will be replaced by the free category on a graph $\mathcal{G}$). In Chapter 3 we will show that many of the concepts and properties of left Rees monoids and self-similar group actions can be generalised to the context of left Rees categories and self-similar groupoids. In Section 3.2 we will briefly describe how one arrives at the correspondence between left Rees categories and self-similar groupoid actions from the work of \cite{LawsonSubshifts} and \cite{JonesLawsonGraph}. In Sections 3.3  we will see that our results about monoid HNN-extensions can be rephrased for the categorical context. This will then lead to further connections with Bass-Serre theory. We will then see in Sections 3.4 and 3.5 that, under suitable assumptions on the graphs and categories in question, we can replace the concepts of automorphism group of the regular rooted tree and the wreath recursion by suitable categorical notions. We also, in Section 3.6, define automaton groupoids analogously to the group situation. A different form of automaton groupoids had previously been considered in \cite{ChoJorgensen}. An indication will be given in Section 3.7 of how one might want to generalise the results about iterated function systems to graph iterated function systems. In Section 3.8 we will consider the representation theory of left Rees categories. One of the curious aspects about left Rees categories (unlike the situation for left Rees monoids) is that one can get finite examples which are not just groupoids. This will then lead to a connection with the representation theory of finite-dimensional algebras.
Finally, in Section 3.9 we will see how one can naturally associate an inverse semigroup to a left Rees category, and this section will act as a bridge between the work of Chapters 2 and 3 and that of Chapter 4. Examples that can be constructed in this manner include the polycyclic monoids and graph inverse semigroups.

In recent decades it has been realised that there exist deep connections between three mathematical structures: inverse semigroups, topological groupoids and $C^{\ast}$-algebras (for example, see \cite{Exel}, \cite{KhoshkamSkandalis}, \cite{LawsonShortNCSD}, \cite{LawsonLongNCSD}, \cite{LawsonLenz}, \cite{LawsonMargolisSteinberg}, \cite{LawsonSteinberg}, \cite{Lenz}, \cite{MatsnevResende}, \cite{MilanSteinberg}, \cite{Paterson}, \cite{Renault}, \cite{Resende}, \cite{Tu}). Good introductions to inverse semigroups, topology and $C^{\ast}$-algebras are \cite{LawsonBook}, \cite{Simmons} and \cite{Landsman}. The connection between $C^{\ast}$-algebras and topological groupoids can be seen as a generalisation of the Gelfand representation theorem viewing commutative $C^{\ast}$-algebras as rings of functions over topological spaces. Chapter 4 of this thesis can be considered as fitting within this framework. Jones and Lawson have shown that the representation theory of the Cuntz $C^{\ast}$-algebras as considered in \cite{BratteliJorgensen} can be thought in terms of the representation theory of the polycyclic monoid (\cite{JonesLawsonReps}). It was mentioned earlier that semigroups can often be thought of in similar ways to rings. An example of this is in the study of Morita equivalence in the context of semigroups (\cite{Knauer}, \cite{LawsonLocal}, \cite{Talwar}). More recently it has been realised that one can describe Morita equivalence for inverse semigroups in a manner analogous to $C^{\ast}$-algebras (\cite{AfaraLawson}, \cite{FunkLawStein}, \cite{LawsonHeaps}, \cite{Steinberg}). 
%Leinster (\cite{Leinster}) has shown how ultrafilters can be viewed as probability measures and using this one can then integrate with respect to %ultrafilters. 
One example of the correspondence between these three mathematical structures is given by the polycyclic monoids $P_{n}$, the Cuntz groupoids $G_{n}$ and Cuntz algebras $\mathcal{O}_{n}$. The Cuntz algebra was introduced by Cuntz in \cite{Cuntz} and can be constructed from the Cuntz groupoid. The Cuntz monoid considered by Lawson in \cite{LawPolyThomp} is the distributive completion of the polycyclic monoid. One can construct the Cuntz groupoid from the Cuntz monoid via the theory of \cite{LawsonLongNCSD} and \cite{LawsonLenz}. This construction is an example of a non-commutative Stone duality. Graph inverse semigroups are a generalisation of the polycyclic monoids. The $C^{\ast}$-algebra associated to a graph inverse semigroup is then the Cuntz-Krieger algebra, and again one can use the Cuntz-Krieger semigroups, the distributive completions of graph inverse semigroups, to construct the Cuntz-Krieger algebras. Leavitt path algebras (\cite{Tomforde}) are the algebras generated in the same way as the Cuntz-Krieger algebras without requiring one ends up with a $C^{\ast}$-algebra. It has been shown in \cite{AraMorenoPardo} that these are Morita equivalent to the Cuntz-Krieger algebras. Lawson \cite{LawsonAF} has introduced AF-monoids as the inverse semigroup counterpart to AF-algebras \cite{VershikKerov}.

The aim of Chapter 4 of this thesis is to define a functor $K$ from the category of orthogonally complete inverse semigroups to the category of abelian groups, in analogy with the $K_{0}$-group of algebraic $K$-theory. Other homological approaches to inverse semigroups include those in \cite{GilbertDerivations}, \cite{Lausch}, \cite{Loganathan}. The $K$-theory of $C^{\ast}$-algebras associated to inverse semigroups has previously been investigated in \cite{EphremSpielberg}, \cite{Mcclanahan} and \cite{Norling}. Standard references on $K$-theory include \cite{Atiyah}, \cite{BlackadarKT} and \cite{Rosenberg}. We will give a brief overview of some aspects of topological and algebraic $K$-theory in Section 1.4. $K$-theory was originally introduced by Grothendieck in the study of coherent sheaves over algebraic varieties. Atiyah and Hirzebruch then introduced topological $K^{0}$-groups by observing that vector bundles over manifolds are in some sense akin to coherent sheaves over algebraic varieties. The Serre-Swan theorem then says that these vector bundles are in one-one correspondence with the finitely generated projective modules of a $C^{\ast}$-algebra of continuous functions. This then gave rise to algebraic and operator $K$-theories. It is also possible to define higher $K$-groups. $K$-theory is used in the classification of topological spaces, rings and operator algebras. The author believes that the classification of semigroup $C^{\ast}$-algebras by their $K_{0}$-groups is really a $K$-theoretic classification of inverse semigroups. We will see that if $S$ is a (unital) Boolean algebra then $K(S)$ will be isomorphic to (topological) $K^{0}(B(S))$, where $B(S)$ denotes the corresponding Boolean space and if $CK_{\mathcal{G}}$ is the Cuntz-Krieger semigroup of a graph $\mathcal{G}$ then $K(CK_{\mathcal{G}})$ will be isomophic to (operator) $K^{0}(\mathcal{O}_{\mathcal{G}})$. Topological $K$-theory is used in the study of characteristic classes in differential topology and operator $K$-theory is employed in Connes' programme of non-commutative differential geometry (\cite{Connes}). Operator $K$-theory is also used in the gap-labelling theory of tilings (\cite{BaakeMoody}, \cite{Bellissard}, \cite{ExelGoncalvesStarling}, \cite{KellendonkGap}, \cite{KellendonkSubstitution}, \cite{KellendonkNonComm}, \cite{KellendonkInteger}, \cite{KellendonkLawson1}, \cite{KellendonkLawson2}, \cite{KellendonkPutnam}). It is believed that it might be possible to describe this gap-labelling theory in terms of the tiling semigroups by using inverse semigroup $K$-theory.

An \emph{inverse semigroup} $S$ is a semigroup such that for each element $s\in S$ there exists a unique element $s^{-1}\in S$ with $ss^{-1}s = s$ and $s^{-1}ss^{-1} = s^{-1}$. Inverse semigroups were introduced independently by Viktor Wagner and Gordon Preston in the 1950's. In the same way that we can think of groups as describing \emph{symmetry}, we can view inverse semigroups as describing \emph{partial symmetry}. The idea here is that each element of the semigroup can be thought of as describing a bijective map from part of a structure to another part of the structure. For example, if this structure is a set, then our inverse semigroup is simply a subsemigroup of the symmetric inverse monoid on that set. In fact, the \emph{Wagner-Preston representation theorem} says that every inverse semigroup embeds in a symmetric inverse monoid. This can be thought of as being analogous to Cayley's theorem for groups. Another example of our semigroup describing partial symmetry is when our structure is a topological space and our inverse semigroup is a pseudogroup of transformations of this space. Associated with a pseudogroup of transformations one has the groupoid of germs of the action. This is an example of how one can naturally associate topological groupoids to inverse semigroups, and Paterson's universal groupoid is a generalisation of this idea. Some of the connections between inverse semigroups and topological groupoids can be thought of as non-commutative versions of the natural dualities between certain classes of lattice-like algebraic structures and discrete topological spaces. These dualities are collectively known as Stone dualities, named in honour of Marshall Stone who introduced the original example relating Boolean algebras and Boolean spaces (\cite{Johnstone}). The important point here is that an inverse semigroup comes equipped with a natural partial order. We say $s\leq t$ if $s = ts^{-1}s$. In the case where these are maps on a set, this should be interpreted as saying that the domain of $s$ is a subset of the domain of $t$, and that $t$ restricted to this subset is equal to $s$. The set of idempotents $E(S)$ of a semigroup $S$ forms a meet semilattice, where the natural partial order on $E(S)$ is given by $e\leq f$ if and only if $e = ef$, so that in general the product $ef$ of elements $e,f\in E(S)$ should be thought of as their order-theoretic meet (greatest lower bound). Given an element $s\in S$ there are two idempotents which we associate with $s$: the \emph{range} of $s$ given by $\ran(s) = ss^{-1}$ and the \emph{domain} of $s$ given by $\dom(s) = s^{-1}s$. We write $e \stackrel{s}{\rightarrow} f$ to mean $e = \dm(s)$ and $f = \rn(s)$. In terms of the symmetric monoid this is saying that the set-theoretic domain of $s$ has identity map on this set given by $s^{-1}s$ and likewise for the range of $s$. All the inverse semigroups we will be considering will have a zero: an element $0$ with $0 = 0s = s0$ for all $s\in S$. This will be the least element in the natural partial order. We will say that two elements $s,t\in S$ are \emph{orthogonal}, and write $s\perp t$, if $st^{-1} = s^{-1}t = 0$. Again, thinking in terms of the symmetric inverse monoid, this means that the domains and ranges of $s$ and $t$ do not intersect. An equivalent condition for elements $s,t\in S$ to be orthogonal is that $\dom(s)\wedge\dom(t) = 0$ and $\ran(s)\wedge \ran(t) = 0$. We will denote, if it exists, the least upper bound (join) of two elements $s,t\in S$ by $s\vee t$. We will call an inverse semigroup with $0$ \emph{orthogonally complete} if every pair of orthogonal elements has a join and multiplication distributes over finite orthogonal joins. It was shown in \cite{LawOrthCompPoly} that every inverse semigroup $S$ with $0$ has an \emph{orthogonal completion}; that is, we take the semigroup $S$ and force every pair of orthogonal elements to have a join in such a way that we end up with an orthogonally complete inverse semigroup. 

\section{Categories and groupoids}

A few remarks might now be helpful to clarify the notation and terminology which will be used with respect to categorical constructions found in this thesis. We will treat small and large categories in different ways. All categories in Chapter 3 will be assumed to be small and all categories in Chapter 4 will be assumed to be large. A \emph{small category} is one such that the classes of objects and arrows of the category are actually sets. For us, a \emph{large category} will simply be any category which is not assumed to be small. That is, a large category may in fact be small. The point of this distinction is that the objects of a large category will be important whereas we will merely be interested in the arrows of a small category. Large categories will be denoted by bold font, as in $\C$. The class of objects of the category $\bf{C}$ will be denoted by $\Obj(\C)$ and the class of arrows will be denoted by $\Arr(\C)$. The class of arrows from an object $A\in \Obj(\bf{C})$ to an object $B\in \Obj(\bf{C})$ will be denoted $\C(A,B)$. Our categories will mainly be \emph{locally small}; that is, the classes $\C(A,B)$ are all sets, in which case we call $\C(A,B)$ the \emph{hom-set} between $A$ and $B$. 

We will treat small categories as algebraic structures, i.e. as sets with partially defined binary operations. 
The \emph{elements} of these small categories are the arrows, and we will replace objects by identity arrows. 
Each arrow $x$ has a {\em domain}, denoted by $\dom(x)$, and a {\em codomain} denoted by $\ran(x)$,
both of these are identities and $x = x\dom(x) = \ran(x)x$. We will write this as $\dom(x) \stackrel{x}{\rightarrow} \ran(x)$.
The set of all identity arrows of a small category $C$ will be denoted by $C_{0}$ and the set of all non-identity arrows by $C_{1}$, so that $C$ is the disjoint union of $C_{0}$ and $C_{1}$. 
Given an identity $e$ the set $eCe$ of all arrows that begin and end at $e$ forms a monoid called the {\em local monoid at} $e$.
An arrow $x$ is {\em invertible} if there is an arrow $x^{-1}$ such that $x^{-1}x = \dom(x)$ and $xx^{-1} = \ran(x)$. We call the element $x^{-1}$ the \emph{inverse} of $x$; this element is necessarily unique.
We shall say that a pair of identities $e$ and $f$ in a category $C$ are \emph{strongly connected} if and only if $eCf \neq \emptyset$ and $fCe \neq \emptyset$. 
A small category in which every arrow is invertible is called a {\em groupoid}.
We denote the subset of invertible elements of $C$ by $G(C)$.
This forms a groupoid.
If $G(C) = C_{0}$ then we shall say that the groupoid of invertible elements is {\em trivial}.
We say that a category $C$ has {\em trivial subgroups} if the only invertible elements in the local monoids are the identities.
A category $C$ will be said to be \emph{totally disconnected} if $\ran(x) = \dom(x)$ for all $x\in C$. This means that the category $C$ is just a disjoint union of monoids.
Two categories $C$ and $D$ are \emph{isomorphic} if there is a bijective functor $f:C\rightarrow D$ (so that $f|_{C_{0}}$ and $f|_{C_{1}}$ are both bijections).

A \emph{directed graph} $\mathcal{G}$ is a collection of \emph{vertices} $\mathcal{G}_0$ and a collection of \emph{edges} $\mathcal{G}_1$ 
together with two functions $\dom , \ran : \mathcal{G}_1 \rightarrow \mathcal{G}_0$ called the \emph{domain} and the \emph{range}, respectively.
All graphs in this thesis will be assumed to be directed.
Two edges $x$ and $y$ are said to be \emph{composable} if $\ran (y) =\dom (x)$.
A \emph{route} in $\mathcal{G}$ is any sequence of edges $x_1 \ldots x_n$ such that $x_i$ and $x_{i+1}$ are composable for all $i=1, \ldots , n$.
The \emph{free category} $\mathcal{G}^{\ast}$ generated by the directed graph $\mathcal{G}$ is the category with $\mathcal{G}^{\ast}_0=\{1_v : v \in \mathcal{G}_0 \}$,
where we have again identified identity arrows with objects of the category and the non-identity arrows, $\mathcal{G}_{1}^{\ast}$, is the set of all non-empty routes in $\mathcal{G}$ and composition of composable routes is by concatenation. We will view $\mathcal{G}_{1}$ as being a subset of $\mathcal{G}_{1}^{\ast}$ and we will identify $\mathcal{G}_{0}$ and $\mathcal{G}^{\ast}_0$.
Given an edge $x$ in a graph $\mathcal{G}$ we can consider the formal reversed edge $x^{-1}$ which has $\dom(x^{-1}) = \ran(x)$ and $\ran(x^{-1}) = \dom(x)$.
A \emph{path} in $\mathcal{G}$ consists of a sequence $x_{1}^{\epsilon_{1}} \ldots x_{n}^{\epsilon_{n}}$ where each $x_{i}$ is an edge, $\epsilon_{i}$ is either 1 or $-1$ and for each $i$ we have $\ran(x_{i+1}^{\epsilon_{i+1}}) = \dom(x_{i}^{\epsilon_{i}})$.
We will say a path is \emph{reduced} if it has no subpath of the form $xx^{-1}$ or $x^{-1}x$. Two paths will be considered equivalent if they can be reduced to the same path.
The \emph{free groupoid} $\mathcal{G}^{\dagger}$ generated by the directed graph $\mathcal{G}$ will have $\mathcal{G}^{\dagger}_0=\{1_v : v \in \mathcal{G}_0 \}$ and the non-identity arrows are all reduced paths in $\mathcal{G}$. Multiplication in $\mathcal{G}^{\dagger}$ will consist of concatenation of composable reduced paths plus reduction if possible. 
Observe that if a graph $\mathcal{G}$ has a single vertex and has edge set $\mathcal{G}_{1} = X$ then the free category on $\mathcal{G}$ is isomorphic to the free monoid $X^{\ast}$ on the set $X$ and the free groupoid on $\mathcal{G}$ is isomorphic to the free group $FG(X)$ on the set $X$. 

A \emph{category presentation} for a small category $C$ is written as follows
$$C = \langle \mathcal{G} | x_{i} = y_{i}, x_{i},y_{i}\in \mathcal{G}^{\ast}, i\in I \rangle\ ,$$
where $\mathcal{G}$ is a directed graph, $I$ is an index set, elements of $C$ are equivalence classes of elements of $\mathcal{G}^{\ast}$, $\dom(x_{i}) = \dom(y_{i})$ and $\ran(x_{i}) = \ran(y_{i})$ for each $i\in I$, and the relation $x_{i} = y_{i}$ tells us that every time we have a route $wx_{i}v$ in $\mathcal{G}$ then this is equivalent to the route $wy_{i}v$ and vice versa. %We implicitly assume that in categories described in this way that identity elements behave as identity elements. 

In a similar manner a \emph{groupoid presentation} for a groupoid $G$ is written as follows 
$$G = \langle \mathcal{G} | x_{i} = y_{i}, x_{i},y_{i}\in \mathcal{G}^{\dagger}, i\in I \rangle\ $$
where $\mathcal{G}$ is a directed graph, $I$ is an index set, elements of $G$ are now equivalence classes of elements of $\mathcal{G}^{\dagger}$, $\dom(x_{i}) = \dom(y_{i})$ and $\ran(x_{i}) = \ran(y_{i})$ for each $i\in I$, and the relation $x_{i} = y_{i}$ tells us that every time we have an element $wx_{i}v$ in $\mathcal{G}^{\dagger}$ then this equivalent to $wy_{i}v$ and vice versa. %We implicitly assume now that in groupoids described in this way that identity elements behave as identity elements and groupoid inverses behave as groupoid inverses. 

It is possible by being careful to give presentations of categories and groupoids where $\dom(x_{i})\neq \dom(y_{i})$ (see for example \cite{HigginsCat} or \cite{Moore} for details) but we will not often be considering this situation. In order to avoid confusion whenever both category and groupoid presentations are being used we may denote category presentations by $\langle | \rangle_{C}$ and groupoid presentations by $\langle | \rangle_{G}$. 

Given a small category $C$ there is a (unique up to isomorphism) groupoid $U(C)$ and functor $u:C\rightarrow U(C)$ such that if $f:C\rightarrow G$ is any functor from $C$ to a groupoid $G$ then there is a unique functor $g:U(C)\rightarrow G$ such that $gu = f$. We call the groupoid $U(C)$ the \emph{groupoid of fractions} of $C$ (\cite{GabrielZisman}). Some authors use the terminology \emph{universal groupoid} (and hence our usage of the notation $U(C)$), but this phrase is used to describe a slightly different construction in \cite{HigginsCat} and \cite{Moore}, and there is in addition Paterson's universal groupoid of an inverse semigroup, so to avoid confusion we will always call it the groupoid of fractions. Other authors use the term groupoid of fractions as a synonym for what \cite{GabrielZisman} calls a category of left fractions which is the situation where every element of the groupoid of fractions $U(C)$ has the form $x^{-1}y$ for some $x,y\in C$. In most of our examples what we are calling the groupoid of fractions is not a category of left fractions.

The following is a rephrased version of how to construct the groupoid of fractions found in \cite{GabrielZisman} in terms of our language of category presentations.

\begin{proposition}
\label{univgroupoid}
Let 
$$C = \langle \mathcal{G} | R \rangle_{C}$$
be a category given by category presentation and let
$$G = \langle \mathcal{G} | R \rangle_{G}$$
be the groupoid generated by the same generating graph and relations but such that we are working with a groupoid presentation.
Then $G$ is isomorphic to the groupoid of fractions $U(C)$ of $C$.
\end{proposition}

\begin{proof}
Let $\mathcal{H}$ be the same graph as $\mathcal{G}$ except with all edges reversed. We will identify the vertices of $\mathcal{G}$ and $\mathcal{H}$. The element $x\in G_{1}$ will have corresponding element $x^{-1}$ in $\mathcal{H}_{1}$ so that $\dom(x) = \ran(x^{-1})$ and $\dom(x^{-1}) = \ran(x)$. Let $\mathcal{M}$ be the union of the graphs $\mathcal{G}$ and $\mathcal{H}$. Observe that $G$ can be given in terms of a category presentation as:
$$G = \langle \mathcal{M} | R, S \rangle_{C},$$
where $S$ denotes the set of relations saying $xx^{-1}$ and $x^{-1}x$ are identities for each $x\in \mathcal{G}_{1}$.
We have a functor $u:C\rightarrow G$ given by $u(x) = x$.

Now let $f:C\rightarrow H$ be any functor from the category $C$ to a groupoid $H$. Then since $f$ is a functor we must be able to write $H$ in terms of the following category presentation:
$$H = \langle \mathcal{N} | R, S, T \rangle_{C},$$
where $\mathcal{M}$ is a subgraph of $\mathcal{N}$ and $T$ are any additional relations needed to define $H$ (for example identifying some of the edges of $\mathcal{M}$). We have assumed that $f$ will map $x\in \mathcal{G}_{1}$ to $x\in \mathcal{N}_{1}$.

We now define the functor $g:G\rightarrow H$ to be the one which maps elements of $\mathcal{M}_{1}$ to elements of $\mathcal{M}_{1}$ in $\mathcal{N}_{1}$. Observe that $gu = f$. To see that $g$ is unique, suppose that $h:G\rightarrow H$ is a functor such that $hu = f$. Then $h$ must agree with $g$ on elements of $\mathcal{G}_{1}$. Now let $x\in \mathcal{G}_{1}$ viewed as an element of $G$. Then
$$\dom(x) = h(x^{-1}x) = h(x^{-1})h(x) = h(x^{-1})g(x)$$
and so $h(x^{-1}) = (g(x))^{-1} = g(x^{-1})$. Thus $g = h$. 

Since universal groups are unique up to isomorphism, $G$ and $U(C)$ must be isomorphic as categories.
\begin{comment}
Let $f: M \rightarrow G$ be a homomorphism from $M$ to a group $G$. Since $f$ is a homomorphism, $G$ can be written in the form
$$G = \langle X, X^{\prime}, Y | R, S, T \rangle_{M}$$
where $X^{\prime}$ is a set of formal inverses of elements of $X$, the set $Y$ may contain some more elements of $G$, $R$ are the relations of $M$, $S$ is the set of relations saying $xx^{-1} = 1$ and $x^{-1}x = 1$ for all $x\in X$, where $x^{-1} \in X^{-1}$ is its formal inverse and $T$ are any additional relations (for example, $T$ will have relations of the form $x = y$, where $x,y\in X$ are such that $f(x) = f(y)$). We assume that $f$ maps elements of $X$ in $M$ to elements of $X$ in $G$. Now $U(M)$ can be given by monoid presentation as
$$U(M) = \langle X, X^{\prime} | R, S \rangle_{M}$$
where these sets are exactly as for $\Gamma$. Observe that there is a homomorphism $g:M\rightarrow U(M)$ mapping elements of $X$ in $M$ to elements of $X$ in $U(M)$ and a group homomorphism $h:G\rightarrow U(M)$ mapping mapping elements of $X$ in $U(M)$ to elements of $X$ in $G$ and elements of $X^{-1}$ in $U(M)$ to elements of $X^{-1}$ in $G$. Further $hg = f$. 

To see that $h$ is unique, suppose that $k:U(M)\rightarrow G$ is a homomorphism such that $kg = f$. Then $k$ must agree with $h$ on elements of $X$. Now let $x\in X$. Then
$$1 = k(x^{-1}x) = k(x^{-1})k(x) = k(x^{-1})h(x)$$
and so $k(x^{-1}) = (h(x))^{-1} = h(x^{-1})$. Thus $k = h$.
\end{comment}
\end{proof}

In particular, if $M$ is a monoid given by monoid presentation, then the group $G$ with the same presentation instead viewed as a group presentation will be the group of fractions of $M$.

Now suppose $G$ is a groupoid given by groupoid presentation $G = \langle \mathcal{G} | \mathcal{R}(G) \rangle\ $,
where here we are denoting the relations of $G$ by $\mathcal{R}(G)$ and suppose there is an index set $I$, subgroups $H_{i}:i\in I$ of $G$ and functors $\alpha_{i}:H_{i}\rightarrow G$. Let $e_{i},f_{i}\in \mathcal{G}_{0}$ be such that $H_{i}\subseteq e_{i}Ge_{i}$ and $K_{i} = \alpha_{i}(H_{i}) \subseteq f_{i}Gf_{i}$. Define $\mathcal{H}$ to be the graph with $\mathcal{H}_{0} = \mathcal{G}_{0}$ and 
$$\mathcal{H}_{1} = \mathcal{G}_{1} \cup \left\{t_{i}|i\in I\right\}$$
where $\ran(t_{i}) = e_{i}$ and $\dom(t_{i}) = f_{i}$. We will say that $\Gamma$ is a \emph{groupoid HNN-extension} of $G$ if $\Gamma$ is given by the groupoid presentation:
$$\Gamma = \langle \mathcal{H} | \mathcal{R}(G), xt_{i} = t_{i}\alpha_{i}(x) \forall x\in H_{i}, i\in I \rangle\ .$$
We call the arrows $t_{i}$ \emph{stable letters}.
Note that since $\alpha_{i}$ is injective it follows that $K_{i}$ is a subgroup of $G$ isomorphic to $H_{i}$. Groupoid HNN-extensions have previously been considered by Moore (\cite{Moore}) and Gilbert (\cite{GilbertHNN}). In the case of \cite{Moore}, $H$ is a wide subgroupoid of $G$ rather than being a subgroup and the situation where $H$ is an arbitrary subgroupoid of $G$ is considered in \cite{GilbertHNN}. In both cases, they define the HNN-extension as a pushout of a certain diagram of functors. It can be checked that their definition is equivalent to the one given here when $H$ is a subgroup of $G$. If $G$ is a group then $\Gamma$ is a group HNN-extension.

Let $G$ be a groupoid, $H$ a subgroup of $G$ with identity $e\in G_{0}$ and let 
$$K = \left\{g\in G| \dom(g) = e\right\}.$$
Then a \emph{transversal} $T$ of $H$ is a subset of $K$ such that
$$K = \coprod_{g\in T}{gH}.$$
Each set $gH$ will have cardinality equal to the cardinality of $H$. Furthermore, the cardinality of $T$ is independent of the choice of representatives so we define $|G:H| = |T|$. 

The following is a straightforward generalisation of Higgins' unique normal form theorem for fundamental groupoids (\cite{HigginsFund}), as stated without proof as Theorem 2.1.26 in \cite{Moore}.

\begin{proposition}
\label{normalformgroupoidHNN}
Let 
$$\Gamma = \langle \mathcal{H} | \mathcal{R}(G), xt_{i} = t_{i}\alpha_{i}(x) \forall x\in H_{i}, i\in I \rangle\ $$
be a groupoid HNN-extension of a groupoid $G$, for each subgroup $H_{i}$ let $T_{i}$ be a transversal of $H_{i}$ in $G$ and for each subgroup $K_{i} = \alpha(H_{i})$ let $T_{i}^{\prime}$ be a transversal of $K_{i}$ in $G$. Then each element $g$ of $\Gamma$ can be written uniquely in the form
$$g = g_{1}t_{i_{1}}^{\epsilon_{1}}g_{2}t_{i_{2}}^{\epsilon_{2}}\cdots g_{m}t_{i_{m}}^{\epsilon_{m}}u$$
where $\epsilon_{k}\in \left\{-1,1\right\}$, $g_{k}\in T_{i_{k}}$ if $\epsilon_{k} = 1$ and $g_{k}\in T_{i_{k}}^{\prime}$ if $\epsilon_{k} = -1$, $u\in G$ is arbitrary subject to the condition that the domains and ranges match up appropriately and if $t_{i_{k}} = t_{i_{k+1}}$ and $\epsilon_{k} + \epsilon_{k+1} = 0$ then $g_{k+1}$ is not an identity.
\end{proposition}

\begin{proof}
Let $g\in \Gamma$. Then $g$ can definitely be written in the form
$$g = s_{1}t_{i_{1}}^{\epsilon_{1}}s_{2}t_{i_{2}}^{\epsilon_{2}}\cdots s_{m}t_{i_{m}}^{\epsilon_{m}}u,$$
where $s_{k},u\in G$ are arbitrary but such that all the domains and ranges match up correctly. If $\epsilon_{1} = 1$ then we can write $s_{1}$ uniquely in the form
$$s_{1} = g_{1}h_{1}$$
where $g_{1}\in T_{i_{1}}$ and $h_{1}\in H_{i_{1}}$. We can then rewrite $g$ as
$$g = g_{1}t_{i_{1}}\alpha_{i_{1}}(h_{1})s_{2}t_{i_{2}}^{\epsilon_{2}}\cdots s_{m}t_{i_{m}}^{\epsilon_{m}}u.$$
If $\epsilon_{1} = -1$ then we can write $s_{1}$ uniquely in the form
$$s_{1} = g_{1}h_{1}$$
where $g_{1}\in T_{i_{1}}^{\prime}$ and $h_{1}\in K_{i_{1}}$. We can then rewrite $g$ as
$$g = g_{1}t_{i_{1}}^{-1}\alpha_{i_{1}}^{-1}(h_{1})s_{2}t_{i_{2}}^{\epsilon_{2}}\cdots s_{m}t_{i_{m}}^{\epsilon_{m}}u.$$
We then continue along in a similar manner, by rewriting $\alpha_{i_{k}}(h_{k})s_{k+1} = g_{k+1}h_{k+1}$ where $h_{k+1}\in T_{i_{k+1}}$ if $\epsilon_{k+1} = 1$ and $h_{k+1}\in T_{i_{k+1}}^{\prime}$ if $\epsilon_{k+1} = -1$, and then moving the $h_{k+1}$ beyond $t_{i_{k+1}}$ by applying $\alpha_{i_{k+1}}$ or its inverse, while at the same time cancelling any pair $t_{k}t_{k}^{-1}$ or $t_{k}^{-1}t_{k}$.
We now wish to prove that these normal forms are unique normal forms. We will do this using an Artin-van der Waerden type argument. Let us denote by $X$ the set of normal form words (where words which are equal in $\Gamma$ are not equated) and let $X_{a}$ be the set of normal forms $w\in X$ with $\ran(w) = a$. We can define a groupoid $B$ as follows. $B_{0}$ will just be equal to $\Gamma_{0}$ (and therefore also to $G_{0}$).
Elements of $B_{1}$ will be bijections $\pi:X_{a}\rightarrow X_{b}$, where we define $\dom(\pi) = a$ and $\ran(\pi) = b$ in $B$. It is readily verified that this gives $B$ the structure of a groupoid. We will define a functor $\Gamma \rightarrow B$. For $g\in G$ with $\dom(g) = a$, $\ran(g)=b$ let us define $\pi_{g}:X_{a}\rightarrow X_{b}$ to be the map with
$$\pi_{g}(g_{1}t_{i_{1}}^{\epsilon_{1}}\cdots g_{m}t_{i_{m}}^{\epsilon_{m}}u) 
= g_{1}^{\prime}t_{j_{1}}^{\delta_{1}}\cdots g_{n}t_{j_{n}}^{\delta_{n}}v $$
where what we have done is premultiplied $g_{1}t_{i_{1}}^{\epsilon_{1}}g_{2}t_{i_{2}}^{\epsilon_{2}}\cdots g_{m}t_{i_{m}}^{\epsilon_{m}}u$ by $g$ and then rewritten this in normal form using the algorithm described above. Observe that $\pi_{gh} = \pi_{g}\pi_{h}$ for $g,h\in G$ with $\dom(g)=\ran(h)$. In particular, $\pi_{gg^{-1}} = \pi_{\ran(g)}$ so that $\pi_{g}\in B$ for each $g\in G$. Define $\pi_{t_{k}}:X_{f_{k}}\rightarrow X_{e_{k}}$ as follows. We define
$$\pi_{t_{k}}(t_{k}^{-1}g_{1}t_{i_{1}}^{\epsilon_{1}}\cdots g_{m}t_{i_{m}}^{\epsilon_{m}}u) = g_{1}t_{i_{1}}^{\epsilon_{1}}\cdots g_{m}t_{i_{m}}^{\epsilon_{m}}u.$$
Otherwise we define
$$\pi_{t_{k}}(g_{1}t_{i_{1}}^{\epsilon_{1}}\cdots g_{m}t_{i_{m}}^{\epsilon_{m}}u) = t_{k}g_{1}t_{i_{1}}^{\epsilon_{1}}\cdots g_{m}t_{i_{m}}^{\epsilon_{m}}u.$$
In a similar manner we define $\pi_{t_{k}^{-1}}:X_{e_{k}}\rightarrow X_{f_{k}}$. We define
$$\pi_{t_{k}^{-1}}(t_{k}g_{1}t_{i_{1}}^{\epsilon_{1}}\cdots g_{m}t_{i_{m}}^{\epsilon_{m}}u) = g_{1}t_{i_{1}}^{\epsilon_{1}}\cdots g_{m}t_{i_{m}}^{\epsilon_{m}}u.$$
Otherwise we define
$$\pi_{t_{k}^{-1}}(g_{1}t_{i_{1}}^{\epsilon_{1}}\cdots g_{m}t_{i_{m}}^{\epsilon_{m}}u) = t_{k}^{-1}g_{1}t_{i_{1}}^{\epsilon_{1}}\cdots g_{m}t_{i_{m}}^{\epsilon_{m}}u.$$
Observe that $\pi_{t_{k}}\pi_{t_{k}^{-1}} = \pi_{t_{k}t_{k}^{-1}} = \pi_{e_{k}}$ and $\pi_{t_{k}^{-1}}\pi_{t_{k}} = \pi_{t_{k}^{-1}t_{k}} = \pi_{f_{k}}$ so that $\pi_{t_{k}},\pi_{t_{k}^{-1}}\in B_{1}$ for each $k\in I$.
We will now check that $\pi_{h}\pi_{t_{k}} = \pi_{t_{k}}\pi_{\alpha_{k}(h)}$ for every $h\in H_{k}$. We have two cases. First,
\begin{eqnarray*}
\pi_{h}(\pi_{t_{k}}(t_{k}^{-1}g_{1}t_{i_{1}}^{\epsilon_{1}}\cdots g_{m}t_{i_{m}}^{\epsilon_{m}}u)) &=& \pi_{h}(g_{1}t_{i_{1}}^{\epsilon_{1}}\cdots g_{m}t_{i_{m}}^{\epsilon_{m}}u) \\
&=& g_{1}^{\prime}t_{j_{1}}^{\delta_{1}}\cdots g_{n}t_{j_{n}}^{\delta_{n}}v,
\end{eqnarray*}
where $g_{1}^{\prime}t_{k_{1}}^{\delta_{1}}\cdots g_{n}t_{k_{n}}^{\delta_{n}}v$ is $hg_{1}t_{i_{1}}^{\epsilon_{1}}\cdots g_{m}t_{i_{m}}^{\epsilon_{m}}u$ reduced using the algorithm described above.
On the other hand, noting that $\alpha_{k}(h)t_{k}^{-1} = t_{k}^{-1}h$,
\begin{eqnarray*}
\pi_{t_{k}}(\pi_{\alpha_{k}(h)}(t_{k}^{-1}g_{1}t_{i_{1}}^{\epsilon_{1}}\cdots g_{m}t_{i_{m}}^{\epsilon_{m}}u)) 
&=& \pi_{t_{k}}(t_{k}^{-1}g_{1}^{\prime}t_{k_{1}}^{\delta_{1}}\cdots g_{n}t_{k_{n}}^{\delta_{n}}v)\\
&=& g_{1}^{\prime}t_{j_{1}}^{\delta_{1}}\cdots g_{n}t_{j_{n}}^{\delta_{n}}v.
\end{eqnarray*}
Now the second case:
\begin{eqnarray*}
\pi_{h}(\pi_{t_{k}}(g_{1}t_{i_{1}}^{\epsilon_{1}}\cdots g_{m}t_{i_{m}}^{\epsilon_{m}}u)) &=& \pi_{h}(t_{k}g_{1}t_{i_{1}}^{\epsilon_{1}}\cdots g_{m}t_{i_{m}}^{\epsilon_{m}}u) \\
&=& t_{k} \pi_{\rho_{k}(h)}(g_{1}t_{i_{1}}^{\epsilon_{1}}\cdots g_{m}t_{i_{m}}^{\epsilon_{m}}u) \\
&=& \pi_{t_{k}}(\pi_{\rho_{k}(h)}(g_{1}t_{i_{1}}^{\epsilon_{1}}\cdots g_{m}t_{i_{m}}^{\epsilon_{m}}u)).
\end{eqnarray*}
Thus $\pi_{h}\pi_{t_{k}} = \pi_{t_{k}}\pi_{\alpha_{k}(h)}$ for every $h\in H_{k}$. It follows that the map $\pi:\Gamma\rightarrow B$ defined by
$$\pi(s_{1}t_{i_{1}}^{\epsilon_{1}}\cdots s_{m}t_{i_{m}}^{\epsilon_{m}}u)
= \pi_{s_{1}}\pi_{t_{i_{1}}^{\epsilon_{1}}}\cdots \pi_{s_{m}}\pi_{t_{i_{m}}^{\epsilon_{m}}}\pi_{u}$$
is a functor. Finally, to see that the normal forms are unique note that if 
$$g_{1}t_{i_{1}}^{\epsilon_{1}}\cdots g_{m}t_{i_{m}}^{\epsilon_{m}}u, g_{1}^{\prime}t_{j_{1}}^{\delta_{1}}\cdots g_{n}t_{j_{n}}^{\delta_{n}}v\in \Gamma$$
are elements written in normal form both with domain $e\in \Gamma_{0}$ then
$$\pi(g_{1}t_{i_{1}}^{\epsilon_{1}}\cdots g_{m}t_{i_{m}}^{\epsilon_{m}}u)(e) = g_{1}t_{i_{1}}^{\epsilon_{1}}\cdots g_{m}t_{i_{m}}^{\epsilon_{m}}u$$
while
$$\pi(g_{1}^{\prime}t_{j_{1}}^{\delta_{1}}\cdots g_{n}t_{j_{n}}^{\delta_{n}}v)(e) = g_{1}^{\prime}t_{j_{1}}^{\delta_{1}}\cdots g_{n}t_{j_{n}}^{\delta_{n}}v$$
and thus they are mapped to different elements of $B$, so must be distinct in $\Gamma$.
\end{proof} 

\section{A brief foray into Bass-Serre theory}

We will now give a brief outline of some aspects of Bass-Serre theory. Our definition of graph of groups is taken from \cite{Moore}, except that we do not assume that the underlying graph is connected. The definition of \cite{Serre} is equivalent. For us a graph of groups $\mathcal{G}_{G}$ will consist of:
\begin{itemize}
\item A graph $\mathcal{G}$.
\item An involution $t\mapsto \bar{t}$ on the edges of $\mathcal{G}$.
\item A group $G_{a}$ for each vertex $a\in \mathcal{G}_{0}$.
\item A subgroup $G_{t}\leq G_{\ran(t)}$ for each edge $t\in \mathcal{G}_{0}$.
\item An isomorphism $\phi_{t}:G_{t}\rightarrow G_{\bar{t}}$ for each edge $t\in \mathcal{G}_{1}$ such that $\phi_{\bar{t}} = \phi_{t}^{-1}$.
\end{itemize}

A \emph{path} in $\mathcal{G}_{G}$ consists of a sequence $g_{1}t_{1}g_{2}t_{2}\cdots g_{m}t_{m}g_{m+1}$ where $t_{k}\in \mathcal{G}_{1}$ for each $k$, $g_{k}\in G_{\ran(t_{k})}$ for $k=1,\ldots,m$ and $g_{k+1}\in G_{\dom(t_{k})}$ for $k = 1,\ldots,m$. We allow for the case $m = 0$, i.e. paths of the form $g\in G_{a}$ for some $a\in\mathcal{G}_{0}$. We write $\dom(g_{1}t_{1}g_{2}t_{2}\cdots g_{m}t_{m}g_{m+1}) = \dom(t_{m})$ and $\ran(g_{1}t_{1}g_{2}t_{2}\cdots g_{m}t_{m}g_{m+1}) = \ran(t_{1})$. For $g\in G_{a}$ viewed as a path we write $\dom(g) = \ran(g) = a$. Let $\sim$ be the equivalence relation on paths in $\mathcal{G}_{G}$ generated by $p\bar{t}htq \sim p\phi_{t}(h)q$, where $p,q$ are paths and $h\in G_{t}$. We say that $p\phi_{t}(h)q$ is a \emph{reduction} of $p\bar{t}htq$. It can be shown that every path reduces to a unique fully reduced path.

Given a graph of groups $\mathcal{G}_{G}$, we define its \emph{fundamental groupoid} $\Gamma(\mathcal{G}_{G})$ (\cite{HigginsFund}) to be the groupoid whose arrows correspond to equivalence classes of $\sim$. Composition of arrows is simply concatenation of composable paths multiplying group elements at each end. The fundamental groupoid of a graph of groups is precisely a groupoid HNN-extension of a totally disconnected groupoid. 

To see this, suppose 
$$\Gamma = \langle \mathcal{H} | \mathcal{R}(G), ht_{i} = t_{i}\alpha_{i}(h) \forall h\in H_{i}, i\in I \rangle\ $$ 
is a groupoid HNN-extension of a totally disconnected groupoid $G$. Then the associated graph of groups $\mathcal{G}_{G}$ will have vertices corresponding to the identities of $\Gamma$. The group at the vertex corresponding to the identity $a\in G_{0}$ will be the local monoid $aGa$. The edges of $\mathcal{G}_{G}$ will be the generating elements $t_{i}$ and their inverses. The involution in the graph will map $t_{i}$ to $t_{i}^{-1}$ and $t_{i}^{-1}$ to $t_{i}$. The groups $H_{t_{i}}$ associated with the edges $t_{i}$ will be the groups $H_{i}$ and the groups $H_{t_{i}^{-1}}$ associated to the edges $t_{i}^{-1}$ will be the groups $\alpha_{i}(H_{i})$. We define $\phi_{t_{i}} = \alpha_{i}$ and $\phi_{t_{i}^{-1}} = \alpha_{i}^{-1}$. We then see that the fundamental groupoid of $\mathcal{G}_{G}$ will be isomorphic to $\Gamma$.

On the other hand, suppose $\mathcal{G}_{G}$ is a graph of groups. We let $G$ be the disjoint union of all the vertex groups of $\mathcal{G}_{G}$ viewed as a totally disconnected groupoid with identities corresponding to the vertices of $\mathcal{G}_{G}$. For each pair $\left\{t,\bar{t}\right\}$ where $t$ is an edge in $\mathcal{G}_{G}$ we pick one edge; these edges will be our arrows $t_{i}$. We define $H_{i} = G_{t_{i}}$ and let $\alpha_{i} = \phi_{t_{i}}$. Then it is easy to see that the groupoid HNN-extension $\Gamma$ of $G$ with respect to the subgroups $H_{i}$, stable letters $t_{i}$ and monomorphisms $\alpha_{i}$ will be isomorphic to the fundamental groupoid of $\mathcal{G}_{G}$. 

Let $\Gamma(\mathcal{G}_{G})$ be the fundamental groupoid of a graph of groups $\mathcal{G}_{G}$ and for each edge $t\in \mathcal{G}_{1}$ let $T_{t}$ be a transversal of the left cosets of $H_{t}$ in $G_{\ran(t)}$. Using the normal form result Proposition \ref{normalformgroupoidHNN} we see that each element of $\Gamma(\mathcal{G}_{G})$ can be written uniquely in the form
$$g_{1}t_{1}g_{2}t_{2}\cdots g_{m}t_{m}u$$
where $g_{1}t_{1}g_{2}t_{2}\cdots g_{m}t_{m}u$ is a path in $\mathcal{G}_{G}$, $g_{i}\in T_{t_{i}}$ for $i=1,\ldots,m$ and $u\in G_{\dom(t_{m})}$ is arbitrary, subject to the condition that if $g_{i} = \dom(t_{i-1}) = \ran(t_{i})$ then $\overline{t_{i-1}}\neq t_{i}$.

If $\mathcal{G}_{G}$ is a graph of groups and $a$ is a vertex in $\mathcal{G}$ (which we have identified with the identity element of $G_{a}$) then the \emph{fundamental group of $\mathcal{G}_{G}$ at $a$}, denoted $\pi_{1}(\mathcal{G}_{G},a)$, is $a\Gamma(\mathcal{G}_{G})a$, the local group at $a$, i.e. all paths in $\mathcal{G}_{G}$ which start and end at $a$. Fundamental groups with respect to vertices in the same connected component of $\mathcal{G}$ will be isomorphic. If $\mathcal{G}_{G}$ has a single vertex $a$ then $\pi_{1}(\mathcal{G}_{G},a) = \Gamma(\mathcal{G}_{G})$ will be a group HNN-extension, and every group HNN-extension is the fundamental group of a graph of groups with a single vertex. 

We have seen that given a graph of groups $\mathcal{G}_{G}$ we can construct its fundamental groupoid $\Gamma(\mathcal{G}_{G})$ and the fundamental groups $\pi_{1}(\mathcal{G}_{G},a)$. It will now be shown how the groups $\pi_{1}(\mathcal{G}_{G},a)$ have natural actions on trees.   

Let $\mathcal{G}_{G}$ be a graph of groups, let $a$ be a vertex of $\mathcal{G}_{G}$ and let $P_{a}$ denote the set of paths in $\mathcal{G}_{G}$ with range $a$. For $p,q\in P_{a}$ we will write $p \approx q$ if $\dom(p) = \dom(q)$ and $p \sim qg$ for some $g\in G_{\dom(p)}$. This defines an equivalence relation on $P_{a}$. We will denote the $\approx$-equivalence class containing the path $p$ by $[p]$. We now define the (undirected) Bass-Serre tree $T$ with respect to the vertex $a$ as follows. The vertices of $T$ are $\approx$-equivalence classes of paths in $P_{a}$. Two vertices $[p],[q]\in T_{0}$ are connected by an edge if there are $g\in G_{\dom(p)}$ and $t\in \mathcal{G}_{1}$ such that 
$$q \approx pgt.$$
It can be verified that $T$ is indeed a tree. We will now define an action of $\pi_{1}(\mathcal{G}_{G},a)$ on $T_{0}$ by
$$g\cdot [p] = [gp].$$
This will then naturally extend to an action of $\pi_{1}(\mathcal{G}_{G},a)$ on $T$.

Let us now consider these ideas from the point of view of groupoid HNN-extensions. By definition two paths $p,q$ in $\mathcal{G}_{G}$ are $\sim$-related if they correspond to the same elements of $\Gamma(\mathcal{G}_{G})$. So let $\Gamma$ be an arbitrary groupoid HNN-extension of a totally disconnected groupoid $G$, let $a\in \Gamma_{0} = G_{0}$ and let 
$$P_{a} = \left\{g\in \Gamma| \ran(g) = a\right\}.$$
For $p,q\in P_{a}$, we define $p \approx q$ if $p = qg$ for some $g\in G$. This defines an equivalence relation on $P_{a}$ and we denote the $\approx$-equivalence class containing $p$ by $[p]$. We now define an undirected tree $T$ with respect to the identity $a$ as follows. The vertices of $T$ will correspond to $\approx$-equivalence classes of elements of $P_{a}$. Two vertices $[p],[q]\in T_{0}$ are connected by an edge if 
$$q = pgt_{i}^{\epsilon}h$$
for some $g,h\in G$, $i\in I$, $\epsilon\in \left\{-1,1\right\}$. We then have an action of $a\Gamma a$ on $T_{0}$ given by 
$$g \cdot [p] = [gp]$$
which naturally extends to an action of $a\Gamma a$ on $T$. We can in fact make the tree $T$ directed by specifying that if $[p],[q]\in T_{0}$ then there is an edge $s\in T_{1}$ with $\ran(s) = [p]$ and $\dom(s) = [q]$ if $q = pgt_{i}h$ for some $g,h\in G$, $i\in I$. It is then clear that this construction works for an arbitrary groupoid HNN-extension, so we do not require that $G$ is totally disconnected. 

\begin{comment}
It should be clear that our requirement that $G$ was totally disconnected does not appear in most of the above ideas, so that all of the above works for an arbitrary groupoid $\Gamma$. The only non-trivial statement is that $T$ is a tree. NEED TO CHECK THIS.
\end{comment}

\newpage
\section{Topological and algebraic $K$-theory}

Let us begin by recalling the definition of the Grothendieck group of a commutative semigroup. If $S$ is a commutative semigroup then there is a unique (up to isomorphism) commutative group $G = \mathcal{G}(S)$, called the \emph{Grothendieck group} of $S$, and a homomorphism $\phi:S\rightarrow G$, such that for any commutative group $H$ and homomorphism $\psi:S\rightarrow H$, there is a unique homomorphism $\theta:G\rightarrow H$ with $\psi = \theta \circ \phi$. In fact $\mathcal{G}$ is really a functor from commutative semigroups to abelian groups. It is easy to check that the Grothendieck group of a commutative semigroup is precisely its group of fractions.

Let us now briefly outline topological and algebraic $K$-theory in order to motivate the theory of Chapter 4. Our treatment follows that of \cite{Rosenberg}. Suppose $X$ is a compact Hausdorff topological space (it is possible to extend the definition of $K_{0}$-group to locally-compact spaces, but we will leave that aside for the moment). Let $\mathbb{F}$ be either $\mathbb{R}$ or $\mathbb{C}$. An $\mathbb{F}$-\emph{vector bundle} consists of a topological space $E$ and a continuous open surjective map $p:E\rightarrow X$, with extra structure defined by the following:
\begin{itemize}
\item Each fibre $p^{-1}(x)$ of $p$ for $x\in X$ is a finite-dimensional vector space over $\mathbb{F}$.
\item There are continuous maps $E\times E\rightarrow E$ and $\mathbb{F}\times E\rightarrow E$ which restrict to vector addition and scalar multiplication on each fibre.
\end{itemize}
We will denote such a vector bundle by $E\stackrel{p}{\rightarrow} X$ or by $(E,p)$. One can consider the category $\Vect_{X}^{\mathbb{F}}$ of all $\mathbb{F}$-vector bundles over $X$. The morphisms in this category are continuous maps $f:(E,p)\rightarrow (F,q)$ such that they are linear on each fibre and such that $qf = p$. The category has a binary operation $\oplus$ called \emph{Whitney sum} defined on objects $(E,p),(F,q)$ by
$$E\oplus F = \left\{(x,y)\in E\times F|p(x) = q(y)\right\}$$
with $p\oplus q:E\oplus F\rightarrow X$ given by $(p\oplus q)(x,y) = p(x) = q(y)$.

For a space $X$ and $n\in \mathbb{N}$ the \emph{trivial vector bundle of rank} $n$ is $(X\times \mathbb{F}^{n},\pi_{n})$ where $\pi_{n}:X\times \mathbb{F}^{n}\rightarrow X$ is given by $\pi_{n}(x,z) = z$. A \emph{locally trivial} $\mathbb{F}$-vector bundle is a vector bundle $(E,p)$ such that for each $x\in X$ there is an open set $U$ containing $x$ and vector bundle isomorphism from $p^{-1}(U)\stackrel{p|_{p^{-1}(U)}}{\longrightarrow} U$ to a trivial bundle of some rank over $U$. The \emph{rank} of such a bundle $(E,p)$ is then a continuous function $\rank_{E}:X\rightarrow \mathbb{N}$ given by $\rank_{E}(x) = \dim(p^{-1}(x))$. 

Let us denote the set of locally trivial $\mathbb{F}$-vector bundles over $X$ by $V_{\mathbb{F}}(X)$. Then $(V_{\mathbb{F}}(X),\oplus)$ is a commutative monoid with identity the trivial vector bundle of rank $0$. We define
$$K^{0}_{\mathbb{F}}(X) = \mathcal{G}(V_{\mathbb{F}}(X)).$$
We will only be concerned with complex topological $K$-theory in this thesis so we write $K^{0}(X) = K^{0}_{\mathbb{C}}(X)$.

There is an alternative way of computing the $K_{0}$-group of a compact Hausdorff space $X$. Let $C(X)$ be the set of complex-valued continuous functions on $X$. $C(X)$ has the structure of a commutative ring under pointwise addition and multiplication (in fact it can be given the structure of a $C^{\ast}$-algebra). Let $\Gamma_{X}$ be the set of finitely generated projective modules of $C(X)$. Then $(\Gamma_{X},\oplus)$ is a commutative monoid. In fact, we have the following theorem:

\begin{thm}
(Serre-Swan)
There is a monoid isomorphism $\phi:V_{\mathbb{C}}(X)\rightarrow \Gamma_{X}$.
\end{thm}

It then follows that $K^{0}(X) \cong \mathcal{G}(\Gamma_{X})$. This then leads to the definition of algebraic $K$-theory. If we let $\Proje(R)$ denote the set of finitely generated projective modules of a ring $R$ then we define $K_{0}(R) = \mathcal{G}(\Proje(R))$. Viewing $C(X)$ as a $C^{\ast}$-algebra we can give another definition of $K^{0}(X)$ in terms of this structure, and when generalised this gives operator $K$-theory. 

It is possible give an alternative description of algebraic $K$-theory. Let $M_{n}(R)$ be the set of $n\times n$ matrices over $R$ and let $M(R)$ denote the set of $\mathbb{N}$ by $\mathbb{N}$ matrices over $R$ with finitely many non-zero entries. One can think of $M(R)$ as being the union of all the $M_{n}(R)$. Given an idempotent matrix $E\in M(R)$, viewed as a homomorphism $R^{n}\rightarrow R^{n}$, the image of $E$ is a projective $R$-module.
On the other hand if $P$ is a projective module, there is an idempotent matrix $E$ with image $P$. We will say idempotent matrices $E, F\in M_{n}(R)$ are \emph{similar}, and write $E\sim F$, if $E = XY$ and $F = YX$ for some matrices $X,Y\in M(R)$. This will define an equivalence relation on the set of idempotent matrices $\Idem(R)$. We have the following proposition:
\begin{proposition}
Idempotent matrices $E,F\in \Idem(R)$ define the same projective module if and only if $E\sim F$.
\end{proposition}

Denote the set of idempotent matrices by $\Idem(R)$ and define a binary operation on $\Idem(R)/\sim$ by
$$[E]+[F] = [E^{\prime} + F^{\prime}],$$
where if a row in $E^{\prime}$ has non-zero entries then that row in $F^{\prime}$ has entries only zeros, similarly for columns of $E^{\prime}$, and for rows and columns of $F^{\prime}$, and such that $E^{\prime}\sim E$ and $F^{\prime}\sim F$. We then have the following result:
\begin{proposition}
This is a well-defined operation and the monoids $\Idem(R)/\sim$ and $\Proj_{R}$ are isomorphic.
\end{proposition}
This then gives us an alternative way of viewing $K_{0}(R)$. We have
$$K_{0}(R) = \mathcal{G}(\Idem(R)/\sim).$$

\chapter{Left Rees Monoids}

\section{Outline of chapter}

The aim of this chapter is to study left Rees monoids in detail. We will consider the correspondence found in \cite{LawLRM1} between left Rees monoids and self-similar group actions in Section 2.2. In Section 2.3 we will see that left Rees monoids and monoid HNN-extensions of groups are one and the same thing. We will then use this to investigate the structure of left Rees monoids in more detail. In Section 2.4 we will show that the group of fractions of a symmetric Rees monoid is a Zappa-Sz\'{e}p product of groups. It will also be shown that every Rees monoid with finite group of units is in fact a symmetric Rees monoid. From this we deduce that a group HNN-extension of a finite group $G$ is isomorphic to a Zappa-Sz\'{e}p product of a free group and the group $G$. Sections 2.5 and 2.6 are devoted to the study of Rees monoids arising from fractals. In Section 2.7 we will look at examples of left Rees monoids described in terms of automata. Finally, in Section 2.8, we will briefly explore the representation theory of left Rees monoids.

%Section~2
\section{The correspondence}\setcounter{theorem}{0}

All unproved assertions in this section are proved in \cite{LawLRM1}. Recall from the introduction that a monoid $M$ will be called a {\em left Rees monoid} if it satisfies the following conditions:
\begin{description}
\item[{\rm (LR1)}] $M$ is a left cancellative monoid.

\item[{\rm (LR2)}] $M$ is right rigid: incomparable principal right ideals are disjoint.

\item[{\rm (LR3)}] Each principal right ideal is properly contained in only a finite number of principal
right ideals.
\end{description}

We shall always assume that left Rees monoids are not groups.
We define {\em right Rees monoids} dually.
Every left Rees monoid $M$ admits a surjective homomorphism $\lambda \colon \: M \rightarrow \mathbb{N}$ such that $\lambda^{-1}(0) = G(M)$,
the group of units of $M$. Any such homomorphism we call a {\em length function}.
Such functions can always be chosen so that their value on generators of maximal proper principal right ideals is one.
Left Rees monoids with trivial groups of units are precisely the free monoids,
and so our monoids are natural generalisations of free monoids.
It is worth recalling here that a free monoid $X^{\ast}$ on a set $X$ consists of
all finite sequences of elements of $X$ called {\em strings}, 
including the empty string $\varepsilon$, which we often denote by 1, 
with multiplication given by {\em concatenation}
of strings. The length $|x|$ of a string $x$ is the total number of elements of $X$ that occur in it.
If $x = yz$ then $y$ is called a {\em prefix} of $x$.
A left Rees monoid which is cancellative is automatically a right Rees monoid,
and a monoid which is both a left Rees monoid and a right Rees monoid is called a {\em Rees monoid}.

%%%%%%%%%%%%%%%%%%%%%%%%%%%%%%%%%%%%%%%%%%%%%%%%%%%%%%%%%%%%%%%%%%%%%%%%%%%%%%%%%%%%%%%%%%%%%%%%%%%%%%%%%%%%%%%%%%%%%%%%%%%%%%%%%%%%%%%%%%%%%5

We will now describe the construction of the Zappa-Sz\'ep product of two monoids. These were first considered by Zappa (\cite{Zappa}) for groups and then later developed in a series of papers by Sz\'ep, beginning with \cite{Szep}. Kunze then considered the setup for two semigroups (\cite{Kunze}); in this situation the lack of identities means one only uses the first 4 of the axioms listed below. Our treatment follows that of Wazzan's PhD thesis (\cite{Wazzan}).

We will say two monoids $A$ and $S$ form a \emph{matched pair} if there are two maps $A \times S \rightarrow S$ denoted $(a,s)\mapsto a\cot s$ and $A \times S \rightarrow A$ denoted $(a,s)\mapsto a|_{s}$ satisfying the following eight axioms, for $a,b\in A$, $s,t\in S$ and $1_{A}$, $1_{S}$ denoting the identities, respectively, of $A$ and $S$:
\begin{description}
\item[{\rm (ZS1)}] $(ab) \cdot s = s$.
\item[{\rm (ZS2)}] $a \cdot (st) = (a \cdot s)(a|_{s} \cdot t)$.
\item[{\rm (ZS3)}] $a|_{st} = (a|_{s})|_{t}$.
\item[{\rm (ZS4)}] $(ab)|_{s} = a|_{b \cdot s} b|_{s}$.
\item[{\rm (ZS5)}] $a \cdot 1_{S} = 1_{S}$.
\item[{\rm (ZS6)}] $a|_{1_{S}} = a$.
\item[{\rm (ZS7)}] $1_{A}\cdot s = s$.
\item[{\rm (ZS8)}] $1_{A}|_{s} = 1|_{A}$.
\end{description}

Given a matched pair $(A,S)$ denote by $S \bowtie A$ the Cartesian product of $S$ and $A$ endowed with the following binary operation:
$$(s,a)(t,b) = (s (a\cdot t), (a|_{t})b).$$
We call this the \emph{Zappa-Sz\'{e}p product} of $S$ and $A$. One can check that $S \bowtie A$ is in fact a monoid with identity $(1_{S},1_{A})$, the sets 
$$S^{\prime} = \left\{(s,1_{A})|s\in S\right\}$$
and
$$A^{\prime} = \left\{(1_{S},a)|a\in A\right\}$$
are isomorphic, respectively, to $S$ and $A$ as monoids and that $S \bowtie A = S^{\prime}A^{\prime}$ uniquely.

On the other hand, if $M$ is a monoid and $S$, $A$ are submonoids of $M$ such that $M = SA$ uniquely then one can define maps $A \times S \rightarrow S$ and $A \times S \rightarrow A$ by 
$$as = (a\cdot s)(a|_{s})$$
and one can check that these maps will satisfy (ZS1) - (ZS8).

Thus we have the following, originally proved in \cite{Kunze}:

\begin{thm}
\label{sufconlrm}
Let $M$ be a monoid and let $A$, $S$ be submonoids of $M$. Then $M = SA$ uniquely if and only if there are maps $A \times S \rightarrow S$ and $A \times S \rightarrow A$ satisfying (ZS1) - (ZS8) such that $M \cong S\bowtie A$.
\end{thm}

We will be interested in a particular case of Zappa-Sz\'{e}p products where $A$ is a group, now denoted $G$, and $S$ is the free monoid on a set $X$. We will identify the identities of $G$ and $X^{\ast}$ and we will now relabel the axioms as follows for this special case:
\begin{comment}
Let $G$ be a group, $X$ a set, $G \times X^{\ast} \rightarrow X^{\ast}$ an operation, called the {\em action}, denoted by
$(g,x) \mapsto g \cdot x$, and $G \times X^{\ast} \rightarrow G$ an operation, called the {\em restriction}, denoted by $(g,x) \mapsto g|_{x}$,
such that the following eight axioms hold:
\end{comment}
\begin{description}
\item[{\rm (SS1)}] $1 \cdot x = x$.
\item[{\rm (SS2)}] $(gh) \cdot x = g \cdot (h \cdot x)$.
\item[{\rm (SS3)}] $g \cdot 1 = 1$.
\item[{\rm (SS4)}] $g \cdot (xy) = (g \cdot x)(g|_{x} \cdot y)$.
\item[{\rm (SS5)}] $g|_{1} = g$.
\item[{\rm (SS6)}] $g|_{xy} = (g|_{x})|_{y}$.
\item[{\rm (SS7)}] $1|_{x} = 1$.
\item[{\rm (SS8)}] $(gh)|_{x} = g|_{h \cdot x} h|_{x}$.
\end{description}
We will then say that there is a {\em self-similar action of the group $G$ on the free monoid $X^{\ast}$}.
When we refer to a `self-similar group action $(G,X)$', we shall assume that the action and restriction have been chosen and are fixed.
It is easy to show that such an action is {\em length-preserving}, in the sense that $|g \cdot x| = |x|$ for all $x \in X^{\ast}$,
and {\em prefix-preserving}, in the sense that $x = yz$ implies that $g \cdot x = (g \cdot y)z'$ for some string $z'$.

The following was proved in \cite{LawLRM1}.

\begin{lemma} Let $(G,X)$ be a self-similar group action.
\label{inverserest}
\begin{description}
\item[{\rm (i)}]  $(g|_{x})^{-1} = g^{-1}|_{g \cdot x}$ for all $x \in X^{\ast}$ and $g \in G$.
\item[{\rm (ii)}] $(g^{-1}|_{x})^{-1} = g|_{g^{-1} \cdot x}$ for all $x \in X^{\ast}$ and $g \in G$. \bbox \\
\end{description}
\end{lemma}

If $x \in G$ then $G_{x}$ is the stabiliser of $x$ in $G$ with respect to the action and so a subgroup of $G$.
The following lemma will play a useful role in what follows.

\begin{lemma} 
\label{usefullemma1}
Let $(G,X)$ be a self-similar group action.
\begin{description}
\item[{\rm (i)}] The function $\phi_{x} \colon \: G_{x} \rightarrow G$ given by $g \mapsto g|_{x}$ is a homomorphism.
\item[{\rm (ii)}] Let $y = g \cdot x$. 
Then $G_{y} = gG_{x}g^{-1}$
and
$$\phi_{y}(h) = g|_{x}\phi_{x}(g^{-1}hg) (g|_{x})^{-1}.$$
\item[{\rm (iii)}] If $\phi_{x}$ is injective then $\phi_{g \cdot x}$ is injective.
\item[{\rm (iv)}] $\phi_{x}$ is injective for all $x \in X$ iff $\phi_{x}$ is injective for all $x \in X^{\ast}$.
\item[{\rm (v)}] The function $\rho_{x}$ from $G$ to $G$ defined by $\rho_{x}:g \mapsto g|_{x}$ is injective for all $x \in X$ iff it is injective for all $x \in X^{\ast}$.
\item[{\rm (vi)}] The function $\rho_{x}$ from $G$ to $G$ defined by $\rho_{x}:g \mapsto g|_{x}$ is injective for all $x \in X$ iff for all $x \in X$, if $g|_{x} = 1$ then $g = 1$.
\item[{\rm (vii)}] The function $\phi_{x}$ is surjective for all $x \in X$ iff it is surjective for all $x \in X^{\ast}$.
\item[{\rm (viii)}] The function $\rho_{x}$ from $G$ to $G$ given by $\rho_{x}:g \mapsto g|_{x}$ is surjective for all $x \in X$ iff it is surjective for all $x \in X^{\ast}$.
\end{description}
\end{lemma}

\begin{proof}
(i) Let $g,h \in G_{x}$.
Then 
$$\phi_{x}(gh) = (gh)|_{x} = g|_{h \cdot x} h|_{x} = g|_{x} h|_{x} = \phi_{x}(g) \phi_{x}(h),$$
using (SS8), as required.

(ii) Let $h\in gG_{x}g^{-1}$. Then $h = gkg^{-1}$ for some $k\in G_{x}$ and so
$$h\cdot y = (gkg^{-1})\cdot (g\cdot x) = g\cdot (k\cdot x) = g\cdot x = y.$$
Thus $h\in G_{y}$. On the other hand, let $h\in G_{y}$. Then
$$(g^{-1}hg)\cdot x = g^{-1}\cdot y = g^{-1}\cdot (g\cdot x)  = x$$
and so $h\in gG_{x}g^{-1}$. If $h\in G_{y}$ then
$$\phi_{y}(h) = h|_{y} = (gg^{-1}hgg^{-1})|_{g\cdot x} = (gg^{-1}hg)|_{x}g^{-1}|_{g\cdot x} = g|_{x}\phi_{x}(g^{-1}hg)(g|_{x})^{-1}.$$

(iii) This follows by (ii) above.

(iv) We need only prove one direction.
We prove the result by induction on the length of $x$.
The result is true for strings of length one by assumption.
We assume the result is true for strings of length $n$.
We now prove it for strings of length $n+1$.
Let $y\in X^{\ast}$ be of length $n+1$.
Then $y = zx$ where $z$ has length $n$ and $x$ has length one.
We prove that $\phi_{y}$ is injective on $G_{y}$.
Let $h,k \in G_{y}$.
Then $h \cdot y = y = k \cdot y$.
By comparing lengths, it follows that 
$h \cdot z = z = k \cdot z$
and
$h|_{z} \cdot x = x = k|_{z} \cdot x$.
Suppose that $\phi_{y}(h) = \phi_{y}(k)$.
Then $h|_{y} = k|_{y}$.
By axiom (SS6), we have that
$(h|_{z})|_{x} = (k|_{z})|_{x}$.
But $h|_{z}, k|_{z} \in G_{x}$, and so
by injectivity for letters $h|_{z} = k|_{z}$.
Also $h,k \in G_{z}$, and so by the induction hypothesis $h = k$, as required.

(v) Just one direction needs proving. We again prove the result by induction. It is true for strings of length one by assumption. Let us assume it is true for strings of length $n$. Let $y\in X^{\ast}$ be a string of length $n+1$ and suppose $g|_{y} = h|_{y}$ for some $g,h\in G$. Then $y = zx$ for some $z,x\in X^{\ast}$ with $|z|=n$ and $|x| =1$. It follows from (SS8) that $(g|_{z})|_{x} = (h|_{z})|_{x}$. Since $\rho_{x}$ is injective we see that $g|_{z} = h_{z}$ and since $\rho_{z}$ is injective we must have $g = h$.

(vi) One direction is clear.
We prove the other direction.
Suppose that for all $x \in X$, if $g|_{x} = 1$ then $g = 1$.
We prove that the function from $G$ to $G$ defined by $g \mapsto g|_{x}$ is injective for all $x \in X$.
Suppose that $g|_{x} = h|_{x}$.
Then $g|_{x}(h|_{x})^{-1} = 1$.
By Lemma \ref{inverserest}, $(h|_{x})^{-1} = h^{-1}|_{h \cdot x}$.
Put $y = h \cdot x$.
Then 
$$1 = g|_{x}(h|_{x})^{-1} = (g|_{h^{-1} \cdot y})( h^{-1}|_{y}) = (gh^{-1})|_{y}$$
by (SS8).
By assumption $gh^{-1} = 1$ and so $g = h$.

(vii) Only one direction needs to be proved.
We assume the result holds for strings of length 1.
Suppose that the result holds for strings $n$.
Let $y$ be a string of length $n+1$.
Then $y = zx$ where $x$ is a letter and $z$ has length $n$.
Let $g \in G$.
Then because $\phi_{x}$ is surjective, 
there exists $h \in G_{x}$ such that $\phi_{x}(h) = g$.
By the induction hypothesis, there exists $k \in G_{z}$ such that $\phi_{z}(k) = h$.
We now calculate 
$$k \cdot y = k \cdot (zx) = (k \cdot z)(k|_{z} \cdot x) = zx = y.$$
Thus $k \in G_{y}$ and $\phi_{y}(k) = k|_{zx} = (k|_{z})|_{x} = h|_{x} = g$,
as required using axiom (SS6).

(viii) We need only prove one direction. Again we prove by induction.
Assume that $\rho_{x}$ is surjective for all $x\in X$. 
Suppose $\rho_{x}$ is surjective for all $x$ of length $n$.
Let $y$ be a string of length $n+1$ so that $y = zx$ for some strings $z,x$ with $|z| = n$ and $|x| = 1$ and let $h\in G$ be arbitrary.
Then there exist $g\in G$ with $g|_{x} = h$ and $k\in G$ with $k|_{z} = g$ by the induction hypotheses so that using (SS6) we have
$$k|_{y} = k|_{zx} = (k|_{z})|_{x} = g|_{x} = h.$$
Thus $\rho_{y}$ is surjective.
\end{proof}

%%%%%%%%%%%%%%%%%%%%%%%%%%%%%%%%%%%%%%%%%%%%%%%%%%%%%%%%%%%%%%%%%%%%%%%%%%%%%%%%%%%%%%%%%%%%%%%%%%%%%%%%%%%%%%%%%%%%%%%%%%%%%%%%%%
Let $M$ be a left Rees monoid, let $G = G(M)$ be its group of units,
let $X$ be a transversal of the generators of the maximal proper principal right ideals,
and denote by $X^{\ast}$ the submonoid generated by the set $X$.
Then $X^{\ast}$ is free, 
$M = X^{\ast}G$,
and each element of $M$ can be written uniquely as a product of an element of $X^{\ast}$ and an element of $G$.
Let $g \in G$ and $x \in X^{\ast}$.
Then $gx \in M$ and so can be written uniquely in the form
$gx = x'g'$ where $x' \in X^{\ast}$ and $g' \in G$.
Define $x' = g \cdot x$ and $g' = g|_{x}$.
Then it is easy to check that this defines a self-similar action of $G$ on $X^{\ast}$.

Let $(G,X)$ be an arbitrary self-smilar group action.
On the set $X^{\ast} \times G$ define its Zappa-Sz\'{e}p product as above by
$$(x,g)(y,h) = (x(g \cdot y), g|_{y}h).$$ 
Then $X^{\ast} \times G$ is a left Rees monoid containing copies of $X^{\ast}$ and $G$
such that $X^{\ast} \times G$ can be written as a unique product of these copies.
% This monoid is called the {\em Zappa-Sz\'ep product of $X^{\ast}$ and $G$} and is denoted $X^{\ast} \bowtie G$.

It follows that a monoid is a (non-group) left Rees monoid 
if and only if it is isomorphic to a Zappa-Sz\'ep
product of a free monoid by a group. 

In turn, Zappa-Sz\'ep products of free monoids by groups determine,
and are determined by, self-similar group actions.
We have therefore set up a correspondence between left Rees monoids
and self-similar group actions in which
each determines the other up to isomorphism.

%%%%%%%%%%%%%%%%%%%%%%%%%%%%%%%%%%%%%%%%%%%%%%%%%%%%%%%%%%%%%%%%%%%%%%%%%%%%%%%%%%%%%%%%%%%%%%%%%%%%%%%%%%%%%%%%%%%%%%%%%%%%%%%%%%%%%%%%%%%%%%%%%%%%%
Throughout this section let $M = X^{\ast}G$ be a left Rees monoid.
Define
$$\mathcal{K}(M) = \{g \in G \colon \: gs \in sG \mbox{ for all } s \in S \},$$
a definition due to Rees \cite{Rees}.
This is a normal subgroup of $G$
which we call the {\em kernel} of the left Rees monoid.
Left Rees monoids $S$ for which $\mathcal{K}(M) = \{1 \}$ are said to be {\em fundamental}.
It can be checked that $\mathcal{K}(M) = \bigcap_{x \in X^{\ast}}G_{x}$,
and so a left Rees monoid is fundamental iff the corresponding group action is faithful.

%%%%%%%%%%%%%%%%%%%%%%%%%%%%%%%%%%%%%%%%%%%%%%%%%%%%%%%%%%%%%%%%%%%%%%%%%%%%%%%%%%%%%%%%%%%%%%%%%%%%%%%%%%%%%%%%%%%%%%%%%%%%%%%%%%%%%%%%%%%%%%%%%%%%%%%%%%%%%%%%%%%%%

Let us summarise some facts and notions relating to self-similar group actions which are described in detail in \cite{NekrashevychBook}. A group $G$ acts by \emph{automorphisms} on a regular rooted tree if the action is level-preserving, if it does not move the root and if $\dom(g\cdot x)=g\cdot \dom(x)$ and $\ran(g\cdot x)=g\cdot \ran(x)$ for each edge $x$. Viewing $X^{*}$ as a tree, we see that in a self-similar action $G$ acts on $X^{*}$ in a length-preserving manner and therefore by automorphisms. We see that $G\leq Aut(X^{*})$ if and only if $G$ acts faithfully. 

Let $G\leq \Aut(X^{*})$ be a subgroup acting on the left on the rooted tree $X^{*}$ (so, in particular, it acts faithfully). Then for each $x\in X^{*}$ and $g\in G$ there is a unique automorphism $g|_{x}\in \Aut(X^{*})$ such that $g\cdot (xy) = (g\cdot x)(g|_{x}\cdot y)$ for each $y\in X^{*}$. Call this the \emph{restriction} of $g$ by $x$. Denote both the identity of $G$ and the root of the tree by $1$. It can be checked that restrictions satisfy the following properties, for all $g,h\in G$ and $x,y\in X^{*}$:
\begin{enumerate}
\item $g|_{1}=g$
\item $g|_{xy}=(g|_{x})|_{y}$
\item $1|_{x}=1$
\item $(gh)|_{x}=g|_{h\cdot x}g|_{x}$
\end{enumerate}
So we see that subgroups of the automorphism group of $X^{*}$ which are closed under restriction give rise to unique fundamental left Rees monoids. On the other hand, given a fundamental left Rees monoid $M = X^{\ast}G$, then $G$ is a subgroup of $\Aut(X^{*})$ closed under the restriction maps.

Let $H$ be a group acting on the left by permutations on a set $X$ and let $G$ be an arbitrary group. Then the \emph{(permutational) wreath product} $H\wr G$ is the semi-direct product $H \ltimes G^{X}$, where $H$ acts on the direct power $G^{X}$ by the respective permutations of the direct factors.

Let $M=X^{*}G$ be a left Rees monoid, $|X|=d$ and let $\mathcal{S}(X)$ denote the symmetric group on the set $X$. Then we have a homomorphism $\psi:G\rightarrow \mathcal{S}(X)\wr G$ given by:
$$\psi(g)=\sigma(g|_{x_{1}},\ldots,g|_{x_{d}}),$$
where $\sigma$ is the permutation on $X$ determined by the action of $g$ on $X$. On the other hand, given a homomorphism $\psi:G\rightarrow \mathcal{S}(X)\wr G$, we have a unique induced self-similar action. The map $\psi$ is called the \emph{wreath recursion}.

We know that the definition of left Rees monoids involves principal right ideals. Green's $\mathcal{R}$-relation is defined on monoids $M$ by $s\mathcal{R} t$ if $sM = tM$; that is, they generate the same principal right ideals. In our situation we have the following lemma:

\begin{comment}
The structure of the principal right ideals will play an important role in our calculations.
These are handled using Green's relations $\mathcal{R}$, $\mathcal{L}$, $\mathcal{H}$, $\mathcal{J}$ and $\mathcal{D}$. We will now summarise some results about these relations from LRMPAPER

Green's relation $\mathcal{R}$ is defined by $a \,\mathcal{R}\, b$ iff $aS = bS$.
It is worth recalling a well-known result here.
If $a \, \mathcal{R} \, b$ then $a = bg$ and $b = ah$ for some $g,h \in S$.
Then $a = bg = ahg$ and $b = ah = bgh$.
Thus by left cancellation we have that $1 = hg = gh$
and so $g$ and $h$ are mutually inverse units.
The proof of the following is now immediate using the uniqueness of the decomposition into elements of $X^{\ast}$ and units.

\end{comment}

\begin{lemma}
\label{Rclasslrm}
Let $xg, yh \in M$. Then
$xg \,\mathcal{R}\, yh$
iff
$x = y$.
In particular, each $\mathcal{R}$-class contains exactly one element from $X^{\ast}$.
\end{lemma}

\begin{proof}
We see that $xgg^{-1}h = xh$ and $xhh^{-1}g = xg$ so that $xg\mathcal{R}xh$. 
On the other hand if $xg\mathcal{R}yh$ then $xgu = yh$ and $yhv = xg$ for some $u,v\in M$. 
Thus $xguv = xg$ and $yhvu = yh$. By left cancellativity $uv = vu = 1$ and so $u,v\in G$.
Thus since elements of $M$ can be uniquely written in the form $xg$ for $x\in X^{\ast}$, $g\in G$ it follows that $x = y$. 
\end{proof}

In fact if $x,y \in X^{\ast}$ then $xM \subseteq yM$ iff $x = yz$ for some $z \in X^{\ast}$.
Combined with Lemma \ref{Rclasslrm}, this tells us that the partially ordered set $M/\mathcal{R}$ of
$\mathcal{R}$-classes is order-isomorphic to the set $X^{\ast}$ equipped with the prefix ordering.\\

Green's $\mathcal{J}$-relation is defined on monoids by $s\, \mathcal{J}\, t$ iff $MsM = MtM$; that is, the principal two-sided ideals generated by $s$ and $t$ are equal. We have the following for left Rees monoids:

\begin{lemma}
\label{greenJLRM}
\begin{description}

\item[{\rm (i)}]  $MxgM\subseteq MyhM$ implies $|y|\leq |x|$.

\item[{\rm (ii)}] Let $xg,yh\in M$. Then
$xg \,\mathcal{J}\, yh$
iff
$x$ and $y$ are in the same orbit under the action of $G$.

\end{description}
\end{lemma}

\begin{proof}
(i) If $MxgM \subseteq MyhM$ then there exist $s,t\in M$ with $syht = xg$ and so $|y|\leq |x|$.

(ii) By (i), if $MxgM = MyhM$ then there exist $u,v,w,z\in G$ with $uxgv = yh$ and $wyhz = xg$ and so by the unique normal form of elements of $M$ we have $y = u\cdot x$. Thus $x$ and $y$ are in the same orbit under the action of $G$. 

Let $g,h\in G$ be arbitrary. If $x,y\in X^{\ast}$ are such that $y = u\cdot x$ for some $u\in G$ then $uxgg^{-1}(u|_{x})^{-1}h = yh$ and $u^{-1}yhh^{-1}u|_{x}g = xg$. Thus $MxgM = MyhM$.
\end{proof}

We will say the a self-similar group action $(G,X)$ is \emph{transitive} if the action of $G$ on $X$ is transitive and \emph{level-transitive} if the action of $G$ on $X^{n}$ is transitive for each $n$. We then have the following corollaries of Lemma \ref{greenJLRM}:

\begin{corollary}
\begin{description}
\item[{\rm (i)}]  A self-similar group action $(G, X)$ is transitive if and only if the associated left Rees monoid has a unique maximal proper principal two-sided ideal.

\item[{\rm (ii)}] A self-similar group action $(G, X)$ is level-transitive if and only if the principal two-sided ideals of the associated left Rees monoid form an infinite descending chain.
\end{description}
\end{corollary}

It will be useful later to know whether our left Rees monoid is in fact cancellative (and therefore a Rees monoid). The following will be proved for the more general case of left Rees categories as Lemma \ref{LRCcancphi} in Chapter 3:

\begin{lemma} 
\label{cancLRM}
Let $M$ be a left Rees monoid.
Then the following are equivalent.
\begin{description}

\item[{\rm (i)}] The functions $\phi_{x} \colon \: G_{x} \rightarrow G$ are injective for all $x \in X^{\ast}$.

\item[{\rm (ii)}] The monoid $M$ is right cancellative (and so cancellative).

\end{description}
\end{lemma}

\section{Monoid HNN-extensions}

%Throughout this section, certain finiteness conditions will be tacitly assumed.

We have seen above that left Rees monoids and self-similar group actions are two different, but equivalent, ways of viewing the same mathematical idea. In this section, we describe a third way, namely, in terms of monoid HNN-extensions. We will only sketch the proof as the theorem will appear in greater generality as Theorem \ref{HNNLRC} in Chapter 3. We will then explain some consequences of this result, and touch on the relationship with Bass-Serre theory which will be expanded further in Chapter 3. Finally, we will demonstrate the main theorem with a simple example.

First, let us briefly explain the motivation behind the main result, which comes from a result of Cohn on proving certain monoids embed in their groups of fractions, as found in \cite{Cohn71}. It is difficult to see how his proof works and he uses a completely different proof in terms of string rewriting appears in the second edition of the same book (\cite{Cohn85}). Let us outline his argument. Let $M$ be a cancellative right rigid monoid, let $G$ be its group of units and for each $a\in M$ let 
$$G_{1}(a) = \left\{u\in G|ua\in aG\right\}$$
and
$$G_{-1}(a) = \left\{u\in G|au\in Ga\right\}.$$
Define $a\sim b$ in $M$ if $a = ubv$ for some $u,v\in G$. One can easily verify that $G_{1}(a), G_{-1}(a)$ are subgroups of $G$, if $a\sim b$ then $G_{1}(a)\cong G_{1}(b)$ and for any $u\in G$ we have $G_{1}(u) = G_{-1}(u) = G$. Let $T_{1}(a),T_{-1}(a)$ be complete sets of left coset representatives of $G_{1}(a),G_{-1}(a)$ respectively in $G$ with $1$ represented by itself and let $A$ be a complete set of representatives of $\sim$-classes of $M$ with $G$ represented by $1$. Cohn then considers the set of expressions 
$$t_{1}a_{1}^{\epsilon_{1}}t_{2}a_{2}^{\epsilon_{2}}\cdots t_{m}a_{m}^{\epsilon_{m}}u$$
where $t_{i}\in T_{\epsilon_{i}}(a_{i})$, $a_{i}\in A$, $\epsilon_{i}\in \left\{-1,1\right\}$, $u\in G$, subject to the condition that if $t_{i} = 1$ and $a_{i-1} = a_{i}$ then $\epsilon_{i-1} + \epsilon_{i} \neq 0$ for $i=2,\ldots,r$. He then claims that it is a routine though tedious exercise to verify that the permutation group on the set of such expressions contains the original monoid as a subsemigroup and that elements of this permutation group have a unique normal form, namely as one of the expressions. To see why I am unclear how this would work suppose $M = X^{\ast}G$ is a Rees monoid. Let $x_{1},x_{2}\in X$ and let $y = x_{1}x_{2}\in X^{\ast}$ be such that these are each in $A$. Then the expressions $x_{1}x_{2}$ and $y$ are distinct according to Cohn's rule but clearly they represent the same element of $M$. It is difficult to see therefore how one would embed $M$ in the permutation group on such expressions. The key idea, however, that seems to have something to it is that there might be a connection between certain cancellative right rigid monoids and Bass-Serre theory (compare the preceding argument and Proposition \ref{normalformgroupoidHNN}).

Let $S$ be a monoid, $I$ an index set, $\breve{S} = S\setminus \left\{1\right\}$, $H_{i}:i\in I$ submonoids of $S$ and let $\rho_{i}:H_{i}\rightarrow S$ be homomorphisms for each $i\in I$. Then $M$ is a \emph{monoid HNN-extension} of $S$ if $M$ can be defined by the following monoid presentation
$$M=\langle \breve{S}, t_{i}:i\in I | \mathcal{R}(S), \quad h t_{i} = t_{i} \rho_{i}(h) \quad \forall h\in H_{i}, i\in I \rangle, $$ 
where $\mathcal{R}(S)$ denotes the relations of $S$. We will call the the generators $t_{i}:i\in I$ \emph{stable letters}, and say that $M$ is a \emph{monoid HNN-extension on a single stable letter} if $|I| = 1$. For the moment, the use of the phrase \emph{monoid HNN-extension} is simply based on the similarity of presentation to that of a group HNN-extension (c.f. Sections 1.2 and 1.3). In Chapter 3 we will see that they in fact appear in the study of graphs of groups and have natural actions on trees.

The following is the main result of this section, though we will only sketch the proof as a more general version will appear as Theorem \ref{HNNLRC} in Chapter 3.

\begin{thm}
\label{mainthm1}
Let $M$ be a monoid HNN-extension of a group $G$ where each associated submonoid $H_{i}$ is in fact a subgroup of $G$. Then $M$ is a left Rees monoid. On the other hand, if $M$ is a left Rees monoid then $M$ is isomorphic to a monoid HNN-extension of a group.
\end{thm}

\begin{comment}
\begin{lemma}
Let $M$ be a monoid HNN-extension of a group $G$ where each associated submonoid $H_{i}$ is in fact a finite index subgroup. Then each element of $M$ has a unique normal form.
\end{lemma}
\end{comment}

\begin{proof}
(Sketch)
Suppose that $M$ is a monoid given by the following presentation:
$$M=\langle \breve{G}, t_{i}:i\in I | \mathcal{R}(G), \quad h t_{i} = t_{i} \rho_{i}(h) \quad \forall h\in H_{i}, i\in I \rangle.$$
For each $i\in I$, let $T_{i}$ be a transversal of left coset representatives of $H_{i}$. Note that for each $i$ an element $u\in G$ can be written uniquely in the form $u = gh$, where $g\in T_{i}$ and $h\in H_{i}$. We further suppose that $1\in T_{i}$ for each $i$.

One can show that every element $s\in M$ can be written in the form
$$s = g_{1}t_{i_{1}}g_{2}t_{i_{2}}\cdots g_{m}t_{i_{m}}u$$
where $g_{k}\in T_{i_{k}}$ and $u\in G$ (in fact it will turn out this is a unique normal form).

Letting $X = \left\{gt_{i}|g\in T_{i}, i\in I\right\}$ one can define a self-similar group action of $G$ on $X^{\ast}$ by rewriting $gx = (g\cdot x)x$ for $g\in G$, $x\in X^{\ast}$. The resulting left Rees monoid will be isomorphic to the original monoid $M$.

Now suppose $M = X^{\ast}G$ is a left Rees monoid. For each orbit of the action of $G$ on $X$ choose an element $x_{i}\in X$, $i\in I$ where $|I|$ is the number of orbits. For each $i\in I$ let $H_{i} = G_{x_{i}}$ be the stabiliser of $G$ acting on $x_{i}$ and let $T_{i}$ be a transversal of left coset representatives of $H_{i}$ in $G$. An arbitrary element $x\in X$ can be written uniquely in the form $x = gx_{i}(\rho_{i}(g))^{-1}$ where $i\in I$ and $g\in T_{i}$. Now define $\Gamma$ to be the monoid given by monoid presentation:
$$\Gamma = \langle \breve{G}, t_{i}:i\in I | \mathcal{R}(G), \quad h t_{i} = t_{i} \rho_{i}(h) \quad \forall h\in H_{i}, i\in I \rangle.$$
One can then check that every element of $\Gamma$ can be written uniquely in the form
$$g_{1}t_{i_{1}}(\rho_{i_{1}}(g_{1}))^{-1}g_{2}t_{i_{2}}(\rho_{i_{2}}(g_{2}))^{-1}\cdots g_{m}t_{i_{m}}(\rho_{i_{m}}(g_{m}))^{-1}u$$
where $g_{k}\in T_{i_{k}}$ and $u\in G$. 
It is then easy to see that the map $\iota:\Gamma\rightarrow M$ given on generators by $\iota(t_{i}) = x_{i}$ for $i\in I$ and $\iota(g) = g$ for $g\in G$ is an isomorphism.   
\end{proof}

Combining Theorem \ref{mainthm1}, Lemma \ref{usefullemma1} (iii) and Lemma \ref{cancLRM}, we have the following.

\begin{corollary}
Let $M$ be a monoid HNN-extension of a group $G$ where each associated submonoid $H_{i}$ is in fact a subgroup of $G$ with associated maps $\rho_{i}:H_{i}\rightarrow G$. Then $M$ is cancellative and therefore a Rees monoid if and only if the maps $\rho_{i}$ are injective for each $i\in I$.
\end{corollary}

Note the above construction tells us that the action of the group $G$ is transitive on $X$ if and only if the HNN-extension is on a single stable letter.
Let $M = X^{\ast}G$ be a left Rees monoid and suppose we split $X$ into its orbits $X_{i}:i\in I$ under the action of $G$, so that 
$$X = \bigcup_{i\in I}{X_{i}}.$$
Given the self-similar action of $G$ on $X^{\ast}$ there is an induced self-similar action of $G$ on $X_{i}^{\ast}$ as none of the axioms of self-similar group actions move elements of $X_{i}^{\ast}$ outside an orbit. We therefore have submonoids $M_{i} = X_{i}^{\ast}G\leq M$, where we are identifying the identity element of $M_{i}$ with that of $M$. Observe that $M_{i}\cap M_{j} = G$ for $i\neq j$. We can then form a \emph{semigroup amalgam}, in the sense of Chapter 8 of \cite{HowieBook}. Viewing the $M_{i}$'s as disjoint monoids and letting $\alpha_{i}$ be the embedding of the group $G$ into each $M_{i}$, $\mathcal{A} = [G; M_{i}; \alpha_{i}]$ is a semigroup amalgam. We can then form the amalgamated free product $S = {\ast}_{G}{M_{i}}$ of the amalgam $\mathcal{A}$. 

We now make use of the following classical theorem of Bourbaki (\cite{Bourbaki}), in the left hand dual of the form given by Dekov in \cite{Dekov}.

\begin{theorem}
\label{bourbakiproduct}
Let $M_{i}:i\in I$ be a family of monoids, let $G$ be a submonoid of $M_{i}$ for each $i\in I$ and let $G = M_{i}\cap M_{j}$ for all $i,j\in I$ with $i\neq j$. Assume that for each $i\in I$ there exists a subset $S_{i}$ of $M_{i}$ containing the identity $1$ and such that the mapping $\psi:S_{i}\times G\rightarrow M_{i}$ given by $\psi(x,g) = xg$ is a bijection. Then every $s\in {\ast}_{G}{M_{i}}$ can be written uniquely in the form 
$$s = x_{1}\cdots x_{n}g$$
where $x_{k} \in S_{i_{k}}\setminus \left\{1\right\}$ and $i_{k}\neq i_{k+1}$ for each $k = 1,\ldots,n$, and $g\in G$. 
\end{theorem}

Applying this theorem to the setup above with $S_{i} = X_{i}^{\ast}$ it follows that each element of $S = {\ast}_{G}{M_{i}}$ can be uniquely written in the form $xg$ where $x\in X^{\ast}$ and $g\in G$. Consequently we have the following result.

\begin{theorem}
Let $M = X^{\ast}G$ be a left Rees monoid, $X_{i}:i\in I$ the orbits of $X$ under the action of $G$ and let $M_{i} = X_{i}^{\ast}G$ for each $i$. Then
$${\ast}_{G}{M_{i}} \cong M.$$
\end{theorem}

In terms of monoid HNN-extensions what we are saying is that the monoids
$$\langle \breve{G}, t_{i}:i\in I | \mathcal{R}(G), h t_{i} = t_{i} \rho_{i}(h) \forall h\in H_{i}, i\in I \rangle$$
and
$${\ast}_{G} \langle \breve{G}, t_{i} | \mathcal{R}(G), h t_{i} = t_{i} \rho_{i}(h) \forall h\in H_{i} \rangle$$
are isomorphic.

Let us therefore now consider the case of left Rees monoids given as monoid HNN-extensions which have a single stable letter. Recall that this is equivalent to the group of units acting transitively on the elements of $X$, and to the monoid having a single maximal proper two-sided principal ideal.

If $H$ is a subgroup of a group $G$, then a homomorphism $\phi:H\rightarrow G$ will be called a \emph{partial endomorphism}. Nekrashevych (\cite{NekrashevychBook}) calls such homomorphisms \emph{virtual endomorphisms} in the case where $H$ is a finite index subgroup of $G$. Two partial endomorphisms $\phi_{1}:H_{1}\rightarrow G$, $\phi_{2}:H_{2}\rightarrow G$ will be said to be \emph{conjugate} if there exist inner automorphisms $\alpha,\beta$ of $G$ with $\alpha(H_{1}) = H_{2}$ and $\beta\phi_{1} = \phi_{2}\alpha$. Partial endomorphisms $\phi_{1}:H_{1}\rightarrow G_{1}$, $\phi_{2}:H_{2}\rightarrow G_{2}$ will be said to be \emph{isomorphic} if there exist group isomorphisms $\alpha,\beta:G_{1}\rightarrow G_{2}$ with $\alpha(H_{1}) = H_{2}$ and $\beta\phi_{1} = \phi_{2}\alpha$. If $\phi:H\rightarrow G$ is a partial endomorphism then we will define $M(\phi)$ to be the left Rees monoid with presentation
$$M(\phi) = \langle \breve{G}, t | \mathcal{R}(G), h t = t \phi(h) \forall h\in H \rangle.$$

\begin{proposition}
\label{virtendLRMiso}
Let $\phi_{1}:H_{1}\rightarrow G$, $\phi_{2}:H_{2}\rightarrow G$ be conjugate partial endomorphisms. Then the monoids $M(\phi_{1})$ and $M(\phi_{2})$ are isomorphic.
\end{proposition}

\begin{proof}
We have 
$$M(\phi_{1}) = \langle \breve{G}, t_{1} | \mathcal{R}(G), h t_{1} = t_{1} \phi_{1}(h) \forall h\in H_{1} \rangle$$
and
$$M(\phi_{2}) = \langle \breve{G}, t_{2} | \mathcal{R}(G), h t_{2} = t_{2} \phi_{2}(h) \forall h\in H_{2} \rangle.$$
Since $\phi_{1}$ and $\phi_{2}$ are conjugate there exist inner automorphisms $\alpha,\beta$ of $G$ with $\alpha(H_{1}) = H_{2}$ and $\beta\phi_{1} = \phi_{2}\alpha$. Suppose $\alpha(g) = a^{-1}ga$ and $\beta(g) = bgb^{-1}$ for some $a,b\in G$. 
Define $f:M(\phi_{1})\rightarrow M(\phi_{2})$ by $f(g) = g$ for $g\in G$ and $f(t_{1}) = at_{2}b$. To verify that $f$ is a homomorphism we just need to check that $f(h)f(t_{1}) = f(t_{1})f(\phi_{1}(h))$ for every $h\in H_{1}$ since it is clear that $f(gh) = f(g)f(h)$ for every $g,h\in G$. For $h\in H_{1}$ we have
\begin{eqnarray*}
f(h)f(t_{1}) &=& hat_{2}b = aa^{-1}hat_{2}b = at_{2}\phi_{2}(a^{-1}ha)b = at_{2}\phi_{2}(\alpha(h))b \\
&=& at_{2}\beta(\phi_{1}(h))b = at_{2}b(\phi_{1}(h))b^{-1}b = at_{2}b(\phi_{1}(h)) = f(t_{1})f(\phi_{1}(h)).
\end{eqnarray*}
Thus $f$ is a homomorphism.
Since $f(a^{-1}t_{1}b^{-1}) = t_{2}$, $f$ is also surjective. 
Let $T_{1}$ and $T_{2}$ be transversals of the left coset representatives of $H_{1}$ and $H_{2}$ in $G$.
Our final task is to check that $f$ is injective. Suppose
$$f(g_{1}t_{1}\cdots g_{m}t_{1}u) = f(g_{1}^{\prime}t_{1}\cdots g_{n}^{\prime}t_{1}v),$$
where $g_{k},g_{k}^{\prime}\in T_{1}$ for each $k$.
Observe that the number of $t_{1}'s$ mapped across is constant so that $m = n$. So
$$g_{1}at_{2}b\cdots g_{m}at_{2}b u = g_{1}^{\prime}at_{2}b\cdots g_{m}^{\prime}at_{2}b v.$$
Reducing this into normal form and using the fact that $\alpha:H_{1}\rightarrow H_{2}$ we see that there must exist unique $c_{1},\ldots,c_{m}\in T_{2}$, $h_{1},\ldots,h_{m},h_{1}^{\prime},\ldots,h_{m}^{\prime}\in H_{1}$ with
$$g_{1}a = c_{1}\alpha(h_{1}), \quad g_{1}^{\prime}a = c_{1}\alpha(h_{1}^{\prime}),$$
$$c_{k}\alpha(h_{k}) = \phi_{2}(\alpha(h_{k-1}))bg_{k}a = \beta(\phi_{1}(h_{k-1}))bg_{k}a = b\phi_{1}(h_{k-1})g_{k}a,$$
$$c_{k}\alpha(h_{k}^{\prime}) = b\phi_{1}(h_{k-1}^{\prime})g_{k}^{\prime}a,$$
for $k = 2,\ldots,m$ and
$$b\phi_{1}(h_{m})b^{-1}bu = b\phi_{1}(h_{m}^{\prime})b^{-1}bv.$$
Thus
$$\phi_{1}(h_{m})u = \phi_{1}(h_{m}^{\prime})v$$
and so
$$h_{m}t_{1}u = h_{m}^{\prime}t_{1}v$$
from which it follows that
$$ac_{m}^{-1}b\phi_{1}(h_{m-1})g_{m}t_{1}u = ac_{m}^{-1}b\phi_{1}(h_{m-1}^{\prime})g_{m}^{\prime}t_{1}v.$$
Cancelling on the left we have
$$\phi_{1}(h_{m-1})g_{m}t_{1}u = \phi_{1}(h_{m-1}^{\prime})g_{m}^{\prime}t_{1}v.$$
From this we deduce that
$$h_{m-1}t_{1}g_{m}t_{1}u = h_{m-1}^{\prime}t_{1}g_{m}^{\prime}t_{1}v.$$
Continuing in this way we find that
$$g_{1}t_{1}\cdots g_{m}t_{1}u = g_{1}^{\prime}t_{1}\cdots g_{m}^{\prime}t_{1}v$$
and so $f$ is indeed injective.
\begin{comment}
(ii) Let
$$M(\phi_{1}) = \langle \breve{G}, t_{1} | \mathcal{R}(G), h t_{1} = t_{1} \phi_{1}(h) \forall h\in H_{1} \rangle$$
and
$$M(\phi_{2}) = \langle \breve{G}, t_{2} | \mathcal{R}(G), h t_{2} = t_{2} \phi_{2}(h) \forall h\in H_{2} \rangle.$$
Since $f$ is an isomorphism, $f(t_{1})$ must generate a maximal proper principal right ideal in $M(\phi_{2})$. Thus $f(t_{1}) = at_{2}b$ for some $a,b\in G$. Define inner automorphisms $\alpha,\beta:G\rightarrow G$ by $\alpha(g) = a^{-1}ga$ and $\beta(g) = bgb^{-1}$. If $h\in H_{1}$ then 
\begin{eqnarray*}
\alpha(h)t_{2} &=& a^{-1}hat_{2} = a^{-1}hat_{2}bb^{-1} = a^{-1}f(ht_{1})b^{-1} = a^{-1}f(t_{1}\phi_{1}(h))b^{-1} \\
&=& a^{-1}f(t_{1})f(\phi_{1}(h))b^{-1} = a^{-1}at_{2}b\phi_{1}(h)b^{-1} = t_{2}\beta(\phi_{1}(h)).
\end{eqnarray*}
Thus $\alpha(H_{1})\subseteq H_{2}$ and $\beta\phi_{1} = \phi_{2}\alpha$. Further, if $h\in H_{2}$ then
\begin{eqnarray*}
f(\alpha^{-1}(h)t_{1}) &=& f(aha^{-1}t_{2}) = aha^{-1}at_{2}b = at_{2}\phi_{2}(h)b \\
&=& at_{2}bb^{-1}\phi_{2}(h)b = f(t_{1} b^{-1}\phi_{2}(h)b).
\end{eqnarray*}
Since $f$ is an isomorphism this therefore implies that $\alpha^{-1}(h)t_{1} = t_{1} b^{-1}\phi_{2}(h)b$ and so $\alpha^{-1}(H_{2}) = H_{1}$. Thus $\phi_{1}$ and $\phi_{2}$ are conjugate partial endomorphisms.
\end{comment}
\end{proof}

We now have the following straightforward corollary.

\begin{corollary}
\label{endomorphismcor1}
Let $G$ be a group, $H_{i},H_{i}^{\prime}:i\in I$ subgroups of $G$ and let $\phi_{i}:H_{i}\rightarrow G$, $\phi_{i}^{\prime}:H_{i}^{\prime}\rightarrow G$ be partial endomorphisms such that $\phi_{i}$ is conjugate to $\phi_{i}^{\prime}$ for each $i\in I$. Then the monoids $\ast_{G}{M(\phi_{i})}$ and $\ast_{G}{M(\phi_{i}^{\prime})}$ are isomorphic. 
\end{corollary}

\begin{proposition}
Let $G_{1},G_{2}$ be groups, $H_{i}:i\in I$ subgroups of $G_{1}$, $H_{j}^{\prime}:j\in J$ subgroups of $G_{2}$, $\phi_{i}:H_{i}\rightarrow G_{1}$, $\phi_{j}^{\prime}:H_{j}^{\ast}\rightarrow G_{2}$ partial endomorphisms for each $i\in I$, $j\in J$ and suppose that $M_{1} = \ast_{G_{1}}{M(\phi_{i})}$ and $M_{2} = \ast_{G_{2}}{M(\phi_{j}^{\prime})}$ are isomorphic left Rees monoids. Then there is a bijection $\gamma:I\rightarrow J$ such that the partial endomorphisms $\phi_{i}$ and $\phi_{\gamma(i)}^{\prime}$ are isomorphic for each $i\in I$. 
\end{proposition}

\begin{proof}
We can write $M_{1}$ and $M_{2}$ in terms of monoid presentations as
$$M_{1} = \langle \breve{G_{1}}, t_{i}:i\in I | \mathcal{R}(G_{1}), h t_{i} = t_{i} \phi_{i}(h) \forall h\in H_{i}, i\in I \rangle $$
and
$$M_{2} = \langle \breve{G_{2}}, r_{j}:j\in J | \mathcal{R}(G_{2}), h r_{j} = r_{j} \phi_{j}^{\prime}(h) \forall h\in H_{j}^{\prime}, j\in J \rangle .$$
Suppose $f:M_{1}\rightarrow M_{2}$ is an isomorphism. Note that $f(G_{1}) = G_{2}$.
Each maximal proper principal two-sided ideal of $M_{1}$ is generated by a $t_{i}$ and likewise for $M_{2}$. 
Since these monoids are isomorphic there must be a bijection between principal two-sided ideals. 
It follows that there is a bijection $\gamma:I\rightarrow J$ and elements $a_{i},b_{i}\in G_{2}$ for each $i\in I$ with $f(t_{i}) = a_{i}r_{\gamma(i)}b_{i}$. Define maps $\alpha_{i},\beta_{i}:G_{1}\rightarrow G_{2}$ for each $i\in I$ by
$\alpha_{i}(g) = a_{i}^{-1}f(g)a_{i}$ and $\beta_{i}(g) = b_{i}f(g)b_{i}^{-1}$. We now verify that $\alpha_{i}:H_{i}\rightarrow H_{\gamma(i)}^{\prime}$ and $\beta_{i}\phi_{i} = \phi_{\gamma(i)}^{\prime}\alpha_{i}$ for each $i\in I$. If $h\in H_{i}$ then
\begin{eqnarray*}
\alpha_{i}(h)r_{\gamma(i)} &=& a_{i}^{-1}f(h)a_{i}r_{\gamma(i)} = a_{i}^{-1}f(h)a_{i}r_{\gamma(i)}b_{i}b_{i}^{-1} = a_{i}^{-1}f(h)f(t_{i})b_{i}^{-1} \\
&=& a_{i}^{-1}f(ht_{i})b_{i}^{-1} = a_{i}^{-1}f(t_{i}\phi_{i}(h))b_{i}^{-1} = a_{i}^{-1}f(t_{i})f(\phi_{i}(h))b_{i}^{-1} \\
&=& r_{\gamma(i)}b_{i}f(\phi_{i}(h))b_{i}^{-1} = r_{\gamma(i)}\beta_{i}(\phi_{i}(h)).
\end{eqnarray*}
Thus $\alpha_{i}(H_{i})\subseteq H_{\gamma(i)}^{\prime}$ and $\beta_{i}\phi_{i} = \phi_{\gamma(i)}^{\prime}\alpha_{i}$.
Further, if $h\in H_{\gamma(i)}^{\prime}$ then
\begin{eqnarray*}
f(\alpha_{i}^{-1}(h)t_{i}) &=& f(f^{-1}(a_{i}ha_{i}^{-1})t_{i}) = a_{i}ha_{i}^{-1}a_{i}r_{\gamma(i)}b_{i} = a_{i}hr_{\gamma(i)}b_{i}
= a_{i}r_{\gamma(i)}\phi_{\gamma(i)}^{\prime}(h)b_{i}\\
&=& a_{i}r_{\gamma(i)}b_{i}b_{i}^{-1}\phi_{\gamma(i)}^{\prime}(h)b_{i} = f(t_{i}f^{-1}(b_{i}^{-1}\phi_{\gamma(i)}^{\prime}(h)b_{i})).
\end{eqnarray*}
Since $f$ is an isomorphism this therefore implies that $\alpha_{i}^{-1}(h)t_{i} = t_{i}f^{-1}(b_{i}^{-1}\phi_{\gamma(i)}^{\prime}(h)b_{i})$ and so 
$\alpha_{i}^{-1}(H_{\gamma(i)}^{\prime}) = H_{i}$. Thus $\phi_{i}$ and $\phi_{\gamma(i)}^{\prime}$ are isomorphic for each $i\in I$.
\end{proof}

Note that if in the previous result we had $G = G_{1} = G_{2}$ and $f(g) = g$ for each $g\in G$ in our isomorphism $f:M_{1}\rightarrow M_{2}$ then the partial endomorphisms $\phi_{i}$ and $\phi_{\gamma(i)}^{\prime}$ would in fact be conjugate.

\begin{comment}

\begin{proof}
Assume $ac = bc$, where $a = h_{1}tu_{1}$, $b = h_{2}tu_{2}$, $c = gtv$, $h_{1},h_{2},g \in G$, $u_{1}, u_{2}, v \in S$.
$$ac = bc \quad \Rightarrow \quad h_{1}tu_{1}gtv = h_{2}tu_{2}gtv \quad \Rightarrow \quad h_{1}tu_{1}gt = h_{2}tu_{2}gt $$
Because of the unique normal form, $h_{1}=h_{2}$ and then from the fact that $\rho$ is injective it follows that $u_{1} = u_{2}$.
\end{proof}
\end{comment}

\begin{comment}
All of the above leads to the following main result:

\begin{thm}
Left Rees monoids and monoid HNN-extensions are one and the same thing. Left Rees monoids with transitive action of $G$ on $X$ correspond to monoid HNN-extensions with single stable letter. Rees monoids correspond to monoid HNN-extensions with injective associated homomorphisms.
\end{thm}
\end{comment}

Let $G$ be a group, $H_{i}:i\in I$ subgroups of $G$ and $\rho_{i}:H_{i}\rightarrow G$ be injective partial endomorphisms for each $i$. Recall that $\Gamma$ is a \emph{group HNN-extension} of $G$ if $\Gamma$ can be defined by the following group presentation
$$\Gamma=\langle G, t_{i}:i\in I | \mathcal{R}(G), \quad h t_{i} = t_{i} \rho_{i}(h) \quad \forall h\in H_{i}, i\in I \rangle, $$ 
where $\mathcal{R}(G)$ denotes the relations of $G$.

\begin{comment}
For each monoid $S$ there is a group $U(S)$ and a homomorphism $\iota: S \rightarrow U(S)$
such that for each homomorphism $\phi: S \rightarrow G$ to a group there is a unique homomorphism
$\bar{\phi}: U(S) \rightarrow G$ such that
$\phi = \bar{\phi} \iota$.
The group $U(S)$ is called the {\em universal group} of $S$.
If $S$ is given in terms of a semigroup presentation then $U(S)$ is just the group given by this presentation viewed as a group presentation (c.f. \cite{Cohn71}). 
The monoid $S$ can be embedded in a group iff $\iota$ is injective.
\end{comment}
If a monoid $M$ embeds in its group of fractions then it has to be cancellative (though the converse is not in general true).
If $M$ is a Rees monoid then combining Theorem \ref{mainthm1} and Proposition \ref{univgroupoid} we see that its group of fractions $U(M)$ is a group HNN-extension and noting the normal form results for monoid HNN-extensions and group HNN-extensions we see that in fact $M$ consists of every element of $U(M)$ which does not contain any $t_{i}^{-1}$. So we have the following:

\begin{lemma}
Rees monoids embed in their groups of fractions.
\end{lemma}

On the other hand, we see that every group HNN-extension of a group $G$ is the group of fractions of a Rees monoid, and so there is an underlying self-similar group action. 

\begin{comment}
A subsemigroup $U$ of a semigroup $S$ is \emph{unitary} if for all $u\in U$, $s\in S$ we have $us\in U$ implies $s\in U$ and $su\in U$ implies $s\in U$. Howie (\cite{Howie62}) proved the following result:

\begin{theorem}
\label{unitamalgam}
If $[U;S_{i};\alpha_{i}]$ is an amalgam of semigroups such that $U$ is a unitary subsemigroup of $S_{i}$ for each $i$ then the amalgam is embeddable.
\end{theorem}

We thus have the following:
\end{comment}
\begin{proposition}
Let $G_{1},G_{2}$ be groups, $H_{i}:i\in I$ subgroups of $G_{1}$, $H_{j}^{\prime}:j\in J$ subgroups of $G_{2}$, $\rho_{i}:H_{i}\rightarrow G_{1}$, $\rho_{j}^{\prime}:H_{i}^{\prime}\rightarrow G_{2}$ partial endomorphisms for each $i\in I$ and $j\in J$ and let 
$$M_{1} = \langle\ \breve{G_{1}}, t_{i}:i\in I | \mathcal{R}(G_{1}), h t_{i} = t_{i} \rho_{i}(h) \forall h\in H_{i}, i\in I \rangle$$
and
$$M_{2} = \langle\ \breve{G_{2}}, r_{j}:j\in J | \mathcal{R}(G_{2}), h r_{j} = r_{j} \rho_{j}^{\prime}(h) \forall h\in H_{j}^{\prime}, j\in J \rangle$$
be the associated monoid HNN-extensions. Let $K$ be a group and let $\alpha_{1}:K\rightarrow M_{1}$, $\alpha_{2}:K\rightarrow M_{2}$ be injective homomorphisms. Then
$M_{1}{\ast}_{K} M_{2}$
is a left Rees monoid. Further $U(M_{1}{\ast}_{K} M_{2}) \cong U(M_{1})\ast_{K} U(M_{2})$.
\end{proposition}

\begin{proof}
Observe that $\alpha_{1}(K)\subseteq G_{1}$ and $\alpha_{2}(K)\subseteq G_{2}$. Since $K$ is a unitary subsemigroup of $M_{1}$ and $M_{2}$ it follows (\cite{Howie62}) that $M_{1}$ and $M_{2}$ embed in $M_{1}{\ast}_{K} M_{2}$. We have
\begin{comment}
are unitary subsemigroups of $M_{1}$ and $M_{2}$ respectively. We can therefore apply Howie's theorem to give
\end{comment}
$$M_{1}{\ast}_{K} M_{2} = \langle\ \breve{G_{1}}, \breve{G_{2}},t_{i}:i\in I,r_{j}:j\in J | \mathcal{R}(G_{1}), \mathcal{R}(G_{2}), h t_{i} = t_{i} \rho_{i}(h) \forall h\in H_{i}, i\in I,$$
$$ h r_{j} = r_{j} \rho_{j}^{\prime}(h) \forall h\in H_{j}^{\prime}, j\in J, \alpha_{1}(g) = \alpha_{2}(g) \forall g\in K \rangle\ .$$
Now let $G = G_{1}\ast_{K} G_{2}$ so that $G$ is given by the following group presentation
$$G = \langle\ G_{1}, G_{2}|\mathcal{R}(G_{1}),\mathcal{R}(G_{2}), \alpha_{1}(g) = \alpha_{2}(g) \forall g\in K \rangle\ .$$
We can therefore write 
$$M_{1}{\ast}_{K} M_{2} \cong \langle\ \breve{G},t_{i}:i\in I,r_{j}:j\in J | \mathcal{R}(G), h t_{i} = t_{i} \rho_{i}(h) \forall h\in H_{i}, i\in I,$$
$$ h r_{j} = r_{j} \rho_{j}^{\prime}(h) \forall h\in H_{j}^{\prime}, j\in J\rangle\ .$$
We then see that $M_{1}{\ast}_{K} M_{2}$ is a monoid HNN-extension of $G$ with associated subgroups $H_{i}:i\in I$ and $H_{j}^{\prime}:j\in J$. Thus $M_{1}{\ast}_{K} M_{2}$ is a left Rees monoid. Note that
$$U(M_{1})\cong \langle\ \breve{G_{1}}, t_{i}:i\in I | \mathcal{R}(G_{1}), h t_{i} = t_{i} \rho_{i}(h) \forall h\in H_{i}, i\in I \rangle$$
and
$$U(M_{2}) = \langle\ \breve{G_{2}}, r_{j}:j\in J | \mathcal{R}(G_{2}), h r_{j} = r_{j} \rho_{j}^{\prime}(h) \forall h\in H_{j}^{\prime}, j\in J \rangle ,$$
where these are now group presentations. So
$$U(M_{1})\ast_{K} U(M_{2}) \cong \langle\ \breve{G_{1}}, \breve{G_{2}},t_{i}:i\in I,r_{j}:j\in J | \mathcal{R}(G_{1}), h t_{i} = t_{i} \rho_{i}(h) \forall h\in H_{i}, i\in I,$$
$$\mathcal{R}(G_{2}), h r_{j} = r_{j} \rho_{j}^{\prime}(h) \forall h\in H_{j}^{\prime}, j\in J, \alpha_{1}(g) = \alpha_{2}(g) \forall g\in K \rangle\  \cong U(M_{1}{\ast}_{K} M_{2}).$$
\end{proof}

To demonstrate the above theory, let us now consider an example. Let $G = \mathbb{Z}\times\mathbb{Z}$, $H = 2\mathbb{Z}\times2\mathbb{Z}$, an index 4 subgroup of $G$, and let $\rho:H\rightarrow G$ be given by
$$\rho(2m,2n) = (m,3n),$$
for $m,n\in \mathbb{Z}$. We see that this is a monomorphism and so we can therefore define an associated group HNN-extension $\Gamma$ of $G$ on a single stable letter $t$ given as a group presentation by
$$\Gamma = \langle\ a,b,t \quad | \quad ab = ba, a^{2}t = ta, b^{2}t = tb^{3}\rangle\ $$
by noting that 
$$\mathbb{Z}\times\mathbb{Z}\cong \langle\ a,b \quad | \quad ab = ba \rangle\ ,$$
where we identify $(1,0)$ with $a$ and $(0,1)$ with $b$.
We see that $\Gamma$ is the group of fractions of the Rees monoid $M$ with monoid presentation
$$M = \langle\ a,a^{-1},b,b^{-1},t \quad |\quad ab = ba, aa^{-1} = a^{-1}a = bb^{-1} = b^{-1}b = 1, a^{2}t = ta, b^{2}t = tb^{3}\rangle\ .$$
Since $|G:H| = 4$, the monoid $M$ has 4 maximal proper principal right ideals so that $M \cong X^{\ast}\bowtie G$ for some $X$ with 4 elements. Observe that $G = H \cup aH \cup bH \cup abH$. Let $x_{1},\ldots,x_{4}$ be defined by
$$x_{1} = t, \quad x_{2} = at, \quad x_{3} = bt, \quad x_{4} = abt$$
and let $X = \left\{x_{1},x_{2},x_{3},x_{4}\right\}$.
We define a self-similar group action of $G$ on $X$ as follows:
$$a \cdot x_{1} = x_{2}, \quad a \cdot x_{2} = x_{1}, \quad b \cdot x_{1} = x_{3}, \quad b\cdot x_{3} = x_{1},$$ 
$$a \cdot x_{3} = x_{4}, \quad a \cdot x_{4} = x_{3}, \quad b \cdot x_{2} = x_{4}, \quad b\cdot x_{4} = x_{2},$$ 
$$a|_{x_{1}} = b|_{x_{1}} = a|_{x_{3}} = b|_{x_{2}} = 1,$$
$$a|_{x_{2}} = a|_{x_{4}} = a$$
and
$$b|_{x_{3}} = b|_{x_{4}} = b^{3}.$$
Note that since $G$ is abelian, there won't be any partial endomorphisms conjugate to $\rho$.

\newpage
\section{Symmetric Rees monoids}

We will say that a left Rees monoid $M = X^{\ast}\bowtie G$ is \emph{symmetric} if the functions $\rho_{x}:G\rightarrow G$ defined in Lemma \ref{usefullemma1} are bijective for every $x\in X$. 

Let $X$ be a set. We will denote by $FG(X)$ the free group on $X$. The Zappa-Sz\'ep product is defined for any monoid $S$ and group $G$ by replacing $x\in X^{*}$ with $s\in S$ in the self-similarity axioms. A natural question now arises: when is it possible to extend a self-similar action of a group $G$ on a free monoid $X^{\ast}$ to an action of $G$ on $FG(X)$ such that $X^{\ast}\bowtie G\leq FG(X)\bowtie G$? The next theorem will give us the necessary and sufficient condition for this to be the case. 

\begin{thm}
\label{symembed}
Let $M=X^{*}\bowtie G$ be a left Rees monoid. Then the Zappa-Sz\'ep product $X^{*} \bowtie G$ can be extended to a Zappa-Sz\'ep product $FG(X) \bowtie G$ respecting the actions if and only if $M$ is symmetric.
\end{thm}

\begin{proof}
($\Rightarrow$) Suppose for a left Rees monoid $M=X^{*}\bowtie G$ the Zappa-Sz\'ep product $\Gamma=FG(X) \bowtie G$ exists such that $M$ is a submonoid of $\Gamma$. We need to show that $\rho_{x}$ is bijective for all $x\in X^{*}$. Let $x,y\in X^{*}$ and $g\in G$. Then since (SS6) says $g|_{xy}=(g|_{x})|_{y}$, we have
$$g = g|_{1} = g|_{x^{-1} x}=(g|_{x^{-1}})|_{x} \quad (1)$$
$$g = g|_{1} = g|_{x x^{-1}}=(g|_{x})|_{x^{-1}} \quad (2)$$
Letting $h=g|_{x^{-1}}$, (1) implies that for every $x\in X^{*}$ and $g\in G$ there exists an $h\in G$ such that $h|_{x}=g$, and so $\rho_{x}$ is surjective for every $x\in X^{*}$. Now suppose $g|_{x}=h|_{x}$. Then (2) implies, upon restriction to $x^{-1}$, that $g=h$, and so $\rho_{x}$ is injective. 

($\Leftarrow$) Let $M=X^{*}G$ be a symmetric left Rees monoid. For $x\in X^{\ast}$, $g\in G$, define $(\rho_{x}\circ \rho_{y})(g) = \rho_{y}(\rho_{x}(g))$. Axiom (SS6) tells us that the map $\rho:X^{\ast}\rightarrow S_{G}$ given by $\rho(x) = \rho_{x}$ is a monoid homomorphism. For $x\in X$, $g\in G$ define
$$g|_{x^{-1}}:=\rho^{-1}_{x}(g).$$
This is well defined since $\rho$ is injective. Now extend the restriction to $g|_{x}$ for $x\in FG(X)$ by using rule (SS6): $$g|_{x_{1}^{\epsilon_{1}}x_{2}^{\epsilon_{2}}\ldots x_{n}^{\epsilon_{n}}}=((g|_{x_{1}^{\epsilon_{1}}})|_{x_{2}^{\epsilon_{2}}})\ldots|_{x_{n}^{\epsilon_{n}}} 
\quad x_{i}\in X, \epsilon_{i}=\pm 1.$$
The preceding remarks tell us that this definition makes sense. Now for $x\in X$, $g\in G$ define
$$g\cdot x^{-1}:=(g|_{x^{-1}}\cdot x)^{-1}.$$
For $x,y\in FG(X)$, define 
$$g\cdot xy:=(g\cdot x)(g|_{x}\cdot y).$$
To see that this is morally the correct definition, let us check that for all $x\in X^{\ast}$, $g\in G$ we have
$$g\cdot(x^{-1}) = (g|_{x^{-1}}\cdot x)^{-1}.$$
We will prove this claim by induction. By definition the claim is true for $|x| = 1$, i.e. $x\in X$. So let us assume that this holds for all $x\in X^{\ast}$ with $|x|\leq n$ for some $n\in \mathbb{N}$. Suppose $z = yx$ where $|y|,|x|\leq n$ and let $g\in G$ be arbitrary. Then
$$g\cdot(z^{-1}) = g\cdot (yx)^{-1} = (g\cdot x^{-1})(g|_{x^{-1}}\cdot y^{-1}).$$
First, let $k=g|_{(yx)^{-1}}$. Applying the rules,
$$(g\cdot x^{-1})(g|_{x^{-1}}\cdot y^{-1})=(g|_{x^{-1}}\cdot x)^{-1}((g|_{x^{-1}})|_{y^{-1}}\cdot y)^{-1}=(g|_{x^{-1}}\cdot x)^{-1}(k\cdot y)^{-1}.$$
Then,
$$(g|_{x^{-1}}\cdot x)^{-1}(k\cdot y)^{-1}=((k\cdot y)(g|_{x^{-1}}\cdot x))^{-1}=((k\cdot y)(k|_{y}\cdot x))^{-1}.$$
But now we can use (SS4) for $M$ to get
$$((k\cdot y)(k|_{y}\cdot x))^{-1}=(k\cdot (yx))^{-1}=(g|_{(yx)^{-1}}\cdot (yx))^{-1} = (g|_{z^{-1}}\cdot z)^{-1},$$
and the claim is proved.

We now need to show that the above definitions taken together give us a well-defined group action of $G$ and $FG(X)$ satisfying axioms (SS1)-(SS8). Note that in our definition, we are assuming (SS4) and (SS6). 
\begin{comment}

Checking (SS6) and (SS4) will be checking that these definitions are well-defined. 
\\(SS6) We want to show that $\rho_{xy}^{-1}(g)=\rho_{x}^{-1}(\rho_{y}^{-1}(g))$. Since $\rho_{xy}$, $\rho_{x}$ and $\rho_{y}$ are bijective, there exist $h\in G$ such that $h|_{xy}=g$, $k\in G$ such that $k|_{y}=g$ and $u\in G$ such that $u|_{x}=k$. We see that $\rho_{xy}^{-1}(g)=h$ and $\rho_{x}^{-1}(\rho_{y}^{-1}(g))=u$. We have to show that $u=h$. Since $(h|_{x})|_{y}=g=k|_{y}$, injectivity of $\rho_{y}$ implies $k=h|_{x}$. Since $u|_{x}=k=h|_{x}$, injectivity of $\rho_{x}$ then implies that $u=h$ and we are done.
\\(SS4) We want to show for every $x,y\in X^{*}$ and $g\in G$ that 
\end{comment}
\begin{description}
\item[{\rm (SS3) and (SS5)}] only involve $M$ and so are true.
\item[{\rm (SS7)}] For $x\in X$, $1|_{x^{-1}}=\rho_{x}^{-1}(1)=1$, and so it follows by (SS6) for all $x\in FG(X)$.
\item[{\rm (SS1)}] For $x\in X$, $1\cdot x^{-1} = (1|_{x^{-1}}\cdot x)^{-1} = (1 \cdot x)^{-1} = x^{-1}$ , and so it follows by (SS4) for all $x\in FG(X)$.
\item[{\rm (SS8)}] We need to show that for every $x\in X$ and $g,h\in G$ 
$$(gh)|_{x^{-1}}=g|_{(h\cdot x^{-1})}h|_{x^{-1}}.$$ 
First note $g|_{(h\cdot x^{-1})} = g|_{(h|_{x^{-1}}\cdot x)^{-1}}$. We will in fact show that $\rho_{x}((gh)|_{x^{-1}}) = \rho_{x}(g|_{(h\cdot x^{-1})}h|_{x^{-1}})$ and the result will follow since $\rho_{x}$ is a bijection. So,
$$\rho_{x}(g|_{(h|_{x^{-1}}\cdot x)^{-1}}h|_{x^{-1}})=(g|_{(h|_{x^{-1}}\cdot x)^{-1}}h|_{x^{-1}})|_{x} = (g|_{(h|_{x^{-1}}\cdot x)^{-1}(h|_{x^{-1}}\cdot x)})(h|_{x^{-1}x}),$$ using (SS8) for $M$. But this is simply
$$g|_{1} h|_{1} = gh = \rho_{x}((gh)|_{x^{-1}}).$$
The result holds for $x\in FG(X)$ by (SS4) and (SS6).
\item[{\rm (SS2)}] We need to show for every $g,h\in G$ and $x\in X$ that $(gh)\cdot x^{-1} = g\cdot(h\cdot x^{-1})$. So,
$$g\cdot(h\cdot x^{-1}) = g\cdot(h|_{x^{-1}}\cdot x)^{-1} = (g|_{(h|_{x^{-1}}\cdot x)^{-1}}\cdot(h|_{x^{-1}}\cdot x))^{-1} $$
$$= ((g|_{(h|_{x^{-1}}\cdot x)^{-1}}h|_{x^{-1}})\cdot x)^{-1} = ((g|_{h\cdot x^{-1}}h|_{x^{-1}})\cdot x)^{-1} = ((gh)|_{x^{-1}}\cdot x)^{-1}.$$
But this is just the definition of $(gh)\cdot x^{-1}$.
\end{description}
\end{proof}

\begin{comment}
\begin{lemma}
\label{univalg1}
Let $\Gamma_{1}=\langle X|\mathcal{R},\mathcal{S}\rangle$ and $\Gamma_{2}=\langle X|\mathcal{R}\rangle$ be two groups given by group presentation. If $f:\Gamma_{1}\rightarrow \Gamma_{2}$ is a homomorphism mapping generators of $\Gamma_{1}$ to generators of $\Gamma_{2}$ (so $f(x)=x$ for every $x\in X$), then $f$ is in fact an isomorphism.
\end{lemma}

\begin{proof}
\par Let us first check injectivity. Suppose $f(x_{1}^{\pm 1}\cdots x_{n}^{\pm 1})=f(y_{1}^{\pm 1}\cdots y_{k}^{\pm 1})$ for $x_{i}, y_{j} \in X$. Then $x_{1}^{\pm 1}\cdots x_{n}^{\pm 1}=y_{1}^{\pm 1}\cdots y_{k}^{\pm 1}$ in $\Gamma_{2}$, but since there are no more relations in $\Gamma_{2}$ than $\Gamma_{1}$, $x_{1}^{\pm 1}\cdots x_{n}^{\pm 1}=y_{1}^{\pm 1}\cdots y_{k}^{\pm 1}$ in $\Gamma_{1}$ and we are done.
\par Now let us check surjectivity. Let $g=x_{1}^{\pm 1}\ldots x_{n}^{\pm 1}$ in $\Gamma_{2}$. Then 
$$g=f(x_{1}^{\pm 1})\cdots f(x_{n}^{\pm 1}).$$
But since $f$ is a homomorphism, we then have 
$$g=f(x_{1}^{\pm 1}\cdots x_{n}^{\pm 1})$$
and we are done.
\end{proof}
\end{comment}

We will now show that if $M = X^{\ast}\bowtie G$ is a symmetric Rees monoid then the group of fractions of $M$ is isomorphic to the extension $FG(X)\bowtie G$ described in Theorem \ref{symembed}.

\begin{thm}
\label{univgroupsym}
Let $M=X^{*}G$ be a symmetric Rees monoid. Then the group of fractions of $M$ is isomorphic to a Zappa-Sz\'{e}p product of the free group on $X$ and $G$. That is,
$$U(M)\cong FG(X)\bowtie G.$$
\end{thm}

\begin{proof}
Let $\left\{x_{i}:i\in I\right\}$ be a set of representatives for orbits of $X$ (where $|X| = |I|$), let $H_{i}^{1} = G_{x_{i}}$ be the stabiliser of $x_{i}$, let $\rho_{i} = \rho_{x_{i}}$, let $T_{i}^{1}$ be a transversal of left coset representatives for $H_{i}^{1}$, let $H_{i}^{-1} = \rho_{i}(H_{i}^{1})$ and let $T_{i}^{-1}$ be a transversal of left coset representatives for $H_{i}^{-1}$. Assume that $1\in T_{i}^{1}$ and $1\in T_{i}^{-1}$ for each $i\in I$.

For each $i\in I$, $\epsilon \in \left\{-1,1\right\}$, define maps $\beta_{i,\epsilon}:G\rightarrow G$ by 
$\beta_{i,1}(g) = (\rho_{i}(g))^{-1}$ and $\beta_{i,-1}(g) = (\rho_{i}^{-1}(g))^{-1}$. Since $\rho_{i}:G\rightarrow G$ is a bijection for each $i\in I$, this latter map is well-defined.

By Theorem \ref{mainthm1}, $M$ is isomorphic to the following monoid presentation:
$$M \cong \langle\ \breve{G}, t_{i}:i\in I| \mathcal{R}(G), ht_{i} = t_{i}\rho_{i}(h), h\in H_{i}^{1}, i\in I \rangle\ .$$

It follows that 
$$U(M) \cong \langle\ \breve{G}, t_{i}:i\in I| \mathcal{R}(G), ht_{i} = t_{i}\rho_{i}(h), h\in H_{i}^{1}, i\in I \rangle\ ,$$
where here we are working with a group presentation.

We know from Proposition \ref{normalformgroupoidHNN} that every element of $U(M)$ can be uniquely written in the form
$$g = g_{1}t_{i_{1}}^{\epsilon_{1}}g_{2}t_{i_{2}}^{\epsilon_{2}}\cdots g_{m}t_{i_{m}}^{\epsilon_{m}}u,$$
where $\epsilon_{k}\in \left\{-1,1\right\}$, $g_{k}\in T_{i_{k}}^{\epsilon_{k}}$, $u\in G$ is arbitrary all subject to the condition that if $t_{i_{k}} = t_{i_{k+1}}$ and $\epsilon_{k} + \epsilon_{k+1} = 0$ then $g_{k+1}$ is not an identity. We call this the \emph{Britton normal form}.

We claim that every element of $U(M)$ can in fact be uniquely written in the form
$$g = g_{1}t_{i_{1}}^{\epsilon_{1}}\beta_{i_{1},\epsilon_{1}}(g_{1})g_{2}t_{i_{2}}^{\epsilon_{2}}\beta_{i_{2},\epsilon_{2}}(g_{2})\cdots g_{m}t_{i_{m}}^{\epsilon_{m}}\beta_{i_{m},\epsilon_{m}}(g_{m})u,$$
where $\epsilon_{k}\in \left\{-1,1\right\}$, $g_{k}\in T_{i_{k}}^{\epsilon_{k}}$, $u\in G$ is arbitrary all subject to the condition that 
$$g_{k}t_{i_{k}}^{\epsilon_{k}}\beta_{i_{k},\epsilon_{k}}(g_{k}) \neq (g_{k+1}t_{i_{k+1}}^{\epsilon_{k+1}}\beta_{i_{k+1},\epsilon_{k+1}}(g_{k+1}))^{-1}$$
for any $k$. An element in such a form will be said to be in \emph{Rees normal form}. Observe that part of our claim is that the elements $gt_{i}^{\pm{1}}\beta_{i,\pm{1}}(g)$ generate a free subgroup of $U(M)$.

Let us first show that every element of $U(M)$ can be written in such a form. Let $g = g_{1}t_{i_{1}}^{\epsilon_{1}}\cdots g_{m}t_{i_{m}}^{\epsilon_{m}}u$ be an arbirtary element of $U(M)$ written in Britton normal form. There exist unique elements $g_{2}^{\prime}\in T_{i_{2}}^{\epsilon_{2}}$, $h_{2}\in H_{i_{2}}^{\epsilon_{2}}$ with $g_{2}^{\prime}h_{2} = (\beta_{i_{1},\epsilon_{1}}(g_{1}))^{-1}g_{2}$. We then define $g_{k}^{\prime}\in T_{i_{k}}^{\epsilon_{k}}$, $h_{k}\in H_{i_{k}}^{\epsilon_{k}}$ inductively for $3\leq k\leq m$ to be the unique elements with 
$$g_{k}^{\prime}h_{k} = (\beta_{i_{k-1},\epsilon_{k-1}}(g_{k-1}^{\prime}))^{-1}\rho_{i_{k-1}}(h_{k-1})g_{k}$$
if $\epsilon_{k-1} = 1$ and 
$$g_{k}^{\prime}h_{k} = (\beta_{i_{k-1},\epsilon_{k-1}}(g_{k-1}^{\prime}))^{-1}\rho_{i_{k-1}}^{-1}(h_{k-1})g_{k}$$
if $\epsilon_{k-1} = -1$. Finally we let $u^{\prime} = (\beta_{i_{m},\epsilon_{m}}(g_{m}^{\prime}))^{-1}\rho_{i_{m}}(h_{m})u$ if $\epsilon_{m} = 1$ and $u^{\prime} = (\beta_{i_{m},\epsilon_{m}}(g_{m}^{\prime}))^{-1}\rho_{i_{m}}^{-1}(h_{m})u$ if $\epsilon_{m} = -1$. One then finds that
$$g = g_{1}t_{i_{1}}^{\epsilon_{1}}\beta_{i_{1},\epsilon_{1}}(g_{1})g_{2}^{\prime}t_{i_{2}}^{\epsilon_{2}}\beta_{i_{2},\epsilon_{2}}(g_{2}^{\prime})\cdots g_{m}^{\prime}t_{i_{m}}^{\epsilon_{m}}\beta_{i_{m},\epsilon_{m}}(g_{m}^{\prime})u^{\prime}.$$
One then reduces if possible by cancelling inverses so that $g$ is in Rees normal form.

Now suppose that 
$$g_{1}t_{i_{1}}^{\epsilon_{1}}\beta_{i_{1},\epsilon_{1}}(g_{1})\cdots g_{m}t_{i_{m}}^{\epsilon_{m}}\beta_{i_{m},\epsilon_{m}}(g_{m})u 
= g_{1}^{\prime}t_{j_{1}}^{\delta_{1}}\beta_{j_{1},\delta_{1}}(g_{1}^{\prime})\cdots g_{n}^{\prime}t_{j_{n}}^{\delta_{n}}\beta_{j_{n},\delta_{n}}(g_{n}^{\prime})v$$
where these are both in Rees normal form and assume $n\leq m$. Then since the Britton normal form is a unique normal form and by our reduction method in Proposition \ref{normalformgroupoidHNN} it follows that $g_{1} = g_{1}^{\prime}$ and $t_{i_{1}}^{\epsilon_{1}} = t_{j_{1}}^{\delta_{1}}$. We therefore cancel to get
$$g_{2}t_{i_{2}}^{\epsilon_{2}}\beta_{i_{2},\epsilon_{2}}(g_{2})\cdots g_{m}t_{i_{m}}^{\epsilon_{m}}\beta_{i_{m},\epsilon_{m}}(g_{m})u 
= g_{2}^{\prime}t_{j_{2}}^{\delta_{2}}\beta_{j_{2},\delta_{2}}(g_{2}^{\prime})\cdots g_{n}^{\prime}t_{j_{n}}^{\delta_{n}}\beta_{j_{n},\delta_{n}}(g_{n}^{\prime})v.$$
We then continue in this way to find 
$$g_{n+1}t_{i_{n+1}}^{\epsilon_{n+1}}\beta_{i_{n+1},\epsilon_{n+1}}(g_{n+1})\cdots g_{m}t_{i_{m}}^{\epsilon_{m}}\beta_{i_{m},\epsilon_{m}}(g_{m})u = v.$$
Suppose $n+1<m$. It then follows that there exists $k$, $n+1\leq k<m$, such that $$g_{k}t_{i_{k}}^{\epsilon_{k}}\beta_{i_{k},\epsilon_{k}}(g_{k})g_{k+1}t_{i_{k+1}}^{\epsilon_{k+1}}\beta_{i_{k+1},\epsilon_{k+1}}(g_{k+1}) \in G.$$
This means that $i_{k} = i_{k+1} = j$ for some $j\in I$ and $\epsilon_{k} + \epsilon_{k+1} = 0$. There are two possibilities: either $\epsilon_{k} = 1$ and $\epsilon_{k+1} = -1$ or $\epsilon_{k} = -1$ and $\epsilon_{k+1} = 1$. Suppose first that $\epsilon_{k} = 1$. Then we are saying that
$$g_{k}t_{j}(\rho_{j}(g_{k}))^{-1}g_{k+1}t_{j}^{-1}(\rho_{j}^{-1}(g_{k+1}))^{-1}\in G$$
with $g_{k}\in T_{j}^{1}$ and $g_{k+1}\in T_{j}^{-1}$. Then $(\rho_{j}(g_{k}))^{-1}g_{k+1} = \rho_{j}(h)$ for some $h\in H_{j}^{1}$. Thus
\begin{eqnarray*}
g_{k}t_{j}(\rho_{j}(g_{k}))^{-1}g_{k+1}t_{j}^{-1}(\rho_{j}^{-1}(g_{k+1}))^{-1} 
&=& g_{k}t_{j}\rho_{j}(h)t_{j}^{-1}(\rho_{j}^{-1}(g_{k+1}))^{-1} \\
&=& g_{k}t_{j}t_{j}^{-1}h(\rho_{j}^{-1}(g_{k+1}))^{-1} \\
&=& g_{k}h(\rho_{j}^{-1}(\rho_{j}(g_{k})\rho_{j}(h)))^{-1} \\
&=& g_{k}h(\rho_{j}^{-1}(\rho_{j}(g_{k}h)))^{-1}\\
&=& g_{k}h(g_{k}h)^{-1} = 1.
\end{eqnarray*}
This contradicts the assumption that our initial word was in Rees normal form and so $n = m$. It follows that $g_{k} = g_{k}^{\prime}$, $i_{k} = j_{k}$, $\delta_{k} = \epsilon_{k}$ for each $k$ and $u=v$. Now suppose $\epsilon_{k} = -1$. Then we have
$$g_{k}t_{j}^{-1}(\rho_{j}^{-1}(g_{k}))^{-1}g_{k+1}t_{j}(\rho_{j}(g_{k+1}))^{-1}\in G$$
with $g_{k}\in T_{j}^{-1}$ and $g_{k+1}\in T_{j}^{1}$. Then $(\rho_{j}^{-1}(g_{k}))^{-1}g_{k+1} = h$ for some $h\in H_{j}^{1}$. Thus
\begin{eqnarray*}
g_{k}t_{j}^{-1}(\rho_{j}^{-1}(g_{k}))^{-1}g_{k+1}t_{j}(\rho_{j}(g_{k+1}))^{-1} 
&=& g_{k}t_{j}^{-1}ht_{j}(\rho_{j}(g_{k+1}))^{-1} \\
&=& g_{k}t_{j}^{-1}t_{j}\rho_{j}(h)(\rho_{j}(g_{k+1}))^{-1}\\
&=& g_{k}\rho_{j}(h)(\rho_{j}(\rho_{j}^{-1}(g_{k})h))^{-1} \\
&=& g_{k}\rho_{j}(h)(g_{k}\rho_{j}(h))^{-1} = 1.
\end{eqnarray*}
Again this contradicts the assumption that our initial word was in Rees normal form and so $n = m$. It follows that $g_{k} = g_{k}^{\prime}$, $i_{k} = j_{k}$, $\delta_{k} = \epsilon_{k}$ for each $k$ and $u=v$.

We have shown that the Rees normal form is a unique normal form for elements of $U(M)$. Let us now consider the monoid $M$. Recall that every element $x \in X$ can be uniquely written in the form $x = gx_{i}(\rho_{i}(g))^{-1}$ for some $i\in I$ and $g\in T_{i}^{1}$. We will now show that every element $x\in X$ can be written uniquely as $x = \rho_{i}^{-1}(g)x_{i}g^{-1}$ with $i\in I$, $g\in T_{i}^{-1}$.

First, let $x\in X$. Then $x = gx_{i}(\rho_{i}(g))^{-1}$ for unique $i\in I$ and $g\in T_{i}^{1}$.  We can write $\rho_{i}(g) = g|_{x_{i}}$ uniquely in the form $g|_{x_{i}} = uh|_{x_{i}}$ where $u\in T_{i}^{-1}$ and $h\in H_{i}^{1}$. Now
\begin{eqnarray*}
x &=& gx_{i}(h|_{x_{i}})^{-1}u^{-1} = gx_{i}(h^{-1})|_{x_{i}}u^{-1} = gh^{-1}x_{i}u^{-1} = \rho_{i}^{-1}((gh^{-1})|_{x_{i}})x_{i}u^{-1} \\
&=& \rho_{i}^{-1}(g|_{x_{i}}h^{-1}|_{x_{i}})x_{i}u^{-1} = \rho_{i}^{-1}(u)x_{i}u^{-1}.
\end{eqnarray*}
Now let $i\in I$, $g\in T_{1}^{-1}$. We will show that $\rho_{i}^{-1}(g)x_{i}g^{-1}\in X$. Let $u\in T_{i}^{1}$, $h\in H_{i}^{1}$ be the unique elements with $uh = \rho_{i}^{-1}(g)$. Then $\rho_{i}(uh) = g$ and so $g = \rho_{i}(u)\rho_{i}(h)$.  Then
\begin{eqnarray*}
\rho_{i}^{-1}(g)x_{i}g^{-1} &=& uhx_{i}g^{-1} = ux_{i}\rho_{i}(h)g^{-1} = ux_{i}(gh|_{x_{i}}^{-1})^{-1} \\
&=& ux_{i}(\rho_{i}(u)h|_{x_{i}}h|_{x_{i}}^{-1})^{-1} =  ux_{i}(\rho_{i}(u))^{-1}\in X.
\end{eqnarray*}
Finally let $\rho_{i}^{-1}(g_{1})x_{i}g_{1}^{-1} = \rho_{j}^{-1}(g_{2})x_{j}g_{2}^{-1}$ in $X$ with $i,j\in I$, $g_{1}\in T_{i}^{-1},g_{2}\in T_{j}^{-1}$. First, since the $x_{i}$'s are representatives of orbits, it follows that $i = j$. Now suppose $\rho_{i}^{-1}(g_{1}) = u_{1}h_{1}$ and $\rho_{i}^{-1}(g_{2}) = u_{2}h_{2}$ for $u_{1},u_{2}\in T_{i}^{1}$, $h_{1},h_{2}\in H_{i}^{1}$. We must have $u_{1} = u_{2}$. We therefore have
$$\rho_{i}(u_{1}) = g_{1}\rho_{i}(h_{1}^{-1}) = g_{2}\rho_{i}(h_{2}^{-1}).$$
Since $g_{1},g_{2}\in T_{i}^{-1}$ and $\rho_{i}(h_{1}^{-1}),\rho_{i}(h_{2}^{-1})\in H_{i}^{-1}$ it follows by the unique decomposition of elements into the product of a coset representative and an element of a subgroup that $g_{1} = g_{2}$. Thus, every element $x\in X$ can be written uniquely as $x = \rho_{i}^{-1}(g)x_{i}g^{-1}$ with $i\in I$, $g\in T_{i}^{-1}$.

Since $FG(X)\bowtie G$ is generated by elements of the form $(1,g)$ for $g\in G$ and $(x_{i},1)$ for $i\in I$ we see that we can write $FG(X)\bowtie G$ in terms of a group presentation as
$$FG(X)\bowtie G \cong \langle\ \breve{G}, x_{i}:i\in I|\mathcal{R}(G), hx_{i} = x_{i}\rho_{i}(h), \mathcal{S} \rangle\ ,$$
where $\mathcal{S}$ is some set of extra relations which are needed to make this really a presentation for $FG(X)\bowtie G$. It follows that there is a surjective homomorphism $f:U(M)\rightarrow FG(X)\bowtie G$ given on generators by $f(g) = (1,g)$ for $g\in G$ and $f(t_{i}) = (x_{i},1)$ for $i\in I$. All of the above argument tells us that two elements of $U(M)$ written in Rees normal form map to the same elements in $FG(X)\bowtie G$ under $f$ if and only if they are equal. Thus $f$ is also injective.  
\end{proof}  

Let $M = X^{*} \bowtie G$ be a left Rees monoid. We will call $X$ a \emph{basis} for $M$. If $Y$ is such that $M\cong Y^{*} \bowtie G$, then we will say $Y$ is a \emph{change of basis} of $X$. 

\begin{lemma}
\label{basis1}
Let $M = X^{*} G$ be a left Rees monoid such that the action of $G$ on $X$ is transitive and $G$ is finite. If, for some $x\in X$, $\rho_{x}$ is bijective, then $\rho_{y}$ is bijective for all $y\in X^{*}$.
\end{lemma}

\begin{proof}
Let $y\in X$ and suppose $y = g\cdot x$ for some $g\in G$. Suppose $\rho_{y}$ is not injective. Then by Lemma \ref{usefullemma1} (vi) there exists $h\in G$ with $h|_{y} = 1$ and $h\neq 1$. Then 
$$(hg)|_{x} = h|_{g\cdot x} g|_{x} = g|_{x}.$$
But by assumption, $\rho_{x}$ was injective, and thus $h = 1$, a contradiction. An injective map from a finite set into itself must also be surjective and thus $\rho_{y}$ must be bijective. It then follows by Lemma \ref{usefullemma1} (v) and (viii) that $\rho_{y}$ is bijective for all $y\in X^{*}$.
\end{proof}

\begin{comment}

The following follows from the right cancellativity of Rees monoids, but we will prove it a slightly different way anyway.

\begin{lemma}
\label{basis2}
Let $M = X^{*} G$ be a Rees monoid. If for some $x\in X$, $g,h\in G$, $gx = hx$, then $g = h$.
\end{lemma}

\begin{proof}
Suppose $gx = hx$. Then $gh^{-1}hx = hx$ and so $gh^{-1} \in G_{h\cdot x}$. Now
$$(gh^{-1})|_{h\cdot x} = (g|_{x})(h|_{x})^{-1} = 1.$$
Since $\phi_{h\cdot x}$ is injective, it follows $gh^{-1} = 1$ and so $g=h$.
\end{proof}
\end{comment}

\begin{proposition}
\label{basis2}
Let $M = X^{*} G$ be a Rees monoid with $G$ finite. Then there exists a change of basis $Y$ of $X$ such that $M \cong Y^{*} \bowtie G$ is a symmetric Rees monoid.
\end{proposition}

\begin{proof}
In what follows, we will be working with orbits of elements and so without loss of generality let us assume the action of $G$ on $X$ is transitive. Let $x\in X$. We know $\phi_{x}$ is injective. We will form a change of basis such that $\rho_{x}$ is injective. So suppose $g,h\in G$ are such that $g|_{x} = h|_{x}$. By the right cancellativity of $M$, we know $g\cdot x \neq h\cdot x$. Suppose $y = g\cdot x$ and suppose $k\notin im(\rho_{x})$. Let $y^{\prime} = y (g|_{x}) k^{-1}$. Then
$$g x = y g|_{x} = y^{\prime} k.$$ 
So changing $y$ to $y^{\prime}$, we have $g\cdot x = y^{\prime}$ and $g|_{x} = k \neq h|_{x}$. Repeat this process for each $g\in G$ and we will have constructed a change of basis so that $\rho_{x}$ is bijective, and thus by Lemma \ref{basis1} the theorem has been proven. 
\end{proof}

Combining Theorem \ref{univgroupsym} and Proposition \ref{basis2} we have the following:

\begin{corollary}
\label{ZSFGUM}
Let $\Gamma$ be a group HNN-extension of a finite group $G$. Then there is a set $X$ such that $\Gamma \cong FG(X) \bowtie G$.
\end{corollary}

If $M = X^{\ast}G$ is a Rees monoid with $\phi_{x}:G_{x}\rightarrow G$ bijective for each $x\in X^{\ast}$ (e.g. the adding machine Rees monoid described in Section 2.7.1) then one cannot use the same change of basis argument as Proposition \ref{basis2} to write $U(M)$ as the Zappa-Sz\'{e}p product of a free group and $G$ since every element of $G$ is in the image of $\phi_{x}$ for each $x\in X^{\ast}$. It therefore seems unlikely that Corollary \ref{ZSFGUM} will still be true in general if the finiteness assumption on $G$ is removed. On the other hand, if $M = X^{\ast}G$ is a Rees monoid such that $|G:G_{x}| = |G:\phi_{x}(G_{x})|$ for each $x\in X$ then one may be able adapt the argument of Proposition \ref{basis2} for this situation. 

\begin{comment}

\begin{thm}
Let $M = X^{*} \bowtie D_{n}$ be a (symmetric) Rees monoid. Then there exists a change of basis $Y$ of $X$ such that $\rho_{x}$ is an isomorphism for each $x\in Y$
\end{thm}

\begin{proof}
Let $D_{n}$ be given by the following group presentation
$$D_{n} = \langle \alpha, \beta | \alpha^{n} = \beta^{2} = 1, \alpha\beta = \beta\alpha^{n-1} \rangle.$$
By the above, assume wlog $\rho_{x}$ is bijective for each $x\in X$. We can perform a change of basis so that $\alpha|_{x} = \alpha$ for each $x\in Y$. Since $\rho_{x}$ is a bijection, it follows that $\beta|_{x} = \alpha^{i} \beta$ for some $i\in \mathbb{N}$. Fix $x\in Y$. We have to show that if $\rho_{x}(\alpha) = \alpha$ and $\rho_{x}(\beta) = \alpha^{k}\beta$, for any $g,h\in G$, 
$$\rho_{x}(gh) = \rho_{x}(g)\rho_{x}(h).$$
This is a straightforward calculation, by checking different cases of possible $g$'s and $h$'s.
\end{proof}

\end{comment}

\newpage
\section{Iterated function systems}

In this section we will provide examples of iterated function systems which give rise to fractals with a Rees monoid as similarity monoid. For undefined notions from fractal geometry see \cite{Falconer}.

Let $D$ be a compact subset of $\mathbb{R}^{k}$. A map $f:D\rightarrow D$ is a \emph{similarity contraction} if $f$ is continuous, injective and there exists a constant $0<c<1$ such that $d(f(x),f(y))=c d(x,y)$ for every $x,y\in D$. 

Let $M(D)$ denote the monoid of all similarity contractions and isometries of $D$ (where $(ab)(x)=a(b(x))$ for $a,b\in M(D)$ and $x\in D$). We will denote by $\dim_{H}(D)$ the Hausdorff dimension of $D$. Since injective maps are monics in the category \textbf{Top} this monoid $M(D)$ will be left cancellative. We will now investigate further this monoid $M(D)$.

\begin{comment}
\begin{lemma}
\label{fracleft}
Let $D\subseteq \mathbb{R}^{k}$ be compact and non-empty. Then $M(D)$ is left cancellative.
\end{lemma}

\begin{proof}
Follows from the fact that injective maps are monics in the category \bf{Top}
\end{proof}
\end{comment}

\begin{lemma}
\label{condim}
Let $D\subseteq \mathbb{R}^{k}$ be compact and let $a\in M(D)$. Then
$$\dim_{H}(D)=\dim_{H}(a(D))$$
\end{lemma}

\begin{proof}
Let $\delta > 0$. Denote by $|a|$ the contraction factor of $a$. Suppose $\left\{U_{i}\right\}$ is a $\delta$-cover for $D$. Then $\left\{a(U_{i})\right\}$ will be a $\delta/|a|$-cover of $a(D)$. It is clear that all coverings of $a(D)$ can be constructed in this manner. It therefore follows that for each $s$ we have $H^{s}(a(D))=\frac{1}{|a|^{s}}H^{s}(D)$ and so $\dim_{H}(D)=\dim_{H}(a(D))$.
\end{proof}

\begin{lemma}
\label{diml}
Let $D$ be a compact subset of $\mathbb{R}^{k}$ such that $\dim_{H}(D)>k-1$. Let $Y\subset D$ be such that for some $b,c\in M$, $b\neq c$, $b(x)=c(x)$ for all $x\in Y$. Then $\dim_{H}(Y)\leq k-1$
\end{lemma}

\begin{proof}
We will prove for the case when $k=2$. The argument can easily be generalised to the case $k\geq 3$ by working with $k-1$-dimensional hyperplanes. So let $D$ be a subset of $\mathbb{R}^{2}$ such that $\dim_{H}(D)>1$, let $Y\subseteq D$ be such that for some $b,c\in M$, $b(x)=c(x)$ for all $x\in Y$ and suppose that $\dim_{H}(Y) > 1$. Let $x,y,z \in Y$ and assume wlog that $x,y,z$ are not collinear (if all points in $Y$ are collinear then $\dim_{H}(Y)\leq 1$). Now let $T$ be the triangle in $\mathbb{R}^{2}$ with vertices at $x,y,z$. Then since $b$ and $c$ must have the same contraction factor, by length considerations, $b(t)=c(t)$ for every $t\in T\cap D$. It then follows since $b$ and $c$ are similarity transformations that $b(t)=c(t)$ for every $t\in D$. 

\end{proof}

\begin{lemma}
\label{fracright}
Let $D$ be a compact subset of $\mathbb{R}^{k}$ such that $\dim_{H}(D)>k-1$. Then $M(D)$ is right cancellative.
\end{lemma}

\begin{proof}
Suppose $a,b,c\in M$ are such that $ac=bc$. Let $Y = c(D)$. By Lemma \ref{condim} we have $\dim_{H}(Y) > k-1$ and $a(x) = b(x)$ for all $x\in Y$. It thus follows from Lemma \ref{diml} that $a=b$. 
\end{proof}

Let $D$ be a compact subset of $\mathbb{R}^{k}$. An \emph{iterated function system} (IFS) is a finite family of similarity contractions $f_{1},\ldots,f_{n}:D\rightarrow D$. Theorem 9.1 of \cite{Falconer} says that there is a unique non-empty compact subset $F$ of $D$ satisfying 
$$F=\bigcup_{i=1}^{n}{f_{i}(F)}$$
which we call the \emph{attractor} of $f_{1},\ldots,f_{n}$.

\begin{thm}
\label{mainthm2}
Let $D\subseteq \mathbb{R}^{k}$ be a compact path-connected subspace, let $f_{1},\ldots,f_{n}:D\rightarrow D$ be an IFS with attractor $F\subseteq D$, $d=\dim_{H}(F)$ and let $\mu = H^{d}$, the $d$-dimensional Hausdorff measure and assume $d>k-1$ and $0<\mu(F)<\infty$. Let $G$ be the group of isometries of $F$ and denote by $X=\left\{f_{1},\ldots,f_{n}\right\}$, $I=\left\{1,\ldots,n\right\}$ and suppose further the following:
\begin{enumerate}
\item [(i)] For every $i,j\in I$, we have $\mu(f_{i}(F)\cap f_{j}(F))=0$.
\item [(ii)] There are no contractions $h$ of $F$ such that $f_{i}(F)\subset h(F)$ for some $i\in I$.
\item [(iii)] For every $i\in I$, $g\in G$ there exists $j\in I$ such that $g(f_{i}(F))=f_{j}(F)$
\end{enumerate}
Then we have the following:
\begin{enumerate}
\item $\langle X\rangle$ is a free subsemigroup of $M(F)$ (and so we denote by $X^{\ast} = \langle X, 1\rangle$).
\item Let $M:=\langle X, G\rangle \subseteq M(F)$. Then $M = X^{*}G$ uniquely. 
\item $M$ is a Rees monoid.
\begin{comment}
\item If in fact $f_{i}(C)\cap f_{j}(C)=\emptyset$ for every $i,j\in I$, then each element of $F$ can be identified with a unique element of $X^{\omega}$ and vice-versa.
\end{comment}
\item If for every element $s\in M(F)$ there is an $f\in X^{\ast}$ with $s(F) = f(F)$ then $M = M(F)$ and $M$ is a fundamental Rees monoid.
\end{enumerate}
\end{thm}

\begin{proof}
\begin{enumerate}
\item We know that $f_{i_{1}}(\cdots (f_{i_{r}}(F)))\subseteq f_{i_{r}}(F)$ and Lemma \ref{condim} tells us that $\mu(f_{i_{1}}(\cdots (f_{i_{r}}(F)))) = d$. Suppose $f_{i_{1}}\cdots f_{i_{r}} = f_{j_{1}}\cdots f_{j_{s}}$. Then condition (i) and the previous remark tells us that $f_{i_{r}} = f_{j_{s}}$. These are elements of $M(F)$ which is right cancellative and thus $f_{i_{1}}\cdots f_{i_{r-1}} = f_{j_{1}}\cdots f_{j_{s-1}}$. Continuing in this way and using condition (ii) (where $h$ here is in fact an element of $X$) tells us that $r = s$ and $f_{i_{k}} = f_{j_{k}}$ for each $k$. Thus $<X>$ is free.
\item Let $f_{i}\in X$ and $g\in G$. We know by (i) and (iii) that $g(f_{i}(F)) = f_{j}(F)$ for a unique $j\in I$. Further the group of isometries of $f_{j}(F)$ is isomorphic to $G$ and each isometry of $f_{j}(F)$ is the restriction from $F$ to $f_{j}(F)$ of a unique element of $G$. Thus there exists a unique $h\in G$ with $f_{j}h = gf_{i}$ as maps $F\rightarrow F$. We can then use this argument and (i) to show that for each $x\in X^{\ast}$ and $g\in G$ there are unique elements $y\in X^{\ast}$ and $h\in G$ with $gx = yh$ as maps $F\rightarrow F$. Thus $M = X^{\ast}G$ uniquely. 
\item We know by Lemma \ref{fracright} that $M$ is cancellative. We see that $M$ satisfies the conditions of Theorem \ref{sufconlrm}. Thus $M$ is a Rees monoid. 
\item Let $s\in M(F)$ and let $f\in X^{\ast}$ be such that $s(F) = f(F)$. Note that $f$ is necessarily unique. Since $s$ is a similarity transformation there must exist $g\in G(f(F))$ with $s = gf$. But as noted above we can extend every $g\in G(f(F))$ to a $g\in G$. Thus $s\in M$. The fact that $G$ is the group of isometries tells us that $M$ is fundamental. 
\end{enumerate}
\end{proof}

Suppose we have two fractals $F_{1}\subseteq\mathbb{R}^{n}$ and $F_{2}\subseteq\mathbb{R}^{n}$ satisfying the conditions of Theorem \ref{mainthm2}. If $\theta:F_{1}\rightarrow F_{2}$ is an isometry between them, we see that we can map similarity transformations of $F_{1}$ bijectively to similarity transformations of $F_{2}$ by defining $\phi(s) = \theta s \theta^{-1}$. Thus we have the following: 

\begin{thm}
Let $F_{1}\subseteq\mathbb{R}^{n}$ and $F_{2}\subseteq\mathbb{R}^{n}$ be compact spaces satisfying the conditions of Theorem \ref{mainthm2}, and let $M_{1} = M(F_{1})$ and $M_{2} = M(F_{2})$ be their associated similarity monoids. If $F_{1}$ and $F_{2}$ are isometric, then $M_{1}$ and $M_{2}$ are isomorphic.
\end{thm}

\begin{comment}

We extend this argument from isometries as far as homeomorphisms:

\begin{lemma}
\label{homnatact}
Let $F$ and $G$ be fractals satisfying the conditions of the theorem such that $F$ and $G$ are homeomorphic, and let $M$ and $N$ be their associated similarity monoids. Then $N$ acts naturally on $F$.
\end{lemma}

\begin{proof}
Let $f:F\rightarrow G$ be the homeomorphism. For $s\in N$ and $x\in F$ define 
$$s\cdot x:= f^{-1}(s\cdot f(x))$$
Then this defines an action of $N$ on $F$. Firstly, if $s=1$, then $1\cdot f(x) = f(x)$ and since $f$ is a bijection we have $1\cdot x = 1$. Now we need to show $(st)\cdot x = s\cdot (t\cdot x)$. 
$$s\cdot (t\cdot x) = f^{-1}(s\cdot f(t\cdot x)) = f^{-1}(s\cdot f(f^{-1}(t\cdot f(x)))) $$
$$= f^{-1}(s\cdot (t\cdot f(x))) = f^{-1}((st)\cdot f(x)) = (st)\cdot x $$
\end{proof}

\end{comment}

It would be nice if one could prove a result of the following kind:

\begin{conjecture}
Let $C$ be the category with objects fractals as in Theorem \ref{mainthm2} and arrows suitable homeomorphisms and let $D$ be the category with objects Rees monoids and arrows suitable homomorphisms. Then there is a functor from $C$ to $D$.
\end{conjecture}

\begin{comment}
\begin{conjecture}
Let $F$ and $G$ be fractals satisfying the conditions of the theorem, and let $M$ and $N$ be their associated similarity monoids. If $M$ and $N$ are isomorphic, then $F$ and $G$ are homeomorphic.
\end{conjecture}

In each of the following examples, using the terminology of \cite{N1}, we have the following:
\begin{itemize}
\item The \emph{digit tile} is in fact equal to the set of all infinite strings on $X$. We therefore have a nicely defined action of the inverse semigroup defined in section 6 of \cite{Lawson07a} on the digit tile.
\item Each equivalence class of the \emph{limit $G$-space} contains two elements of $\Omega(M)$
\end{itemize}
\end{comment}

We will now show that a number of interesting fractals satisfy the conditions of Theorem \ref{mainthm2}.

\subsection{Sierpinski gasket}

This example appeared in \cite{LawLRM1} and was in fact the motivation for the above theorem. Consider the monoid $M$ of similarities of the Sierpinski gasket (Figure 1). Let $R$, $L$ and $T$ be the maps which quarter the size of the gasket and translate it, respectively, to the right, left and top of itself, $\rho$ be rotation by $2\pi/3$ degrees and $\sigma$ be reflection in the verticle axis. Then the monoid generated by $L$, $R$ and $T$ is free and the group of units $G=\langle \sigma,\rho\rangle\cong D_{6}$. We see that the conditions of Theorem \ref{mainthm2} are satisfied and so $M = \langle R,L,T,\sigma,\rho\rangle$ is a symmetric Rees monoid. Explicitly,
$$\rho T = R \rho, \quad \rho L = T\rho, \quad \rho R = L\rho$$
and
$$\sigma T = T\sigma, \quad \sigma L = R\sigma, \quad \sigma R = L\sigma.$$
We see that the action of $G$ on $X = \left\{L,R,T\right\}$ is transitive and noting that $G_{T} = \left\{1,\sigma\right\}$ we can apply Theorem \ref{mainthm1} to give $M$ by the following monoid presentation:
$$M=\langle \sigma,\rho,t|\sigma^{2}=\rho^{3}=1,\sigma \rho = \rho^{2} \sigma, \sigma t= t\sigma\rangle.$$

\begin{center}
\includegraphics[angle=90,width=60mm]{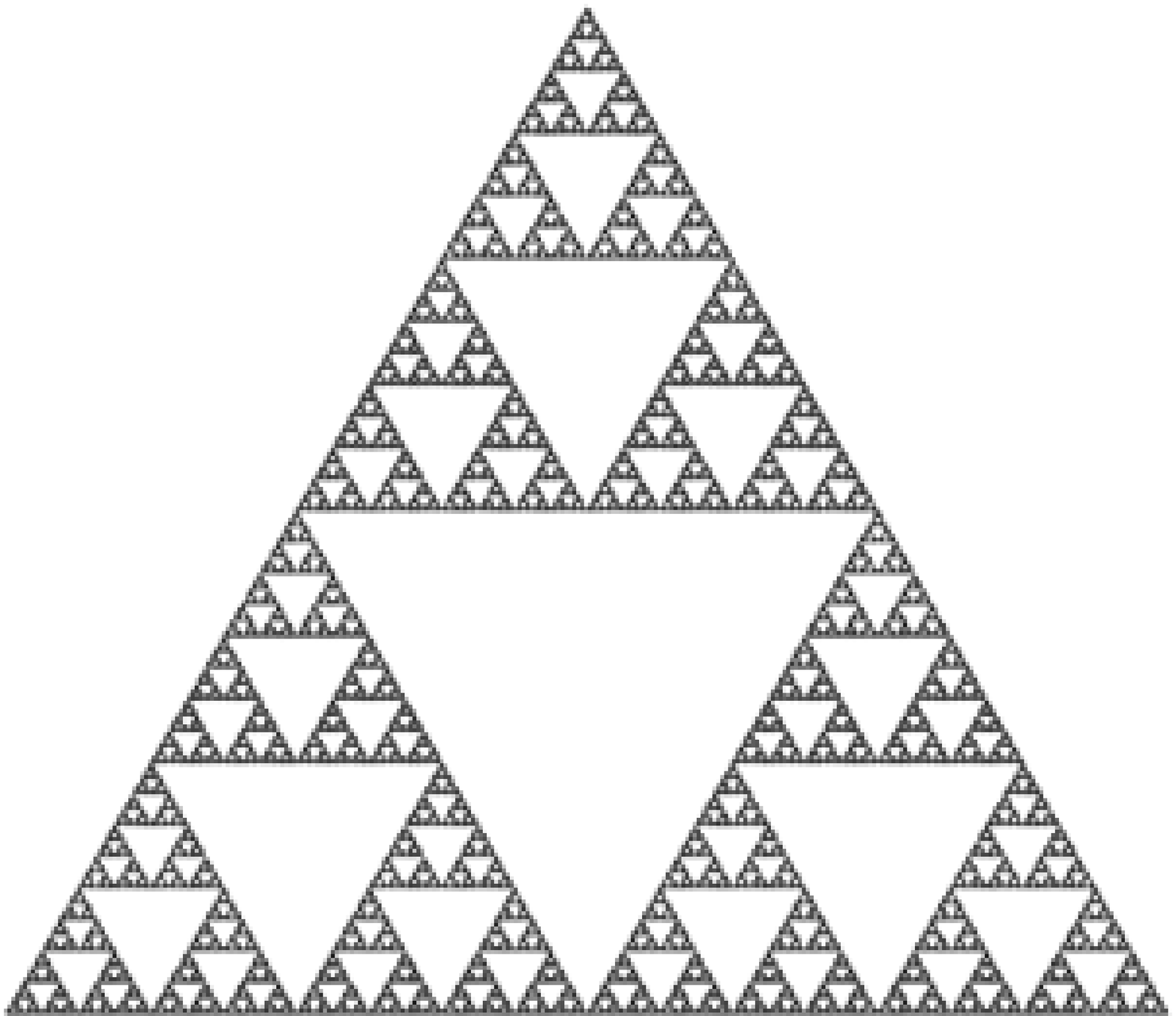}
\end{center}
\begin{center}
Figure 1: Sierpinski gasket (source \cite{SGS})
\end{center}

Let us suppose the corners of the Sierpinski gasket to be at the points $(0,1)$, $(-\frac{\sqrt{3}}{2},-0.5)$ and $(\frac{\sqrt{3}}{2},-0.5)$ so that it is centred on $(0,0)$. Then simple calculations give:
$$\rho = \begin{pmatrix}
  -0.5 & \frac{\sqrt{3}}{2}\\
  -\frac{\sqrt{3}}{2} & -0.5 
 \end{pmatrix}$$
and
$$\sigma = \begin{pmatrix}
  -1 & 0\\
  0 & 1 
 \end{pmatrix}$$
Further, 
$$L(\mathbf{x}) = \frac{1}{2}(\mathbf{x} - \frac{1}{2}(\sqrt{3},1)),$$
$$R(\mathbf{x}) = \frac{1}{2}(\mathbf{x} + \frac{1}{2}(\sqrt{3},-1)),$$
and
$$T(\mathbf{x}) = \frac{1}{2}(\mathbf{x} + \frac{1}{2}(0,1)).$$

Since $M$ is symmetric, it can be extended to a Zappa-Sz\'ep product of a free group and a group, which is the universal group of $M$. So,
$$U(M) \cong FG(X) \bowtie G \cong \langle \sigma,\rho,t|\sigma^{2}=\rho^{3}=1,\sigma \rho = \rho^{2} \sigma, \sigma t= t\sigma\rangle,$$
where this is a group presentation. 
\begin{comment}
If we consider an extended Sierpinski gasket, call it $\mathcal{T}$, where each element of $\mathcal{T}$ is a bi-infinite string of $L$'s, $R$'s and $T$'s then we see that in some sense $U(M)$ is the similarity monoid of $\mathcal{T}$. 

We can apply lemma \ref{homnatact} to this example with the Cantor-$3$ set.
\end{comment}

\subsection{Cantor set}

Consider the monoid $M$ of similarities of the Cantor set $F$ (construction shown in Figure 2).  Let $R$ and $L$ be the maps which divide the Cantor set by 3 and move it, respectively, to the right and left of itself and $\sigma$ be reflection in the verticle axis. We have the following relations:
$$\sigma L = R\sigma, \quad \quad \sigma R = L\sigma.$$
Then the monoid generated by $L$ and $R$ is free, the group of units $G=\langle \sigma\rangle\cong C_{2}$ and $M=\langle R,L,\sigma\rangle$. We see that the conditions of Theorem \ref{mainthm2} are satisfied and so $M$ is a symmetric Rees monoid. Since $G_{R} = G_{L} = \left\{1\right\}$, we find that $M$ is given by the following monoid presentation:

$$M=\langle \sigma,t|\sigma^{2} = 1 \rangle.$$

\begin{center}
\includegraphics[angle=90,width=60mm]{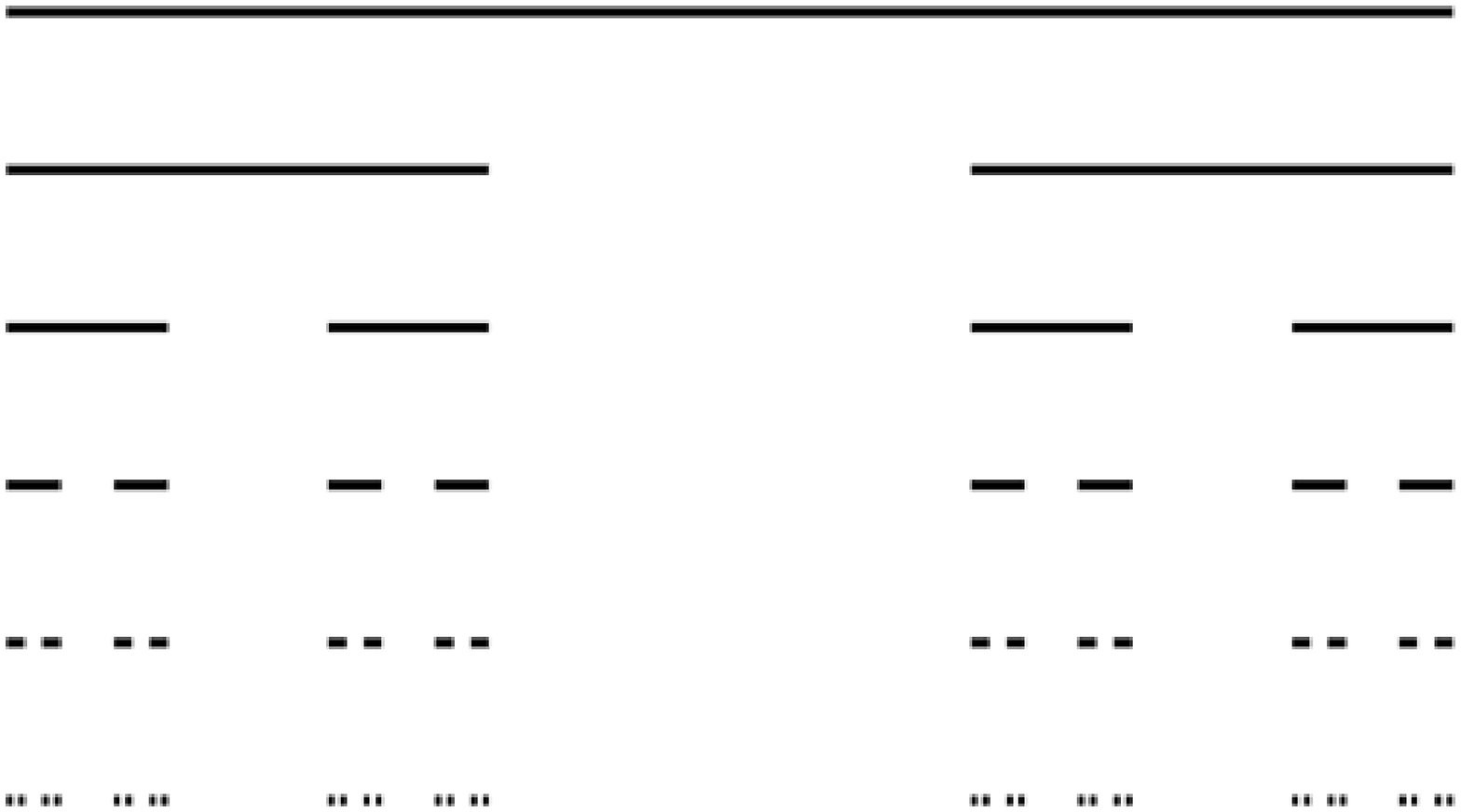}
\end{center}
\begin{center}
Figure 2: Construction of Cantor set (source \cite{CSS})
\end{center}

Notice that each element of $F$ can be written as an infinite word over $X$.
We see in fact that we can identify elements of the free monoid with self-similar subsets of the Cantor set. Using this same identification with the random Cantor set, we see that if we add a random element to the construction, as in Chapter 15 of \cite{Falconer}, then we have an action of a Rees monoid on a random fractal, where the action is piece-wise.

\subsection{Sierpinski carpet}

In this section we will see a group action which is not transitive. Consider the monoid $M$ of similarities of the Sierpinski carpet $F$ (Figure 3).  Let $L_{1}$, $L_{2}$, $R_{1}$, $R_{2}$, $T$, $S_{1}$, $S_{2}$ and $B$ be the maps which map $F$, respectively, to the top left, bottom left, top right, bottom right, top centre, left centre, right centre and bottom centre of itself, $\rho$ be rotation by $\pi/4$ degrees and $\sigma$ be reflection in the verticle axis. We have the following relations:
$$\sigma L_{1} = R_{1} \sigma, \quad \sigma L_{2} = R_{2} \sigma, \quad \sigma R_{1} = L_{1} \sigma, \quad \sigma R_{2} = L_{2} \sigma,$$
$$\sigma T = T \sigma, \quad \sigma B = B \sigma, \quad \sigma S_{1} = S_{2} \sigma, \quad \sigma S_{2} = S_{1} \sigma,$$
$$\rho L_{1} = R_{1} \rho, \quad \rho L_{2} = L_{1} \rho, \quad \rho R_{1} = R_{2} \rho, \quad \rho R_{2} = L_{2} \rho$$
and
$$\rho T = S_{2} \rho, \quad \rho S_{2} = B \rho, \quad \rho B = S_{1} \rho, \quad \rho S_{1} = T \rho.$$

Let $X = \left\{L_{1},L_{2},R_{1},R_{2},T,S_{1},S_{2},B\right\}$. Then the monoid generated by $X$ is free, the group of units $G=\langle \sigma,\rho\rangle\cong D_{8}$ and $M=\langle L_{1},L_{2},R_{1},R_{2},T,S_{1},S_{2},B,\sigma,\rho\rangle$. Again, the conditions of Theorem \ref{mainthm2} are satisfied and so we see that $M$ is a symmetric Rees monoid. We see there are two orbits of $G$ on $X$; $\left\{L_{1},L_{2},R_{1},R_{2}\right\}$ and $\left\{T,S_{1},S_{2},B\right\}$. We find that
$$G_{L_{1}} = \left\{1,\sigma\rho\right\}, \quad \quad G_{T} = \left\{1,\sigma\right\}.$$
Applying Theorem \ref{mainthm1}, $M$ is thus given by the following monoid presentation:

$$M=\langle \sigma,\rho,t,r|\sigma^{2}=\rho^{4}=1, \sigma \rho \sigma = \rho^{3}, \sigma \rho t= t\sigma\rho, \sigma r = r \sigma \rangle.$$

\begin{center}
\includegraphics[scale=0.2]{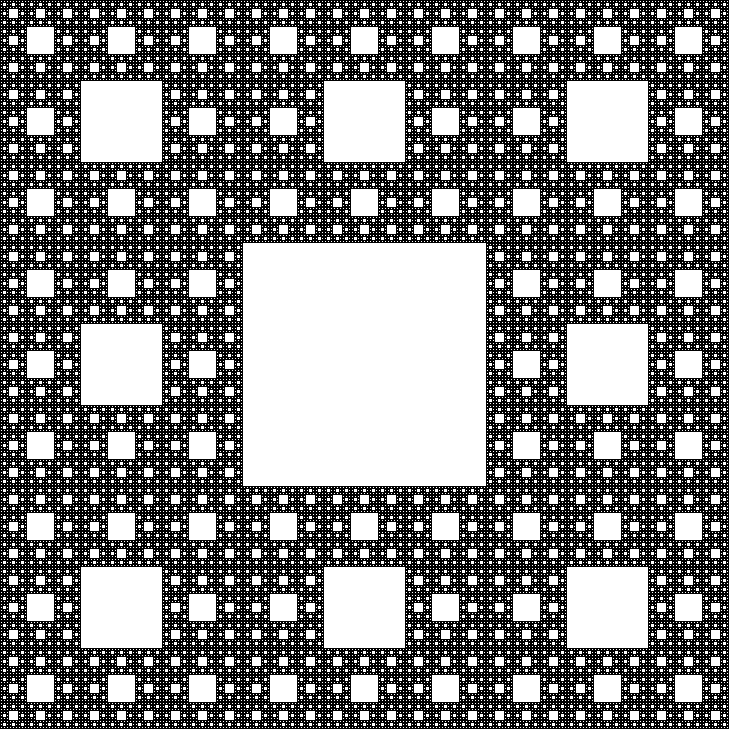}
\end{center}
\begin{center}
Figure 3: Sierpinski carpet (source \cite{SCS})
\end{center}

\subsection{Von Koch curve}

Consider the monoid $M$ of similarities of the von Koch curve $F$ (Figure 4). Let $L$ be the map which rotates $F$ by $3\pi/4$ radians about the central axis and sends it to the left hand side and let $R$ be the map which rotates $F$ by $5\pi/4$ radians and sends it to the right hand side. Letting $X = \left\{L,R\right\}$, $G = C_{2} = \langle \sigma \rangle$ we see that $\sigma L = R \sigma$ and $\sigma R = L \sigma$. The conditions of Theorem \ref{mainthm2} are satisfied and so 
$$M = \langle \sigma,t|\sigma^{2} = 1 \rangle,$$
which is isomorphic to the monoid for the Cantor set. This demonstrates that fractals with the same similarity monoids can have very different geometric structures. 

Let $L_{1}$, $L_{2}$, $R_{1}$ and $R_{2}$ be the maps which map $F$, respectively, to the far left, the left diagonal, the right diagonal and the far right of itself and $\sigma$ again be reflection in the verticle axis. We see that $L_{1} = LR$, $L_{2} = L^{2}$, $R_{1} = R^{2}$ and $R_{2} = RL$. We have the following relations:
$$\sigma L_{1} = R_{2} \sigma, \quad \sigma L_{2} = R_{1} \sigma, \quad \sigma R_{1} = L_{2} \sigma, \quad \sigma R_{2} = L_{1} \sigma.$$

Then the monoid generated by $X = \left\{L_{1},L_{2},R_{1},R_{2}\right\}$ is free, the group of units $G=\langle \sigma\rangle\cong C_{2}$ and $N=\langle L_{1}, L_{2}, R_{1},R_{2},\sigma\rangle$ will again be a Rees monoid, this time a submonoid of the monoid of similarity transformations of $F$. $N$ is given by the following monoid presentation: 

$$N=\langle \sigma,t,r|\sigma^{2} =1\rangle.$$

\begin{center}
\includegraphics[angle=90,width=60mm]{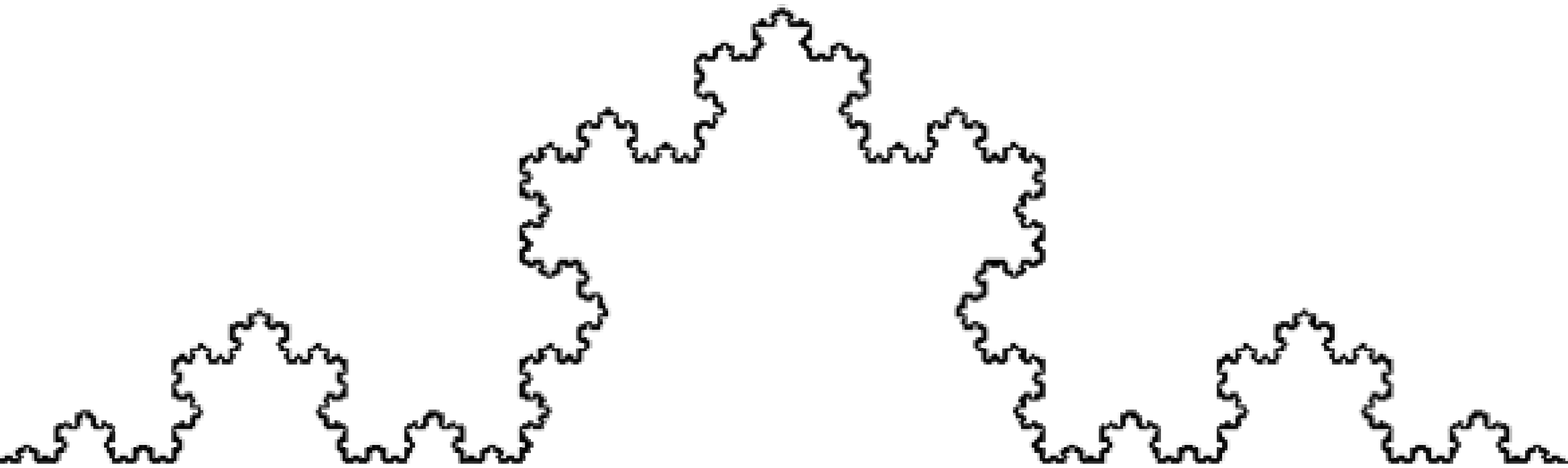}
\end{center}
\begin{center}
Figure 4: Von Koch curve (source \cite{VKCS})
\end{center}

\subsection{Some examples in $\mathbb{R}^{3}$} 
We can also consider examples in 3 dimensional space. We can define the Sierpinski tetrahedron, Cantor cylinder and Sierpinski cube in analogy with the constructions described in the previous section. Note that the Sierpinski tetrahedron and Sierpinski cube satisfy the conditions of Theorem \ref{mainthm2}. We find that for the Sierpinski tetrahedron $|X| = 4$ and $G$ is the isometry group of a tetrahedron, and for the Sierpinski cube $|X| = 20$ and $G$ is the isometry group of a cube. 

The Cantor cylinder doesn't quite work as the obvious contraction maps are not similarity transformations, but there is still a nice associated Rees monoid, where we set $|X| = 2$, $G = S^{1} \rtimes C_{2}$, $G_{x} = G_{y} = S^{1}$ and $gx = xg$ for all $g\in S^{1}$. 

\section{Topological fractals}

In this section we describe a construction of a fractal-like topological space found in Bandt and Retta \cite{BandtRetta}. They show that certain fractals are really determined up to homeomorphism. We prove that the monoid of into-homeomorphisms of certain examples arising from their construction is a left Rees monoid. 

Let $S = \left\{1,\ldots,m\right\}$ be a finite set, $C = S^{\infty}$ the space of sequences $s = s_{1}s_{2}\ldots$ with the product topology, $S^{*}$ the free monoid on $S$ and $S^{<n}$ the set of words of length smaller than $n$. If $s\in C \cup S^{*}$, then the prefix word of length $k$ of $s$ is denoted by $s|_{k} = s_{1}\ldots s_{k}$. An equivalence relation $\sim$ on $C$ will be called \emph{invariant} if $\sim$ is a closed set in $C\times C$ and for all $s,t\in C$ and $i\in S$
$$s \sim t \Leftrightarrow  is\sim it,$$
that is, it is a left congruence with respect to $S$. $A = C/\sim$ will be called an \emph{invariant factor} for $C$. Let us now fix such a relation $\sim$. Let 
$$M = \left\{s\in C| \exists t\in C: \quad s\sim t \quad \wedge \quad s_{1} \neq t_{1} \right\}.$$
We will call $Q = M/\sim$ the \emph{critical points} of $A$. Let $p:C \rightarrow A$ be the associated projection with respect to $\sim$. We will say $p$ is \emph{finite-to-one} if $M$ is finite and contains no periodic sequence. If there do not exist $s\in M$ and $w\in S^{*}$ such that $ws$ also belongs to $M$, the relation $\sim$ and the factor $A$ are called \emph{simple}.

For $w\in S^{\ast}$ denote by $C_{w} = \left\{ws | s\in C\right\}$ and let $A_{w} = p(C_{w})$. It was proved in \cite{BandtKeller} that there is a unique homeomorphism $f_{w}: A\rightarrow A$ such that $f_{w}(p(s)) = p(ws)$.

We will now define a sequence of undirected hypergraphs $G_{n}$. The vertex set of $G_{n}$ is $S^{n}$ and the edge set is $S^{<n} \times Q$. The endpoints of the edge $(v,q)$ will be the words $(vs)|_{n}$, with $s\in q$. If each equivalence class $q\in Q$ contains less than 2 elements, then these will in fact be graphs. 

A connected graph $G$ is said to be \emph{2-connected} if $G\setminus\left\{u\right\}$ is connected for each $u\in V(G)$. A connected graph $G$ with $m$ vertices and $c$ edges is said to be \emph{edge-balanced} if for each $k$ with $1<k<m$, the graph cannot be divided into $k$ components by deleting $(k-1)c/(m-1)$ or less edges.

Let $M(A)$ denote the set of homeomorphisms from $A$ into subspaces of $A$ and let $G(A)$ be the group of homeomorphisms from $A$ to itself. Bandt and Retta proved the following result \cite{BandtRetta}:

\begin{thm}
\label{toplrmclass}
Let $A$ be a simple finite-to-one invariant factor such that $G_{1}$ is edge-balanced and $G_{2}$ is 2-connected. Then for each $f\in M(A)$, there exists $w\in S^{*}$ such that $im(f) = A_{w}$.
\end{thm}

We then get the following immediate corollary.

\begin{corollary}
Each element of $h\in M(A)$ can be written $h = f_{w}g$, where $w\in S^{*}$ and $g\in G(A)$.
\end{corollary}

\begin{proof}
If $h(A) = A_{w}$, then $g = f_{w}^{-1}h$.
\end{proof}

\begin{proposition}
For $A$ as above, $M(A)$ is a left Rees monoid.
\end{proposition}

\begin{proof}
$M(A)$ is clearly left cancellative since by assumption its elements are injective. Identifying an element $f_{w}$ with $w\in S^{\ast}$, we see that $M(A) = S^{*}G(A)$ uniquely. Thus by Theorem \ref{sufconlrm},  $M(A)$ is a left Rees monoid.
\end{proof}

\begin{ex} 
Let $S =\left\{1,2,3\right\}$, denote by $\bar{i}$ the sequence consisting just of the character $i$, say $i\bar{j} \sim j\bar{i}$ for $i,j = 1,2,3$, and extend this equivalence relation to $C$. Then $A = C/\sim$ is homeomorphic to the Sierpinski gasket and the conditions of Theorem \ref{toplrmclass} are satisfied. 
\end{ex}

\begin{proposition}
Let $A$ and $B$ be two invariant factors such that $f:A\rightarrow B$ is a homeomorphism. Then $M(A)$ is isomorphic to $M(B)$.
\end{proposition}

\begin{proof}
Define $S:M(A)\rightarrow M(B)$ by $S(h) = fhf^{-1}$. Then $S$ is clearly an isomorphism.
\end{proof}

\begin{comment}
\begin{conjecture}
Let $h:M(A)\rightarrow M(B)$ be a homomorphism (or isomorphism) and define a map $f:A\rightarrow B$ by $f(x_{1}x_{2}\ldots) = h(x_{1})h(x_{2})\ldots$. Then $f$ is interesting.
\end{conjecture}
\end{comment}

\section{Automaton presented groups}

We now describe how self-similar group actions arise from automata. This is described in detail in \cite{NekrashevychBook}. For the present an \emph{automaton} $\mathcal{A}=(A,X,\lambda,\pi)$ consists of 
\begin{itemize}
\item a set $A$ whose elements are called \emph{states};
\item a set $X$ called the \emph{alphabet};
\item a map $\lambda:A\times X\rightarrow X$ called the \emph{output function}; 
\item a map $\pi:A\times X\rightarrow A$ called the \emph{transition function}. 
\end{itemize}
In theoretical computer science, these structures are normally called \emph{deterministic real-time synchronous transducers} \cite{cain}. We will denote $\lambda(q,x)$ by $q\cdot x$ and $\pi(q,x)$ by $q|_{x}$.

We can use \emph{Moore diagrams} to represent automata as follows. The states are represented by labelled circles. For $a\in A$ and $x\in X$, there exists an arrow from $a$ to $a|_{x}$ labelled by the ordered pair $(x,a\cdot x)$. 

Let us define:
\begin{itemize}
\item $q|_{\emptyset}=q$
\item $q\cdot\emptyset= \emptyset$
\end{itemize}

Given an automaton $(A,X)$ we can construct an automaton $(A,X^{n})$ for each $n$ by defining the transition and output functions recursively as follows, for $x,y\in X^{*}$: 
\begin{enumerate}
\item $q|_{xy}=(q|_{x})|_{y}$
\item $q\cdot(xy)=q\cdot(x)q|_{x}\cdot y$
\end{enumerate}
Let $(A,X)$, $(B,X)$ be two automata. Then we can define their composition automaton $(A\times B, X)$ with transition and output functions as follows:
\begin{enumerate}
\item $(pq)\cdot x=p\cdot(q\cdot x)$
\item $(pq)|_{x}=p|_{(q\cdot x)}q|_{x}$
\end{enumerate}

\par
We see that states of automata describe endomorphisms of free monoids as trees. An automaton is \emph{invertible} if each of its states describes an invertible transformation of a free monoid. That is, an automaton is invertible if and only if $\lambda(a,\cdot)$ is a bijection for each $a\in A$. Given an invertible automaton $(A,X)$, we can construct an automaton $(A^{-1},X)$ whose states are in bijective correspondence with those of $A$ and whose transition and output functions are inverted. We call an automaton whose states each define different endomorphism of $X^{\ast}$ \emph{reduced}. We thus have a homomorphism $FG(A) \rightarrow Aut(X^{*})$. The group generated by the image of this homomorphism we will denote by $G$ or $G(\mathcal{A})$. Each of the elements of $G$ will correspond to one or more compositions of states of $A$. We see that $G$ acts on $X^{\ast}$ self-similarly and faithfully. Note that the kernel of this homomorphism will be
$$\mathcal{K}(\mathcal{A}) = \bigcap_{x \in X^{\ast}}G_{x}.$$
So by the first isomorphism theorem, we have
$$G \cong FG(A) / \mathcal{K}(\mathcal{A}). $$
\par
We call a left Rees monoid \emph{finite-state} if each of the sets $\left\{g|_{x}:x\in X^{*}\right\}$ for $g\in G$ is finite.

\par
We now see that given an invertible reduced automaton $\mathcal{A}$ with a finite number of states over a finite alphabet, we can then construct a fundamental left Rees monoid $M(\mathcal{A})$. On the other hand, given a finite state fundamental left Rees monoid with finitely generated group of units and finite $X$, then we can describe it by a finite-state reduced invertible automaton.

Let $\mathcal{A} = (A,X,\lambda_{1},\pi_{1})$ and $\mathcal{B} = (B,Y,\lambda_{2},\pi_{2})$ be finite state automata. We will say $\mathcal{A}$ and $\mathcal{B}$ are \emph{computationally equivalent} and write $\mathcal{A} \sim \mathcal{B}$ if 
\begin{enumerate}
\item There is an isomorphism $\theta:X^{\ast} \rightarrow Y^{\ast}$.
\item For all $a\in FG(A)$ there exists $b\in FG(B)$ such that $a\cdot x = b\cdot \theta(x)$ for all $x\in X^{*}$.
\item For all $b\in FG(B)$ there exists $a\in FG(B)$ such that $a\cdot \theta^{-1}(y) = b\cdot y$ for all $y\in Y^{*}$.
\end{enumerate}

It is clear that $\sim$ defines an equivalence relation.

\begin{proposition}
Let $\mathcal{A}$, $\mathcal{B}$ be automata. Then
$\mathcal{A} \sim \mathcal{B}$ if and only if $M(\mathcal{A})$ and $M(\mathcal{B})$ are isomorphic monoids.
\end{proposition}

\begin{proof}
($\Rightarrow$) Let $\mathcal{A} = (A,X,\lambda_{1},\pi_{1}) \sim \mathcal{B} = (B,Y,\lambda_{2},\pi_{2})$. Then we can assume $X = Y$ by (1). By (2) and (3), we have a bijective map $f:G(\mathcal{A})\rightarrow G(\mathcal{B})$. Further since in group actions $(ab)\cdot x = a\cdot(b\cdot x)$, $f$ is a homomorphism and thus an isomorphism
\\($\Leftarrow$) This is clear.
\end{proof}

\subsection{Adding machine}

Let us describe an example found in \cite{NekrashevychBook}. Let $G=\langle a\rangle\cong\mathbb{Z}$ and $X=\left\{x,y\right\}$. Then define $a\cdot x= y$, $a\cdot y= x$, $a|_{x}=1$ and $a|_{y}= a$. This defines a self-similar action. In terms of the wreath recursion we have $a=\sigma(1,a)$, where $\sigma$ is the permutation in $\mathcal{S}(X)$ permuting $x$ and $y$. Let us denote this left Rees monoid by $M$. It is called the \emph{dyadic adding machine}. Identify $x$ with 0 and $y$ with 1 so that finite and infinite words over $X$ become expansions of dyadic integers. The action of $a$ on a word $w$ is then equivalent to adding 1 to the dyadic integer corresponding to $w$. We can see that $M$ acts on the Cantor set by identifying elements of the Cantor set with dyadic integers. The associated automaton has the following Moore diagram:

\begin{center}
\setlength{\unitlength}{0.75mm}
\begin{picture}(60,60)

\put(9,30){\circle{17}}
\put(9,38.5){\vector(1,0){1}}
\put(4,41){$(1,0)$}

\put(20,30){$a$}
\put(21,31){\circle{7}}

\put(24.5,31){\vector(1,0){27}}
\put(32,33){$(0,1)$}

\put(54,30){$1$}
\put(55,31){\circle{7}}

\put(55,43){\circle{17}}
\put(55,51.5){\vector(1,0){1}}
\put(50,53.5){$(1,1)$}

\put(55,19){\circle{17}}
\put(55,10.5){\vector(1,0){1}}
\put(50,6){$(0,0)$}

\end{picture}
\end{center}

\begin{center}
Figure 5: Moore diagram of the dyadic adding machine
\end{center}

\vspace{3 mm}

We see that $G_{x}=G_{y}=\left\{a^{2n}:n\in\mathbb{Z}\right\}$. We have that
$$\phi_{x}(a^{2})=(a^{2})|_{x}=(a|_{a\cdot x})(a|_{x})=(a|_{y})(a|_{x})=a.$$
Then for $n>1$ we have 
$$\phi_{x}(a^{2n})=(a^{2n})|_{x}=(a^{2n-2}|_{a^{2}\cdot x})(a^{2}|_{x})=\phi_{x}(a^{2n-2})a$$
and so by induction $\phi_{x}(a^{2n})=a^{n}$.
Similarly $\phi_{y}(a^{2n})=a^{n}$. Therefore $\phi_{x}$ and $\phi_{y}$ are both injective so $M$ is in fact a Rees monoid. The action is transitive and so the Rees monoid will have the following monoid presentation  
$$M=\langle a,a^{-1},t|aa^{-1} = a^{-1}a = 1, ta^{n}=a^{2n}t,n\in\mathbb{Z}\rangle,$$
which can be further reduced to give
$$M = \langle a,a^{-1},t|aa^{-1} = a^{-1}a = 1,ta=a^{2}t\rangle.$$ 
Therefore the universal group will be
$$U(M)=\langle a,t|tat^{-1}=a^{2}\rangle \cong BS(1,2),$$
\par
where the Baumslag-Solitar group $BS(m,n)$ is given by the following group presentation
$$BS(m,n) = \langle a,t|ta^{m}t^{-1}=a^{n}\rangle.$$
We can in fact generalise the above to construct an automaton whose associated monoid's universal group is $BS(k,n)$, where $k < n$. Let $A = \left\{a,1\right\}$ and $X = \left\{0,\ldots,n-1 \right\}$. There will be $n$ arrows starting at $a$ and these will be labelled by tuples $(x, x+1 \mod n)$, $x\in X$. The first $k$ arrows will be from $a$ to itself, and the remaining $n-k$ arrows will go from $a$ to $1$. The arrows from $1$ to itself will be labelled by pairs $(x,x)$, $x\in X$. We have that $G = \mathbb{Z}$ and for each $x \in X$, we have $G_{x} = \langle a^{n}\rangle$ and $a^{n}|_{x}=a^{k}$. This gives a monoid with presentation: 
$$M=\langle a,a^{-1},t|aa^{-1} = a^{-1}a = 1, ta^{k}=a^{n}t\rangle$$
and so
$$U(M)=\langle a,t|ta^{k}t^{-1}=a^{n}\rangle \cong BS(k,m).$$
Note that $BS(k,m)\cong BS(m,k)$ and so this really gives us all of the Baumslag-Solitar groups.

\subsection{Baumslag-Solitar group actions}

The following example is adapted from one given in \cite{BartholdiHenriquesNekrashevych}. Consider the automaton $\mathcal{A}$ given by the following Moore diagram. 

\begin{center}
\setlength{\unitlength}{0.75mm}
\begin{picture}(60,60)

\put(-20,30){\circle{17}}
\put(-20,38.5){\vector(1,0){1}}
\put(-26.5,41){$(x,x)$}

\put(-9,30){$\alpha$}
\put(-8,31){\circle{7}}

\put(-4.5,34){\vector(1,0){30}}
\put(3,37){$(y,y)$}

\put(26,28){\vector(-1,0){30}}
\put(3,22){$(x,y)$}

\put(28,30){$\beta$}
\put(29,31){\circle{7}}

\put(32.5,34){\vector(1,0){30}}
\put(39.5,37){$(y,x)$}

\put(63,28){\vector(-1,0){30}}
\put(39.5,22){$(x,x)$}

\put(65,30){$\gamma$}
\put(66,31){\circle{7}}

\put(78,30){\circle{17}}
\put(78,38.5){\vector(1,0){1}}
\put(72,41){$(y,y)$}

\end{picture}
\end{center}

\begin{center}
Figure 6: Baumslag-Solitar machine
\end{center}

Thinking of $x$ as representing 0 mod 2 and $y$ as 1 mod 2 and identifying $X^{\omega}$ with $\mathbb{Z}_{2}$ we can consider $\alpha$, $\beta$ and $\gamma$ as the maps defined on the dyadic integers given by $\alpha(X) = 3X$, $\beta(X) = 3X+1$ and $\gamma(X) = 3X+2$.
Letting 
$$G = FG(\alpha,\beta,\gamma)/\mathcal{K}(\mathcal{A}),$$
we have $G\cong BS(1,3)$ where the isomorphism $\theta:BS(1,3)\rightarrow G$ is given on generators by $\theta(t) = \alpha$ and $\theta(a) = \beta \alpha^{-1}$, viewing $t$ and $a$ as the maps on $\mathbb{Z}_{2}$ given by $t(X) = 3X$ and $a(X) = X+1$. The action of $G$ on $X = \left\{x,y\right\}$ is transitive. An arbitrary element $g\in G$ can be written
$$g = (\prod_{k=1}^{n}{t^{i_{k}}a^{j_{k}}})t^{r},$$
where $i_{k},j_{k},r\in \mathbb{Z}$.
We see that 
$$\theta(t)\cdot x = \alpha\cdot x = x$$
and
$$\theta(a)\cdot x = (\beta \alpha^{-1})\cdot x = y.$$
Similarly, $\theta(t)\cdot y = y$ and $\theta(a)\cdot y = x$.
From this we deduce that 
$$G_{x} = G_{y} = \left\{(\prod_{k=1}^{n}{t^{i_{k}}a^{j_{k}}})t^{r}|\sum_{k=1}^{n}{j_{k}} \hbox{ even}\right\}.$$
Note that
$$\alpha|_{x} = \alpha, \quad \alpha|_{y} = \beta = \theta(at), \quad (\beta \alpha^{-1})|_{x} = \alpha \alpha^{-1} = 1, \quad (\beta \alpha^{-1})|_{y} = \gamma \beta^{-1} = \beta \alpha^{-1}.$$
Now the group $G_{x}$ is generated as a group by the elements $a^{2}$, $ata$ and $t$. 
We see that when we consider the monoid presentation of the monoid $M$ of the automaton $\mathcal{A}$, if we have the relation $gr = rh$ for some $g,h\in G$ then this gives for free 
$$g^{-1}r = g^{-1}grh^{-1} = rh^{-1}.$$
Thus $M$ has monoid presentation
\begin{comment}
Denote by $[z]$ the floor function where we adopt the convention that $[10.5] = 10$ and $[-10.5] = [-11]$. Now define $\rho: G_{x}\rightarrow G$ by
$$\rho((\prod_{k=1}^{n}{t^{i_{k}}a^{j_{k}}})t^{r}) = (\prod_{k=1}^{n}{s_{k}^{i_{k}}a^{r_{k}}})t^{r},$$
where
$$r_{n} = [\frac{i_{n}}{2}],$$
and for $1 < k \leq n$,
$$c_{k} = \sum_{m=k}^{n}{j_{m}},$$
$$s_{k} = t, \quad r_{k-1} = [\frac{i_{n}}{2}]$$
for $c_{k}$ even and
$$s_{k} = at, \quad r_{k-1} = [\frac{i_{n}+1}{2}]$$
for $c_{k}$ odd.
\end{comment}
$$M = \langle a,a^{-1},t, t^{-1},r\quad | \quad aa^{-1} = a^{-1}a = tt^{-1} = t^{-1}t = 1, ta = a^{3}t, $$
$$atar = ra^{2}t, tr = rt, a^{2}r = ra \rangle .$$
It is not clear whether $M$ is right cancellative or not. Its universal group $U(M)$ has group presentation
$$U(M) = \langle a,t,r|ta = a^{3}t, atar = ra^{2}t, tr = rt, a^{2}r = ra \rangle .$$

\begin{comment}
Let us now consider an action of $BS(1,3)$ on $X^{*}$, where $X=\left\{x,y\right\}$ as described in. Firstly, define $a:\mathbb{R}\rightarrow \mathbb{R}$ by $a(X)=X+1$ and $t:\mathbb{R}\rightarrow \mathbb{R}$ by $t(X)=3X$. Then $BS(1,3)$ is equal to the group generated by $a$ and $t$. Now define $f:\mathbb{R}\rightarrow \mathbb{R}$ by $f(X)=3X$, $g:\mathbb{R}\rightarrow \mathbb{R}$ by $g(X)=3X+1$ and $h:\mathbb{R}\rightarrow \mathbb{R}$ by $h(X)=3X+2$. Then the group generated by $f$, $g$ and $h$ is also isomorphic to $BS(1,3)$, under the isomorphism $\theta$ given by $\theta(t) = f$, $\theta(a) \mapsto gf^{-1}$. Now instead of being defined on $\mathbb{R}$, consider these as maps on $\mathbb{Z}_{2}$, the dyadic integers. This then gives rise to an action of $BS(1,3)$ on $X^{*}$, with wreath recursion given by $\psi(f)=(f,g)$, $\psi(g)=\sigma(f,h)$, $\psi(h)=(g,h)$, where $\sigma$ is the flip in $\mathcal{S}(X) \cong C_{2}$. 
The embedding of the triadic adding machine into $BS(1,3)$ then maps, from our original Rees monoid,
$$x\mapsto a, y\mapsto b, z\mapsto c, a\mapsto ba^{-1}$$
\end{comment}

\subsection{Sierpinski gasket}

Let $G = D_{6}$, $X = \left\{L, R, T\right\}$ and suppose $M$ is the monoid of similarity transformations of the Sierpinski gasket as described in Section 2.5.1. Observe that $M$ is in fact the monoid associated with the following automaton:

\begin{center}
\setlength{\unitlength}{0.75mm}
\begin{picture}(130,50)

\put(9,30){\circle{17}}
\put(9,38.5){\vector(1,0){1}}
\put(4,41){(L,T)}

\put(30,20){\circle{17}}
\put(30,28.5){\vector(1,0){1}}
\put(26,31){(T,R)}

\put(11,13){\circle{17}}
\put(11,4.5){\vector(1,0){1}}
\put(6,0){(R,L)}

\put(17,21.5){$\rho$}
\put(18,22.5){\circle{7}}

\put(79,30){\circle{17}}
\put(79,38.5){\vector(1,0){1}}
\put(74,41){(L,R)}

\put(100,20){\circle{17}}
\put(100,28.5){\vector(1,0){1}}
\put(96,31){(T,T)}

\put(81,13){\circle{17}}
\put(81,4.5){\vector(1,0){1}}
\put(76,0){(R,L)}

\put(87,21.5){$\sigma$}
\put(88,22.5){\circle{7}}

\end{picture}
\end{center}

\begin{center}
Figure 7: Moore diagram of Sierpinski gasket automaton
\end{center}

\subsection{Grigorchuk group}

Here we give an example taken from \cite{NekrashevychBook} of a left Rees monoid which is not a Rees monoid. The \emph{Grigorchuk group} $G$ is defined to be the group of units of the left Rees monoid generated by four elements $a,b,c,d$ with $X=\left\{0,1\right\}$ and wreath recursion
$$a=\sigma, \quad b=(a,c), \quad c=(a,d), \quad d=(1,b),$$
where $\sigma\in\mathcal{S}(X)$ is again the flip map and the associated Moore diagram is: 

\begin{center}
\includegraphics[scale=0.7]{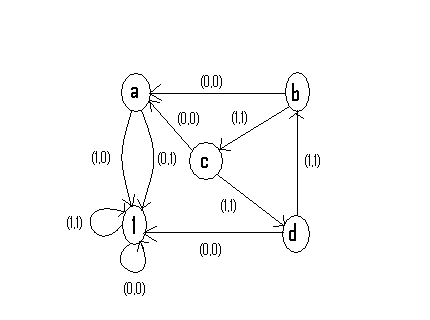}
\end{center}

\begin{center}
Figure 8: Moore diagram of Grigorchuk group action automaton
\end{center}

It is not a Rees monoid because $b,c\in G_{0}$ and $\phi_{0}(b)=\phi_{0}(c)=a$.

\section{Associated bialgebra} 

The following construction is based on ideas found in \cite{Kassel}. Let $K$ be a field and let $M$ be a monoid. We can form the monoid bialgebra $KM$ as follows. An element $v$ of $KM$ is a finite sum 
$$v = \sum_{i=1}^{n}\alpha_{i}x_{i},$$
where $\alpha_{i}\in K$ and $x_{i}\in M$. We define addition $+$, convolution $\circ$ and scalar multiplication as follows:
$$\sum_{i=1}^{n}\alpha_{i}x_{i} + \sum_{i=1}^{m}\beta_{i}y_{i} = \sum_{i=1}^{n+m}\alpha_{i}x_{i},$$
where for $n+1\leq i\leq n+m$, $\alpha_{i} = \beta_{i-n}$ and $x_{i} = y_{i-n}$,
$$\sum_{i=1}^{n}\alpha_{i}x_{i} \circ \sum_{i=1}^{m}\beta_{i}y_{i} = \sum_{i=1}^{n}\sum_{j=1}^{m} \alpha_{i}\beta_{j}x_{i}y_{j},$$
where $x_{i}y_{j}$ is the product in $M$ and
$$\lambda \sum_{i=1}^{n}\alpha_{i}x_{i} = \sum_{i=1}^{n}\lambda\alpha_{i}x_{i},$$
where for all of the above $\lambda, \alpha_{i}, \beta_{i}\in K$ and $x_{i},y_{i}\in M$. This gives $KM$ the structure of a unital $K$-algebra. (Note if $K =\mathbb{C}$, we may want to take the complex conjugate of the $\beta_{i}$'s in the definition of the convolution.) We can make $KM$ a cocommutative bialgebra by specifying the comultiplication $\Delta$ on the elements $x\in M$ to be
$$\Delta(x) = x\otimes x$$
and counit $\epsilon$ to be $\epsilon(x) = 1$. If $M$ is a group, we can make $KM$ into a Hopf algebra by defining the antipode $S(g) = g^{-1}$. Now suppose $M = X^{*}\bowtie G$ is a left Rees monoid. Then $KM$ is isomorphic to the \emph{bicrossed product bialgebra} (\cite{Majid}) $KX^{*}\bowtie KG$ with unit $1\otimes 1$, multiplication on generators given by
$$(x\otimes g)(y\otimes h) = x(g\cdot y)\otimes (g|_{y})h,$$
comultiplication $\Delta(x\otimes g) = (x\otimes g)\otimes (x\otimes g)$, counit $\epsilon(x\otimes g) = 1$. If $M$ is symmetric, then we can form $K\Gamma = KFG(X)\bowtie KG$, as above with antipode $S(x\times g) = (g^{-1}\cdot x^{-1})\otimes(g^{-1}|_{x^{-1}})$. 

As in the representation theory of finite groups, we see that if $f: M\rightarrow M_{n}(K)$ is a homomorphism, then $K^{n}$ can naturally be given the structure of a finitely generated left $KM$-module, by setting 
$$\left(\sum_{i}{\alpha_{i}s_{i}}\right)\cdot x = \sum_{i}{\alpha_{i}(f(s_{i})\cdot x)},$$
for $\alpha_{i}\in K$, $s_{i}\in M$ and $x\in K^{n}$.

We now consider what the ring $KM$ looks like. Throughout $M = X^{*} G$ is a left Rees monoid and $R = KM$. We will view $R$ as a left $R$-module.

\begin{lemma}
$R$ is not artinian.
\end{lemma}

\begin{proof}
Consider the chain of left ideals $Rx\supseteq Rx^{2} \supseteq Rx^{3} \supseteq ...$. Then this chain is infinite and so $R$ does not satisfy the descending chain condition.
\end{proof}

\begin{lemma}
If $|X|\geq 2$ then $R$ is not noetherian.
\end{lemma}

\begin{proof}
Let $J_{k}$ be the left ideal generated by the set $\left\{yx, yx^{2}, \ldots, yx^{k} \right\}$. Then $J_{k}\subseteq J_{k+1}$ for $k\geq 1$, and so $R$ does not satisfy the ascending chain condition.
\end{proof}

\begin{comment}
\begin{corollary}
$R$ has left and right ideals which are not finitely generated. 
\end{corollary}

Question: Are there finitely generated 2-sided ideals?

\begin{lemma}
$R$ has no non-zero nilpotent or nil left, right or 2-sided ideals
\end{lemma}

\begin{lemma}
$R$ satisfies Schur's lemma
\end{lemma}

\end{comment}

We can give $R$ a grading by letting $R_{k}$ be the set of elements of $R$ with maximal length of a string from $X^{\ast}$ being $k$.

For a ring $R$ let the \emph{Jacobson radical} be defined as follows
$$J(R) = \bigcap_{I}{I},$$
where the intersection is taken over all maximal proper right ideals $I$ of $R$. If $J(R) = 0$, then $R$ is said to be \emph{semisimple}.
By Lemma I.1.3 of \cite{Assem} and the fact that an element $x\in X$ is not invertible we have

\begin{lemma}
If $|X| \geq 1$, $G$ arbitrary then $R$ is semisimple.
\end{lemma}

\begin{comment}
Again, from now on $R = KM$, where $M = X^{\ast} G$.

Firstly, we assume $X$ is finite because otherwise $X^{\ast}$ is very complicated and we assume $G$ is finite because of the following lemma:

\begin{lemma}
(Maschke) If $G$ is finite and $|X|=0$, then $R$ is semisimple.
\end{lemma}

In fact by Lemma I.1.3 of \cite{Assem} and the fact that an element $x\in X$ is far from invertible we have

\begin{lemma}
If $|X| > 1$, $G$ arbitrary then $R$ is semisimple.
\end{lemma}

\end{comment}

Now consider $KX^{n}$ as a $k^{n}$-dimensional vector space, viewed as the $n$th tensor power of $KX$. We will now construct an embedding of $KG$ into $M_{k^{n}}(K)$, which will act in a nice way on $KX^{n}$.

Let $s_{1}:KG\rightarrow M_{k}(K)$ be defined as follows. For each $g\in G$, $1\leq i \leq k$, if $g\cdot x_{i} = x_{j}$, then there will be a $1$ in the $(j,i)$th entry of the matrix $s_{1}(g)$. All other entries will be $0$. We see that $s_{1}(g)$ is a doubly stochastic matrix with a single $1$ in every row and column. Now define 
$$s_{1}(\sum_{i=1}^{|G|}{\alpha_{i}g_{i}}) = \sum_{i=1}^{|G|}{\alpha_{i}s_{1}(g_{i})}.$$

We will now describe inductively $s_{k}:KG\rightarrow M_{k^{n}}(K)$ for $k>1$. For $g\in G$, let $s_{k}(g)$ be as $s_{1}(g)$, except that the $1$ in the $(j,i)$th position is replaced by a $k^{n-1}\times k^{n-1}$ matrix $A_{j}$, and the $0$'s are replaced by blocks of $0$'s. Here 
$$A_{j} = s_{k-1}(g|_{x_{i}}).$$
Similarly,
$$s_{k}(\sum_{i=1}^{|G|}{\alpha_{i}g_{i}}) = \sum_{i=1}^{|G|}{\alpha_{i}s_{k}(g_{i})}.$$

For example, let us consider the monoid $M$ of similarity transformations of the Sierpinski gasket. In order to clarify the construction, we will change the basis from the one used above. We will define $L$ to be the map which rotates the gasket by $2\pi/3$ radians and then maps to the bottom left hand corner. $R$ and $T$ will still map the gasket to the bottom right and top parts and $\rho$ and $\sigma$ will remain unchanged. So, our new relations will be

$$\rho T = R \rho, \quad \rho L = T\rho^{2}, \quad \rho R = L$$
and
$$\sigma T = T\sigma, \quad \sigma L = R\rho^{2}\sigma, \quad \sigma R = L\rho^{2}\sigma.$$

\begin{comment}
$X=\left\{L^{\prime},R,T\right\}$, $G = D_{3}$ and let $M = X^{*} \bowtie G$ be  Here we have replaced $L$ by $L^{\prime} = L\rho$, so that $\rho|_{x} \neq \rho$ for all $x\in X$ (I've just done that to make the construction clearer). So,

$$\rho L^{'} = T \rho^{2}.$$  
\end{comment}

Then we have

\begin{displaymath}
     s_{1}(\rho) =
      \begin{pmatrix}
        0 & 1 & 0 \\
        0 & 0 & 1 \\
        1 & 0 & 0
      \end{pmatrix},
\end{displaymath}
\begin{displaymath}
     s_{1}(\sigma) =
      \begin{pmatrix}
        0 & 1 & 0 \\
        1 & 0 & 0 \\
        0 & 0 & 1
      \end{pmatrix},
\end{displaymath}
\begin{displaymath}
     s_{2}(\rho) =
      \begin{pmatrix}
        0 & 0 & 0 & 1 & 0 & 0 & 0 & 0 & 0 \\
				0 & 0 & 0 & 0 & 1 & 0 & 0 & 0 & 0 \\
				0 & 0 & 0 & 0 & 0 & 1 & 0 & 0 & 0 \\
				0 & 0 & 0 & 0 & 0 & 0 & 0 & 1 & 0 \\
				0 & 0 & 0 & 0 & 0 & 0 & 0 & 0 & 1 \\
				0 & 0 & 0 & 0 & 0 & 0 & 1 & 0 & 0 \\
				0 & 0 & 1 & 0 & 0 & 0 & 0 & 0 & 0 \\
				1 & 0 & 0 & 0 & 0 & 0 & 0 & 0 & 0 \\
				0 & 1 & 0 & 0 & 0 & 0 & 0 & 0 & 0 
      \end{pmatrix}
\end{displaymath}
and
\begin{displaymath}
     s_{2}(\sigma) =
      \begin{pmatrix}
        0 & 0 & 0 & 0 & 0 & 1 & 0 & 0 & 0 \\
				0 & 0 & 0 & 0 & 1 & 0 & 0 & 0 & 0 \\
				0 & 0 & 0 & 1 & 0 & 0 & 0 & 0 & 0 \\
				0 & 0 & 1 & 0 & 0 & 0 & 0 & 0 & 0 \\
				0 & 1 & 0 & 0 & 0 & 0 & 0 & 0 & 0 \\
				1 & 0 & 0 & 0 & 0 & 0 & 0 & 0 & 0 \\
				0 & 0 & 0 & 0 & 0 & 0 & 0 & 1 & 0 \\
				0 & 0 & 0 & 0 & 0 & 0 & 1 & 0 & 0 \\
				0 & 0 & 0 & 0 & 0 & 0 & 0 & 0 & 1 
      \end{pmatrix}.
\end{displaymath}

\chapter{Left Rees Categories}

\section{Outline of chapter}

The aim of this chapter is to generalise ideas from the theory of left Rees monoids and self-similar group actions to the context of left Rees categories and self-similar groupoid actions. The hope is that by generalising these ideas, one can apply the algebraic theory of self-similarity more widely. 
\begin{comment}
It has been known for some time that there exist connections between self-similar groups, automaton presented groups and automorphism groups of trees.
In our earlier work, we proved further connections between the structures of self-similar groups, monoid HNN-extensions and left Rees monoids.
In this paper, we will generalise each of the above concepts from monoids / groups to categories / groupoids. Self-similar groups become self-similar groupoids, automaton presented groups become automaton presented groupoids, automorphism groups of regular rooted trees become path automorphism groupoids of directed graphs, monoid HNN-extensions become category HNN-extensions and left Rees monoids become left Rees categories.
\end{comment}
In Section 3.2 we will describe how the correspondence between left Rees monoids and self-similar group actions outlined in Section 2.2 can naturally be generalised to a correspondence between left Rees categories and self-similar groupoid actions. This can essentially be deduced from the work of \cite{JonesLawsonGraph} and \cite{LawsonSubshifts}. We will here flesh out the details for the sake of completeness. 
%In Section 3.3 we will discover when Zappa-Sz\'ep products of free categories and self-similar groupoids can be extended to Zappa-Sz\'ep products of free groupoids and self-similar groupoids. 
We will show in Section 3.3 that every left Rees category is the category HNN-extension of a groupoid. From this we will deduce several facts relating self-similar groupoid actions to Bass-Serre theory, in particular showing that fundamental groupoids of graphs of groups are precisely the groupoids of fractions of Rees categories with totally disconnected groupoids of invertible elements.
In Section 3.4 we will encounter the notion of a path automorphism groupoid of a graph. This is a direct generalisation of the automorphism group of a regular rooted tree, as the vertices of a regular $n$-rooted tree can be viewed as paths in a graph with one vertex and $n$ edges. We will see that certain self-similar groupoid actions can be described in terms of a functor into a path automorphism groupoid. In Section 3.5 we will consider how one might define wreath products in the context of groupoid theory and we will see how one might generalise the wreath recursion to the context of self-similar groupoid actions.
We will see in Section 3.6 a method of obtaining self-similar groupoid actions from automata. 
In Section 3.7 an indication will be given as to how one might generalise the ideas of self-similar group actions arising from iterated function systems to self-similar groupoid actions arising from graph iterated function systems.
We will investigate the representation theory of left Rees categories in Section 3.8 and will show a connection with the representation theory of finite-dimensional algebras when the left Rees category is finite.
Finally, in Section 3.9 we will see how one can associate an inverse semigroup to a left Rees category in a natural way. This will allow us to connect this work with the work of Nivat and Perrot, and will be useful in tackling examples in Chapter 4.

\section{Left Rees categories and self-similar groupoid actions}

We begin by giving the background definitions required for this chapter. As stated in the introduction all categories in this chapter will be assumed to be small. 

A {\em principal right ideal} in a category $C$ is a subset of the form $xC$ where $x \in C$. 
Analogously to the case of monoids, a category $C$ will be said to be {\em right rigid} if $xC \cap yC \neq \emptyset$ implies that $xC \subseteq yC$ or $yC \subseteq xC$.
We will then use the term {\em left Rees category} to describe a left cancellative, right rigid small category in which each principal right ideal is properly contained in only finitely many distinct principal right ideals. 
A left Rees monoid is then precisely a left Rees category with a single object. 
A \emph{Leech category} is a left cancellative small category such that any pair of arrows with a common range that can be completed to a commutative square have a pullback and so left Rees categories are examples of Leech categories. 
Analogously, a \emph{right Rees category} is a right cancellative, left rigid category in which each principal left ideal is properly contained in only finitely many distinct principal left ideals. A category is \emph{Rees} if it is both a left and right Rees category.

An element $x$ in a category $C$ is said to be {\em indecomposable} iff $x = yz$ implies that either $y$ or $z$ is invertible.
A principal right ideal $xC$ is said to be {\em submaximal} if $xC \neq \ran(x)C$ and there are no proper principal right ideals between $xC$ and $\ran(x)C$.

We will now summarise some results about left cancellative categories whose proofs can be found in \cite{JonesLawsonGraph}.

\begin{lemma} Let $C$ be a left cancellative category. 
\label{proplcc}
\begin{enumerate}
\item If $a = xy$ is an identity then $x$ is invertible with inverse $y$.
\item We have that $xC = yC$ iff $x = yg$ where $g$ is an invertible element.
\item $xC = aC$ for some identity $a\in C_{0}$ iff $x$ is invertible.
\item The maximal principal right ideals are those generated by identities.
\item The non-invertible element $x$ is indecomposable iff $xC$ is submaximal.
\item The set of invertible elements is trivial iff for all identities $a\in C_{0}$ we have that $a = xy$ implies that either $x$ or $y$ is an identity.
\end{enumerate}
\end{lemma}

One is then (\cite{JonesLawsonGraph}) led to the following result which is a generalisation of a similar result for free monoids.

\begin{proposition} 
\label{freetriv}
A category is free if and only if it is a left Rees category having a trivial groupoid of invertible elements.
\end{proposition}

It follows from Lemma 3.6 of \cite{LawsonSubshifts} that a left Rees category which is right cancellative is in fact a Rees category.

We shall now describe the structure of arbitrary left Rees categories in terms of free categories. One can view this as a generalisation of the connection between self-similar group actions and left Rees monoids.
Let $G$ be a groupoid with set of identities $G_{0}$ and let $C$ be a category with set of identities $C_{0}$.
We shall suppose that there is a bijection between $G_{0}$ and $C_{0}$ and, to simplify notation, we shall identify these two sets.
Denote by $G \ast C$ the set of pairs
$(g,x)$ such that $g^{-1}g = \ran (x)$. 
We shall picture such pairs as follows:
$$\spreaddiagramrows{2pc}
\spreaddiagramcolumns{2pc}
\diagram
&
&
\\
&\uto^{g}
&\lto^{x} 
\enddiagram$$
We suppose that there is a function 
$$G \ast C \rightarrow C \text{ denoted by } (g,x) \mapsto g \cdot x$$
which gives a left action of $G$ on $C$
and a function
$$G \ast C \rightarrow G \text{ denoted by } (g,x) \mapsto g|_{x}$$
which gives a right action of $C$ on $G$
such that
\begin{description}
\item[{\rm (C1)}] $\ran (g \cdot x) = gg^{-1}$.
\item[{\rm (C2)}] $\dom (g \cdot x) = g|_{x}(g|_{x})^{-1}$.
\item[{\rm (C3)}] $\dom (x) = (g|_{x})^{-1}g|_{x}$.
\end{description}
This information is summarised by the following diagram
$$\spreaddiagramrows{2pc}
\spreaddiagramcolumns{2pc}
\diagram
&
&\lto_{g \cdot x}
\\
&\uto^{g}
&\lto^{x} \uto_{g|_{x}}
\enddiagram$$
We also require that the following axioms be satisfied:
\begin{description}
\item[{\rm (SS1)}] $\ran (x) \cdot x = x$.

\item[{\rm (SS2)}] If $gh$ is defined and $h^{-1}h = \ran (x)$ then $(gh) \cdot x = g \cdot (h \cdot x)$.

\item[{\rm (SS3)}] $gg^{-1} = g \cdot g^{-1}g$.

\item[{\rm (SS4)}] $\ran (x)|_{x} = \dom (x)$.

\item[{\rm (SS5)}] $g|_{g^{-1}g} = g$.

\item[{\rm (SS6)}] If $xy$ is defined and $g^{-1}g = \ran (x)$ then $g|_{xy} = (g|_{x})|_{y}$.

\item[{\rm (SS7)}] If $gh$ is defined and $h^{-1}h = \ran (x)$ then $(gh)|_{x} = g|_{h \cdot x} h|_{x}$.

\item[{\rm (SS8)}] If $xy$ is defined and $g^{-1}g = \ran (x)$ then $g \cdot (xy) = (g \cdot x)(g|_{x} \cdot y)$.
\end{description}
If there are maps $g \cdot x$ and $g|_{x}$ satisfying (C1)--(C3) and (SS1)--(SS8) then we say that there is a {\em self-similar action of $G$ on $C$.}

Put 
$$C \bowtie G = \{(x,g) \in C \times G \colon \: \dom (x) = gg^{-1}\}.$$
We represent $(x,g)$ by the diagram
$$\spreaddiagramrows{2pc}
\spreaddiagramcolumns{2pc}
\diagram
&
&\lto_{x}
\\
&
&\uto_{g}
&
\enddiagram$$
Given elements $(x,g)$ and $(y,h)$ satisfying
$g^{-1}g = \ran (y)$
we then have the following diagram
$$\spreaddiagramrows{2pc}
\spreaddiagramcolumns{2pc}
\diagram
&
&\lto_{x}
&
\\
&
&\uto_{g}
&\lto_{y}
\\
&
&
&\uto_{h}
\enddiagram$$
Completing the square enables us to define a partial binary operation on $C \bowtie G$ by 
$$(x,g)(y,h) = (x(g \cdot y), g|_{y}h).$$

The following is now a straightforward reinterpretation of Theorem 4.2 of \cite{LawsonSubshifts}.

\begin{proposition}
\label{propzappa}
Let $G$ be a groupoid having a self-similar action on the category $C$.
\begin{enumerate}

\item $C \bowtie G$ is a category.

\item $C \bowtie G$ contains copies $C^{\prime}$ and $G^{\prime}$ of $C$ and $G$ respectively
such that each element of $C \bowtie G$ can be written as a product
of a unique element from $C^{\prime}$ followed by a unique element from $G^{\prime}$.

\item If $C$ has trivial invertible elements then the set of invertible elements of $C \bowtie G$ is $G'$.

\item If $C$ is left cancellative then so too is $C \bowtie G$.

\item If $C$ is left cancellative and right rigid then so too is $C \bowtie G$.

\item If $C$ is a left Rees category then so too is $C \bowtie G$.

\end{enumerate}
\end{proposition}
\begin{proof}
(1) Define $\dom (x,g) = (g^{-1}g,g^{-1}g)$ and $\ran (x,g) = (\ran (x), \ran (x))$.
The condition for the existence of $(x,g)(y,h)$ is that $\dom (x,g) = \ran (y,h)$.
Axioms (C1),(C2) and (C3) then guarantee the existence of $(x(g \cdot y), g|_{y}h)$
and we can see from the diagram that
$$\dom ((x,g)(y,h)) = \dom (y,h)
\text{ and }
\ran ((x,g)(y,h)) = \ran (x,g).$$
It remains to prove associativity.

Suppose first that 
$$[(x,g)(y,h)](z,k)$$
exists.
The product $(x,g)(y,h)$ exists and so we have the following diagram
$$\spreaddiagramrows{2pc}
\spreaddiagramcolumns{2pc}
\diagram
&
&\lto_{x}
&\lto_{g \cdot y}
\\
&
&\uto_{g}
&\uto_{g|_{y}} \lto_{y}
\\
&
&
&\uto_{h}
\enddiagram$$
similarly 
$[(x,g)(y,h)](z,k)$ exists
and so we have the following diagram
$$\spreaddiagramrows{2pc}
\spreaddiagramcolumns{2pc}
\diagram
&
&\lto_{x(g \cdot y)}
&\lto_{(g|_{y}h) \cdot z}
\\
&
&\uto_{g|_{y}h}
&\uto_{(g|_{y}h)|_{z}} \lto_{z}
\\
&
&
&\uto_{k}
\enddiagram$$
resulting in the product
$$(x(g \cdot y)[(g|_{y}h) \cdot z],(g|_{y}h)|_{z}k).$$
By assumption, 
$x(g \cdot y)[(g|_{y}h) \cdot z]$ exists and so
$(g \cdot y)[(g|_{y}h) \cdot z]$ is non-zero.
Premultiplying by $g^{-1}$ we find that $y(h \cdot z)$ exists
and we use (SS7) and (SS6) to show that
$$(g|_{y}h)|_{z}k = g|_{y(h \cdot z)}h|_{z}k.$$
By (SS2),
$$x(g \cdot y)[(g|_{y}h) \cdot z]
=
x(g \cdot y)(g|_{y} \cdot (h \cdot z)).$$
It now follows that
$$(y,h)(z,k) = (y(h \cdot z),h|_{z}k)$$
exists.
It also follows that 
$(x,g)[(y,h)(z,k)]$ exists
and is equal to 
$$[(x,g)(y,h)](z,k).$$

Next suppose that 
$$(x,g)[(y,h)(z,k)]$$
exists.
This multiplies out to give
$(x[g \cdot (y(h \cdot z))], g|_{y(h \cdot z)}h|_{z}k)$.
By (SS6) and (SS7) we get that
$$g|_{y(h \cdot z)}h|_{z}k
=
(g|_{y}h)|_{z}k,$$
and by (SS8) and (SS2) we get that
$x[g \cdot (y(h \cdot z))]
=
x(g \cdot y)[(g|_{y}h) \cdot z]$.
This completes the proof that 
$C \bowtie G$ is a category.

(2) Define $\iota_{C} \colon \: C \rightarrow C \bowtie G$ by $\iota_{C}(x) = (x,\dom (x))$.
Denote the image of $\iota_{C}$ by $C^{\prime}$.
Note that there exists $\iota_{C}(x)\iota_{C}(y)$ iff $\dom(x) = \ran(y)$ iff $\exists xy$.
In this case 
$$\iota_{C}(x)\iota_{C}(y) = (x,\dom(x))(y,\dom(y)) = (xy,\dom(xy)) = \iota_{C}(xy).$$
Thus the categories $C$ and $C^{\prime}$ are isomorphic.

Now define $\iota_{G} \colon \: G \rightarrow C \bowtie G$
by $\iota_{G}(g) = (gg^{-1},g)$ and denote the image by $G^{\prime}$.
Then once again the categories $G$ and $G^{\prime}$ are isomorphic.

Finally, if we now pick an arbitrary non-zero element $(x,g)$,
then we can write it as
$(x,g) = (x,\dom(x))(gg^{-1},g)$
using the fact that 
$gg^{-1}|_{gg^{-1}} = gg^{-1}$
and
$\dom (x) \cdot gg^{-1} = gg^{-1}$.

(3) Suppose that $C$ has trivial invertible elements.
We need to check that $(x,g)$ is invertible if and only if $x$ is an identity. Suppose $(x,g)$ is invertible. Let $(y,h)$ be its inverse.
Calculating $(x,g)(y,h)$ and $(y,h)(x,g)$ gives $y(h\cdot x) = r(y)$, $x(g\cdot y) = r(x)$, $g^{-1} = h|_{x}$ and $h^{-1} = g|_{y}$. 
To show that $x$ is invertible, we just need to show that $(g\cdot y)x=d(x)$ and we will have proved $x$ is invertible and thus by assumption an identity.
We have that $d(x) = g\cdot r(y) = g\cdot (y(h\cdot x)) = (g\cdot y)(g|_{y}h)\cdot x = (g\cdot y)x$.

Now suppose $x$ is an identity. Then $(x,g) = (gg^{-1},g) \in G'$ and since $G$ is a groupoid, we have $(x,g)$ is invertible.

(4) Suppose that $C$ is left cancellative.
We prove that $C \bowtie G$ is left cancellative.
Suppose that $(x,g)(y,h) = (x,g)(z,k)$.
Then $x(g \cdot y) = x(g \cdot z)$ and $g|_{y}h = g|_{z}k$.
By left cancellation in $C$ it follows that $g \cdot y = g \cdot z$
and by (SS1) we deduce that $y = z$.
Hence $h = k$.
We have therefore proved that $(y,h) = (z,k)$, as required.

(5) Suppose now that $C$ is left cancellative and right rigid.
By (4), we know that $C \bowtie G$ is left cancellative 
so it only remains to be proved that $C \bowtie G$ is right rigid.
Suppose that
$$(x,g)(y,h) = (u,k)(v,l)$$
From the definition of the product it follows that
$x(g \cdot y) = u(k \cdot v)$
and
$g|_{y}h = k|_{v}l$.
From the first equation we see that $xC \cap uC \neq \emptyset$.
Without loss of generality, suppose that $x = uw$.
Then by left cancellation $w(g \cdot y) = k \cdot v$.
Observe that $k^{-1} \cdot (k \cdot v)$ is defined
and so $k^{-1} \cdot (w(g \cdot y))$ is defined by (SS2).
Thus by (SS8), $k^{-1} \cdot w$ is defined.
It is now easy to check that
$$(x,g) = (u,k)(k^{-1} \cdot w, (k|_{k^{-1} \cdot w})^{-1}g).$$

(6) Let $C$ be a left Rees category and let $M = C\bowtie G$.
By (4) and (5), it remains to prove that every principal right ideal is only contained in finitely many distinct principal right ideals. 
We show that $(x,g)M\subseteq (y,h)M$ iff $xC \subseteq yC$, from which it will follow that $M$ is a left Rees category.
If $(x,g)M\subseteq (y,h)M$ then there exists $(z,k)\in M$ with $(x,g) = (y,h)(z,k)$. That is,
$$(x,g) = (y(h\cdot z), h|_{z}k)$$
and so $xC \subseteq yC$.
Now suppose that $x,y\in C$ are such that $xC\subseteq yC$. Then there exists $z\in C$ with $x=yz$.
Let $g,h\in G$ be arbitrary elements with $\dom(x) = gg^{-1}$ and $\dom(y) = hh^{-1}$.
It can easily be verified that
$$(h^{-1}\cdot z, (h|_{h^{-1}\cdot z})^{-1}g)\in M$$
and that
$$(x,g) = (y,h)(h^{-1}\cdot z,(h|_{h^{-1}\cdot z})^{-1}g).$$
\begin{comment}
Now we let us show that $xC = yC$ implies $(x,g)M = (y,h)M$, where $\dom(x) = gg^{-1}$ and $\dom(y) = hh^{-1}$.
If $xC=yC$ then there exists an invertible element $u\in G(C)$ with $x=yu$ and so for $g,h\in G$ with $\dom(x) = gg^{-1}$ and $\dom(y) = hh^{-1}$, we have
$$(x,g) = (y,h)(h^{-1}\cdot u,(h|_{h^{-1}\cdot u})^{-1}g)$$
and
$$(y,h) = (x,g)(g^{-1}\cdot (u^{-1}),(g|_{g^{-1}\cdot (u^{-1})})^{-1}h).$$
Thus $(x,g)M = (y,h)M$. It now follows that $M$ is a left Rees category.
\end{comment}
\end{proof}

We call $C \bowtie G$ the {\em Zappa-Sz\'ep product} of the category $C$ by the groupoid $G$ by analogy to the monoid situation.
It follows from Proposition \ref{propzappa} that the Zappa-Sz\'ep product of a free category by a groupoid is a left Rees category. 
In fact an arbitrary left Rees category is a Zappa-Sz\'ep product of a free category by a groupoid. 

\begin{proposition} 
Every left Rees category is isomorphic to a Zappa-Sz\'ep product of a free category by a groupoid.
\end{proposition}

\begin{proof}
Let $M$ be a left Rees category. 
First, let $X$ be a transversal of the generators of the submaximal principal right ideals of $M$. 
We claim that $X^{*}$, the subcategory of $M$ consisting of all allowed products of elements of $X$, is free.
Suppose 
$$x_{1}\ldots x_{m} = y_{1}\ldots y_{n},$$
where $x_{i},y_{i}\in X$ and this product exists.
Then from the above $y_{1}\ldots y_{n}M \subseteq x_{1}M$. Thus $y_{1}M\cap x_{1}M \neq \emptyset$.
By assumption $y_{1}$ and $x_{1}$ are indecomposable and so $y_{1}M = x_{1}M$. On the other hand,
$X$ was defined to be a transversal and so $x_{1} = y_{1}$. By left cancellativity we thus have 
$$x_{2}\ldots x_{m} = y_{2}\ldots y_{n}.$$
Suppose $m < n$. Continue cancelling and we get $y_{m+1}\ldots y_{n} = e$ for some identity $e$.
But that would imply $eM \subseteq y_{m+1}M$, which cannot happen by Lemma \ref{proplcc}.
Thus $m = n$, and we have $x_{i} = y_{i}$ for each $i$. 

Let $\mathcal{G}$ be the graph with edges elements of $X$ and vertices identities and let $\mathcal{G}^{\ast}$ be the free category on $\mathcal{G}$.
We have shown that $\mathcal{G}^{\ast}$ and $X^{*}$ are isomorphic, so view $\mathcal{G}^{\ast}$ as the subcategory of $M$ containing products of elements of $X$.
Let $G = G(M)$ be the groupoid of invertible elements of $M$ and let $s\in M\setminus G$ be arbitrary.
Since the submaximal ideals of $M$ are generated by indecomposable elements it follows that $sM \subseteq x_{1}M$ for some $x_{1}\in X$.
If this is equality then $s = x_{1}g$ for some $g\in G$. Otherwise $s = x_{1}y_{1}$ for some $y_{1}\in M$. Now we repeat the same argument for $y_{1}$ to get $y_{1} = x_{2}y_{2}$ for some $x_{2}\in X$, $y_{2}\in M$. Continuing in this way we find $s = x_{1}\ldots x_{n}g$ for some $x_{1},\ldots,x_{n}\in X$ and $g\in G$, this process terminating since $s$ is only contained in finitely many principal right ideals. To see that this decomposition is unique, suppose $x_{1}\ldots x_{n}g = y_{1}\ldots y_{m}h$ where $x_{i},y_{j}$ are in $X$ and $g,h\in G$. It follows that $x_{1}M \cap y_{1}M \neq \emptyset$. Since $x_{1},y_{1}$ are indecomposable, we must have $x_{1} = y_{1}$. We then cancel on the left and continue in this manner to find that $m = n$, $x_{i} = y_{i}$ for each $i$ and $g = h$. Thus every element $s\in M$ can be written uniquely in the form $s = xg$ where $x\in \mathcal{G}^{\ast}$ and $g\in G$.

Now define, for $g\in G$, $x\in \mathcal{G}^{\ast}$ such that $\exists gx$, 
$$gx =: (g\cdot x)(g|_{x}).$$
By the above this is well-defined. We claim that this gives a self-similar action of $G$ on $\mathcal{G}^{\ast}$. 
We thus need to show it satisfies (C1) - (C3) and (SS1) - (SS8).

\begin{description}
\item[{\rm (C1)}] $\ran(g\cdot x) = \ran((g\cdot x)(g|_{x})) = \ran(gx) = \ran(g) = gg^{-1}$.

\item[{\rm (C2)}] $\dom(g\cdot x) = \ran(g|_{x}) = g|_{x}(g|_{x})^{-1}$.

\item[{\rm (C3)}] $\dom(x) = \dom(gx) = \dom((g\cdot x)(g|_{x})) = \dom(g|_{x}) = (g|_{x})^{-1}g|_{x}$.

\item[{\rm (SS1) and (SS4)}] $x\dom(x) = x = \ran(x)x = (\ran(x)\cdot x)(\ran(x)|_{x})$ giving $\ran(x)\cdot x = x$ and $\ran(x)|_{x} = \dom(x)$.

In a similar manner, using uniqueness of the decomposition,

\item[{\rm (SS2) and (SS7)}] $(gh)x = ((gh)\cdot x)(gh)|_{x}$ and 
$$g(hx) = g (h\cdot x)(h|_{x}) = g\cdot (h\cdot x) g|_{h\cdot x}h|_{x}.$$

\item[{\rm (SS3) and (SS5)}] $g = gg^{-1}g = g\cdot (g^{-1}g) g|_{g^{-1}g}$.

\item[{\rm (SS6) and (SS8)}] $g(xy) = g\cdot(xy)g|_{xy}$ and
$$(gx)y = (g\cdot x)(g|_{x})y = (g\cdot x)(g|_{x}\cdot y)((g|_{x})|_{y}).$$

\end{description}

\end{proof}

Let $M$ be a left Rees category which is the Zappa-Sz\'ep product of a free category $\mathcal{G}^{\ast}$ and a groupoid $G$. 
For $x\in \mathcal{G}^{\ast}$, let 
$$G^{x} = \left\{ g\in G| r(x) = g^{-1}g\right\}$$ 
and let
$${}^{x}G = \left\{ g\in G| d(x) = g^{-1}g\right\}.$$

We define the map $\rho_{x}:G^{x}\rightarrow {}^{x}G$ by $\rho_{x}(g) = g|_{x}$. A left Rees category is \emph{symmetric} if the maps $\rho_{x}:G^{x}\rightarrow {}^{x}G$ are bijections for each $x\in \mathcal{G}^{\ast}$.

Let us define for $x\in \mathcal{G}^{\ast}$ the \emph{stabiliser of} $x$, $G_x$, and the \emph{orbit of} $x$, $\Omega_{x}$,  as follows:
$$G_{x} = \left\{g\in G^{x} \quad | \quad g\cdot x = x \right\}$$
and
$$\Omega_{x} = \left\{y\in \mathcal{G}^{\ast} \quad | \quad \exists g\in G:g\cdot x = y \right\}.$$

It follows from the fact that $\ran(g) = \ran(g\cdot x)$ that $G_{x}$ is in fact a group. We define the map $\phi_{x}:G_{x}\rightarrow {}^{x}G$ by $\phi_{x}(g) = g|_{x}$, so that $\phi_{x}$ is the restriction of $\rho_{x}$ to the stabiliser of $x$.

Analogously to Lemma \ref{usefullemma1} of Chapter 2 it is easy to see that we have the following:

\begin{lemma} 
\label{usefulLRClem}
Let $(G,\mathcal{G}^{\ast})$ be a self-similar groupoid action.
\begin{description}
\item[{\rm (i)}] The map $\phi_{x}$ is a functor for each $x\in \mathcal{G}^{\ast}$.
\item[{\rm (ii)}] Let $y = g \cdot x$. 
Then $G_{y} = gG_{x}g^{-1}$
and
$$\phi_{y}(h) = g|_{x}\phi_{x}(g^{-1}hg) (g|_{x})^{-1}.$$
\item[{\rm (iii)}] If $\phi_{x}$ is injective then $\phi_{g \cdot x}$ is injective.
\item[{\rm (iv)}] $\phi_{x}$ is injective for all $x \in \mathcal{G}_{1}$ iff $\phi_{x}$ is injective for all $x \in \mathcal{G}^{\ast}$.
\item[{\rm (v)}] $\phi_{x}$ is surjective for all $x \in \mathcal{G}_{1}$ iff $\phi_{x}$ is surjective for all $x \in \mathcal{G}^{\ast}$.
\item[{\rm (vi)}] $\rho_{x}$ is injective for all $x \in \mathcal{G}_{1}$ iff $\rho_{x}$ is injective for all $x \in \mathcal{G}^{\ast}$.
\item[{\rm (vii)}] $\rho_{x}$ is surjective for all $x \in \mathcal{G}_{1}$ iff $\rho_{x}$ is surjective for all $x \in \mathcal{G}^{\ast}$.
\end{description}
\end{lemma}

Let us define the \emph{length} of a non-identity element of $\mathcal{G}^{\ast}$ to be the number of elements of $\mathcal{G}_{1}$ in its unique decomposition, and say that an identity has length 0.

\begin{lemma}
The action of $G$ on $\mathcal{G}^{\ast}$ is length-preserving
\end{lemma}

\begin{proof}
Consider an identity $e\in \mathcal{G}^{\ast}$. 
Then $g\cdot e$ exists iff $e=g^{-1} g$ and so $g\cdot e= g g^{-1}$ by (SS3), which is an identity. 
Suppose the claim is true for all $x\in \mathcal{G}^{\ast}$ with $l(x)< n$ for some $n\geq 1$. 
Let $x\in \mathcal{G}^{\ast}$ be such that $l(x) = n$. 
We see that if $g\in G^{x}$ then $l(g\cdot x) \geq n$ as otherwise $l(x) = l((g^{-1}g)\cdot x) < n$, a contradiction.
So suppose $g \cdot x = yz$ for some $y,z \in \mathcal{G}^{\ast}$ with $l(y) = n-1$. 
Then $g^{-1} \cdot (yz)$ exists and equals $x$. But $g^{-1}\cdot(yz)=(g^{-1}\cdot y)(g^{-1}|_{y}\cdot z)$. 
Thus $l(g^{-1}|_{y}\cdot z) = 1$ and so $l(g\cdot x) = n$.
\end{proof}

\begin{lemma}
\label{LRCcancphi}
A left Rees category $M$ is right cancellative if and only if we have that $\phi_{x}$ is injective for every $x\in \mathcal{G}^{\ast}$.
\end{lemma}

\begin{proof}
($\Rightarrow$) Suppose $g|_{y}=h|_{y}$ for some $g,h\in G_{y}$. Then 
$$(x,g)(y,\ran(y)) = (xy, g|_{y}) = (xy, h|_{y}) = (x,h)(y,\ran(y)).$$
It then follows by right cancellativity that $g=h$.
\\
($\Leftarrow$) Suppose $(x,g)(y,h)=(z,k)(y,h)$. We want to show $x=z$ and $g=k$. Since
$$(x (g\cdot y), g|_{y} h) = (z (k\cdot y), k|_{y} h),$$
we must have $x (g\cdot y) = z (k\cdot y)$ and $g|_{y}h=k|_{y}h$. By the cancellativity of $G$, length-preservation and uniqueness, $x=z$, $g\cdot y = k\cdot y$ and $g|_{y}=k|_{y}$. Let $t = g\cdot y = k\cdot y $. We have
$$(gk^{-1})|_{t} = g|_{k^{-1}\cdot t}k^{-1}|_{t} = g|_{y}k^{-1}|_{t} = k|_{y}k^{-1}|_{t} = k|_{k^{-1}\cdot t}k^{-1}|_{t} = (kk^{-1})|_{t}.$$
Since $g k^{-1}\in G_{t}$ and $\phi_{t}$ is injective, we have $g k^{-1} = k k^{-1}$ and so $g = k$.
\end{proof}

\begin{proposition}
\label{locmonLRCLRM}
Let $M$ be a left Rees category and let $a\in M_{0}$ be an identity. Then the local monoid $aMa$ is a left Rees monoid.
\end{proposition}

\begin{proof}
Let $M = \mathcal{G}^{\ast}G$ be a left Rees category and let $a\in M_{0}$.
A subcategory of a left cancellative category will again be left cancellative, so $aMa$ must be left cancellative.
Suppose $x,y\in \mathcal{G}^{\ast}$, $g,h\in G$ are such that $xg,yh\in aMa$ and $xgaMa\cap yhaMa \neq \emptyset$. Then there exist $z_{1},z_{2}\in \mathcal{G}^{\ast}$ and $u_{1},u_{2}\in G$ such that $xgz_{1}u_{1} = yhz_{2}u_{2}$. It therefore follows that there exists $t\in \mathcal{G}^{\ast}$ with $xt = y$ or $yt = x$. Suppose $xt = y$. Observe that $g^{-1}th\in aMa$ and so $xg g^{-1}th = yh$ in $aMa$. Thus $yhaMa \subseteq xgaMa$. In a similar way if $yt = x$ we find $xgaMa\subseteq yhaMa$. It therefore follows that $aMa$ is right rigid. 
Note that if $xg,yh\in aMa$ are such that $xgaMa \subset yhaMa$ then $y$ is a prefix of $x$. Since $x$ has finite length, there are only finitely many prefixes of $x$ and so there can only be finitely many principal right ideals containing $xgaMa$. Thus $aMa$ is a left Rees monoid.
\end{proof}

\section{Category HNN-extensions and Bass-Serre theory}

\begin{comment}
A \emph{category presentation} for a small category $C$ is written as follows
$$C = \langle \mathcal{G} | x_{i} = y_{i}, x_{i},y_{i}\in \mathcal{G}^{\ast}, i\in I \rangle\ ,$$
where $\mathcal{G}$ is a directed graph, $I$ is an index set, elements of $C$ are equivalence classes of elements of $\mathcal{G}^{\ast}$ and the relation $x_{i} = y_{i}$ tells us that every time we have a path $wx_{i}v$ in $\mathcal{G}$ then this is equivalent to the path $wy_{i}v$ and vice versa. We implicitly assume that in categories described in this way that identity elements behave as identity elements. 

In a similar manner a \emph{groupoid presentation} for a groupoid $G$ is written as follows 
$$G = \langle \mathcal{G} | x_{i} = y_{i}, x_{i},y_{i}\in \mathcal{G}^{\dagger}, i\in I \rangle\ $$
where $\mathcal{G}$ is a directed graph, $I$ is an index set, elements of $G$ are now equivalence classes of elements of $\mathcal{G}^{\dagger}$ and the relation $x_{i} = y_{i}$ tells us that every time we have an element $wx_{i}v$ in $\mathcal{G}^{\dagger}$ then this equivalent to $wy_{i}v$ and vice versa. We implicitly assume now that in groupoids described in this way that identity elements behave as identity elements and groupoid inverses behave as groupoid inverses. 
\end{comment}

In this section we will prove left Rees categories are precisely what we will call category HNN-extensions of groupoids. We will further see how one can interpret ideas from Bass-Serre theory in the context of Rees categories. 

Suppose $C$ is a category given by category presentation $C = \langle \mathcal{G} | \mathcal{R}(C) \rangle\ $,
where here we are denoting the relations of $C$ by $\mathcal{R}(C)$ and suppose there is an index set $I$, submonoids $H_{i}:i\in I$ of $C$ and functors $\alpha_{i}:H_{i}\rightarrow C$. Let $e_{i},f_{i}\in \mathcal{G}_{0}$ be such that $H_{i}\subseteq e_{i}Ce_{i}$ and $\alpha_{i}(H_{i}) \subseteq f_{i}Cf_{i}$. Define $\mathcal{H}$ to be the graph with $\mathcal{H}_{0} = \mathcal{G}_{0}$ and 
$$\mathcal{H}_{1} = \mathcal{G}_{1} \cup \left\{t_{i}|i\in I\right\}$$
where $\ran(t_{i}) = e_{i}$ and $\dom(t_{i}) = f_{i}$. We will say that $M$ is a \emph{category HNN-extension} of $C$ if $M$ is given by the category presentation:
$$M = \langle \mathcal{H} | \mathcal{R}(C), xt_{i} = t_{i}\alpha_{i}(x) \forall x\in H_{i}, i\in I \rangle\ .$$

\begin{comment}
Let $G$ be a groupoid, $H$ a subgroup of $G$ with identity $e\in G_{0}$ and let 
$$K = \left\{g\in G| \dom(g) = e\right\}.$$
Then a \emph{transversal} $T$ of $H$ is a subset of $K$ such that
$$K = \coprod_{g\in T}{gH}.$$
Each set $gH$ will have cardinality equal to the cardinality of $H$. Furthermore, the cardinality of $T$ is independent of the choice of representatives so we define $|G:H| = |T|$. If $|T|$ is finite then we say that $H$ is a \emph{finite index subgroup} of $G$. 

We can now state the main theorem of this section. The proof is almost identitical to the proofs of Theorems \ref{HNNLRM} and \ref{mainthm1}, and so we relegate it to appendix B.
\end{comment}

\begin{thm}
\label{HNNLRC}
Category HNN-extensions $M$ of groupoids $G$ such that each associated submonoid $H_{i}$ in the definition above is a subgroup of $G$ are precisely left Rees categories $M \cong \mathcal{G}^{\ast}\bowtie G$ for some graph $\mathcal{G}$. 
\end{thm}

\begin{proof}
Let $G$ be a groupoid, $H_{i}:i\in I$ subgroups of $G$, $\alpha_{i}:H_{i}\rightarrow G$ group homomorphisms and let
$M$ be the category HNN-extensions of the groupoid $G$ with associated submonoids $H_{i} = e_{i}H_{i}e_{i}$, stable letters $t_{i}:i\in I$ and let $f_{i}\in G_{0}$ be such that $f_{i} = f_{i}\alpha_{i}(H_{i})f_{i}$. We will now prove that $M$ is a left Rees category such that 
$$M \cong \mathcal{G}^{\ast}\bowtie G$$
for some graph $\mathcal{G}$.

For each $i\in I$, let $T_{i}$ be a transversal of left coset representatives of $H_{i}$. Note that for each $i$ an element $u\in G$ with $\dom(u) = e_{i}$ can be written uniquely in the form $u = gh$, where $g\in T_{i}$ and $h\in H_{i}$. We will assume that $e_{i}\in T_{i}$ for each $i$.

We claim that a normal form for elements $s\in M$ is
$$s = g_{1}t_{i_{1}}g_{2}t_{i_{2}}\cdots g_{m}t_{i_{m}}u$$
where $g_{k}\in T_{i_{k}}$ and $u\in G$.

An element $s\in M$ can definitely be written in the form 
$$s = v_{1}t_{i_{1}}v_{2}t_{i_{2}}\cdots v_{m}t_{i_{m}}w$$
with $v_{k},w\in G$. There will be a unique $g_{1}\in T_{i_{1}}$, $h_{1}\in H_{i_{1}}$ such that
$$v_{1} = g_{1}h_{1}.$$
So
$$s = g_{1}h_{1}t_{i_{1}}v_{2}t_{i_{2}}\cdots v_{m}t_{i_{m}}w.$$
We see that $h_{1}t_{i_{1}} = t_{i_{1}}\rho_{i_{1}}(h_{1})$ and thus
$$s = g_{1}t_{i_{1}}\rho_{i_{1}}(h_{1})v_{2}t_{i_{2}}\cdots v_{m}t_{i_{m}}w.$$
We can continue this process by writing
$$\rho_{i_{k}}(h_{k})v_{k+1} = g_{k+1}h_{k+1}$$
with $g_{k+1}\in T_{i_{k+1}}$, $h_{k+1}\in H_{i_{k+1}}$ and then noting
$$h_{k+1}t_{i_{k+1}} = t_{i_{k+1}}\rho_{i_{k+1}}(h_{k+1}).$$

So we see that we can write $s$ in the form 
$$s = g_{1}t_{i_{1}}g_{2}t_{i_{2}}\cdots g_{m}t_{i_{m}}u$$
where $g_{k}\in T_{i_{k}}$ and $u\in G$. We will see in due course that this is in fact a unique normal form.

Let $\mathcal{G}$ be the graph with $\mathcal{G}_{0} = G_{0} = M_{0}$ and
$$\mathcal{G}_{1} = \left\{gt_{i}|g\in T_{i}, i\in I\right\},$$
where the domain of the edge $gt_{i}$ will be $\dom(gt_{i})\in M_{0}$, similarly for ranges.

We will now consider the free category $\mathcal{G}^{\ast}$. Note that since we haven't yet shown that the normal forms above are unique normal forms, distinct elements of $\mathcal{G}^{\ast}$ might correspond to the same element of $M$. We will now define a self-similar action of $G$ on $\mathcal{G}^{\ast}$. 

Let $y = x_{1}\ldots x_{m}\in \mathcal{G}^{\ast}$ and $g\in G$. Each $x_{k}$ is of the form $x_{k} = u_{k} t_{i_{k}}$ where $u_{k} \in T_{i_{k}}$. Now there exist unique elements $g_{1}\in T_{i_{1}}$, $h_{1}\in H_{i_{1}}$ with $g u_{1} = g_{1} h_{1}$ in $G$. 
%Further there are unique elements $g_{2}\in T_{i_{2}}$, $h_{2}\in H_{i_{2}}$ with $\rho_{i_{1}}(h_{1}) u_{2} = g_{2} h_{2}$ in $G$. 
Then for each $1<k\leq m$ we will let $g_{k}\in T_{i_{k}}$, $h_{k}\in H_{i_{k}}$ be the unique elements with $\alpha_{i_{k-1}}(h_{k-1}) u_{k} = g_{k} h_{k}$ in $G$. Finally we will let $u = \alpha_{i_{m}}(h_{m})$. We thus define
$$g\cdot (x_{1}\ldots x_{m}) = y_{1}\ldots y_{m}$$
where $y_{k} = g_{k}t_{i_{k}}$ and
$$g|_{(x_{1}\ldots x_{m})} = u.$$  
We will define $g\cdot \dom(g) = \ran(g)$ and $g|_{\dom(g)} = g$ for $g\in G$.
We now check this describes a self-similar groupoid action:
\begin{description}
\item[{\rm (SS3), (SS5), (SS6) and (SS8)}] These are true by construction.
\item[{\rm (SS1) and (SS4)}] These follow from the fact that $e_{i}\in T_{i}$ and $\alpha_{i}(e_{i}) = f_{i}$ for each $i\in I$.
\item[{\rm (SS2) and (SS7)}] If $h u = g_{1}h_{1}$ and $(gh) u = g_{2}h_{2}$ for $g,h\in G$, $u,g_{1},g_{2}\in T_{i_{k}}$, $h_{1},h_{2}\in H_{i_{k}}$ then 
$$g g_{1} = ghu h_{1}^{-1} = g_{2}h_{2}h_{1}^{-1}$$
and since $\alpha_{i_{k}}$ is a functor we also have
$$\alpha_{i_{k}}(h_{2}) = \alpha_{i_{k}}(h_{2}h_{1}^{-1}) \alpha_{i_{k}}(h_{1}).$$
Thus $(gh) \cdot x = g \cdot (h \cdot x)$ and $(gh)|_{x} = g|_{h \cdot x} h|_{x}$ for all $g,h\in G$ and $x\in \mathcal{G}^{\ast}$ with $\dom(g)=\ran(h)$ and $\dom(h) = \ran(x)$.
\end{description}
Let $C = \mathcal{G}^{\ast} \bowtie G$ be the associated Zappa-Sz\'{e}p product. We define a map $\theta:C\rightarrow M$ by
$$\theta(x_{1}\ldots x_{m},g) = g_{1}t_{i_{1}}\ldots g_{m}t_{i_{m}} g$$
where $x_{k} = g_{k}t_{i_{k}}$ for each $1\leq k\leq m$. By the above work on normal forms for $M$, we see that $\theta$ is surjective. We now check that $\theta$ is a functor. Let $(x_{1}\ldots x_{m},v_{1}),(y_{1}\ldots y_{r},v_{2})\in C$ be arbitrary with $\dom(v_{1}) = \ran(y_{1})$. Suppose $x_{k} = u_{k}t_{i_{k}}$ and $y_{k} = g_{k}t_{j_{k}}$ for each $k$. Let $g_{1}^{\prime}\in T_{j_{1}}, h_{1}\in H_{j_{1}}$ be such that $v_{1} g_{1} = g_{1}^{\prime} h_{1}$, for each $1<k\leq r$ let $g_{k}^{\prime}\in T_{j_{k}}, h_{k}\in H_{j_{k}}$ be such that $\alpha_{j_{k-1}}(h_{k-1}) g_{k} = g_{k}^{\prime}h_{k}$, let $u = \alpha_{j_{r}}(h_{r})v_{2}$ and let $y_{k}^{\prime} = g_{k}^{\prime}t_{j_{k}} \in \mathcal{G}$ for $1\leq k\leq r$. Then
$$(x_{1}\ldots x_{m},v_{1})(y_{1}\ldots y_{r},v_{2}) = (x_{1}\ldots x_{m}y_{1}^{\prime}\ldots y_{r}^{\prime}, u)$$
and so
\begin{eqnarray*}
\theta((x_{1}\ldots x_{m},v_{1})(y_{1}\ldots y_{r},v_{2})) &=& \theta(x_{1}\ldots x_{m}y_{1}^{\prime}\ldots y_{r}^{\prime}, u) \\
&=& u_{1}t_{i_{1}}\ldots u_{m}t_{i_{m}}g_{1}^{\prime}t_{j_{1}}\ldots g_{r}^{\prime}t_{j_{r}}u\\
&=& u_{1}t_{i_{1}}\ldots u_{m}t_{i_{m}}v_{1}g_{1}t_{j_{1}}\ldots g_{r}t_{j_{r}} v_{2} \\
&=& \theta(x_{1}\ldots x_{m},v_{1})\theta(y_{1}\ldots y_{r},v_{2}).
\end{eqnarray*}
Thus $\theta$ is a functor.
Since $e_{i}\in T_{i}$ for each $i$ there exist $x_{i}\in \mathcal{G}_{1}$ with $x_{i} = t_{i}$. We see that $\theta(x_{i},\dom(x_{i})) = t_{i}$ for each $i$ and $\theta(\ran(g),g) = g$ for each $g\in G$. We know from the earlier theory that 
$$\mathcal{G}^{\ast} \cong \left\{(x,\dom(x)) | x\in \mathcal{G}^{\ast}\right\}\subseteq C$$
and
$$G \cong \left\{(\ran(g),g) | g\in G\right\}\subseteq C.$$
Let us denote the element $(x_{i},\dom(x_{i}))\in C$ by $y_{i}$, the image of $G$ in $C$ by $G^{\prime}$, the image of $H_{i}$ in $C$ by $H_{i}^{\prime}$ and let us denote by $\alpha_{i}^{\prime}$ the functor $\alpha_{i}^{\prime}:H_{i}^{\prime}\rightarrow G^{\prime}$ given by 
$\alpha_{i}^{\prime}(1,h) = (f_{i},\alpha_{i}(h))$. 
Then we see that $G^{\prime}\cup \left\{y_{i}:i\in I\right\}$ generates $C$. Further, for each $h\in H_{i}^{\prime}$ we have 
$hy_{i} = y_{i}\alpha_{i}^{\prime}(h)$. Let $\mathcal{H}$ be the graph with $\mathcal{H}_{0} = G_{0}^{\prime}$ and $\mathcal{H}_{1} = G_{1}^{\prime}\cup\left\{y_{i}:i\in I\right\}$. Then we see that $C$ is given by the category presentation
$$C\cong \langle \mathcal{H} \quad | \quad \mathcal{R}(G^{\prime}), \quad hy_{i} = y_{i}\alpha_{i}^{\prime}(h), i\in I, \quad \mathcal{S}\rangle,$$
where $\mathcal{R}(G^{\prime})$ denotes the relations of $G$ in terms of $G^{\prime}$ and $\mathcal{S}$ denotes whatever relations are needed so that $C$ really is given by this presentation. Since $\theta(\ran(g),g) = g$ for each $(\ran(g),g)\in G^{\prime}$, $\theta(y_{i}) = t_{i}$ for each $i$ and all the relations of $M$ hold in $C$ it follows from the fact that $\theta$ is an surjective functor that $\theta$ is in fact an isomorphism.

We have therefore shown that every category HNN-extension of a groupoid is a left Rees category. We will now show that every left Rees category is a category HNN-extension.

Let $M=\mathcal{G}^{*}G$ be a left Rees category. We will say two elements $x,y\in \mathcal{G}_{1}$ are in the same orbit under the action of $G$ if there exist elements $g,h\in G$ with $gx = yh$. This defines an equivalence relation on $\mathcal{G}_{1}$. Let $X$ be a subset of $\mathcal{G}_{1}$ such that $X$ contains precisely one element in each orbit of the action of $G$ on $\mathcal{G}_{1}$. 

Let us write 
$$X = \left\{t_{i}|i\in I \right\}$$
and let $X^{\ast}$ denote the set of all allowed products of elements of $X$ together with all the identity elements of $M$.

Define $\rho_{i}:=\rho_{t_{i}}$, the map which sends an element $g\in G$ with $\dom(g) = \ran(t_{i})$ to the element $g|_{t_{i}}\in G$. Let $H_{i}=G_{t_{i}}$ be the stabiliser of $t_{i}$ under the action of $G$ let $T_{i}$ be a transversal of $H_{i}$ and let $\mathcal{H}$ be the directed graph with $\mathcal{H}_{0} = \mathcal{G}_{0} = G_{0}$ and $\mathcal{H}_{1} = G_{1} \cup X$. 

Define $\Gamma$ by the following category presentation:

$$\Gamma=\langle \mathcal{H} | \mathcal{R}(G), \quad h t_{i} = t_{i} \rho_{i}(h) \quad \forall h\in H_{i}, i\in I \rangle,$$

where $\mathcal{R}(G)$ denotes the relations of $G$, so that $\Gamma$ is a category HNN-extension of $G$.

A few observations:
\begin{enumerate}
\item Each $x\in \mathcal{G}_{1}$ is given uniquely by $x = g\cdot t_{i}$ for some $i$, where $g\in T_{i}$. 
\item For $g\in T_{i}$ we have $g\cdot t_{i} = g t_{i}(g|_{x_{i}})^{-1} = g t_{i}(\rho_{i}(g))^{-1}$.
\item For each $i$ every element $u\in G$ with $\dom(u) = \ran(t_{i})$ can be written uniquely in the form $u=gh$, where $g\in T_{i}$ and $h\in H_{i}$.
\end{enumerate}

One can check in exactly the same way as for the first half of this theorem that every element of $\Gamma$ can be written in the form
$$g_{1}t_{i_{1}}(\rho_{i_{1}}(g_{1}))^{-1}g_{2}t_{i_{2}}(\rho_{i_{2}}(g_{2}))^{-1}\cdots g_{m}t_{i_{m}}(\rho_{i_{m}}(g_{m}))^{-1}u$$
where $g_{k}\in T_{i_{k}}$ and $u\in G$. 

Let us check that this is a unique normal form for elements of $\Gamma$. Note first that the relations of $\Gamma$ do not allow us to swap or remove $t_{i}$'s, so two equal elements of $\Gamma$ must have the same number of $t_{i}$'s and they must be in the same positions relative to each other.

Now suppose
$$g_{1}t_{i_{1}}(\rho_{i_{1}}(g_{1}))^{-1}\cdots g_{m}t_{i_{m}}(\rho_{i_{m}}(g_{m}))^{-1}u
= g_{1}^{\prime}t_{i_{1}}(\rho_{i_{1}}(g_{1}^{\prime}))^{-1}\cdots g_{m}^{\prime}t_{i_{m}}(\rho_{i_{m}}(g_{m}^{\prime}))^{-1}v$$
in $\Gamma$, where these elements are written in the above form. Then by the unique normal form for elements of $\Gamma$ we must have $g_{1} = g_{1}^{\prime}$. Thus by left cancellativity of $\Gamma$ we have
$$g_{2}t_{i_{2}}(\rho_{i_{2}}(g_{2}))^{-1}\cdots g_{m}t_{i_{m}}(\rho_{i_{m}}(g_{m}))^{-1}u 
=g_{2}^{\prime}t_{i_{2}}(\rho_{i_{2}}(g_{2}^{\prime}))^{-1}\cdots g_{m}^{\prime}t_{i_{m}}(\rho_{i_{m}}(g_{m}^{\prime}))^{-1}v.$$
Continuing in this way one see that $g_{k} = g_{k}^{\prime}$ for each $k$ and $u = v$.

Observe that the left Rees category $M$ can be given by category presentation as
$$M=\langle \mathcal{H} | \mathcal{R}(G), \quad h t_{i} = t_{i} \rho_{i}(h) \quad \forall h\in H_{i}, i\in I, \quad \mathcal{S} \rangle,$$
where $\mathcal{S}$ are any additional relations required to make this really a presentation for $M$. It is now easy to see that the map $f:\Gamma\rightarrow M$ defined on generators by $f(t_{i}) = t_{i}$ for $i\in I$ and $f(g) = g$ for $g\in G$ is an isomorphism of categories.

\end{proof}

If $G$ be a groupoid and $H$ a subgroup of $G$ then we will call a functor $\phi:H\rightarrow G$ a \emph{partial endomorphism} of $G$. Given a groupoid $G$ and partial endomorphisms $\phi_{i}:H_{i}\rightarrow G$ then Theorem \ref{HNNLRC} says that we can form a left Rees category $M(\phi_{i}:i\in I)$ as follows. For each $i\in I$ let $a_{i},b_{i}\in G_{0}$ be such that $H_{i} = a_{i}H_{i}a_{i}$ and $\phi_{i}(H_{i}) = b_{i}\phi_{i}(H_{i})b_{i}$. Define $\mathcal{H}$ to be the graph with $\mathcal{H}_{0} = G_{0}$ and $\mathcal{H}_{1} = G_{1}\cup \left\{t_{i}:i\in I\right\}$ where the edge $t_{i}$ has $\ran(t_{i}) = a_{i}$ and $\dom(t_{i}) = b_{i}$. Then $M(\phi_{i}:i\in I)$ will have category presentation
$$M(\phi_{i}:i\in I) = \langle \mathcal{H} | \mathcal{R}(G), \quad h t_{i} = t_{i} \rho_{i}(h) \quad \forall h\in H_{i}, i\in I \rangle.$$

By Lemma \ref{LRCcancphi} and Lemma \ref{usefulLRClem} (iii) and (iv) we see that a left Rees category $M(\phi_{i}:i\in I)$ is right cancellative (and so a Rees category) if and only if $\phi_{i}$ is injective for each $i\in I$.

We have the following theorem which describes in terms of partial endomorphisms when two left Rees categories are isomorphic.

\begin{theorem}
\label{LRCconjclass}
Let $G,G^{\prime}$ be groupoids, $H_{i}:i\in I$ subgroups of $G$, $H_{j}^{\prime}:j\in J$ subgroups of $G^{\prime}$ and suppose $\phi_{i}:H_{i}\rightarrow G$, $\phi_{j}^{\prime}:H_{j}^{\prime}\rightarrow G^{\prime}$ are partial endomorphisms for each $i\in I$, $j\in J$. Then the left Rees categories $M(\phi_{i}:i\in I)$ and $M(\phi_{j}^{\prime}:j\in J)$ are isomorphic if and only if there is an isomorphism $f:G\rightarrow G^{\prime}$, a bijection $\gamma:I\rightarrow J$ and elements $u_{i},v_{i}\in G^{\prime}$ with $u_{i}^{-1}f(H_{i})u_{i} = H_{\gamma(i)}^{\prime}$ and $v_{i}f(\phi_{i}(h))v_{i}^{-1} = \phi_{\gamma(i)}^{\prime}(u_{i}^{-1}f(h)u_{i})$ for every $i\in I$ and $h\in H_{i}$.
\end{theorem} 

\begin{proof}
($\Rightarrow$) For each $i\in I$, $j\in J$ let $a_{i},b_{i}\in G$, $a_{j}^{\prime},b_{j}^{\prime}\in G^{\prime}$ be the identities with $H_{i} = a_{i}H_{i}a_{i}$, $\phi_{i}(H_{i}) = b_{i}\phi_{i}(H_{i})b_{i}$, $H_{j}^{\prime} = a_{j}^{\prime}H_{j}^{\prime}a_{j}^{\prime}$ and $\phi_{j}^{\prime}(H_{j}^{\prime}) = b_{j}^{\prime}\phi_{j}^{\prime}(H_{j}^{\prime})b_{j}^{\prime}$.
Let $\mathcal{G}$ and $\mathcal{H}$ be the graphs with $\mathcal{G}_{0} = G_{0}$ and $\mathcal{G}_{1} = G_{1}\cup \left\{t_{i}:i\in I\right\}$ where $a_{i} \stackrel{t_{i}}{\leftarrow} b_{i}$ and $\mathcal{H}_{0} = G_{0}^{\prime}$ and $\mathcal{H}_{1} = G_{1}^{\prime}\cup \left\{r_{j}:j\in J\right\}$ where $a_{j}^{\prime} \stackrel{r_{j}}{\leftarrow} b_{j}^{\prime}$.
We can write $M(\phi_{i}:i\in I)$ and $M(\phi_{j}^{\prime}:j\in J)$ in terms of category presentation as
$$M(\phi_{i}:i\in I) = \langle \mathcal{G} | \mathcal{R}(G), \quad h t_{i} = t_{i} \phi_{i}(h) \quad \forall h\in H_{i}, i\in I \rangle $$
and
$$M(\phi_{j}^{\prime}:j\in J) = \langle \mathcal{H} | \mathcal{R}(G), \quad h r_{j} = r_{j} \phi_{j}^{\prime}(h) \quad \forall h\in H_{j}^{\prime}, j\in J \rangle .$$

Suppose $f:M(\phi_{i}:i\in I)\rightarrow M(\phi_{j}^{\prime}:j\in J)$ is an isomorphism. Note that $f(G) = G^{\prime}$.
Each submaximal principal two-sided ideal of $M(\phi_{i}:i\in I)$ is generated by a $t_{i}$ and likewise for $M(\phi_{j}^{\prime}:j\in J)$. 
It follows that there is a bijection $\gamma:I\rightarrow J$ and elements $u_{i},v_{i}\in G^{\prime}$ for each $i\in I$ with $f(t_{i}) = u_{i}r_{\gamma(i)}v_{i}$.
Define maps $\alpha_{i}:a_{i}Ga_{i}\rightarrow a_{i}^{\prime}G^{\prime}a_{i}^{\prime}$, $\beta_{i}:b_{i}Gb_{i}\rightarrow b_{i}^{\prime}G^{\prime}b_{i}^{\prime}$ for each $i\in I$ by $\alpha_{i}(g) = u_{i}^{-1}f(g)u_{i}$ and $\beta_{i}(g) = v_{i}f(g)v_{i}^{-1}$.
It is clear that $\alpha_{i}$ and $\beta_{i}$ are isomorphisms for each $i\in I$ since local groups on the same connected component of a groupoid are isomorphic.

We now verify that $\alpha_{i}:H_{i}\rightarrow H_{\gamma(i)}^{\prime}$ and $\beta_{i}\phi_{i} = \phi_{\gamma(i)}^{\prime}\alpha_{i}$ for each $i\in I$. If $h\in H_{i}$ then
\begin{eqnarray*}
\alpha_{i}(h)r_{\gamma(i)} &=& u_{i}^{-1}f(h)u_{i}r_{\gamma(i)} = u_{i}^{-1}f(h)u_{i}r_{\gamma(i)}v_{i}v_{i}^{-1} = u_{i}^{-1}f(h)f(t_{i})v_{i}^{-1} \\
&=& u_{i}^{-1}f(ht_{i})v_{i}^{-1} = u_{i}^{-1}f(t_{i}\phi_{i}(h))v_{i}^{-1} = u_{i}^{-1}f(t_{i})f(\phi_{i}(h))v_{i}^{-1} \\
&=& r_{\gamma(i)}v_{i}f(\phi_{i}(h))v_{i}^{-1} = r_{\gamma(i)}\beta_{i}(\phi_{i}(h)).
\end{eqnarray*}
Thus $\alpha_{i}(H_{i})\subseteq H_{\gamma(i)}^{\prime}$ and $\beta_{i}\phi_{i} = \phi_{\gamma(i)}^{\prime}\alpha_{i}$.
Further, if $h\in H_{\gamma(i)}^{\prime}$ then
\begin{eqnarray*}
f(\alpha_{i}^{-1}(h)t_{i}) &=& f(f^{-1}(u_{i}hu_{i}^{-1})t_{i}) = u_{i}hu_{i}^{-1}u_{i}r_{\gamma(i)}v_{i} = u_{i}hr_{\gamma(i)}v_{i}
= u_{i}r_{\gamma(i)}\phi_{\gamma(i)}^{\prime}(h)v_{i}\\
&=& u_{i}r_{\gamma(i)}v_{i}v_{i}^{-1}\phi_{\gamma(i)}^{\prime}(h)v_{i} = f(t_{i}f^{-1}(v_{i}^{-1}\phi_{\gamma(i)}^{\prime}(h)v_{i})).
\end{eqnarray*}
Since $f$ is an isomorphism this therefore implies that $\alpha_{i}^{-1}(h)t_{i} = t_{i}f^{-1}(v_{i}^{-1}\phi_{\gamma(i)}^{\prime}(h)v_{i})$ and so 
$\alpha_{i}^{-1}(H_{\gamma(i)}^{\prime}) = H_{i}$. 
%Thus $\phi_{i}$ and $\phi_{\gamma(i)}^{\prime}$ are isomorphic for each $i\in I$.

($\Leftarrow$) For each $i\in I$, $j\in J$ let $a_{i},b_{i}\in G$, $a_{j}^{\prime},b_{j}^{\prime}\in G^{\prime}$ be the identities with $H_{i} = a_{i}H_{i}a_{i}$, $\phi_{i}(H_{i}) = b_{i}\phi_{i}(H_{i})b_{i}$, $H_{j}^{\prime} = a_{j}^{\prime}H_{j}^{\prime}a_{j}^{\prime}$ and $\phi_{j}^{\prime}(H_{j}^{\prime}) = b_{j}^{\prime}\phi_{j}^{\prime}(H_{j}^{\prime})b_{j}^{\prime}$.
Let $\mathcal{G}$ and $\mathcal{H}$ be the graphs with $\mathcal{G}_{0} = G_{0}$ and $\mathcal{G}_{1} = G_{1}\cup \left\{t_{i}:i\in I\right\}$ where $a_{i} \stackrel{t_{i}}{\leftarrow} b_{i}$ and $\mathcal{H}_{0} = G_{0}^{\prime}$ and $\mathcal{H}_{1} = G_{1}^{\prime}\cup \left\{r_{j}:j\in J\right\}$ where $a_{j}^{\prime} \stackrel{r_{j}}{\leftarrow} b_{j}^{\prime}$.
We can write $M(\phi_{i}:i\in I)$ and $M(\phi_{j}^{\prime}:j\in J)$ in terms of category presentation as
$$M(\phi_{i}:i\in I) = \langle \mathcal{G} | \mathcal{R}(G), \quad h t_{i} = t_{i} \phi_{i}(h) \quad \forall h\in H_{i}, i\in I \rangle $$
and
$$M(\phi_{j}^{\prime}:j\in J) = \langle \mathcal{H} | \mathcal{R}(G), \quad h r_{j} = r_{j} \phi_{j}^{\prime}(h) \quad \forall h\in H_{j}^{\prime}, j\in J \rangle .$$
Define $\bar{f}:M(\phi_{i}:i\in I)\rightarrow M(\phi_{j}^{\prime}:j\in J)$ on generators by $\bar{f}(g) = f(g)$ for each $g\in G$ and $\bar{f}(t_{i}) = u_{i}r_{\gamma(i)}v_{i}$ for each $i\in I$. Observe that by construction $\bar{f}(gh) = \bar{f}(g) \bar{f}(h)$ for each $g,h\in G$. 
We now check that for each $i\in I$ and $h\in H_{i}$ we have $\bar{f}(h)\bar{f}(t_{i}) = \bar{f}(t_{i})\bar{f}(\phi_{i}(h))$. Let $i\in I$ and $h\in H_{i}$. Then
\begin{eqnarray*}
\bar{f}(h)\bar{f}(t_{i}) &=& f(h)u_{i}r_{\gamma(i)}v_{i} = u_{i}u_{i}^{-1}f(h)u_{i}r_{\gamma(i)}v_{i} = u_{i}r_{\gamma(i)}\phi_{\gamma(i)}^{\prime}(u_{i}^{-1}f(h)u_{i})v_{i} \\
&=& u_{i}r_{\gamma(i)}v_{i}f(\phi_{i}(h))v_{i}^{-1}v_{i} = u_{i}r_{\gamma(i)}v_{i}f(\phi_{i}(h)) = \bar{f}(t_{i})\bar{f}(\phi_{i}(h))
\end{eqnarray*}
and so $\bar{f}$ is a functor. 
To see that $\bar{f}$ is surjective note that for each $j\in J$ there exists $i\in I$ with $\gamma(i) = j$, $f^{-1}(u_{i}^{-1})t_{i}f^{-1}(v_{i}^{-1}) \in M(\phi_{i}:i\in I)$ and $\bar{f}(f^{-1}(u_{i}^{-1})t_{i}f^{-1}(v_{i}^{-1})) = r_{j}$. 
We must finally verify that $\bar{f}$ is injective. Suppose $g_{1}t_{i_{1}}\ldots g_{m}t_{i_{m}}x_{1}, g_{1}^{\prime}t_{i_{1}}\ldots g_{m}^{\prime}t_{i_{m}}x_{2} \in M(\phi_{i}:i\in I)$ are such that
$$\bar{f}(g_{1}t_{i_{1}}\ldots g_{m}t_{i_{m}}x_{1}) = \bar{f}(g_{1}^{\prime}t_{i_{1}}\ldots g_{m}^{\prime}t_{i_{m}}x_{2}).$$
Note that we have assumed that the $t's$ in both expressions are the same since our relations don't allow us to swap them.
We therefore have
$$f(g_{1})u_{i_{1}}r_{\gamma(i_{1})}v_{i_{1}}\ldots f(g_{m})u_{i_{m}}r_{\gamma(i_{m})}v_{i_{m}}f(x_{1}) = f(g_{1}^{\prime})u_{i_{1}}r_{\gamma(i_{1})}v_{i_{1}}\ldots f(g_{m}^{\prime})u_{i_{m}}r_{\gamma(i_{m})}v_{i_{m}}f(x_{2}).$$
For each $i\in I$ let $T_{i}$ be a transversal of $H_{i}$ in $G$ and for each $j\in J$ let $T_{j}^{\prime}$ be a transversal of $H_{j}^{\prime}$ in $G^{\prime}$. 
Let $y_{1}\in T_{\gamma(i_{1})}^{\prime}$, $h_{1},h_{1}^{\prime}\in H_{\gamma(i_{1})}^{\prime}$ be such that
$$y_{1}h_{1} = f(g_{1})u_{i_{1}}$$
and
$$y_{1}h_{1}^{\prime} = f(g_{1}^{\prime})u_{i_{1}}.$$
For $k = 2,\ldots,m$ let $y_{k}\in T_{\gamma(i_{k})}^{\prime}$, $h_{k},h_{k}^{\prime}\in H_{\gamma(i_{k})}^{\prime}$ be such that
$$y_{k}h_{k} = \phi_{\gamma(i_{k-1})}^{\prime}(h_{k-1})v_{i_{k-1}}f(g_{k})u_{i_{k}}$$
and
$$y_{k}h_{k}^{\prime} = \phi_{\gamma(i_{k-1})}^{\prime}(h_{k-1}^{\prime})v_{i_{k-1}}f(g_{k}^{\prime})u_{i_{k}}.$$
Then by the uniqueness of normal forms we have
$$\phi_{\gamma(i_{m})}^{\prime}(h_{m})v_{i_{m}}f(x_{1}) = \phi_{\gamma(i_{m})}^{\prime}(h_{m}^{\prime})v_{i_{m}}f(x_{2}).$$
By assumption, $h_{m} = u_{i_{m}}^{-1}f(z_{m})u_{i_{m}}$ for some $z_{m}\in H_{i_{m}}$ and $h_{m}^{\prime} = u_{i_{m}}^{-1}f(z_{m}^{\prime})u_{i_{m}}$ for some $z_{m}^{\prime}\in H_{i_{m}}$. It follows that
$$\phi_{\gamma(i_{m})}^{\prime}(u_{i_{m}}^{-1}f(z_{m})u_{i_{m}})v_{i_{m}}f(x_{1}) = \phi_{\gamma(i_{m})}^{\prime}(u_{i_{m}}^{-1}f(z_{m}^{\prime})u_{i_{m}})v_{i_{m}}f(x_{2})$$
and so
$$v_{i_{m}}f(\phi_{i_{m}}(z_{m}))v_{i_{m}}^{-1}v_{i_{m}}f(x_{1}) = v_{i_{m}}f(\phi_{i_{m}}(z_{m}^{\prime}))v_{i_{m}}^{-1}v_{i_{m}}f(x_{2})$$
giving
$$f(\phi_{i_{m}}(z_{m}))f(x_{1}) = f(\phi_{i_{m}}(z_{m}^{\prime}))f(x_{2}).$$
Since $f$ is an isomorphism this implies that
$$\phi_{i_{m}}(z_{m})x_{1} = \phi_{i_{m}}(z_{m}^{\prime})x_{2}$$
and so
$$z_{m}t_{i_{m}}x_{1} = z_{m}^{\prime}t_{i_{m}}x_{2}.$$
Rewriting this in terms of $h_{m}$ and $h_{m}^{\prime}$ we have
$$f^{-1}(u_{i_{m}}h_{i_{m}}u_{i_{m}}^{-1})t_{i_{m}}x_{1} = f^{-1}(u_{i_{m}}h_{i_{m}}^{\prime}u_{i_{m}}^{-1})t_{i_{m}}x_{2}$$
and so
$$f^{-1}(u_{i_{m}}y_{m}^{-1}\phi_{\gamma(i_{m-1})}^{\prime}(h_{m-1})v_{i_{m-1}}f(g_{m}))t_{i_{m}}x_{1} = f^{-1}(u_{i_{m}}y_{m}^{-1}\phi_{\gamma(i_{m-1})}^{\prime}(h_{m-1}^{\prime})v_{i_{m-1}}f(g_{m}^{\prime}))t_{i_{m}}x_{2}.$$
Cancelling on the left gives
$$f^{-1}(\phi_{\gamma(i_{m-1})}^{\prime}(h_{m-1})v_{i_{m-1}})g_{m}t_{i_{m}}x_{1} = f^{-1}(\phi_{\gamma(i_{m-1})}^{\prime}(h_{m-1}^{\prime})v_{i_{m-1}})g_{m}^{\prime}t_{i_{m}}x_{2}.$$
Note that $h_{m-1} = u_{i_{m-1}}^{-1}f(z_{m-1})u_{i_{m-1}}$ for some $z_{m-1}\in H_{i_{m-1}}$ and $h_{m-1}^{\prime} = u_{i_{m-1}}^{-1}f(z_{m-1}^{\prime})u_{i_{m-1}}$ for some $z_{m-1}^{\prime}\in H_{i_{m-1}}$. Using this we have
$$f^{-1}(\phi_{\gamma(i_{m-1})}^{\prime}(u_{i_{m-1}}^{-1}f(z_{m-1})u_{i_{m-1}})v_{i_{m-1}})g_{m}t_{i_{m}}x_{1} = f^{-1}(\phi_{\gamma(i_{m-1})}^{\prime}(u_{i_{m-1}}^{-1}f(z_{m-1}^{\prime})u_{i_{m-1}})v_{i_{m-1}})g_{m}^{\prime}t_{i_{m}}x_{2}.$$
Thus 
$$f^{-1}(v_{i_{m-1}}f(\phi_{i_{m-1}}(z_{m-1}))v_{i_{m-1}}^{-1}v_{i_{m-1}})g_{m}t_{i_{m}}x_{1} = f^{-1}(v_{i_{m-1}}f(\phi_{i_{m-1}}(z_{m-1}^{\prime}))v_{i_{m-1}}^{-1}v_{i_{m-1}})g_{m}^{\prime}t_{i_{m}}x_{2}.$$
This gives
$$\phi_{i_{m-1}}(z_{m-1})g_{m}t_{i_{m}}x_{1} = \phi_{i_{m-1}}(z_{m-1}^{\prime})g_{m}^{\prime}t_{i_{m}}x_{2}$$
and so
$$z_{m-1}t_{i_{m-1}}g_{m}t_{i_{m}}x_{1} = z_{m-1}^{\prime}t_{i_{m-1}}g_{m}^{\prime}t_{i_{m}}x_{2}.$$
We then continue in this way to discover that
$$g_{1}t_{i_{1}}\ldots g_{m}t_{i_{m}}x_{1} = g_{1}^{\prime}t_{i_{1}}\ldots g_{m}^{\prime}t_{i_{m}}x_{2}$$
and so $\bar{f}$ is injective. 
\end{proof}

Let $\phi:H\rightarrow G$, $\phi^{\prime}:H^{\prime}\rightarrow G$ be partial endomorphisms of a groupoid $G$ and suppose that $a_{1},a_{2},b_{1},b_{2}\in G$ are identities such that $H = a_{1}Ha_{1}$, $\phi(H) = b_{1}\phi(H)b_{1}$, $H^{\prime} = a_{2}H^{\prime}a_{2}$ and $\phi^{\prime}(H^{\prime}) = b_{2}\phi^{\prime}(H^{\prime})b_{2}$. Then $\phi$ and $\phi^{\prime}$ will be said to be \emph{conjugate} if there exist 
$a_{1} \stackrel{u}{\leftarrow} a_{2}$, $b_{1} \stackrel{v}{\leftarrow} b_{2}$ in $G$ 
such that the maps $\alpha:a_{1}Ga_{1}\rightarrow a_{2}Ga_{2}$, $\beta:b_{1}Gb_{1}\rightarrow b_{2}Gb_{2}$ defined by $\alpha(g) = u^{-1}gu$, $\beta(g) = vgv^{-1}$ satisfy $\alpha(H) = H^{\prime}$ and $\beta\phi = \phi^{\prime}\alpha$.

\begin{corollary}
Let $G$ be a groupoid, $H_{i},H_{i}:i\in I$ subgroups of $G$ and let $\phi_{i}:H_{i}\rightarrow G$ and $\phi_{i}^{\prime}:H_{i}^{\prime}$ be partial endomorphisms for each $i\in I$. If $\phi_{i}$ is conjugate to $\phi_{i}^{\prime}$ for every $i\in I$ then the categories $M(\phi_{i}:i\in I)$ and $M(\phi_{i}^{\prime}:i\in I)$ are isomorphic.
\end{corollary}

Let $\phi:H\rightarrow G$, $\phi^{\prime}:H^{\prime}\rightarrow G^{\prime}$ be partial endomorphisms of groupoids $G$ and $G^{\prime}$ and suppose that $a_{1},b_{1}\in G$ and $a_{2},b_{2}\in G^{\prime}$ are identities such that $H = a_{1}Ha_{1}$, $\phi(H) = b_{1}\phi(H)b_{1}$, $H^{\prime} = a_{2}H^{\prime}a_{2}$ and $\phi^{\prime}(H^{\prime}) = b_{2}\phi^{\prime}(H^{\prime})b_{2}$. Then $\phi$ and $\phi^{\prime}$ will be said to be \emph{isomorphic} if there exist isomorphisms $\alpha:a_{1}Ga_{1}\rightarrow a_{2}G^{\prime}a_{2}$, $\beta:b_{1}Gb_{1}\rightarrow b_{2}G^{\prime}b_{2}$ with $\alpha(H) = H^{\prime}$ and $\beta\phi = \phi^{\prime}\alpha$.

\begin{corollary}
\label{LRCisovirtiso}
Let $G,G^{\prime}$ be groupoids, $H_{i}:i\in I$ subgroups of $G$, $H_{j}^{\prime}:j\in J$ subgroups of $G^{\prime}$, $\phi_{i}:H_{i}\rightarrow G$, $\phi_{j}^{\prime}:H_{j}^{\prime}\rightarrow G^{\prime}$ partial endomorphisms for each $i\in I$, $j\in J$ and suppose that 
$M(\phi_{i}:i\in I)$ and $M(\phi_{j}^{\prime}:j\in J)$ are isomorphic left Rees categories. Then there is a bijection $\gamma:I\rightarrow J$ such that the partial endomorphisms $\phi_{i}$ and $\phi_{\gamma(i)}^{\prime}$ are isomorphic for each $i\in I$. 
\end{corollary}

Using Proposition \ref{univgroupoid} we see that the groupoid of fractions of a Rees category is a groupoid HNN-extension and in addition by comparing the normal forms of elements of a Rees category and Proposition \ref{normalformgroupoidHNN} it is clear that a Rees category embeds in its groupoid of fractions. Since the fundamental groupoid of a graph of groups is a groupoid HNN-extension of a totally disconnected groupoid, every fundamental groupoid of a graph groups is the groupoid of fractions of a Rees category with totally disconnected groupoid of invertible elements, and so there is an underlying self-similar groupoid action.

To explore this connection further we will require the notion of a \emph{diagram of partial homomorphisms} which we now define. A \emph{diagram of partial homomorphisms} $\mathcal{G}_{G}$ consists of
\begin{itemize}
\item A (not necessarily connected) graph $\mathcal{G}$.
\item A group $G_{a}$ for each vertex $a\in \mathcal{G}_{0}$.
\item A subgroup $G_{t}\leq G_{\ran(t)}$ for each edge $t\in \mathcal{G}_{0}$.
\item A homomorphism $\phi_{t}:G_{t}\rightarrow G_{\dom(t)}$ for each edge $t\in \mathcal{G}_{1}$.
\end{itemize}
In other words, a diagram of partial homomorphisms is just a graph of groups without an involution on the underlying graph and such that the maps $\phi_{t}$ are not necessarily injective.

We will say two diagrams of partial homomorphisms $\mathcal{G}_{G}$ and $\mathcal{G}^{\prime}_{G}$ with underlying graphs $\mathcal{G}$ and $\mathcal{G}^{\prime}$ are \emph{equivalent} if
\begin{itemize}
\item There are bijections $\gamma_{0}:\mathcal{G}_{0}\rightarrow \mathcal{G}_{0}^{\prime}$ and $\gamma_{1}:\mathcal{G}_{1}\rightarrow \mathcal{G}_{1}^{\prime}$ with
$\gamma_{0}(\dom(t)) = \dom(\gamma_{1}(t))$ and $\gamma_{0}(\ran(t)) = \ran(\gamma_{1}(t))$ for each $t\in \mathcal{G}_{1}$.
\item For each $a\in \mathcal{G}_{0}$ there is an isomorphism $f_{a}:G_{a}\rightarrow G_{\gamma_{0}(a)}$.
\item For each $t\in \mathcal{G}_{1}$ there are elements $u_{t}\in G_{\ran(\gamma_{1}(t))}$, $v_{t}\in G_{\dom(\gamma_{1}(t))}$ with $$u_{t}^{-1}f_{\ran(t)}(G_{t})u_{t} = G_{\gamma_{1}(t)}$$
and 
$$v_{t}f_{\dom(t)}(\phi_{t}(h))v_{t}^{-1} = \phi_{\gamma_{1}(t)}(u_{t}^{-1}f_{\ran(t)}(h)u_{t})$$
for every $h\in G_{t}$.
\end{itemize}

A \emph{route} in $\mathcal{G}_{G}$ consists of a sequence $g_{1}t_{1}g_{2}t_{2}\cdots g_{m}t_{m}g_{m+1}$ where $t_{k}\in \mathcal{G}_{1}$ for each $k$, $g_{k}\in G_{\ran(t_{k})}$ for $k=1,\ldots,m$ and $g_{k+1}\in G_{\dom(t_{k})}$ for $k = 1,\ldots,m$. We allow for the case $m = 0$, i.e. routes of the form $g\in G_{a}$ for some $a\in\mathcal{G}_{0}$. We write $\dom(g_{1}t_{1}g_{2}t_{2}\cdots g_{m}t_{m}g_{m+1}) = \dom(t_{m})$ and $\ran(g_{1}t_{1}g_{2}t_{2}\cdots g_{m}t_{m}g_{m+1}) = \ran(t_{1})$. For $g\in G_{a}$ viewed as a route we write $\dom(g) = \ran(g) = a$. Let $\sim$ be the equivalence relation on routes in $\mathcal{G}_{G}$ generated by $phtq \sim pt\phi_{t}(h)q$, where $p,q$ are routes and $h\in G_{t}$.

Given a diagram of partial homomorphisms $\mathcal{G}_{G}$, we define its \emph{fundamental category} $\mathcal{C}(\mathcal{G}_{G})$ to be the category whose arrows correspond to equivalence classes of $\sim$. Composition of arrows is simply concatenation of composable paths multiplying group elments at each end. 

Let $\mathcal{G}_{G}$ be a diagram of partial homomorphisms, let $G$ be the groupoid which is the disjoint union of all the vertex groups of $\mathcal{G}_{G}$ and let $\mathcal{H}$ be the graph with $\mathcal{H}_{0} = G_{0}$ and $\mathcal{H}_{1} = G_{1}\cup \left\{t:t\in \mathcal{G}_{1}\right\}$. We can then write the fundamental category of $\mathcal{G}_{G}$ in terms of a category presentation as
$$\mathcal{C}(\mathcal{G}_{G}) \cong \langle \mathcal{H}|\mathcal{R}(G), ht = t\phi_{t}(h) \forall h\in G_{t}, t\in \mathcal{G}_{1}\rangle .$$
It then follows that $\mathcal{C}(\mathcal{G}_{G})$ is a left Rees category with totally disconnected groupoid of invertible elements. On the other hand, if we have a left Rees category with totally disconnected groupoid of invertible elements we can just reverse this process to get a fundamental category of a diagram of partial homomorphisms. Thus,

\begin{proposition}
\label{fundcatLRC}
Fundamental categories of diagrams of partial homomorphisms are precisely left Rees categories with totally disconnected groupoids of invertible elements.
\end{proposition}

Combining Proposition \ref{fundcatLRC} and Theorem \ref{LRCconjclass} we have

\begin{proposition}
Two diagrams of partial homomorphisms are equivalent if and only if their fundamental categories are isomorphic.
\end{proposition}

If $\mathcal{G}_{G}$ is a diagram of partial homomorphisms we will denote by $T_{t}$ a transversal of the left cosets of $G_{t}$ in $G_{\ran(t)}$ for each edge $t$. We then see that an arbitrary element $s$ of $\mathcal{C}(\mathcal{G}_{G})$ can be written uniquely in the form
$$s = g_{1}t_{1}\cdots g_{m}t_{m}u$$
where $\dom(t_{k}) = \ran(t_{k+1})$ for each $k=1,\ldots,m-1$, $g_{k}\in T_{t_{k}}$ for each $k=1,\ldots,m$ and $u\in G_{\dom(t_{m})}$ is arbitrary.

Given a diagram of partial homomorphisms $\mathcal{G}_{G}$ and a vertex $a\in \mathcal{G}_{0}$ we define the \emph{fundamental monoid of $\mathcal{G}_{G}$ at $a$} to be $M(\mathcal{G}_{G},a) = a\mathcal{C}(\mathcal{G}_{G})a$, the local monoid at $a$ of $\mathcal{C}(\mathcal{G}_{G})$. By Proposition \ref{locmonLRCLRM} $M(\mathcal{G}_{G},a)$ will be a left Rees monoid with group of units $G_{a}$.

Let $\mathcal{G}_{G}$ be a diagram of partial homomorphisms, let $a$ be a vertex of $\mathcal{G}_{G}$ and let $P_{a}$ denote the set of routes in $\mathcal{G}_{G}$ with range $a$. For $p,q\in P_{a}$ we will write $p \approx q$ if $\dom(p) = \dom(q)$ and $p \sim qg$ for some $g\in G_{\dom(p)}$. 
This defines an equivalence relation on $P_{a}$.
We will denote the $\approx$-equivalence class containing the route $p$ by $[p]$.
An arbirtrary element of $P_{a}/\approx$ can then be written uniquely in the form
$$[x] = [g_{1}t_{1}\cdots g_{m}t_{m}]$$
where $\ran(t_{1}) = a$, $\dom(t_{k}) = \ran(t_{k+1})$ for each $k=1,\ldots,m-1$ and $g_{k}\in T_{t_{k}}$ for each $k=1,\ldots,m$.
We now define the Bass-Serre tree $T$ with respect to the vertex $a$ as follows. 
The vertices of $T$ are $\approx$-equivalence classes of routes in $P_{a}$. 
Two vertices $[x],[y]\in T_{0}$ are connected by an edge $s\in T_{1}$ if there are $g\in G_{a}$ and $t\in \mathcal{G}_{1}$ such that 
$$y \approx xgt.$$
Here $\dom(s) = y$ and $\ran(s) = x$.
In other words there is an edge connecting $[g_{1}t_{1}\cdots g_{m}t_{m}]$ and $[g_{1}t_{1}\cdots g_{m}t_{m}g_{m+1}t_{m+1}]$ where $g_{k}\in T_{t_{k}}$ for each $k$, and every edge arises in this way.
It therefore follows that $T$ is a tree. 
We will now define an action of $M(\mathcal{G}_{G},a)$ on $T_{0}$ by
$$p\cdot [x] = [px].$$
This will then naturally extend to an action of $M(\mathcal{G}_{G},a)$ on $T$.

\begin{comment}
Let us now consider these ideas from the point of view of groupoid HNN-extensions. By definition two paths $p,q$ in $\mathcal{G}_{G}$ are $\sim$-related if they correspond to the same elements of $\Gamma(\mathcal{G}_{G})$. So let $\Gamma$ be an arbitrary groupoid HNN-extension of a totally disconnected groupoid $G$, let $a\in \Gamma_{0} = G_{0}$ and let 
$$P_{a} = \left\{g\in \Gamma| \ran(g) = a\right\}.$$
For $p,q\in P_{a}$, we define $p \approx q$ if $p = qg$ for some $g\in G$. This defines an equivalence relation on $P_{a}$ and we denote the $\approx$-equivalence class containing $p$ by $[p]$. We now define an undirected tree $T$ with respect to the identity $a$ as follows. The vertices of $T$ will correspond to $\approx$-equivalence classes of elements of $P_{a}$. Two vertices $[p],[q]\in T_{0}$ are connected by an edge if 
$$q = pgt_{i}^{\epsilon}$$
for some $g\in G$, $i\in I$, $\epsilon\in \left\{-1,1\right\}$. We then have an action of $a\Gamma a$ on $T_{0}$ given by 
$$g \cdot [p] = [gp]$$
which naturally extends to an action of $a\Gamma a$ on $T$. It should be clear that our requirement that $G$ was totally disconnected does not appear in most of the above ideas, so that all of the above works for an arbitrary groupoid $\Gamma$. The only non-trivial statement is that $T$ is a tree. NEED TO CHECK THIS.
\end{comment}

Let us rewrite this in our earlier notation for left Rees categories. Suppose $M = \mathcal{G}^{\ast}G$ is a left Rees category and $a\in M_{0}$ is an identity. We will define $T$ to be the tree with vertices
$$T_{0} = \left\{x\in \mathcal{G}^{\ast}|\ran(x) = a\right\}$$
and two vertices $x,y\in T_{0}$ will be connected by an edge $s\in T_{1}$ with $\dom(s) = y$ and $\ran(s) = x$ if $y = xz$ for some $z\in \mathcal{G}^{\ast}$. We then have an action of $aMa$ on $T$ given on vertices by
$$(xg)\cdot y = x(g\cdot y)$$
and then extended to the edge $s$ with $\ran(s) = y$ and $\dom(s) = yz$ by defining the edge $(xg)\cdot s$ to be the one connecting $x(g\cdot y)$ and $x(g\cdot y)(g|_{y}\cdot z)$.
If $\mathcal{G}_{G}$ has a single vertex, then then this action just described will essentially be the action of a left Rees monoid $M = X^{\ast}G$ on the tree $X^{\ast}$ given by
$$(xg)\cdot y = x(g\cdot y).$$

As a final remark to this section, we note that all of the results of Section 2.4 should transfer to the categorical setting without any problems, so that a Zappa-Sz\'{e}p product of a free category and a groupoid can only be extended to a Zappa-Sz\'{e}p product of a free groupoid and a groupoid if it is symmetric, the groupoid of fractions of a symmetric Rees category is isomorphic to the Zappa-Sz\'{e}p product of a free groupoid and a groupoid and every Rees category with finite groupoid of invertible elements is isomorphic to a symmetric Rees category.

\section{Path automorphism groupoids}

In this section we will define the path automorphism groupoid of a graph. This is a generalisation of the notion of the automorphism group of a regular rooted tree. Throughout $\mathcal{G}$ will denote an arbitrary directed graph. In addition, in both this section and the following section we will use the word \emph{path} to mean what we earlier called a \emph{route}, since all routes are paths in the previous sense.

For each $e\in \mathcal{G}_{0}$, let $l(e)$ be the length of the longest path $p$ with $\ran(p) = e$. If there are no paths $p$ (aside from the empty path) with $\ran(p) = e$ then we will say $l(e) = 0$ and if there are paths of arbitrary length then we say $l(e)= \infty$. 

For example, let $\mathcal{G}$ be the following graph:

\begin{center}
\setlength{\unitlength}{0.75mm}
\begin{picture}(60,60)

\put(20,40){$e$}

\put(51,41){\vector(-1,0){27}}
\put(37,43){$x$}

\put(54,40){$f$}

\put(10,10){\circle{17}}
\put(10,18.5){\vector(1,0){1}}
\put(9,21){$t$}

\put(20,10){$g$}

\put(51,11){\vector(-1,0){27}}
\put(37,13){$z$}

\put(54,10){$h$}

\put(55,15){\vector(0,1){22}}
\put(57,25){$y$}

\end{picture}
\end{center}

% vertices $e,f,g,h$ and 4 edges $x,y,z,t$ with $\ran(x) = e$, $\dom(x) = \ran(y) = f$, $\dom(y) = \dom (z) = h$ and $\ran(z) = \dom(t) = \ran(t) = g$. 
Here we have $l(e) = 2$, $l(f) = 1$, $l(g) =\infty$ and $l(h) = 0$.
If $M = \mathcal{G}^{*}G$ is a left Rees category and if $e,f\in \mathcal{G}_{0}$ are such that $g\cdot e = f$ for some $g\in G$, then $l(e) = l(f)$ (here we are again identifying $\mathcal{G}_{0}$ and $\mathcal{G}^{\ast}_{0}$). 

We say a graph $\mathcal{G}$ satisfies the \emph{infinite path condition} (IPC) if $l(e) =\infty$ for every $e\in \mathcal{G}_{0}$.
Let $\mathcal{G}$ be a directed graph satisfying (IPC). For each $e\in \mathcal{G}_{0}$, let $P_{e}$ be the set of infinite paths $p$ in $\mathcal{G}$ with $\ran(p) = e$, $P_{e}^{\ast}$ be the set of finite paths $p$ in $\mathcal{G}$ with $\ran(p) = e$ and let $F_{e,f}$ be the set of bijective maps $g:P_{e}\cup P_{e}^{\ast}\rightarrow P_{f} \cup P_{f}^{\ast}$ satisfying the following conditons:
\begin{itemize}
\item If $p\in P_{e}^{\ast}$ then $l(p) = l(g(p))$
\item If $r\in P_{e}^{\ast}$ is a subpath of $p\in P_{e}$, then $g(p) = g(r)q$ for some infinite path $q$ in $\mathcal{G}$. In other words, if $p,q\in P_{e}$ are of the form $p = r \tilde{p}$, $q = r \tilde{q}$ where $r\in P_{e}^{\ast}$ with $|r| = n$, then $g(p) = s p^{\prime}$, $g(q) = s q^{\prime}$ where $|s| = n$. 
\end{itemize}
We will call such maps $g:P_{e}\cup P_{e}^{\ast}\rightarrow P_{f} \cup P_{f}^{\ast}$ \emph{path automorphisms}. Note that in general in a graph satisfying (IPC) often $F_{e,f}$ will be empty. Let 
$$\mathscr{G} = \bigcup_{e,f\in\mathcal{G}_{0}}{F_{e,f}}.$$
Then we can give $\mathscr{G}$ the structure of a groupoid by composing path automorphisms whose domains and ranges match up and we call this the \emph{path automorphism groupoid} of $\mathcal{G}$. 
When $\mathcal{G}$ has a single vertex and edge set $X$, then $\mathscr{G}$ will be the automorphism group of $X^{*}$, where we view $X^{\ast}$ as a regular rooted tree.

\begin{proposition}
\label{PAGSS}
Let $\mathcal{G}$ be a graph satisfying (IPC) and $\mathcal{G}^{\ast}$ the free category on $\mathcal{G}$. Then the path automorphism groupoid $\mathscr{G}$ of $\mathcal{G}_{0}$ has a natural faithful self-similar action on $\mathcal{G}^{\ast}$. 
\end{proposition}

\begin{proof}
Firstly, identify the idenities of $\mathscr{G}$ and $\mathcal{G}^{\ast}$. Let $x\in \mathcal{G}^{\ast}$, let $e=\ran(x)$, let $f\in \mathcal{G}_{0}$ and let $g\in F_{e,f}$. Define $g\cdot x$ to be $g(x)$ and define $g|_{x}\in F_{\dom(x),\dom(g(x))}$ to be the map which satisfies the following: for every $q\in P_{d(x)}$,
$$g(xq) = g(x)g|_{x}(q).$$
We need to check this satisfies the axioms for a self-similar action. Firstly, $\dom(g) = \ran(x)$, so this is all well-defined.  
We thus need to show it satisfies (C1) - (C3) and (SS1) - (SS8).

\begin{description}

\item[{\rm (C1) - (C3)}] These follow from how we have defined $g\cdot x$ and $g|_{x}$.

\item[{\rm (SS8)}] This follows from the definition of the restriction.

\item[{\rm (SS1), (SS3), (SS4) and (SS5)}] These are all clear.

\item[{\rm (SS2)}] This follows from the definition of composition of functions.

\item[{\rm (SS6)}] We will prove this by computing $g\cdot (xyz)$ in 2 different ways:
$$g\cdot (xyz) = (g\cdot (xy))(g|_{xy}\cdot z)$$
and
$$g\cdot (xyz) = (g\cdot x)(g|_{x}\cdot (yz)) = (g\cdot x)(g|_{x}\cdot y)((g|_{x})|_{y}\cdot z).$$
By using $(g\cdot (xy)) = (g\cdot x)(g|_{x}\cdot y)$ and cancelling we get the desired result.

\item[{\rm (SS7)}] We will prove this by computing $(gh)\cdot (xy)$ in 2 different ways:
$$(gh)\cdot (xy) = ((gh)\cdot x)((gh)|_{x}\cdot y)$$
and
$$(gh)\cdot (xy) = g\cdot (h\cdot(xy)) = g\cdot ((h\cdot x)(h|_{x}\cdot y)) = ((gh)\cdot x)((g|_{h\cdot x}h|_{x})\cdot y).$$
Cancelling gives the desired result.
\end{description}
\end{proof}

It follows from Proposition \ref{PAGSS} that if $G$ is a groupoid acting faithfully and self-similarly on $\mathcal{G}^{\ast}$ for a graph $\mathcal{G}$ satisfying the infinite path condition then $G$ is a subgroupoid of $\mathscr{G}$.

Let $\mathcal{G}$ be a directed graph. An \emph{automorphism} $g$ of $\mathcal{G}$ consists of bijective maps $\mathcal{G}_{0}\rightarrow \mathcal{G}_{0}$ and $\mathcal{G}_{1}\rightarrow \mathcal{G}_{1}$ such that $\dom(g(x)) = g(\dom(x))$ and $\ran(g(x)) = g(\ran(x))$ for each edge $x\in \mathcal{G}_{1}$. The set of all automorphisms forms a group under composition, which we will denote by $\Aut(\mathcal{G})$. If $\mathcal{G}$ satisfies (IPC) then every element of $\Aut(\mathcal{G})$ can be extended to a path automorphism. Let us denote the set of such path automorphisms by $G$. We see that $G$ is a subgroupoid of $\mathscr{G}$ which is closed under restriction and thus acts self-similarly on $\mathcal{G}^{\ast}$. 

\section{Wreath products}

In this section we will define wreath products for groupoids. This definition is not equivalent to that of Houghton \cite{Houghton}. Essentially his definition generalises to groupoids that of functions from a set $X$ to a group $G$, whereas ours generalises to groupoids the notion of the $X$th direct power of the group $G$. Throughout this section all graphs will be finite and will be assumed to satisfy (IPC). We also suppose that $F_{e,f}$ is non-empty for all $e,f\in \mathcal{G}_{0}$. 

Let $\mathcal{G}$ be a graph, let $e\in \mathcal{G}_{0}$ and let $E_{e}$ be the set of edges $x\in\mathcal{G}_{1}$ with $\ran(x) = e$. Let $e,f\in \mathcal{G}_{0}$ be such that $|E_{e}|=|E_{f}|$. Then a bijection $E_{e}\rightarrow E_{f}$ will be called an \emph{edge bijection}. Let $B(\mathcal{G})$ be the groupoid of all edge bijections where the product is composition whenever it is defined.

Let $\mathcal{G}$ be a graph, let $H$ be a subgroupoid of $B(\mathcal{G})$ and let $G$ be a groupoid such that we can identify $G_{0} = H_{0} = \mathcal{G}_{0}$. For each $e$ fix an order on $E_{e}$. Then the \emph{permutational wreath product of} $G$ and $H$, denoted $H\wr G$, is defined to be the set of elements
$$\sigma(g_{x_{1}},\ldots,g_{x_{n}}),$$
where $\sigma\in H$, $x_{1},\ldots,x_{n}$ are all the edges in $\mathcal{G}$ such that $\ran(x_{i}) = \dom(\sigma)$, and for all $i$, $g_{x_{i}}\in G$, $\dom(g_{x_{i}}) = \dom(x_{i})$ and $\ran(g_{x_{i}}) = \dom(\sigma(x_{i}))$. 
We define a product between two elements $\sigma(g_{x_{1}},\ldots,g_{x_{n}})$ and $\tau(h_{y_{1}},\ldots,h_{y_{m}})$ iff $\sigma\tau$ is defined (in which case $n=m$). The product is defined as follows:
$$\sigma(g_{x_{1}},\ldots,g_{x_{n}})\tau(h_{y_{1}},\ldots,h_{y_{n}}) = \sigma\tau(g_{\tau(y_{1})}h_{y_{1}},\ldots,g_{\tau(y_{n})}h_{y_{n}}).$$

\begin{lemma}
With $G,H$ as in the previous definition, $H\wr G$ is a groupoid.
\end{lemma}

\begin{proof}
We have that
$$\dom(\sigma(g_{x_{1}},\ldots,g_{x_{n}})) = \sigma^{-1}\sigma(\dom(x_{1}),\ldots,\dom(x_{n}))$$
and
$$\ran(\sigma(g_{x_{1}},\ldots,g_{x_{n}})) = \sigma\sigma^{-1}(\dom(\sigma(x_{1})),\ldots,\dom(\sigma(x_{n}))),$$
noting that for each $i$ we have $\sigma^{-1}\sigma = \ran(x_{i})$ and $\sigma\sigma^{-1} = \ran(\sigma(x_{i}))$.
It is easy to see that 
$$(\sigma(g_{x_{1}},\ldots,g_{x_{n}}))^{-1} = \sigma^{-1}(h_{y_{1}},\ldots,h_{y_{n}}),$$
where $h_{\sigma(x_{i})} = g_{x_{i}}^{-1}$.
The difficult thing to see is that this multiplication is associative. So let $\sigma(g_{x_{1}},\ldots,g_{x_{n}})$, $\tau(h_{y_{1}},\ldots,h_{y_{n}})$ and $\pi(k_{z_{1}},\ldots,k_{z_{n}})$ be such that $\sigma\tau\pi$ exists. Then
$$\sigma(g_{x_{1}},\ldots,g_{x_{n}})\tau(h_{y_{1}},\ldots,h_{y_{n}}) = \sigma\tau(g_{\tau(y_{1})}h_{y_{1}},\ldots,g_{\tau(y_{n})}h_{y_{n}})
= \sigma\tau(u_{y_{1}},\ldots,u_{y_{n}}),$$
so 
$$(\sigma(g_{x_{1}},\ldots,g_{x_{n}})\tau(h_{y_{1}},\ldots,h_{y_{n}}))\pi(k_{z_{1}},\ldots,k_{z_{n}}) 
= \sigma\tau\pi(u_{\pi(z_{1})}k_{z_{1}},\ldots,u_{\pi(z_{n})}k_{z_{n}}).$$
On the other hand,
$$\tau(h_{y_{1}},\ldots,h_{y_{n}})\pi(k_{z_{1}},\ldots,k_{z_{n}}) = \tau\pi(h_{\pi(z_{1})}k_{z_{1}},\ldots,h_{\pi(z_{n})}k_{z_{n}})
= \tau\pi(v_{z_{1}},\ldots,v_{z_{n}}),$$
so
$$\sigma(g_{x_{1}},\ldots,g_{x_{n}})(\tau(h_{y_{1}},\ldots,h_{y_{n}})\pi(k_{z_{1}},\ldots,k_{z_{n}})) 
= \sigma\tau\pi(g_{\tau\pi(z_{1})}v_{z_{1}},\ldots,g_{\tau\pi(z_{n})}v_{z_{n}}).$$
Now
$$g_{\tau\pi(z_{i})}v_{z_{i}} = g_{\tau\pi(z_{i})}(h_{\pi(z_{i})}k_{z_{i}}) = (g_{\tau\pi(z_{i})}h_{\pi(z_{i})})k_{z_{i}} = u_{\pi(z_{i})}k_{z_{i}}$$
and so we are done.
\end{proof}

We can now prove a result analogous to Proposition 1.4.2 of \cite{NekrashevychBook}.

\begin{proposition}
Let $\mathcal{G}$ be a finite directed graph satisfying (IPC), $H = B(\mathcal{G})$ and $G = \mathscr{G}$. Then there is a bijective functor
$$\psi:G\rightarrow H\wr G.$$
\end{proposition}

\begin{proof}
Define $\psi:G\rightarrow H\wr G$ by 
$$\psi(g) = \sigma(g|_{x_{1}},\ldots,g|_{x_{n}}),$$
where $\left\{x_{i}\right\}$ are the edges with $\ran(x_{i}) = \dom(g)$, $\sigma$ describes the action of $g$ on the edges $x_{i}$ and $g|_{x_{i}}$ is just the restriction of $g$ by $x_{i}$. We have that $\sigma\in B(\mathcal{G})$, $\dom(g|_{x_{i}}) = \dom(x_{i})$ and $\ran(g|_{x_{i}}) = \dom(\sigma(x_{i}))$. Thus $\psi(x)\in H\wr G$.

Let us prove first that $\psi$ is a functor. Let $g,h\in G$ be such that $gh$ exists. Then
\begin{eqnarray*}
\psi(g)\psi(h) &=& \sigma(g|_{x_{1}},\ldots,g|_{x_{n}})\tau(h|_{y_{1}},\ldots,h|_{y_{n}}) 
= \sigma\tau(g|_{\tau(y_{1})}h|_{y_{1}},\ldots,g|_{\tau(y_{n})}h|_{y_{n}})\\
&=& \sigma\tau((gh)|_{y_{1}},\ldots,(gh)|_{y_{n}}) = \psi(gh).
\end{eqnarray*}

Now suppose $\psi(g)=\psi(h)$. Then $\sigma = \tau$ and for each $i$ we have $g|_{x_{i}} = h|_{y_{i}}$. But this means the actions of $g$ and $h$ are equivalent, and so $g = h$ in $G$, since the action of $G$ on $\mathcal{G}^{\ast}$ is faithful.

Finally, let $\sigma(g_{1},\ldots,g_{n})\in H\wr G$, $e = \dom(\sigma)$ and $f = \ran(\sigma)$. Since $\sigma(g_{1},\ldots,g_{n})\in H\wr G$, there are $n$ edges $x_{1},\ldots,x_{n}\in\mathcal{G}_{1}$ with $\ran(x_{i}) = e$ and such that $\dom(x_{i}) = \dom(g_{i})$ and $\dom(\sigma(x_{i})) = \ran(g_{i})$. Define $g$ to be the unique element of $F_{v,w}$ satisfying 
$$g(x_{i}p) = \sigma(x_{i})g_{i}(p).$$ 
for $p\in P_{\dom(x_{i})}$. Then 
$$\psi(g) = \sigma(g_{1},\ldots,g_{n}).$$ 
\end{proof}

Thus if $G$ is a groupoid acting self-similarly on $\mathcal{G}^{\ast}$ then there is a functor
$$\psi:G\rightarrow B(\mathcal{G})\wr \mathscr{G}.$$
On the other hand, given a groupoid $G$ with finitely many identities, any functor $\psi:G\rightarrow B(\mathcal{G})\wr \mathscr{G}$, where $\mathcal{G}$ is a finite graph satisfying (IPC) and with $F(e,f)$ non-zero for all $e,f\in \mathcal{G}_{0}$ such that $\psi$ is surjective on identities gives rise to a self-similar action of $G$ on $\mathcal{G}^{\ast}$.

\section{Automaton groupoids}

%Mark: This section will need a bit of work (I'll do this over Christmas).
We can generalise the notion of an automaton group as defined in Section 2.7. The following definition describes a typed-automaton in the sense of \cite{BarrWells}, but which also has an output function. 

A \emph{finite-state generalised invertible automaton} $\mathcal{A}=(A,X,f,\lambda,\pi)$ will consist of 
\begin{itemize}
\item a finite set $A$ whose elements are called \emph{states};
\item a finite set $X$ called the \emph{alphabet};
\item a subset $I_{a}\subseteq X$ for each $a\in A$ called the \emph{input alphabet of $a$};
\item a subset $P_{a}\subseteq X$ for each $a\in A$ called the \emph{output alphabet of $a$};
\item a bijection $\lambda_{a}:I_{a}\rightarrow P_{a}$ for each $a\in A$;
\item a map $\pi_{a}:I_{a}\rightarrow A$ for each $a\in A$
\end{itemize}
satisfying the following:
\begin{enumerate}
\item for each $a\in A$ there exists $b\in A$ such that $I_{a}=P_{b}$;
\item for each $a\in A$ there exists $b\in A$ such that $P_{a}=I_{b}$;
\item for every $a,b\in A$ either $I_{a}=I_{b}$ or $I_{a}\cap I_{b} = \emptyset$;
\item for every $a,b\in A$ either $P_{a}=P_{b}$ or $P_{a}\cap P_{b} = \emptyset$;
\item for each $x\in X$, if $a = \pi_{b}(x)$ and $c = \pi_{d}(x)$, then $I_{a} = I_{c}$ - and so we define $T_{x}:= I_{a}$;
\item if $x\in I_{a}$ then $P_{\pi_{a}(x)} = T_{\lambda_{a}(x)}$.
% \item if $x\in I_{a}$ then $T_{\lambda_{a}(x)} = I_{\pi_{a}(x)}$.
% $I_{a} = P_{b}$ then $T_{\pi_{b}(x)} = P_{\lambda_{b}(x)}$ for each $x\in I_{b}$.
\end{enumerate}

Axioms 1-4 say that the input and output sets partition $X$, and both do so in the same way. Axiom 5 will allows us to define a multiplication on the alphabet and axiom 6 will allow us to construct a self-similar action. We can describe these automata by Moore diagrams, in an analogous fashion to Section 2.7.

Suppose we have partitioned $X$ into $n$ subsets $X_{i}$, so that 
$$X = \bigcup_{i=1}^{n}{X_{i}},$$
where each $X_{i}=I_{a}$ for some $a$. 
Let us now create two graphs, $\mathcal{G}$ and $\mathcal{H}$. Both $\mathcal{G}$ and $\mathcal{H}$ will have as their vertex sets $\mathcal{G}_{0} = \mathcal{H}_{0} =\left\{e_{1},\ldots,e_{n}\right\}$. $\mathcal{G}$ will have as its edge set $\mathcal{G}_{1} = X$ and $\mathcal{H}$ will have as its edge set $\mathcal{H}_{1} = A$. Edges will connect vertices as follows in $\mathcal{G}$:
\begin{itemize}
\item $\dom(x) = e_{i}$ iff $X_{i} = T_{x}$  
\item $\ran(x) = e_{i}$ iff $x\in X_{i}$
\end{itemize}
Edges will connect vertices as follows in $\mathcal{H}$:
\begin{itemize}
\item $\dom(a) = e_{i}$ iff $X_{i} = I_{a}$  
\item $\ran(a) = e_{i}$ iff $X_{i} = P_{a}$
\end{itemize}

We can define a partial action of $\mathcal{G}_{1}$ on $\mathcal{H}_{1}$ by $a|_{x} = \pi_{a}(x)$ for $\ran(x) = \dom(a)$ and an action of $\mathcal{H}_{1}$ on $\mathcal{G}_{1}$ by $a\cdot x = \lambda_{a}(x)$ for $\ran(x) = \dom(a)$. Let $\mathcal{H}^{\dagger}$ be the free groupoid on $\mathcal{H}$ and $\mathcal{G}^{\ast}$ the free category on $\mathcal{G}$. In a similar manner to Section 2.7 we can extend the actions of $\mathcal{G}_{1}$ on $\mathcal{H}_{1}$ and $\mathcal{H}_{1}$ on $\mathcal{G}_{1}$ in a unique way to actions of $\mathcal{G}^{\ast}$ on $\mathcal{H}^{\dagger}$ and $\mathcal{H}^{\dagger}$ on $\mathcal{G}^{\ast}$ by using axioms (SS1)-(SS8) and requiring that $a\cdot(a^{-1}\cdot x) = x$ and $a^{-1}|_{x} = (a|_{a^{-1}\cdot x})^{-1}$. This then gives a self-similar groupoid action of $\mathcal{H}^{\dagger}$ on $\mathcal{G}^{\ast}$. 

If $g,h\in \mathcal{H}^{\dagger}$ are such that $g^{-1}g = h^{-1}h$ and $gg^{-1} = hh^{-1}$ then we will write $g\sim h$ if $g\cdot x = h\cdot x$ for all $x\in \mathcal{G}^{\ast}$ with $\ran(x) = g^{-1}g$. This defines a congruence on $\mathcal{H}^{\dagger}$. If $a\in A$ is such that $I_{a} = P_{a} = \emptyset$ then we say $a\sim a^{-1}a = aa^{-1}$. We then define $G$ by
$$G = \mathcal{H}^{\dagger}/\sim$$
and call $G$ the \emph{automaton groupoid} of $\mathcal{A}$. We see from its construction that $G$ will act faithfully on $\mathcal{G}^{\ast}$ in a self-similar manner.

\begin{comment}
Let 
$$K(\mathcal{A}) = \bigcap_{x\in \mathcal{G}_{1}^{\ast}}{G_{x}},$$
where $G_{x}$ is the stabiliser of $x\in \mathcal{G}_{1}^{\ast}$. $K = K(\mathcal{A})$ is a normal subgroupoid of $\mathcal{H}^{\dagger}$. Let $G = \mathcal{H}^{\dagger}/K$ be the factor groupoid. 

\begin{thm}
Let $G$ and $\mathcal{G}^{\ast}$ be as above. Then the Zappa-Sz\'ep product of $G$ and $\mathcal{G}^{\ast}$ is a fundamental left Rees category.
\end{thm}

\begin{proof}
We see that $G$ is a subgroupoid of the path automorphism groupoid of the underlying graph of $G_{1}^{\ast}$ by the fact that the input strings are the same length as output strings (showing it is length-preserving) and by defintion of the action of states on strings, it satisfies the second condition. The self-similarity of the action then follows from the fact that $G$ is closed under restriction.
\end{proof}

\end{comment}

Note that if $M$ is a fundamental left Rees category generated by a finite generalised invertible automaton then $M$ is finite if, and only if, $\mathcal{A}$ has no cycles. On the other hand, given a finite fundamental left Rees category with a single sink, we can construct such an automaton which generates it.

\begin{example}
The following example is analogous to the dyadic adding machine, except that in addition to adding two dyadic integers, it turns $0$'s and $1$'s into $x$'s and $y$'s. Here is the Moore diagram:

\begin{center}
\setlength{\unitlength}{0.75mm}
\begin{picture}(60,60)

\put(9,30){\circle{17}}
\put(9,38.5){\vector(1,0){1}}
\put(4,41){$(1,x)$}

\put(20,30){$a$}
\put(21,31){\circle{7}}

\put(24.5,31){\vector(1,0){27}}
\put(32,33){$(0,y)$}

\put(54,30){$b$}
\put(55,31){\circle{7}}

\put(55,43){\circle{17}}
\put(55,51.5){\vector(1,0){1}}
\put(50,53.5){$(1,y)$}

\put(55,19){\circle{17}}
\put(55,10.5){\vector(1,0){1}}
\put(50,6){$(0,x)$}

\end{picture}
\end{center}

\begin{center}
\setlength{\unitlength}{0.75mm}
\begin{picture}(60,60)

\put(9,30){\circle{17}}
\put(9,38.5){\vector(1,0){1}}
\put(4,41){$(x,1)$}

\put(20,30){$c$}
\put(21,31){\circle{7}}

\put(24.5,31){\vector(1,0){27}}
\put(32,33){$(y,0)$}

\put(54,30){$d$}
\put(55,31){\circle{7}}

\put(55,43){\circle{17}}
\put(55,51.5){\vector(1,0){1}}
\put(50,53.5){$(y,1)$}

\put(55,19){\circle{17}}
\put(55,10.5){\vector(1,0){1}}
\put(50,6){$(x,0)$}

\end{picture}
\end{center}
\begin{center}
Figure 9: Moore diagram of analogue of dyadic adding machine
\end{center}

Here $X=\left\{0,1,x,y\right\}$ and $A = \left\{a,b,c,d\right\}$. Let $X_{1} = \left\{0,1\right\}$ and $X_{2} = \left\{x,y\right\}$ (so that $X = X_{1}\cup X_{2})$. Then 
$$X_{1} = I_{a} = I_{b} = P_{c} = P_{d} = T_{x} = T_{y}$$
and 
$$X_{2} = P_{a} = P_{b} = I_{c}  = I_{d} = T_{0} = T_{1}.$$
Let $V = \left\{e_{1}, e_{2}\right\}$. The associated graphs $\mathcal{G}$ and $\mathcal{H}$ will be as follows:

\begin{center}
\setlength{\unitlength}{0.75mm}
\begin{picture}(30,30)

\put(-30,20){$\mathcal{G}$}

\put(-20,10){$e_{1}$}

\put(-15,13.5){\vector(1,0){27}}
\put(-15,16){\vector(1,0){27}}
\put(-6,19){$x,y$}

\put(13,8.5){\vector(-1,0){27}}
\put(13,6){\vector(-1,0){27}}
\put(-2,1){$0,1$}

\put(14,10){$e_{2}$}

\put(25,20){$\mathcal{H}$}

\put(35,10){$e_{1}$}

\put(40,13.5){\vector(1,0){27}}
\put(40,16){\vector(1,0){27}}
\put(49,19){$a,b$}

\put(68,8.5){\vector(-1,0){27}}
\put(68,6){\vector(-1,0){27}}
\put(53,1){$c,d$}

\put(69,10){$e_{2}$}

\end{picture}
\end{center}
% Both graph will have vertex set $V$. $\mathcal{G}$ will have edge set $X$ and $\mathcal{H}$ will have edge set $A$. On $\mathcal{G}$, $x$ and $y$ will go from $e_{1}$ to $e_{2}$ and $0$ and $1$ will go from $e_{2}$ to $e_{1}$. On $\mathcal{H}$, $a$ and $b$ will go from $e_{1}$ to $e_{2}$ and $c$ and $d$ will go from $e_{2}$ to $e_{1}$. 

Now 
$$(bd)\cdot (xw) = x (bd)|_{x}\cdot w = x (bd)\cdot w$$
and
$$(bd)\cdot (yw) = y (bd)|_{y}\cdot w = y (bd)\cdot w.$$
Thus $bd \sim e_{2}$. In a similar way we have $db \sim e_{1}$ and so $b^{-1} \sim d$. Now
$$(ac)\cdot (xw) = x (ac)|_{x}\cdot w = x (ac)\cdot w$$
and
$$(ac)\cdot (yw) = y (ac)|_{y}\cdot w = y (bd)\cdot w = yw.$$
Thus $ac \sim e_{2}$, by symmetry $ca \sim e_{1}$ and $a^{-1} \sim c$. So the groupoid $G = \mathcal{H}^{\dagger}/\sim$ will have 4 non-identity elements.

We can view the element $a$ as adding $1$ to a dyadic integer, where we have identified $0$'s and $x$'s, and $1$'s and $y$'s.
%We will have that $K(\mathcal{A})$ will be the subgroupoid generated by $ac$, $bd$, $ca$ and $db$.

\begin{comment}
\subsection{Different example}

A more interesting example. Here is the Moore diagram:

\vspace{10 mm}

\begin{center}
\setlength{\unitlength}{0.75mm}
\begin{picture}(130,50)

\put(9,30){\circle{17}}
\put(9,38.5){\vector(1,0){1}}
\put(4,41){(2,2)}

\put(11,13){\circle{17}}
\put(11,4.5){\vector(1,0){1}}
\put(6,0){(3,3)}

\put(17,21.5){$c$}
\put(18,22.5){\circle{7}}

\put(54,32){$b$}
\put(55,33){\circle{7}}

\put(55,45){\circle{17}}
\put(55,53.5){\vector(1,0){1}}
\put(50,55.5){$(0,1)$}

\put(54,10){$a$}
\put(55,11){\circle{7}}

\put(51.5,32){\vector(-4,-1){31}}
\put(31,32){$(1,0)$}

\put(51.5,12){\vector(-4,1){31}}
\put(32,10){$(1,3)$}

\put(55,14.5){\vector(0,0){15}}
\put(56,21){$(0,2)$}

\end{picture}
\end{center}

\vspace{3 mm}

Here $X=\left\{0,1,2,3\right\}$, $X_{1} = \left\{0,1\right\}$ and $X_{2} = \left\{2,3\right\}$. We have
$$T_{0} = I_{a} = I_{b} = P_{b} = X_{1}$$
and

$$T_{1} = T_{2} = T_{3} = I_{c} = P_{a} = P_{c} = X_{1}$$

\subsection{Some Finite Left Rees Categories}

\begin{ex}
\end{comment}
\end{example}
\begin{example}
Consider the automaton described by the Moore diagram in Figure 10.

\begin{center}
\setlength{\unitlength}{0.75mm}
\begin{picture}(60,60)

\put(-11,40){$g_{1}$}
\put(-9,41){\circle{8}}

\put(-5.5,41){\vector(1,0){27}}
\put(2,43){$(x_{1},x_{3})$}

\put(23,40){$g_{2}$}
\put(25,41){\circle{8}}

\put(-11,8){$g_{3}$}
\put(-9,9){\circle{8}}

\put(-5.5,9){\vector(1,0){27}}
\put(2,11){$(x_{4},x_{2})$}

\put(23,8){$g_{4}$}
\put(25,9){\circle{8}}

\put(56,12){\vector(-1,0){27}}
\put(35,14){$(x_{5},x_{6})$}

\put(56,6){\vector(-1,0){27}}
\put(35,2){$(x_{6},x_{5})$}

\put(57,8){$g_{5}$}
\put(59,9){\circle{8}}

\put(-9,37.5){\vector(0,-1){25}}
\put(-27,25){$(x_{3},x_{1})$}

\put(25,37.5){\vector(0,-1){25}}
\put(7,25){$(x_{2},x_{4})$}

\end{picture}
\end{center}
\begin{center}
Figure 10: Moore diagram of automaton
\end{center}

Let $X_{1} = \left\{x_{1},x_{3}\right\}$, $X_{2} = \left\{x_{2}\right\}$, $X_{3} = \left\{x_{4}\right\}$, $X_{4} = \left\{x_{5},x_{6}\right\}$ and $X_{5} = \emptyset$. Then
$$X_{1} = I_{g_{1}} = P_{g_{1}}, \quad X_{2} = I_{g_{2}} = P_{g_{3}} = T_{x_{1}}, \quad X_{3} = I_{g_{3}} = P_{g_{2}} = T_{x_{3}},$$
$$X_{4} = I_{g_{5}} = P_{g_{5}}, \quad X_{5} = T_{x_{2}} = T_{x_{4}} = T_{x_{5}} = T_{x_{6}}.$$

In this case we have the following graphs:

\begin{comment}
In this case $\mathcal{G}$ is the graph with 5 vertices $e_{1},\ldots,e_{5}$ and 6 edges $x_{1},\ldots,x_{6}$ such that 
$$e_{1} = \ran(x_{1}) = \ran(x_{3}), \quad e_{2} = \dom(x_{3}) = \ran(x_{4}), \quad e_{3} = \dom(x_{1}) = \ran(x_{2}),$$
$$e_{4} = \ran(x_{5}) = \ran(x_{6}), \quad e_{5} = \dom(x_{2}) = \dom(x_{4}) = \dom(x_{5}) = \dom(x_{6})$$
and $\mathcal{H}$ is the graph with 5 vertices $e_{1},\ldots,e_{5}$ and 5 edges $g_{1},\ldots,g_{5}$ such that 
$$e_{1} = g_{1}g_{1}^{-1} = g_{1}^{-1}g_{1}, \quad e_{2} = g_{2}^{-1}g_{2} = g_{3}g_{3}^{-1}, \quad e_{3} = g_{2}g_{2}^{-1} = g_{3}^{-1}g_{3},$$
$$e_{4} = g_{4}^{-1}g_{4} = g_{4}g_{4}^{-1}, \quad e_{5} = g_{5}^{-1}g_{5} = g_{5}g_{5}^{-1}.$$
\end{comment}

\begin{center}
\setlength{\unitlength}{0.75mm}
\begin{picture}(60,60)

\put(-5,50){$\mathcal{G}$}

\put(5,40){$e_{1}$}

\put(37,41){\vector(-1,0){27}}
\put(22,43){$x_{3}$}

\put(39,40){$e_{2}$}

\put(6,15){\vector(0,1){22}}
\put(8,25){$x_{1}$}

\put(5,10){$e_{3}$}

\put(37,11){\vector(-1,0){27}}
\put(22,13){$x_{2}$}

\put(39,10){$e_{5}$}

\put(40,15){\vector(0,1){22}}
\put(42,25){$x_{4}$}

\put(73,10){$e_{4}$}

\put(44,13){\vector(1,0){27}}
\put(56,15){$x_{5}$}

\put(44,8){\vector(1,0){27}}
\put(56,4){$x_{6}$}

\end{picture}
\end{center}

\begin{center}
\setlength{\unitlength}{0.75mm}
\begin{picture}(60,60)

\put(-10,50){$\mathcal{H}$}

\put(10,40){\circle{17}}
\put(10,48.5){\vector(1,0){1}}
\put(9,51){$g_{1}$}

\put(20,40){$e_{1}$}

\put(44,40){$e_{2}$}

\put(10,10){\circle{17}}
\put(10,18.5){\vector(1,0){1}}
\put(9,21){$g_{4}$}

\put(20,10){$e_{4}$}

\put(44,10){$e_{3}$}

\put(44,15){\vector(0,1){22}}
\put(47,37){\vector(0,-1){22}}
\put(48,25){$g_{2}$}
\put(39,25){$g_{3}$}

\put(70,40){\circle{17}}
\put(70,48.5){\vector(1,0){1}}
\put(69,51){$g_{5}$}

\put(80,40){$e_{5}$}

\end{picture}
\end{center}

Note that the free category $\mathcal{G}^{\ast}$ on the graph $\mathcal{G}$ is finite. Now we assume $g_{4} \sim e_{4}$ since $I_{g_{4}} = P_{g_{4}} = \emptyset$. Now
$$(g_{3}g_{2})\cdot (x_{2}w) = x_{2}(g_{3}g_{2})|_{x_{2}}\cdot w = x_{2} (g_{4}^{2})\cdot (w) = x_{2}w.$$
Thus $g_{3}g_{2}\sim e_{3}$. By symmetry $g_{2}g_{3}\sim e_{2}$ and so $g_{2}^{-1} = g_{3}$. We have
$$g_{1}^{2}\cdot (x_{1}w) = x_{1}(g_{1}^{2})|_{x_{1}}\cdot w = x_{1} (g_{3}g_{2})\cdot w = x_{1}w.$$
Similarly $g_{1}^{2}\cdot x_{3}w = x_{3}w$. Thus $g_{1}^{2} = e_{1}$. Finally,
$$g_{5}^{2}\cdot (x_{5}w) = x_{5}(g_{5}^{2})|_{x_{5}}\cdot w = x_{5} (g_{4}^{2})\cdot (w) = x_{5}w$$
and $g_{5}^{2}\cdot (x_{6}w) = x_{6}w$. Thus $g_{5}^{2} = e_{5}$. It therefore follows that $G = \mathcal{H}^{\dagger}/\sim$ will be finite and consequently $M = \mathcal{G}^{\ast}\bowtie G$ will be finite. This makes sense since $\mathcal{A}$ is finite and acyclic.

\end{example}

\section{Graph iterated function systems}

In Section 2.5 it was shown that many fractals defined by iterated function systems have a left Rees monoid as their monoid of similarity transformations. In this section we suggest how this might be generalised to graph iterated function systems by considering the example of the Von Koch snowflake. The von Koch snowflake can be regarded as 3 von Koch curves attached to each other in a triangle, giving the following fractal:
\begin{comment}

\subsection{Von Koch curve}

Consider the monoid $M$ of similarities of one of the sides $F$ of the von Koch snowflake (Figure 4).  Let $L_{1}$, $L_{2}$, $R_{1}$ and $R_{2}$ be the maps which map $F$, respectively, to the far left, the left diagonal, the right diagonal and the far right of itself and $\sigma$ be reflection in the verticle axis. We have the following relations
$$\sigma L_{1} = R_{2} \sigma \quad \quad \sigma L_{2} = R_{1} \sigma$$
$$\sigma R_{1} = L_{2} \sigma \quad \quad \sigma R_{2} = L_{1} \sigma.$$
Then the monoid generated by $L_{1}$, $L_{2}$, $R_{1}$ and $R_{2}$ is free, the group of units $G=\langle \sigma\rangle\cong C_{2}$ and $M=\langle L_{1}, L_{2}, R_{1},R_{2},\sigma\rangle$, and it is not hard to see that $M$ is a Rees monoid. $M$ is given by the following monoid presentation

$$M=\langle \sigma,t,r|\sigma^{2} =1\rangle .$$

\begin{center}
\includegraphics[angle=90,width=60mm]{VonKochCurve.pdf}
\end{center}
\begin{center}
Figure 4
\end{center}

\begin{center}
\includegraphics[scale=0.3]{fern.png}
\end{center}
\begin{center}
Figure n
\end{center}

\subsection{von Koch snowflake}
\end{comment}

\begin{center}
\includegraphics[scale=0.5]{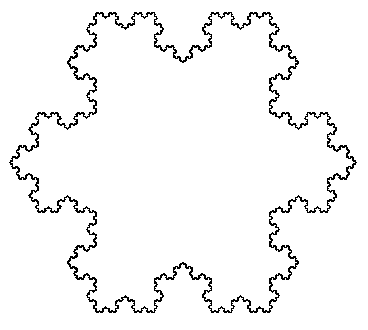}
\end{center}
\begin{center}
Figure 11: von Koch snowflake (source \cite{VKSS})
\end{center}

One possible way to construct the von Koch snowflake is as the attractor of a graph iterated function system. Let us describe each von Koch curve $C_1$, $C_2$ and $C_3$ by iterated function systems. $C_1$ is the attractor with maps $L_{1},R_{1}$, $C_2$ with maps $L_{2},R_{2}$ and $C_3$ with maps $L_{3},R_{3}$. Then consider a graph $Y$ with 3 vertices $e_{1},e_{2},e_{3}$, and maps $L_{i}$ and $R_{i}$ represented as edges from vertex $i$ to itself. Let $C$ be the free category of $Y$ and let $G$ be the groupoid with 3 objects, and 3 non-identity maps $\sigma_{1}$, $\sigma_{2}$ and $\sigma_{3}$ each from the $k$th object to iself, such that $\sigma_{k}^{2}=id_{k}$. Then we have the same self-similar action of $G$ on $C$ as with the von Koch curve above, giving rise to a Rees category $M$. 

\section{Algebras and representation theory} 

Here we generalise the ideas of Section 2.8. Let $K$ be a field and let $M$ be a category. Assume $M_{0}$ is finite. We can form the category algebra $KM$ as follows. An element $v$ of $KM$ is a finite sum 
$$v = \sum_{i=1}^{n}\alpha_{i}x_{i},$$
where $\alpha_{i}\in K$ and $x_{i}\in M$. We define addition $+$, convolution $\circ$ and scalar multiplication as follows:
$$\sum_{i=1}^{n}\alpha_{i}x_{i} + \sum_{i=1}^{m}\beta_{i}y_{i} = \sum_{i=1}^{n+m}\alpha_{i}x_{i},$$
where for $n+1\leq i\leq n+m$, $\alpha_{i} = \beta_{i-n}$ and $x_{i} = y_{i-n}$,
$$\sum_{i=1}^{n}\alpha_{i}x_{i} \circ \sum_{i=1}^{m}\beta_{i}y_{i} = \sum_{i=1}^{n}\sum_{j=1}^{m} \alpha_{i}\beta_{j}x_{i}y_{j},$$
where $x_{i}y_{j}$ is the product in $M$ (set it equal to $0$ if it does not exist) and
$$\lambda \sum_{i=1}^{n}\alpha_{i}x_{i} = \sum_{i=1}^{n}\lambda\alpha_{i}x_{i},$$
where for all of the above $\lambda, \alpha_{i}, \beta_{i}\in K$ and $x_{i},y_{i}\in M$. 

Let 
$$e = \sum_{e_{i}\in M_{0}}{e_{i}}.$$
The above gives $KM$ the structure of a unital $K$-algebra with unit $e$. 

\begin{comment}
Let $N$ be the subcategory of $M$ consisting of all elements $x\in M$ with $\dom(x) = \ran(x)$. Then in the same way as above $KN$ becomes a category algebra.

We can make $KN$ a bialgebra by specifying the comultiplication $\Delta$ on the elements $x\in N$ to be
$$\Delta(x) = x\otimes x$$
and counit $\epsilon$ to be $\epsilon(x) = \dom(x)$. If $N$ is a groupoid, we can make $KN$ into a Hopf algebroid by defining the antipode $S(g) = g^{-1}$. 

\end{comment}

\begin{comment}
Now suppose $M = \mathcal{G}^{*}\bowtie G$ is a left Rees category. Let $e_{1}$ be the unit of $K\mathcal{G}^{*}$ and $e_{2}$ be the unit of $KG$.
Then $KM$ is isomorphic to the \emph{bicrossed product algebra} $K\mathcal{G}^{*}\bowtie KG$ with unit $e_{1}\otimes e_{2}$, multiplication on generators given by
$$(x\otimes g)(y\otimes h) = x(g\cdot y)\otimes (g|_{y})h.$$
\end{comment}
Observe that if $M = \mathcal{G}^{*} \bowtie G$ is a left Rees category then $KM$ will be finite dimensional over $K$ if and only if $G$, $\mathcal{G}$ are finite and $\mathcal{G}$ is acyclic.
% then $\mathcal{K}M$ is finite dimensional over $K$. In fact it is clear that $KM$ will be finite dimensional iff $G$, $\mathcal{G}$ are finite and %$\mathcal{G}$ is acyclic.
Note that in representation theory the algebra $K\mathcal{G}^{\ast}$ is often called a \emph{quiver algebra}.

\begin{comment}

For $N$ as defined above, we have comultiplication (for $\exists xg$ and $\exists g\cdot x$) given by
$$\Delta(x\otimes g) = (x\otimes g)\otimes (x\otimes g)$$
and counit $\epsilon(x\otimes g) = \ran(x\otimes g)$.

If $M$ is a symmetric Rees category, then sometimes we can form $K\Gamma = K\mathcal{G}^{\dagger}\bowtie KG$, as above with antipode 
$$S(x\otimes g) = (g^{-1}\cdot x^{-1})\otimes(a^{-1}|_{x^{-1}}).$$

\subsection{Proving Things Work}

In the previous section, a number of assertions were made. Here we check all the necessary details. 

\end{comment}

Let $\mathcal{G}$ be a finite directed graph. A ($K$-linear) \emph{representation} $R$ of $\mathcal{G}$ is defined by the following data:
\begin{enumerate}
\item To each vertex $e\in \mathcal{G}_{0}$ is associated a $K$-vector space $R_{e}$.
\item To each arrow $\alpha:e\rightarrow f$ in $\mathcal{G}_{1}$ is associated a $K$-linear map $\phi_{\alpha}:R_{e} \rightarrow R_{f}$.
\end{enumerate} 

Now there is an abelian category well-studied in the representation theory of associative algebras whose objects are all $K$-linear representations of some specified directed graph (see Chapter 3 of \cite{Assem}). The morphisms in this category motivated the following theory.

We call a representation \emph{finite dimensional} if each of the vertex vector spaces are finite dimensional over $K$. We will assume from now on that all our representations are finite dimensional.

Observe that in the above definition it is possible that in a representation of a directed graph some of the vector spaces assigned to vertices are isomorphic.

Let $\mathcal{G}$ be a finite directed graph and let $R$ be a representation of $\mathcal{G}$. A \emph{collection of inner morphisms} $G$ for $R$ consists of a collection $G$ of invertible linear maps $g$ between vector spaces appearing in the representation satisfying the following properties:
\begin{enumerate}
\item For every $e\in \mathcal{G}_{0}$ there exists an identity map $g:R_{e}\rightarrow R_{e}$ in $G$.
\item For all $e_{1},e_{2},e_{3}\in \mathcal{G}_{0}$, for every $g:R_{e_{1}}\rightarrow R_{e_{2}}$ in $G$ and every arrow map $\phi_{\alpha}:R_{e_{2}}\rightarrow R_{e_{3}}$, there exists $e_{4}\in \mathcal{G}_{0}$, a unique arrow map $\phi_{\beta}:R_{e_{1}}\rightarrow R_{e_{4}}$ and a unique map $h:R_{e_{4}}\rightarrow R_{e_{3}}$ in $G$ such that $h\phi_{\alpha} = \phi_{\beta} g$.
\end{enumerate}

Note that for a particular representation there might be infinitely many such collections or there might be none.

Let $\mathcal{G}$ be a finite graph, let $R$ be a representation of $\mathcal{G}$ and let $G$ be a collection of inner morphisms for $R$. If for every vertex $e\in\mathcal{G}_{0}$ there is an edge $\alpha\in\mathcal{G}_{1}$ with $\dom(\alpha) = e$ then $G$ has the structure of a groupoid whose arrows are the elements of $G$ and composition of arrows is just composition of linear maps.

\begin{proposition}
Let $\mathcal{G}$, $R$ and $G$ be as in the preceding paragraph. Then there is a natural self-similar action of $G$ on $\mathcal{G}^{\ast}$.
\end{proposition}

\begin{proof}
For $x\in \mathcal{G}_{1}$, $g\in G$ with $\ran(x) = \dom(g)$ define $g\cdot x$ and $g|_{x}$ to be the unique elements of, respectively, $G$ and $\mathcal{G}_{1}$ satisfying the equation
$$gx = (g\cdot x)(g|_{x}).$$
One considers the identity elements of $\mathcal{G}^{\ast}$ to be the identity maps on the vertex vector spaces and paths in $\mathcal{G}^{\ast}$ to be the composition of linear maps. It then follows by a categorical version of Theorem \ref{sufconlrm} that this extends to a self-similar action of $G$ on $\mathcal{G}^{\ast}$ satisfying axioms (C1)-(C3) and (SS1)-(SS8).
\end{proof} 

Morally, the self-similarity in the last result follows from the associativity of matrix multiplication.

Let $C$ be a small category with $C_{0}$ finite. A (finite dimensional $K$-linear) \emph{representation} $R$ of $C$ consists of:
\begin{enumerate}
\item For each idenitity $e\in C_{0}$ there is associated a (finite dimensional) $K$-vector space $R_{e}$ and corresponding identity morphism $\phi_{e}:R_{e}\rightarrow R_{e}$
\item For each each arrow $x:e\rightarrow f$ in $\mathcal{C}$ is associated a $K$-linear map $\phi_{x}:R_{e} \rightarrow R_{f}$
\end{enumerate}  
such that $xy = z$ in $C$ implies $\phi_{x}\phi_{y} = \phi_{z}$.

Note that representations of free categories and directed graphs are effectively the same.

Now suppose that we have a self-similar action of a groupoid $G$ on a free category $\mathcal{G}^{\ast}$, and $M$ is the associated left Rees category. Suppose that $R$ is a representation of $M$. Then this gives rise to a representation $S$ of the quiver $\mathcal{G}$ with $G$ a collection of inner morphisms for $S$.

\begin{comment}

\section{Background semigroup theory}

\begin{defin}
An inverse semigroup $S$ is said to be \emph{$E$-unitary} if $e\leq s$ where $e\in E(S)$ then $s\in E(S)$. An inverse semigroup $S$ with zero is said to be \emph{$E^{\ast}$-unitary} if $0\neq e\leq s$ where $e\in E(S)$ then $s\in E(S)$.
\end{defin}

\begin{defin}
An inverse semigroup $S$ is said to be \emph{combinatorial} if $\mathcal{H}$ is trivial.
\end{defin}

\begin{defin}
An inverse semigroup $S$ with zero is said to be \emph{0-bisimple} if it contains exactly 2 $\mathcal{D}$-classes.
\end{defin}

\begin{defin}
An $X$-monoid is a $0$-bisimple inverse monoid in which the idempotents form a tree with base $X$.
\end{defin}

\begin{defin}
An \emph{$X^{\ast}$-monoid} is an $X$-monoid whose non-zero idempotents form a tree isomorphic to $X^{\ast}$ under the partial order.
\end{defin}

\begin{defin}
An inverse subsemigroup $T$ of a semigroup $S$ is said to be \emph{wide} if $E(S)\subseteq T$.
\end{defin}

\begin{defin}
A wide inverse subsemigroup $T$ of a semigroup $S$ is said to be \emph{normal} if $sts^{-1}\in T$ for all $s\in S$ and $t\in T$.
\end{defin}

\begin{defin}
Let $S$ be a semigroup. Then a subsemigroup $U$ is said to be a \emph{right denominator set} if $U$ satisfies the following two conditions:
\begin{enumerate}
\item $su = t u$ or $us = ut$ implies $s = t$ for $s,t\in S$, $u \in U$.
\item $sU\cap uS\neq \emptyset$ for all $s\in S$, $u\in U$.
\end{enumerate}
\end{defin}

\end{comment}

\newpage
\section{Associated inverse semigroup}

In this section we will see how the work of the past two chapters connects to the work of Nivat and Perrot, and how it relates to ideas in the following chapter. We will also use the results of this section for calculations in Section 4.9.

Given a Leech category $C$, of which left Rees categories are an example, there is a general way of forming an inverse semigroup $S(C)$, which we will call the \emph{associated inverse semigroup}. We will briefly describe this construction (see \cite{JonesLawsonGraph}, \cite{LawsonCatInv}, \cite{Law0EUnitary} and \cite{LawsonOrdered} for more details). Let $C$ be a Leech category, let $G(C)$ denote the groupoid of invertible elements of $C$ and let  
$$U = \left\{(x,y)\in C\times C | \dom(x) = \dom(y) \right\}.$$
Define $(x,y)\sim (z,w)$ in $U$ if there is an isomorphism $g\in G(C)$ with $(x,y) = (zg,wg)$. This is an equivalence relation and so we let 
$$S(C) = U/\sim \bigcup \left\{0\right\}.$$
We denote the equivalence class containing $(x,y)$ by $[x,y]$.
We define a multiplication for elements $[x,y],[z,w]\in S(C)$ as follows. If there are elements $u,v\in C$ with $yu = zv$ we define $[x,y][z,w] = [xu,wv]$. Otherwise the product is defined to be $0$. It turns out that $S(C)$ is an inverse semigroup with $0$. The inverse of an element $[x,y]$ is $[y,x]$ and idempotents are of the form $[x,x]$.

We have the following, proved in \cite{JonesLawsonGraph}:
\begin{lemma}
\label{asssemfacts}
\begin{enumerate}
\item $[x, y]\,\mathcal{L}\,[z, w]$ if and only if $y = wg$ for some isomorphism $g\in G(C)$.
\item $[x, y]\,\mathcal{R}\,[z, w]$ if and only if $x = zg$ for some isomorphism $g\in G(C)$
\item $[x, y]\,\mathcal{D}\,[z, w]$ if and only if $\dom(x)$ and $\dom(z)$ are isomorphic.
\item $[x, y]\,\mathcal{J}\,[z, w]$ if and only if the identities $\dom(x)$ and $\dom(z)$ are strongly connected.
\item $S(C)$ is $E^{\ast}$-unitary if and only if the Leech category $C$ is right cancellative.
\end{enumerate}
\end{lemma}

If $M$ is a left Rees monoid then it follows by Lemma \ref{asssemfacts} (1) that the $\mathcal{L}$-class of $[1,1]$ in $S(M)$ is isomorphic to $M$. This is how Nivat and Perrot came across self-similar group actions - they were studying a particular class of inverse semigroups for which this turns out to be the case. 

It follows from Lemma \ref{asssemfacts} (3) that $S(C)$ is $0$-bisimple if and only if $C$ is equivalent to a monoid. If $C$ is a free monoid $X^{\ast}$ with $|X| = 1$ then $S(C)$ is the \emph{bicyclic monoid} plus a $0$ adjoined and if $C$ is a free monoid $X^{\ast}$ with $|X| > 1$ then $S(C)$ is the \emph{polycyclic monoid} $P_{X}$. If $C$ is a free category $\mathcal{G}^{\ast}$ then $S(C) = P_{\mathcal{G}}$ is a \emph{graph inverse semigroup}. We see that both polycyclic monoids and graph inverse semigroups are $E^{\ast}$-unitary since free categories are right cancellative. This means that they are inverse $\wedge$-semigroups; that is, $s\wedge t$ exists for all $s,t\in S(C)$. Inverse $\wedge$-semigroups $S$ have a distributive completion, which we denote by $D(S)$, which means they are in particular orthogonally complete and so Rees categories give natural examples to which we can apply the theory of the following chapter. If $C$ is a left Rees monoid, then the semigroups $S(C)$ give rise to the Cuntz-Pimsner algebras of \cite{NekrashevychC} in much the same way as polycyclic monoids give rise to Cuntz algebras. 

If $M = \mathcal{G}^{\ast}G$ is a left Rees category then because of the unique decomposition of elements of left Rees categories we can write an arbitrary element of $S(M)$ in the form $[xg,y]$, where $x,y\in \mathcal{G}^{\ast}$ and $g\in G$. We now consider the natural partial order for $S(M)$.

\begin{lemma}
\label{orderasssem}
Let $M$ be a left Rees category and $S(M)$ be its associated inverse semigroup. Then $[xg,y]\leq [zh,w]$ in $S(M)$ if and only if there is a $v\in \mathcal{G}^{\ast}$ with $y = wv$, $x = z(h\cdot v)$ and $g = h|_{v}$.
\end{lemma}

\begin{proof} 
Let $[xg,y]\leq [zh,w]$ in $S(M)$. Then 
$$[xg,y] = [zh,w][y,y].$$
First suppose that $y$ is a prefix of $w$. Then $w = yv$ for some $v\in \mathcal{G}^{\ast}$ and so
$$[xg,y] = [zh,yv][y,y] = [zh,yv] = [zh,w].$$
Thus $w$ must be a prefix of $y$, so $y = wv$ for some $v\in \mathcal{G}^{\ast}$. Now 
$$[xg,y] = [zh,w][wv,wv] = [zhw,wv] = [z(h\cdot w)h|_{w},y]$$
and so $x = z(h\cdot w)$ and $g = h|_{w}$. On the other hand,
$$[zh,w][wv,wv] = [zhv,wv] = [z(h\cdot v)h|_{v},wv]$$
and so $[z(h\cdot v)h|_{v},wv] \leq [zh,w]$.
\end{proof}

The following curious result may be deduced from Lemma 1.7 of \cite{JonesLawsonGraph}.

\begin{lemma}
\label{curlemasssem}
Let $M$ be a Rees category and let $S(M)$ be its associated inverse semigroup. If $s,t\in S(M)$ are such that $s\wedge t \neq 0$ then $s\leq t$ or $t\leq s$.
\end{lemma}

\begin{proof}
Now suppose $[z_{1}h_{1},w_{1}],[z_{2}h_{2},w_{2}]\in S$ are such that $[z_{1}h_{1},w_{1}]\wedge [z_{2}h_{2},w_{2}]\neq 0$. Then there exists $[xg,y]\in S$ with 
$$[xg,y]\leq [z_{1}h_{1},w_{1}],[z_{2}h_{2},w_{2}].$$
Thus Lemma \ref{orderasssem} tells us there exist $u,v\in \mathcal{G}^{\ast}$ with 
$$y = w_{1}u = w_{2}v, \quad x = z_{1}(h_{1}\cdot u) = z_{2}(h_{2}\cdot v)$$
and 
$$g = h_{1}|_{u} = h_{2}|_{v}.$$
We must have either $w_{1}$ is a prefix of $w_{2}$ or $w_{2}$ is a prefix of $w_{1}$. Suppose without loss of generality that $w_{1}$ is a prefix of $w_{2}$. Then there is $r\in \mathcal{G}^{\ast}$ with $w_{2} = w_{1}r$. It then follows that $u = rv$. Now 
$$h_{1}\cdot u = (h_{1}\cdot r)(h_{1}|_{r}\cdot v)$$
and
$$h_{1}|_{u} = (h_{1}|_{r})|_{v}.$$
By the uniqueness of the decomposition of elements of $M$ and length considerations we must have $z_{2} = z_{1}(h_{1}\cdot r)$ and $h_{2}\cdot v = h_{1}|_{r}\cdot v$. Thus
$$h_{1}|_{r}v = (h_{1}|_{r}\cdot v)h_{1}|_{u} = (h_{2}\cdot v)h_{2}|_{v} = h_{2}v$$
and so by right cancellativity $h_{1}|_{r} = h_{2}$. Now
$$[z_{1}h_{1},w_{1}][w_{2},w_{2}] = [z_{1}h_{1}r,w_{2}] = [z_{1}(h_{1}\cdot r)h_{1}|_{r}, w_{2}] = [z_{2}h_{2},w_{2}]$$
and so $[z_{2}h_{2},w_{2}] \leq [z_{1}h_{1},w_{1}]$. If $w_{2}$ had been a prefix of $w_{1}$ then an identical argument would have shown that $[z_{1}h_{1},w_{1}] \leq [z_{2}h_{2},w_{2}]$. Thus the claim is proved.
\end{proof}

We know from the above that the associated inverse semigroups of Rees monoids are $E^{\ast}$-unitary. In fact, they are strongly $E^{\ast}$-unitary.

\begin{lemma}
\label{strongrees}
Let $M = X^{\ast}G$ be a Rees monoid, let $S(M)$ be its associated inverse semigroup and let $U(M)$ be the universal group of $M$. Then there is an idempotent pure partial homomorphism 
$$\theta: S(M)\rightarrow U(M)$$
given by 
$$\theta([xg,y]) = xgy^{-1}.$$
\end{lemma}

\begin{proof}
We can describe elements of $U(M)$ as products of elements of $X$, $G$ and their inverses. We know from the above theory that $M$ actually embeds in $U(M)$ if $M$ is a Rees monoid.
Firstly, 
$$\theta([xgh,yh]) = xghh^{-1}y^{-1} = xgy^{-1} = \theta([xg,y])$$
and so this map is well-defined. 
Let $[xg,y],[zh,w]\in S(M)$ be such that $[xg,y][zh,w]\neq 0$. First suppose $z = yu$ for some $u\in X^{\ast}$. Then
$$\theta([xg,y][zh,w]) = \theta([xguh,w]) = xguhw^{-1} = xgy^{-1}yuhw^{-1} = \theta([xg,y])\theta([zh,w]).$$
Now suppose $y = zu$. Then
\begin{eqnarray*}
\theta([xg,y][zh,w]) & = & \theta([xg,zu][z,wh^{-1}]) = \theta([xg,wh^{-1}u]) = xgu^{-1}hw^{-1}\\
& = & xgu^{-1}z^{-1}zhw^{-1} = \theta([xg,y])\theta([zh,w]).
\end{eqnarray*}
To see it is idempotent pure, note that $\theta([xg,y]) = 1$ implies $xgy^{-1} = 1$ and so $xg = y$. Since the decomposition of elements of $M$ is unique and the homomorphism from $M$ to $U(M)$ is injective, we must have $x = y$ and $g = 1$. Thus $[xg,y]$ is an idempotent.
\end{proof} 

Let $F\subseteq \mathbb{R}^{n}$ be a fractal-like structure satisfying the conditions of Theorem \ref{mainthm2} and let $M$ be the monoid of similarity transformations of $F$ which we know from earlier is a Rees monoid. Then $U(M)$ is a subgroup of the affine group of $\mathbb{R}^{n}$. Lemma \ref{strongrees} tells us we can view elements of $S(M)$ as restrictions of affine transformations to certain subsets of $F$.

Now suppose $M$ is an arbitrary left Rees category. Consider the subset $T(M)$ of $S(M)$ given by
$$T(M) = \left\{0\neq [xg,y]\in S(M) | |x| = |y|\right\} \bigcup \left\{0\right\}.$$
It is easy to check that $T(M)$ is in fact a normal inverse subsemigroup of $S(M)$, which we call the \emph{gauge inverse subsemigroup}. When $S(M)$ is the polycylic monoid this subsemigroup plays an important r\^{o}le in its representation theory \cite{JonesLawsonReps}. When $M$ is the monoid of similarity transformations of a fractal $F$ then $T(M)$ corresponds to the elements of $S(M)$ which are restrictions of Euclidean transformations.

\chapter{$K$-Theory of Inverse Semigroups}

\section{Outline of chapter} 

The aim of this chapter is to define a functor $K$ from the category of orthogonally complete inverse semigroups and orthogonal join preserving maps to the category of abelian groups in analogy with algebraic $K$-theory. In Section 4.2, we give an abstract definition of $K(S)$ for a particular class of inverse semigroups, which we call \emph{K-inverse semigroups}. This is motivated by the definition in terms of idempotents for regular rings and $C^{\ast}$-algebras. We will see in Section 4.3 that this definition does not depend on the inverse semigroup structure and so can in fact be defined for the underlying groupoid. 
Motivated by the module approach to $K_{0}$-groups in algebraic $K$-theory in Section 4.4 we give a definition of a module for an orthogonally complete inverse semigroup, and use this to associate a group $K(S)$ to orthogonally complete inverse semigroups such that for $K$-inverse semigroups this definition agrees with the one of Section 4.2.  
In Section 4.5, we define the $K$-group in terms of idempotent matrices in analogy with algebraic $K$-theory. In particular, we show that the definitions of Section 4.4 and 4.5 are equivalent. 
We will see in Section 4.6 that $K$ is actually a functor from the category of orthogonally complete inverse semigroups and orthogonal join preserving maps to the category of abelian groups.
In Section 4.7 it will be shown that more can be said about $K(S)$ for commutative inverse semigroups.
We will extend the ideas of states and traces of $C^{\ast}$-algebras to the situation of inverse semigroups in Section 4.8 and we will see that, analogously to the case of $C^{\ast}$-algebras, traces extend to homomorphisms on the $K$-groups.
In Section 4.9, we compute the K-group for a number of examples. 

\section{$K$-Inverse semigroups}

Throughout this section let $S$ be an orthogonally complete inverse semigroup. It is well-known (c.f. \cite{LawsonBook}) that two idempotents $e,f\in E(S)$ are $\mathcal{D}$-related if and only if there exists an $s\in S$ with $e \stackrel{s}{\rightarrow} f$. An equivalent statement is that idempotents $e,f\in E(S)$ are $\mathcal{D}$-related if and only if there exist $s,t\in S$ with $st = e$ and $ts = f$. Thus we will replace the concept of similarity from algebraic $K$-theory with the $\mathcal{D}$-relation for inverse semigroups.

\begin{comment}

\begin{defin}
Let $e,f\in E(S)$. Then $e\sim f$ if there exist $x,y\in S$ with $xy = e$ and $yx = f$.
\end{defin}

\begin{lemma}
Let $e,f\in E(S)$ with $e\sim f$. Then there exist $x,y\in S$ with $xy = e$, $yx = f$, $x = exf$, $y = fye$.
\end{lemma}

\begin{proof}
Since $e\sim f$, there exist $s,t\in S$ with $st = e$ and $ts = f$. Let $x = esf$ and $y = fte$. Then
$$xy = (esf)(fte) = esfte = stststst = e$$
and
$$yx = (fte)(esf) = ftesf = tstststs = f.$$
\end{proof}

\begin{lemma}
$\sim$ is an equivalence relation.
\end{lemma}

\begin{proof}
Let $e = xy$, $f = yx = zw$, $g = wz$, with $x = exf$, $y = fye$, $z = fzg$, $w = gwf$. Then
$$(xz)(wy) = xfy = xy = e$$ 
and
$$(wy)(xz) = wfz = wz = g.$$
\end{proof}

\begin{lemma}
Two elements $s,t\in S$ are orthogonal if and only if $\dm(s)\wedge \dm(t) = 0$ and $\rn(s)\wedge \rn(t) = 0$.
\end{lemma}

\begin{proof}
Observe that $st^{-1} = 0$ iff $s^{-1}st^{-1}t = 0$ and $s^{-1}t = 0$ iff $ss^{-1}tt^{-1} = 0$.
\end{proof}

\end{comment}

\begin{lemma}
\label{orthjoinksem}
Let $e_{1},e_{2},f_{1},f_{2}\in E(S)$ be idempotents such that $e_{1}\perp e_{2}$, $f_{1}\perp f_{2}$, $e_{1}\, \mathcal{D}\, f_{1}$ and $e_{2}\, \mathcal{D}\, f_{2}$. Then 
$$e_{1}\vee e_{2}\, \mathcal{D}\, f_{1}\vee f_{2}.$$
\end{lemma}

\begin{proof}
Let $s,t\in S$ be such that $e_{1} \stackrel{s}{\rightarrow} f_{1}$ and $e_{2} \stackrel{t}{\rightarrow} f_{2}$. Since $\dom(s)\wedge\dom(t) = 0$ and $\ran(s)\wedge \ran(t) = 0$ it follows that $s$ and $t$ are orthogonal, and so there exists $s\vee t$. Now 
$$\dm(s\vee t) = \dm(s)\vee \dm(t) = e_{1}\vee e_{2}$$
and 
$$\rn(s\vee t) = \rn(s)\vee \rn(t) = f_{1}\vee f_{2}.$$
Thus $e_{1}\vee e_{2}\, \mathcal{D}\, f_{1}\vee f_{2}$.
\begin{comment}
$y_{i} = e_{i}$, $y_{i}x_{i} = f_{i}$, $x_{i} = e_{i}x_{i}f_{i}$, $y_{i} = f_{i}y_{i}e_{i}$. B
We have $x_{i}x_{j}^{-1} = y_{i}y_{j}^{-1} = 0$ for $i\neq j$. Thus $x_{1},x_{2}$ are compatible and $y_{1},y_{2}$ are compatible, and so there exist $x_{1}\vee x_{2}$, $y_{1}\vee y_{2}$. Then
$$(x_{1}\vee x_{2})(y_{1}\vee y_{2}) = (x_{1}y_{1}\vee x_{1}y_{2})\vee (x_{2}y_{1}\vee x_{2}y_{2}) = (e_{1} \vee 0)\vee(0\vee e_{2}) = e_{1}\vee e_{2}$$
and
$$(y_{1}\vee y_{2})(x_{1}\vee x_{2}) = (y_{1}x_{1}\vee y_{1}x_{2})\vee (y_{2}x_{1}\vee y_{2}x_{2}) = (f_{1} \vee 0)\vee(0\vee f_{2}) = f_{1}\vee f_{2}.$$
\end{comment}
\end{proof}

We will say an inverse semigroup with zero is \emph{orthogonally separating} if for any pair of idempotents $e$ and $f$ there are idempotents $e^{\prime}$ and $f^{\prime}$ such that $e^{\prime}\, \mathcal{D}\, e$, $f^{\prime}\, \mathcal{D}\, f$ and $e^{\prime}\perp f^{\prime}$. A \emph{$K$-inverse semigroup} will be an orthogonally complete orthogonally separating inverse semigroup. The previous lemma tells us that we can define a binary operation on the $\mathcal{D}$ classes of such a semigroup.

So let $S$ be a $K$-inverse semigroup and denote by $[e]$ the $\mathcal{D}$-class of the idempotent $e$ in $E(S)$. Let $A(S) = E(S)/\mathcal{D}$ and define an operation $+$ on $A(S)$ by
$$[e] + [f] = [e^{\prime}\vee f^{\prime}]$$
where $e^{\prime},f^{\prime}\in E(S)$ are such that $e^{\prime}\, \mathcal{D}\, e$, $f^{\prime}\, \mathcal{D}\, f$ and $e^{\prime}\perp f^{\prime}$.
We see from Lemma \ref{orthjoinksem} this operation is well-defined. We in fact have the following:

\begin{lemma}
\label{AScommmon}
$(A(S),+)$ is a commutative monoid.
\end{lemma}

\begin{proof}
Commutativity follows from the commutativity of the join operation on $S$ and the identity element is easily seen to be $[0]$. Thus we just need to check that $+$ is associative.
Let $e,f,g\in E(S)$ be arbitrary. We want to show that
$$([e]+[f])+[g] = [e] + ([f]+[g])$$
in $A(S)$. 

Suppose that $e^{\prime}, e^{\prime\prime}, f^{\prime}, f^{\prime\prime}, g^{\prime}, g^{\prime\prime}, h,h^{\prime}\in E(S)$ are idempotents such that
$$e^{\prime} \stackrel{s_{1}}{\rightarrow} e \stackrel{s_{2}}{\rightarrow} e^{\prime\prime}, \quad f^{\prime} \stackrel{t_{1}}{\rightarrow} f \stackrel{t_{2}}{\rightarrow} f^{\prime\prime},$$
$$g^{\prime} \stackrel{u_{1}}{\rightarrow} g \stackrel{u_{2}}{\rightarrow} g^{\prime\prime}, \quad h \stackrel{w_{1}}{\rightarrow} (e^{\prime}\vee f^{\prime}), \quad (f^{\prime\prime}\vee g^{\prime\prime}) \stackrel{w_{2}}{\rightarrow} h^{\prime}$$
and
\begin{comment}
$$\xymatrix@1{e^{\prime} \ar[r]^{s_{1}} & e \ar[r]^{s_{2}} & e^{\prime\prime}},$$
$$\xymatrix@1{f^{\prime} \ar[r]^{t_{1}} & f \ar[r]^{t_{2}} & f^{\prime\prime}},$$
$$\xymatrix@1{g^{\prime} \ar[r]^{u_{1}} & g \ar[r]^{u_{2}} & g^{\prime\prime}},$$
$$\xymatrix@1{h \ar[r]^{w_{1}} & (e^{\prime}\vee f^{\prime})},$$
$$\xymatrix@1{(f^{\prime\prime}\vee g^{\prime\prime}) \ar[r]^{w_{2}} & h^{\prime}}$$
and
\end{comment}
$$e^{\prime}f^{\prime} = hg^{\prime} = f^{\prime\prime}g^{\prime\prime} = h^{\prime}e^{\prime\prime} = 0$$
for some $s_{1},s_{2},t_{1},t_{2},u_{1}u_{2},w_{1},w_{2}\in S$.

We then have
$$([e]+[f])+[g] = [h\vee g^{\prime}]$$
and
$$[e] + ([f]+[g]) = [e^{\prime\prime}\vee h^{\prime}].$$

So our task is to show that
$$(h\vee g^{\prime})\, \mathcal{D}\, (e^{\prime\prime}\vee h^{\prime}).$$

Let $x_{1} = s_{2}s_{1}w_{1}$, $x_{2} = w_{2}t_{2}t_{1}w_{1}$ and $x_{3} = w_{2}u_{2}u_{1}$. Then
$$x_{1}x_{2}^{-1} = s_{2}s_{1}e^{\prime}(e^{\prime}\vee f^{\prime})f^{\prime}t_{1}^{-1}t_{2}^{-1}w_{2}^{-1} = 0,$$
$$x_{1}^{-1}x_{2} = w_{1}^{-1}s_{1}^{-1}s_{2}^{-1}e^{\prime\prime}h^{\prime}w_{2}t_{2}t_{1}w_{1} = 0,$$
$$x_{1}x_{3}^{-1} = s_{2}s_{1}w_{1}hg^{\prime}u_{1}^{-1}u_{2}^{-1}w_{2}^{-1} = 0,$$
$$x_{1}^{-1}x_{3} = w_{1}^{-1}s_{1}^{-1}s_{2}^{-1}e^{\prime\prime}h^{\prime}w_{2}u_{2}u_{1} = 0,$$
$$x_{2}x_{3}^{-1} = w_{2}t_{2}t_{1}w_{1}hg^{\prime}u_{1}^{-1}u_{2}^{-1}w_{2}^{-1} = 0$$
and
$$x_{2}^{-1}x_{3} = w_{1}^{-1}t_{1}^{-1}t_{2}^{-1}f^{\prime\prime}(f^{\prime\prime}\vee g^{\prime\prime})g^{\prime\prime}u_{2}u_{1} = 0.$$
Thus we may form the orthogonal join $x = x_{1}\vee x_{2} \vee x_{3}$. We see that
\begin{eqnarray*}
xx^{-1} & = & (x_{1}x_{1}^{-1}) \vee (x_{2}x_{2}^{-1}) \vee (x_{3}x_{3}^{-1}) \\
& = & (s_{2}s_{1}w_{1}w_{1}^{-1}s_{1}^{-1}s_{2}^{-1}) \vee (w_{2}t_{2}t_{1}w_{1}w_{1}^{-1}t_{1}^{-1}t_{2}^{-1}w_{2}^{-1}) \vee 
(w_{2}u_{2}u_{1}u_{1}^{-1}u_{2}^{-1}w_{2}^{-1})\\
& = & (s_{2}s_{1}e^{\prime}(e^{\prime}\vee f^{\prime})e^{\prime}s_{1}^{-1}s_{2}^{-1}) \vee (w_{2}t_{2}t_{1}f^{\prime}(e^{\prime}\vee f^{\prime})f^{\prime}t_{1}^{-1}t_{2}^{-1}w_{2}^{-1}) \vee 
(w_{2}u_{2}g u_{2}^{-1}w_{2}^{-1})\\
& = & e^{\prime\prime} \vee (w_{2}f^{\prime\prime}w_{2}^{-1}) \vee (w_{2}g^{\prime\prime}w_{2}^{-1})\\
& = & e^{\prime\prime} \vee (w_{2}(f^{\prime\prime}\vee g^{\prime\prime})w_{2}^{-1}) = e^{\prime\prime}\vee h^{\prime}
\end{eqnarray*}
Similarly, $x^{-1}x = h\vee g^{\prime}$. Thus $(h\vee g^{\prime})\, \mathcal{D}\, (e^{\prime\prime}\vee h^{\prime})$.
\end{proof}

\begin{comment}

Let $S$ be an orthogonally complete inverse semigroup with $0$ such that for all $e,f\in E(S)$ there exist $e^{\prime}, f^{\prime}\in E(S)$ with $e\mathcal{D} e^{\prime}$, $f\mathcal{D} f^{\prime}$ and $e^{\prime}f^{\prime} = 0$. Then define
$$e + f = e^{\prime}\vee f^{\prime}.$$
The above shows this operation is well-defined. 

\begin{lemma}
For $S$ appropriate, $e,f,g\in E(S)$, we have
$$([e]+[f])+[g] = [e]+([f]+[g]).$$
\end{lemma}

This operation is clearly commutative and so define $K_{0}(S)$ in the obvious manner for this class of semigroups.

\end{comment}

For $S$ a K-inverse semigroup we define
$$K(S) = \mathcal{G}(A(S)),$$
where $\mathcal{G}(M)$ is the Grothendieck group of $M$, as defined in Section 1.4.

As an example let $S = I_{f}(\mathbb{N})$ be the symmetric inverse monoid on $\mathbb{N}$ with finite support. Then $S$ is a K-inverse semigroup. Further for $e,f\in E(S)$ we have $e\,\mathcal{D}\, f$ if and only if $|\supp(e)| = |\supp(f)|$. In addition if $e,f\in E(S)$ are such that $ef = 0$ then $|\supp(e\vee f)| = |\supp(e)| + |\supp(f)|$. We therefore have:

$$K(S) \cong \mathbb{Z}.$$

\section{$K$-Groupoids}

In this section it will be demonstrated that we do not require the full inverse semigroup structure of a $K$-inverse semigroup in defining $K(S)$ by showing that we can work through all the arguments above for the underlying groupoid.

If $G$ is a groupoid, we will denote by $\dom(x) = x^{-1}x$ and $\ran(x) = xx^{-1}$ for $x\in G$. We will say two identities $e,f\in G_{0}$ are $\mathcal{D}$-related if they are in the same connected component of $G$.

An \emph{ordered groupoid} $(G,\leq)$ is a groupoid equipped with a partial order $\leq$ satisfying the following four axioms:
\begin{enumerate}
\item If $x\leq y$ then $x^{-1}\leq y^{-1}$.
\item If $x\leq y$, $x^{\prime}\leq y^{\prime}$ and the products $xx^{\prime}$ and $yy^{\prime}$ are defined then $xx^{\prime}\leq yy^{\prime}$.
\item If $e\in G_{0}$ is such that $e\leq \dom(x)$ then there exists a unique element $(x|e)\in G$ such that $(x|e)\leq x$ and $\dom(x|e) = e$.
\item If $e\in G_{0}$ is such that $e\leq \ran(x)$ then there exists a unique element $(e|x)\in G$ such that $(e|x)\leq x$ and $\ran(e|x) = e$.
\end{enumerate}

An ordered groupoid is said to be \emph{inductive} if the partially ordered set of identities forms a meet-semilattice.
An \emph{ordered groupoid with zero} is an ordered groupoid $G$ with a distinguished identity $0\in G_{0}$ such that $0\le e$ for all $e\in G_{0}$ and such that for all $x\in G$, $\dom(x)\neq 0$ and $\ran(x)\neq 0$. 
\begin{comment}
An ordered groupoid is called a \emph{meet ordered groupoid} if $\exists x\wedge y$ for all $x,y\in G$ (so, in particular it is an inductive groupoid).
\end{comment}
We will say two elements $x,y\in G$ are \emph{orthogonal} and write $x\perp y$ if $\dom(x)\wedge \dom(y) = 0$ and $\ran(x)\wedge \ran(y) = 0$.
It is clear that $x\perp y$ implies $x^{-1}\perp y^{-1}$.
An inductive groupoid $G$ with zero is \emph{orthogonally complete} if joins of orthogonal elements always exist and multiplication distributes over orthogonal joins when the multiplication is defined.

We want our groupoid $G$ to satisfy three further conditions:
\begin{enumerate}
\item If $e\perp f$, $e,f\in G_{0}$, and $x\in G$ is such that $\dom(x) = e\vee f$ then $\ran(x|e)\perp \ran(x|f)$ and $\ran(x|e)\vee \ran(x|f) = \ran(x)$ and if $e\perp f$, $e,f\in G_{0}$, and $y\in G$ is such that $\ran(y) = e\vee f$ then $\dom(e|y)\perp \dom(f|y)$ and $\dom(e|y)\vee \dom(f|y) = \dom(y)$. 
\item For every $e,f\in G_{0}$ there exist $e^{\prime},f^{\prime}\in G_{0}$ with $e^{\prime}\perp f^{\prime}$, $e\mathcal{D} e^{\prime}$ and $f\mathcal{D} f^{\prime}$.
\item If $x\perp y$ then $\dom(x)\vee \dom(y) \stackrel{x\vee y}{\longrightarrow} \ran(x)\vee \ran(y)$.
\end{enumerate}

We will define a $K$-\emph{groupoid} to be an orthogonally complete inductive groupoid with $0$ satisfying conditions (1), (2) and (3). If $S$ is a $K$-inverse semigroup, then the associated ordered groupoid by endowing $S$ with the restricted product is a $K$-groupoid. This is the motivating example. Throughout what follows $G$ will be a $K$-groupoid.

\begin{comment}
The following lemma is probably well-known.

\begin{lemma}
Let $x,y\in G$. Then $x \perp y$ iff $x^{-1}\perp y^{-1}$.
\end{lemma}

\begin{proof}
This follows from the fact that $\dom(x) = \ran(x^{-1})$.
\end{proof}
\end{comment}

\begin{lemma}
\label{kgroupoidlemma1}
Let $e_{1},e_{2},f_{1},f_{2}\in G_{0}$ be such that $e_{1}\,\mathcal{D}\, e_{2}$, $f_{1}\,\mathcal{D}\, f_{2}$, $e_{1}\perp f_{1}$ and $e_{2}\perp f_{2}$. Then by the assumptions on $G$ there exist $e_{1}\vee f_{1}$ and $e_{2}\vee f_{2}$, and
$$(e_{1}\vee f_{1})\, \mathcal{D}\, (e_{2}\vee f_{2}).$$
\end{lemma}

\begin{proof}
Since $e_{1}\,\mathcal{D}\, e_{2}$ and $f_{1}\, \mathcal{D}\, f_{2}$, there exist $x,y\in G$ with $e_{1} \stackrel{x}{\rightarrow} e_{2}$ and $f_{1} \stackrel{y}{\rightarrow} f_{2}$. Since $e_{1}\perp f_{1}$ and $e_{2}\perp f_{2}$, $x\perp y$. Then by condition (3), $e_{1}\vee f_{1} \stackrel{x\vee y}{\longrightarrow} e_{2}\vee f_{2}$.
\end{proof}

Let $A(G) = G_{0}/\mathcal{D}$ and define $[e] + [f]$ to be $[e^{\prime} \vee f^{\prime}]$ for $e^{\prime}\, \mathcal{D}\, e$ and $f^{\prime}\, \mathcal{D}\, f$. This is a well-defined binary operation by Lemma \ref{kgroupoidlemma1} and condition (2).

\begin{lemma}
\label{AGcommmon}
(A(G),+) is a commutative monoid.
\end{lemma}

\begin{proof}
Firsly, as above, we see that $+$ is commutative since $\vee$ is commutative and $0$ will be the identity of $A(G)$ (note that by assumption $0$ is not in the same connected component as any other element). So it remains to prove that $+$ is associative. Suppose that $e^{\prime}, e^{\prime\prime}, f^{\prime}, f^{\prime\prime}, g^{\prime}, g^{\prime\prime}, h,h^{\prime}\in G_{0}$ are identities such that
$$e^{\prime} \stackrel{s_{1}}{\rightarrow} e \stackrel{s_{2}}{\rightarrow} e^{\prime\prime}, \quad f^{\prime} \stackrel{t_{1}}{\rightarrow} f \stackrel{t_{2}}{\rightarrow} f^{\prime\prime},$$
$$g^{\prime} \stackrel{u_{1}}{\rightarrow} g \stackrel{u_{2}}{\rightarrow} g^{\prime\prime}, \quad h \stackrel{w_{1}}{\rightarrow} (e^{\prime}\vee f^{\prime}), \quad (f^{\prime\prime}\vee g^{\prime\prime}) \stackrel{w_{2}}{\rightarrow} h^{\prime},$$
$$e^{\prime}\perp f^{\prime}, \quad h\perp g^{\prime}, \quad f^{\prime\prime}\perp g^{\prime\prime}, \quad h^{\prime}\perp e^{\prime\prime}.$$
\begin{comment}
Suppose $e^{\prime}\mathcal{D} e\mathcal{D} e^{\prime\prime}$, $f^{\prime}\mathcal{D} f\mathcal{D} f^{\prime\prime}$, $g^{\prime}\mathcal{D} e\mathcal{D} g^{\prime\prime}$, $e^{\prime}\perp f^{\prime}$, $f^{\prime\prime}\perp g^{\prime\prime}$, $h\mathcal{D} (e^{\prime}\vee f^{\prime})$, $(f^{\prime\prime}\vee g^{\prime\prime}) \mathcal{D} h^{\prime}$, $h\perp g^{\prime}$ and $h^{\prime} \perp e^{\prime\prime}$. We need to check
$$(h\vee g^{\prime}) \mathcal{D} (h^{\prime} \vee e^{\prime\prime}).$$
Let $s_{1},s_{2},t_{1},t_{2},u_{1},u_{2},w_{1},w_{2}\in G$ be such that $\dom(s_{1}) = e^{\prime}$, $\ran(s_{1}) = \dom(s_{2}) = e$, $\ran(s_{2}) = e^{\prime\prime}$, $\dom(t_{1}) = f^{\prime}$, $\ran(t_{1}) = \dom(t_{2}) = f$, $\ran(t_{2}) = f^{\prime\prime}$, $\dom(u_{1}) = g^{\prime}$, $\ran(u_{1}) = \dom(u_{2}) = g$, $\ran(u_{2}) = g^{\prime\prime}$, $\dom(w_{1}) = h$, $\ran(w_{1}) = e^{\prime}\vee f^{\prime}$, $\dom(w_{2}) = f^{\prime\prime}\vee g^{\prime\prime}$ and $\ran(w_{2}) = h^{\prime}$ so that these are the elements connecting up the identities as described above.
\end{comment}
Let $x = s_{2}s_{1}(e^{\prime}|w_{1})$, $y = (w_{2}|f^{\prime\prime})t_{2}t_{1}(f^{\prime}|w_{1})$ and $z = (w_{2}|g^{\prime\prime})u_{2}u_{1}$. These elements are well-defined because all the domains and ranges match up. We have $\dom(x) = \dom(e^{\prime}|w_{1})$, $\ran(x) = e^{\prime\prime}$, $\dom(y) = \dom(f^{\prime}|w_{1})$, $\ran(y) = \ran(w_{2}|f^{\prime\prime})$, $\dom(z) = g^{\prime}$ and $\ran(z) = \ran(w_{2}|g^{\prime\prime})$. By condition (1), $\dom(e^{\prime}|w_{1}) \perp \dom(f^{\prime}|w_{1})$ and $\ran(w_{2}|f^{\prime\prime}) \perp \ran(w_{2}|g^{\prime\prime})$. Further, $\dom(e^{\prime}|w_{1}) \vee \dom(f^{\prime}|w_{1}) = \dom(w_{1}) = h$, $h\perp g^{\prime}$ and so $\dom(x)$, $\dom(y)$ and $\dom(z)$ are all mutually orthogonal. Similarly, $\ran(x)$, $\ran(y)$ and $\ran(z)$ are mutually orthogonal. Thus $\exists x\vee y\vee z$. Further $\dom(x)\vee\dom(y)\vee\dom(z) = h\vee g^{\prime}$ and $\ran(x)\vee\ran(y)\vee\ran(z) = e^{\prime\prime}\vee h^{\prime}$. Hence $(h\vee g^{\prime})\, \mathcal{D}\, (e^{\prime\prime}\vee h^{\prime})$.   
\end{proof}

We then define $K(G) = \mathcal{G}(A(G))$, as in Section 4.2. Now suppose $S$ is a $K$-inverse semigroup and $G(S)$ is the underlying $K$-groupoid obtained by restricting the multiplication in $S$. Then by construction $\mathcal{D}$ is the same in both $G(S)$ and $S$ and the order is the same (thus the same elements are orthogonal and joins of orthogonal elements are the same in both $S$ and $G(S)$). It therefore follows that
$$K(G(S)) \cong K(S).$$

In fact we could even have deduced Lemma \ref{AScommmon} from Lemma \ref{AGcommmon}.

\begin{comment}
\subsection{Morita Equivalence}

\begin{defin}
We say that an ordered groupoid $G$ is an \emph{enlargement} of an ordered groupoid $H$ if $H$ is a full subcategory of $G$, $H$ is an order ideal and if every identity of $G$ is isomorphic to an identity of $H$.
\end{defin}

Idea: Maybe we can use the theory developed in the previous section to relate Morita equivalence of K-inverse semigroups to their $K_{0}$-groups, so that we have a result of the form:

\begin{conjecture}
Let $S$ and $T$ be K-inverse semigroups. Then $K_{0}(S)\cong K_{0}(T)$ iff $S$ and $T$ are Morita equivalent.
\end{conjecture} 

\begin{lemma}
Let $G$ be an enlargement of a K-groupoid $H$. Then $G$ is a K-groupoid.
\end{lemma}

\begin{lemma}
Let $G$ be an enlargement of a K-groupoid $H$. Then $K_{0}(G)\cong K_{0}(H)$.
\end{lemma}

\begin{corollary}
Let $S$ and $T$ be K-inverse semigroups. Then $K_{0}(S)\cong K_{0}(T)$ iff $S$ and $T$ are Morita equivalent.
\end{corollary} 
\end{comment}

\section{Modules over inverse semigroups}

In this section, we define the concept of \emph{module} for an orthogonally complete inverse semigroup. We will use this to define a $K$-group for arbitrary orthogonally complete inverse semigroups in such a way that if the semigroup is a $K$-inverse semigroup this definition will agree with that of Section 4.2.

Let $S$ be a fixed orthogonally complete inverse semigroup. We shall only be dealing with unitary right actions of $S$ \cite{Talwar}; that is, actions $X\times S\rightarrow X$ such that $X\cdot S = X$. Furthermore, rather than arbitrary actions, we shall work with (right) \'{e}tale actions (c.f. \cite{LawsonSteinberg}, \cite{Steinberg}), whose definition we now recall.

An action $X\times S\rightarrow X$ is said to be a \emph{(right) \'{e}tale action} if there is also a function $p:X\rightarrow E(S)$ such that the following two axioms hold:

\begin{description}
\item[{\rm (E1)}] $x\cdot p(x) = x$.
\item[{\rm (E2)}] $p(x\cdot s) = s^{-1}p(x)s$.
\end{description}

We refer to the \emph{\'{e}tale set} $(X,p)$. On such a set, we may define a partial order $\leq$ as follows: $x\leq y$ if and only if $x = y \cdot p(x)$. If $(X,p)$ and $(Y,q)$ are \'{e}tale sets, then a \emph{morphism} is a function $\alpha:X\rightarrow Y$ such that 

\begin{description}
\item[{\rm (EM1)}] $\alpha(x\cdot s) = \alpha(x)\cdot s$.
\item[{\rm (EM2)}] $p(x) = q(\alpha(x))$.
\end{description}

Since we are working with inverse semigroups \emph{with zero}, we shall actually only consider a special class of \'{e}tale sets. An \'{e}tale set $(X,p)$ is called \emph{pointed} if there is a distinguished element $0_{X}\in X$, called a \emph{zero}, such that the following axioms hold:

\begin{description}
\item[{\rm (P1)}] $p(0_{X}) = 0$ and if $p(x) = 0$ then $x = 0_{X}$.
\item[{\rm (P2)}] $0_{X}\cdot s = 0_{X}$ for all $s\in S$.
\item[{\rm (P3)}] $x \cdot 0 = 0_{X}$ for all $x\in X$.
\end{description}

Since $0_{X} = x\cdot 0 = x\cdot p(0_{X})$, we have $0_{X}\leq x$ for all $x\in X$. Thus the distinguished element $0_{X}$ in $X$ is actually the minimum element of the poset $X$. Usually we shall write $0$ instead of $0_{X}$. A \emph{pointed morphism} of pointed \'{e}tale sets is an \'{e}tale morphism which preserves the minimum elements of the \'{e}tale sets (PM). We denote the category of right pointed \'{e}tale $S$-sets and their pointed morphisms by $\Etale_{S}$.

Let $(X,p)$ be a pointed \'{e}tale set and $x,y\in X$. Define $x\perp y$ if $p(x)p(y) = 0$ and say that $x$ and $y$ are \emph{orthogonal}. We will say elements $x,y\in X$ are \emph{strongly orthogonal} if $x\perp y$, $\exists x\vee y$ and $p(x)\vee p(y) = p(x\vee y)$.

A pointed set $(X,p)$ is a \emph{(right) premodule} if it satisfies the following axioms:

\begin{description}
\item[{\rm (PRM1)}] If $x,y\in X$ are strongly orthogonal then for all $s\in S$ we have $x\cdot s$ and $y\cdot s$ are strongly orthogonal and 
$(x\vee y)\cdot s = (x\cdot s)\vee (y\cdot s)$.
\item[{\rm (PRM2)}] If $s,t\in S$ are orthogonal then $x\cdot s$ and $x\cdot t$ are strongly orthogonal for all $x\in X$.
\end{description}

A \emph{premodule morphism} of premodules is a pointed morphism $f:X\rightarrow Y$ such that if $x,y\in X$ are strongly orthogonal, then $f(x),f(y)\in Y$ are strongly orthogonal and $f(x\vee y) = f(x)\vee f(y)$ (PRMM). We will denote the category of right premodules and their premodule morphisms by $\Premod_{S}$.

\begin{proposition}
\label{rightidealpremod}
Let $I$ be a right ideal of $S$, define an action $I\times S\rightarrow I$ by $s\cdot t = st$ and define $p:I\rightarrow S$ by $p(s) = s^{-1}s$. Then $(I,p)$ is a premodule.
\end{proposition}

\begin{proof}
To see that $(I,p)$ is an \'{e}tale set note that
$$s\cdot p(s) = ss^{-1}s = s$$
and 
$$p(s\cdot t) = (st)^{-1}(st) = t^{-1}s^{-1}st = t^{-1}p(s) t.$$
It is pointed since $I$ necessarily contains $0$.
Now we have to be a little cautious as there are potentially two partial orders on elements of $I$: the order in $I$ viewed as an \'{e}tale set and the natural partial order of the semigroup $S$. Fortunately, these two orders coincide since $p(s) = s^{-1}s$ and so $s = t\cdot p(s)$ iff $s = ts^{-1}s$. Consequently, we are able to write $s\leq t$ without there being any ambiguity.
We will now show that if $s,t$ are strongly orthogonal elements of $I$ then $s\perp t$ in $S$.
Let $s, t$ be strongly orthogonal elements of $I$. Then $0 = p(s)p(t) = s^{-1}st^{-1}t$ and so premultiplying by $s$ and postmultiplying by $t^{-1}$ we have $st^{-1} = 0$. Let $u = s\vee t$ be the join of $s$ and $t$ in $I$ (which since the orders coincide will be the join in S). We then have $u\cdot p(s) = us^{-1}s = s$ and $u\cdot p(t) = ut^{-1}t = t$. Since $s^{-1}s\perp t^{-1}t$ in $S$, we must have $us^{-1}s\perp ut^{-1}t$ in $S$ and thus $s\perp t$ in $S$. Let us now check the axioms for a premodule.
\begin{description}
\item[{\rm (PRM1)}] If $s,t\in I$ are strongly orthogonal then since $s\perp t$ in $S$ we must have $su$ and $tu$ are orthogonal in $S$ for all $u\in S$, and therefore also orthogonal in $I$. Further since $S$ is orthogonally complete there exists $su\vee tu$ in $S$ and
$$su\vee tu = (s\vee t) u.$$
Since $I$ is a right ideal and $s\vee t \in I$ by assumption then $(s\vee t)u\in I$. Thus $su\vee tu\in I$ and $p(su\vee tu) = p(su)\vee p(tu)$.
\item[{\rm (PRM2)}] If $s,t\in S$ are orthogonal and $u\in I$ then $us$ and $ut$ are orthogonal in $S$, and $us,ut,u(s\vee t)\in I$. In addition
$$u(s\vee t) = us \vee ut$$
and $p(us)\vee p(ut) = p(us\vee ut)$. Thus $us$ and $ut$ are strongly orthogonal in $I$.
\end{description}
\end{proof} 

If $a\in S$ then we can consider the principal right ideal $aS$ generated by $a$ and this will be a premodule. In this case if $s\perp t$ in $S$ and $s = as$, $t = at$ then 
$$s\vee t = as \vee at = a(s\vee t)\in aS$$
and so $s,t$ are strongly orthogonal in $aS$. Since $aS = aa^{-1}S$ we will mainly be considering principal right ideals generated by idempotents.

The following lemma will be used often:

\begin{lemma}
\label{usefuljoinlem}
Let $X$ be a premodule and $x,y,z\in X$ be such that $p(x)p(y) = 0$, $p(z) = p(x)\vee p(y)$ and $z\geq x,y$. Then $z = x\vee y$
\end{lemma}

\begin{proof}
Since $X$ is a premodule and $p(x)p(y) = 0$, (PRM2) implies that $x = z\cdot p(x)$ and $y = z\cdot p(y)$ are strongly orthogonal, and so there exists $x\vee y$ with $p(x\vee y) = p(x)\vee p(y) = p(z)$. Furthermore, since $x,y\leq z$ it follows that $x\vee y\leq z$. Thus,
$$z = z \cdot p(z) = z\cdot p(x\vee y) = x\vee y.$$
\end{proof}

A pointed set $(X,p)$ is called a \emph{(right) module} if it satisfies the following axioms:

\begin{description}
\item[{\rm (M1)}] If $x\perp y$ then $\exists x\vee y$ and $p(x\vee y) = p(x)\vee p(y)$.
\item[{\rm (M2)}] If $x\perp y$ then $(x\vee y)\cdot s = x\cdot s \vee y\cdot s$.
\end{description}

Observe that (M2) makes sense, because

$$p(x\cdot s)p(y\cdot s) = s^{-1}p(x)ss^{-1}p(y)s = s^{-1}p(x)p(y)s = 0$$
if $x\perp y$. Note that in general for a module $(X,p)$ we cannot simplify $x\cdot s\vee x\cdot t$ to $x\cdot(s\vee t)$ as $s$ and $t$ need not be orthogonal. 
We do however have the following lemma:

\begin{lemma}
If $s$ and $t$ are orthogonal and $(X,p)$ is a module, then 
$$x\cdot s \vee x\cdot t = x\cdot(s\vee t)$$
for all $x\in X$.
\end{lemma}

\begin{proof}
We have $p(x\cdot s)p(x\cdot t) = s^{-1}p(x)st^{-1}p(x)t = 0$ since $s$ and $t$ are orthogonal and thus there exists $x\cdot s \vee x\cdot t$. Further
$$x\cdot s = x \cdot (p(x)s) = x\cdot ((s\vee t)s^{-1}p(x)s) = (x\cdot(s\vee t))\cdot p(x\cdot s)$$
and therefore $(x\cdot s)\vee (x\cdot t) \leq x\cdot(s\vee t)$. Now 
$$p(x\cdot (s\vee t)) = (s\vee t)^{-1} p(x)(s\vee t) = (s^{-1}p(x)s)\vee(t^{-1}p(x)t) = p(x\cdot s) \vee p(x\cdot t).$$
We then have
$$x\cdot s \vee x\cdot t = (x\cdot (s\vee t))\cdot p(x\cdot s\vee x \cdot t) = (x\cdot (s\vee t))\cdot p(x\cdot (s\vee t)) = x\cdot(s\vee t).$$
\end{proof}

Let $(X,p)$ and $(Y,q)$ be modules. A \emph{module morphism} is a pointed morphism $\alpha:X\rightarrow Y$ such that if $x\perp y$ then $\alpha(x\vee y) = \alpha(x)\vee \alpha(y)$ (MM). Observe that this is well-defined because $q(\alpha(x))q(\alpha(y)) = p(x)p(y) = 0$ if $x\perp y$.
We denote the category of right modules of $S$ together with their module morphisms by $\Mod_{S}$. If $(X,p)$ is a module then $Y\subseteq X$ is called a \emph{submodule} of $X$ if $y\in Y$ implies $y\cdot s\in Y$ for all $s\in S$ and if $u,v\in Y$ with $u\perp v$ then $u\vee v\in Y$.

\begin{lemma}
Let $(X,p)$ and $(Y,q)$ be modules. Then the image $\im(\theta)$ of a module morphism $\theta: X\rightarrow Y$ is a submodule of $Y$.
\end{lemma}

\begin{proof}
Suppose $\theta(x)\perp \theta(y)$. Then $q(\theta(x))q(\theta(y)) = 0$. But $q(\theta(x)) = p(x)$ and $q(\theta(y)) = p(y)$. Thus $p(x)p(y) = 0$ and so $x\perp y$. It follows that $x\vee y$ exists and since $\theta$ is a module morphism we have that $\theta(x\vee y) = \theta(x)\vee \theta(y)$. Thus the image of $\theta$ is closed under orthogonal joins. It is immediate that the image of $\theta$ is closed under the action of $S$.
\end{proof}

Let $(X,p)$ be a module. We define a \emph{congruence} on $X$ to be an equivalence relation $\rho$ such that the following conditions hold:

\begin{description}
\item[{\rm (C1)}] $x\,\rho\, y$ implies that $x\cdot s\,\rho\, y\cdot s$.
\item[{\rm (C2)}] $x\,\rho\, y$ implies that $p(x) = p(y)$.
\item[{\rm (C3)}] $x_{1}\perp x_{2}$, $y_{1}\perp y_{2}$ and $x_{i}\,\rho\, y_{i}$ implies that $x_{1}\vee x_{2}\,\rho\, y_{1}\vee y_{2}$.
\end{description}

We will now prove some facts about congruences which we will use later.

\begin{lemma}
Let $\theta:(X,p)\rightarrow (Y,q)$ be a module homomorphism. Define the kernel of $\theta$ by
$$\ker(\theta) = \left\{(x,y)\in X\times X|\theta(x) = \theta(y)\right\}.$$
Then $\ker(\theta)$ is a congruence.
\end{lemma} 

\begin{proof}
$\ker(\theta)$ is clearly an equivalence relation. Let us check the congruence axioms:
\begin{description}
\item[{\rm (C1)}] $\theta(x) = \theta(y)$ implies $\theta(x\cdot s) = \theta(y\cdot s)$.
\item[{\rm (C2)}] $\theta(x) = \theta(y)$ implies $p(x) = q(\theta(x)) = q(\theta(y)) = p(y)$.
\item[{\rm (C3)}] $x_{1}\perp x_{2}$, $y_{1}\perp y_{2}$ and $\theta(x_{i}) = \theta(y_{i})$ implies 
$$\theta(x_{1}\vee x_{2}) = \theta(x_{1})\vee \theta(x_{2}) = \theta(y_{1})\vee \theta(y_{2}) = \theta(y_{1}\vee y_{2}).$$
\end{description}
\end{proof}

\begin{lemma}
Let $\rho$ be a congruence on a module $(X,p)$. Then $X/\rho$ can naturally be endowed with the structure of a module.
\end{lemma}

\begin{proof}
Denote the equivalence class of an element $x\in X$ by $[x]$. Let $x\,\rho\, y$. Then $x\cdot s\,\rho\, y\cdot s$ and $p(x) = p(y)$, thus the action $[x]\cdot s = [x\cdot s]$ and map $p([x]) = p(x)$ are well-defined. Checking the axioms, we have
\begin{description}
\item[{\rm (E1)}] $[x]\cdot p([x]) = [x\cdot p(x)] = [x]$. 
\item[{\rm (E2)}] $p([x]\cdot s) = p([x\cdot s]) = p(x\cdot s) = s^{-1}p(x)s = s^{-1}p([x])s$.
\item[{\rm (P1) - (P3)}] are clear since $0_{X}$ will always be in an equivalence class on its own.
\item[{\rm (M1)}] If $[x]\perp [y]$, then $x\perp y$ and so the axiom follows by (C2) and (C3).
\item[{\rm (M2)}] If $[x]\perp [y]$, then 
$$([x]\vee [y])\cdot s = [x\vee y]\cdot s = [(x\vee y)\cdot s] = [x\cdot s \vee y\cdot s] = [x\cdot s] \vee [y\cdot s].$$
\end{description}
\end{proof}

\begin{lemma}
\label{projectioncong}
Let $\rho$ be a congruence on a module $(X,p)$ and denote the equivalence class of an element $x\in X$ by $[x]$. Then the map $\pi:X\rightarrow X/\rho$ defined by $\pi(x) = [x]$ is a module morphism.
\end{lemma}

\begin{proof}
We check the axioms:
\begin{description}
\item[{\rm (EM1)}] $\pi(x\cdot s) = [x\cdot s] = [x]\cdot s = \pi(x)\cdot s$.
\item[{\rm (EM2)}] $p(x) = p([x]) = p(\pi(x))$.
\item[{\rm (PM)}] $\pi(0) = [0]$ and $\pi(x) = [0]$ implies $x = 0$.
\item[{\rm (MM)}] Suppose $p(x)p(y) = 0$. Then $p([x])p([y]) = 0$ and 
$$\pi(x\vee y) = [x\vee y] = [x]\vee [y].$$
\end{description}
\end{proof}

\begin{lemma}
\label{FITmods}
Let $\theta:(X,p)\rightarrow (Y,q)$ be a module homomorphism. Then $\im(\theta)$ and $X/\ker(\theta)$ are isomorphic as modules.
\end{lemma}

\begin{proof}
Define 
$$\alpha: X/\ker(\theta) \rightarrow \im(\theta)$$
by $\alpha([x]) = \theta(x)$. By construction, if $(x,y)\in \ker(\theta)$ then $\theta(x) = \theta(y)$ and so $\alpha$ is a well-defined map. Let us check that it is a module morphism:
\begin{description}
\item[{\rm (EM1)}] $\alpha([x]\cdot s) = \alpha([x\cdot s]) = \theta(x\cdot s) = \theta(x)\cdot s$.
\item[{\rm (EM2)}] $p([x]) = p(x) = q(\theta(x))$.
\item[{\rm (PM)}] $\theta(0) = 0$ and $\theta(x) = 0$ implies $x = 0$.
\item[{\rm (MM)}] Suppose $p(x)p(y) = 0$. Then $p([x])p([y]) = 0$ and 
$$\alpha([x\vee y]) = \theta(x\vee y) = \theta(x)\vee \theta(y) = \alpha([x])\vee \alpha([y]).$$
\end{description}
\end{proof}

\begin{lemma}
Let $\rho$ and $\sigma$ be congruences on a module $(X,p)$. Then their intersection $\rho\cap \sigma$ is a congruence.
\end{lemma}

\begin{proof}
$\rho\cap \sigma$ is clearly an equivalence relation. We now check the congruence axioms:
\begin{description}
\item[{\rm (C1)}] Suppose $x\, (\rho\cap\sigma)\, y$. Then $x\,\rho\, y$ and $x\,\sigma\, y$. So $x\cdot s\,\rho\, y\cdot s$ and $x\cdot s\,\sigma\, y\cdot s$. 
\item[{\rm (C2)}] $x\,\rho\, y$ and $x\,\sigma\, y$ implies $p(x) = p(y)$.
\item[{\rm (C3)}] $x_{1}\perp x_{2}$, $y_{1}\perp y_{2}$ and $x_{i}\,(\rho\cap\sigma)\, y_{i}$ implies that $x_{1}\vee x_{2}\,\rho\, y_{1}\vee y_{2}$ and $x_{1}\vee x_{2}\,\sigma\, y_{1}\vee y_{2}$.
\end{description}
\end{proof}

\begin{lemma}
Let $(X,p)$ be a module and let $\rho_{\max}$ be the equivalence relation defined on $X$ by $x\,\rho_{\max}\, y$ if $p(x) = p(y)$. Then $\rho_{\max}$ is a congruence. Furthermore, $\rho_{\max}$ is the largest congruence defined on $X$.
\end{lemma}

\begin{proof}
$\rho_{\max}$ is clearly an equivalence relation. We now check the congruence axioms:
\begin{description}
\item[{\rm (C1)}] $p(x) = p(y)$ implies $p(x\cdot s) = p(y\cdot s)$.
\item[{\rm (C2)}] $x\,\rho_{\max}\, y$ implies by definition that $p(x) = p(y)$.
\item[{\rm (C3)}] $x_{1}\perp x_{2}$, $y_{1}\perp y_{2}$ and $p(x_{i}) = p(y_{i})$ implies that 
$$p(x_{1}\vee x_{2}) = p(x_{1})\vee p(x_{2}) = p(y_{1})\vee p(y_{2}) = p(y_{1}\vee y_{2}).$$
\end{description}
It is the largest congruence on $X$ by (C2).
\end{proof} 

Let $(X,p)$ be a module. We will call $X/\rho_{\max}$ the \emph{submodule of $E(S)$ generated by $(X,p)$}. By the above, we see that $X/\rho_{\max}$ is the smallest submodule of $(X,p)$.

A finitely generated order ideal $A$ of a premodule $X$ is said to be \emph{orthogonal} if there exist $x_{1},\ldots,x_{m}\in A$ such that the $x_{i}$'s are pairwise orthogonal and $A = \left\{x_{1},\ldots,x_{m}\right\}^{\downarrow}$; that is, all elements $x\in X$ with $x\leq x_{i}$ for some $i$. Let $\overline{X}$ be the set of all finitely generated orthogonal order ideals of the premodule $X$. We will denote by $x^{\downarrow} = \left\{x\right\}^{\downarrow}$.

\begin{comment}
\begin{remark}
We would have liked it to be the case that if $X$ were a module then $X$ would be isomorphic to $\overline{X}$ and that if $ef = 0$, $e,f\in E(S)$, then $\overline{eS}\bigoplus\overline{fS}$ would be isomorphic to $\overline{(e\vee f)S}$. Unfortunately this is not the case so we will define an equivalence relation on $\overline{X}$ so that these two properties are true.
\end{remark}
\end{comment}

Let $X$ be a premodule. Let $\equiv$ be the smallest equivalence relation on $\overline{X}$ such that if $x_{1}$ and $x_{2}$ are strongly orthogonal then $\left\{x_{1},x_{2},x_{3},\ldots,x_{n}\right\}^{\downarrow}\equiv \left\{(x_{1}\vee x_{2}),x_{3},\ldots,x_{n}\right\}^{\downarrow}$ and let 
$$X^{\sharp} = \overline{X}/\equiv.$$

If $A = \left\{x_{1},\ldots,x_{m}\right\}^{\downarrow}$ is an element of $X^{\sharp}$ (with the $x_{i}$'s pairwise orthogonal), then define
$$p^{\sharp}(A) = \bigvee_{i=1}^{m}{p(x_{i})}$$
and
$$A\cdot s = \left\{x_{1}\cdot s,\ldots,x_{m}\cdot s\right\}^{\downarrow}.$$

Then $p^{\sharp}(A)$ is well-defined and $A\cdot s\in X^{\sharp}$ since $p(x_{i}\cdot s)p(x_{j}\cdot s) = 0$ for $i\neq j$ and if $x\leq y$ and $p(x)p(y)=0$ then $x=0$.

\begin{lemma}
Let $(X,p)$ be a premodule. The above gives $(X^{\sharp},p^{\sharp})$ the structure of a module.
\end{lemma}

\begin{proof}
Let $A = \left\{x_{1},\ldots,x_{m}\right\}^{\downarrow}$ and $B = \left\{y_{1},\ldots,y_{n}\right\}^{\downarrow}$, where the generators are pairwise orthogonal. First, we check it is a pointed set:
\begin{description}
\item[{\rm (E1)}] $A\cdot p^{\sharp}(A) = \left\{x_{1}\cdot p^{\sharp}(A),\ldots,x_{m}\cdot p^{\sharp}(A)\right\}^{\downarrow} = \left\{x_{1},\ldots,x_{m}\right\}^{\downarrow} = A$.
\item[{\rm (E2)}] $p^{\sharp}(A\cdot s) = \vee_{i=1}^{m}{p(x_{i}\cdot s)} = s^{-1}p^{\sharp}(A)s$.
\item[{\rm (P1) - (P3)}] follow from the fact that $p^{\sharp}(A) = 0$ iff $A = 0^{\downarrow}$.
\end{description}
Thus $(X^{\sharp},p^{\sharp})$ is a pointed \'{e}tale set. Let us now show $X^{\sharp}$ is a premodule:
\begin{description}
\item[{\rm (PRM1)}] Let $A = \left\{x_{1},\ldots,x_{m}\right\}^{\downarrow},B = \left\{y_{1},\ldots,y_{n}\right\}^{\downarrow}\in X^{\sharp}$ be strongly orthogonal. Since $p^{\sharp}(A)p^{\sharp}(B) = 0$ we have that $p(x_{i})p(y_{j}) = 0$ for all $i,j$. Let 
$$C = \left\{x_{1},\ldots,x_{m},y_{1},\ldots,y_{n}\right\}^{\downarrow}.$$

Then $A = C\cdot p^{\sharp}(A)$ and $B = C\cdot p^{\sharp}(B)$. So $A,B\leq C$ and therefore $A\vee B\leq C$. Further 
$$p^{\sharp}(C) = p^{\sharp}(A)\vee p^{\sharp}(B) = p^{\sharp}(A\vee B).$$
Thus, 
$$A\vee B = C\cdot p^{\sharp}(A\vee B) = C\cdot p^{\sharp}(C) = C.$$ 

Now suppose $s\in S$. Then $p^{\sharp}(C\cdot s) = p^{\sharp}(A\cdot s)\vee p^{\sharp}(B\cdot s)$ and $A\cdot s, B\cdot s\leq C\cdot s$. So $A\cdot s$, $B\cdot s$ are bounded above. Let $D = \left\{z_{1},\ldots,z_{k}\right\}^{\downarrow}\geq A\cdot s, B\cdot s$ with $p^{\sharp}(D) = p^{\sharp}(C\cdot s)$. Then for each $x_{i}$, $x_{i}\cdot s\leq z_{j}$ for some $z_{j}$ and similarly for the $y_{i}$'s. Suppose $z_{i}\geq x_{i_{1}}\cdot s, \ldots x_{i_{r_{1}}}\cdot s, y_{j_{1}}\cdot s, \ldots y_{j_{r_{2}}}\cdot s$ and 
$$p(z_{i}) = \vee_{k=1}^{r_{1}}{p(x_{i_{k}}\cdot s)}\bigvee \vee_{k=1}^{r_{2}}{p(y_{j_{k}}\cdot s)}.$$
Then, since $X$ is a premodule, Lemma \ref{usefuljoinlem} tells us that
$$z_{i} = \vee_{k=1}^{r_{1}}{(x_{i_{k}}\cdot s)}\bigvee \vee_{k=1}^{r_{2}}{(y_{j_{k}}\cdot s)}$$
and so $x_{i_{1}}\cdot s, \ldots x_{i_{r_{1}}}\cdot s, y_{j_{1}}\cdot s, \ldots y_{j_{r_{2}}}\cdot s$ are strongly orthogonal. 
Thus $D \equiv C\cdot s$. We therefore see that $A\cdot s,B\cdot s$ are strongly orthogonal and 
$$(A\vee B)\cdot s = (A\cdot s) \vee (B\cdot s).$$

\item[{\rm (PRM2)}] Let $A = \left\{x_{1},\ldots,x_{m}\right\}^{\downarrow}\in X^{\sharp}$ be arbitrary, let $s,t\in S$ be orthogonal and let $u =s\vee t$. Then
$$p^{\sharp}(A\cdot s)p^{\sharp}(A\cdot t) = s^{-1}p^{\sharp}(A)st^{-1}p^{\sharp}(A)t = 0$$
and so $A\cdot s\perp A\cdot t$. Now
$$(A\cdot u)\cdot p^{\sharp}(A\cdot s) = A\cdot (us^{-1}p^{\sharp}(A)s) = A\cdot (p^{\sharp}(A)s) = A\cdot s$$
giving $A\cdot s \leq A\cdot u$. In a similar manner we obtain $A\cdot t\leq A\cdot u$. Further,
\begin{eqnarray*}
p^{\sharp}(A\cdot u) &=& u^{-1}p^{\sharp}(A)u = (s^{-1}\vee t^{-1})p^{\sharp}(A)(s\vee t) \\
&=& s^{-1}p^{\sharp}(A)s \vee t^{-1}p^{\sharp}(A)t = p^{\sharp}(A\cdot s)\vee p^{\sharp}(A\cdot t).
\end{eqnarray*}
Let $B = \left\{y_{1},\ldots,y_{n}\right\}^{\downarrow}\in X^{\sharp}$ be such that $B\geq A\cdot s, A\cdot t$ and $p^{\sharp}(B) = p^{\sharp}(A\cdot u)$. Then for each $x_{i}$ we have $x_{i}\cdot s\leq y_{j}$ for some $j$ and $x_{i}\cdot t\leq y_{k}$ for some $k$. Suppose 
$$y_{k}\geq x_{i_{1}}\cdot s,\ldots,x_{i_{r_{1}}}\cdot s,x_{j_{1}}\cdot t,\ldots,x_{j_{r_{2}}}\cdot t$$
and
$$p(y_{k}) = \vee_{l=1}^{r_{1}}p(x_{i_{l}}\cdot s)\bigvee \vee_{l=1}^{r_{2}}p(x_{j_{l}}\cdot t).$$
Then as above we have
$$y_{k} = \vee_{l=1}^{r_{1}}(x_{i_{l}}\cdot s)\bigvee \vee_{l=1}^{r_{2}}(x_{j_{l}}\cdot t)$$
and so $x_{i_{1}}\cdot s,\ldots,x_{i_{r_{1}}}\cdot s,x_{j_{1}}\cdot t,\ldots,x_{j_{r_{2}}}\cdot t$ are strongly orthogonal. Thus $B\equiv A\cdot u$ and so 
$$A\cdot u = (A\cdot s)\vee (A\cdot t),$$
yielding that $A\cdot s$ and $A\cdot t$ are strongly orthogonal.
\end{description}

Thus $X^{\sharp}$ is a premodule. Let us now show $X^{\sharp}$ is a module. Suppose $A = \left\{x_{1},\ldots,x_{m}\right\}^{\downarrow},B = \left\{y_{1},\ldots,y_{n}\right\}^{\downarrow}\in X^{\sharp}$ are such that $p^{\sharp}(A)p^{\sharp}(B) = 0$. Then $p(x_{i})p(y_{j}) = 0$ for all $i,j$. Let

$$C = \left\{x_{1},\ldots,x_{m},y_{1},\ldots,y_{n}\right\}^{\downarrow}.$$

Then $A = C\cdot p^{\sharp}(A)$ and $B = C\cdot p^{\sharp}(B)$. So $A,B\leq C$. Further $p^{\sharp}(C) = p^{\sharp}(A)\vee p^{\sharp}(B)$. Lemma \ref{usefuljoinlem} then tells us that $C = A\vee B$. We have $(A\vee B)\cdot s = A\cdot s \vee B\cdot s$. Thus $(X^{\sharp},p^{\sharp})$ is a module.
\end{proof}

We can think of $\equiv$ in a slightly different way. If $A = \left\{x_{1},\ldots,x_{m}\right\}^{\downarrow},B = \left\{y_{1},\ldots,y_{n}\right\}^{\downarrow}$ are finitely
generated orthogonal order ideals then $A\equiv B$ if and only if for each $x_{i}$ there exist $b_{i1},\ldots,b_{ik_{i}}\in B$ strongly orthogonal with $x_{i} = \vee_{j=1}^{k_{i}} b_{ij}$ and for each $y_{i}$ there exist $a_{i1},\ldots,a_{ik_{i}}\in A$ strongly orthogonal with $y_{i} = \vee_{j=1}^{k_{i}} a_{ij}$.

\begin{lemma}
If $X$ is a module then $X$ is isomorphic to $X^{\sharp}$.
\end{lemma}

\begin{proof}
Firstly, every element of $X^{\sharp}$ is of the form $x^{\downarrow}$ for some $x\in X$ since all orthogonal joins satisfy the required properties. On the other hand $x^{\downarrow}$ is not equivalent to $y^{\downarrow}$ for $x\neq y$. 
Define $g:X\rightarrow X^{\sharp}$ by $g(x) = x^{\downarrow}$. It is now easy to see this is a bijective module morphism.
\end{proof}

Let $\alpha: X\rightarrow Y$ be a premodule morphism and define $\alpha^{\sharp}:X^{\sharp}\rightarrow Y^{\sharp}$ as follows. Let $A = \left\{x_{1},\ldots,x_{m}\right\}^{\downarrow}\in X^{\sharp}$. Define
$$\alpha^{\sharp}(A) = \left\{\alpha(x_{1}),\ldots,\alpha(x_{m})\right\}^{\downarrow}.$$

\begin{lemma}
\label{alphasharp}
With the above definition $\alpha^{\sharp}$ is a module morphism and if $\alpha$ is surjective then $\alpha^{\sharp}$ is surjective.
\end{lemma}

\begin{proof}
Let $\alpha:(X,p)\rightarrow (Y,q)$.
Firstly $\alpha^{\sharp}$ is well-defined since if $A = \left\{x_{1},\ldots,x_{m}\right\}^{\downarrow}$ with the $x_{i}$'s pairwise orthogonal, then $q(\alpha(x_{i}))q(\alpha(x_{j})) = p(x_{i})p(x_{j}) = 0$ for $i\neq j$, so $\alpha^{\sharp}(A)\in Y^{\sharp}$ and if $x_{1},x_{2}$ are strongly orthogonal then
$$\alpha^{\sharp}(\left\{x_{1},x_{2}\right\}^{\downarrow}) = \left\{\alpha(x_{1}),\alpha(x_{2})\right\}^{\downarrow} = \left\{\alpha(x_{1})\vee \alpha(x_{2})\right\}^{\downarrow} $$
$$= \left\{\alpha(x_{1}\vee x_{2})\right\}^{\downarrow} = \alpha^{\sharp}(\left\{x_{1}\vee x_{2}\right\}^{\downarrow}).$$

Let $A = \left\{x_{1},\ldots,x_{m}\right\}^{\downarrow}$ with the $x_{i}$'s pairwise orthogonal. We check the axioms for a module morphism:

\begin{description}
\item[{\rm (EM1)}] $\alpha^{\sharp}(A\cdot s) = \alpha^{\sharp}(A)\cdot s$.
\item[{\rm (EM2)}] $q^{\sharp}(\alpha^{\sharp}(A)) = \vee_{i=1}^{m}{q(\alpha(x_{i}))} = \vee_{i=1}^{m}{p(x_{i})} = p^{\sharp}(A)$.
\item[{\rm (PM)}] $\alpha^{\sharp}(0_{X^{\sharp}}) = \left\{\alpha(0_{X})\right\}^{\downarrow} = \left\{0_{Y}\right\}^{\downarrow} = 0_{Y^{\sharp}}$.
\item[{\rm (MM)}] Suppose $A\perp B$ with $A = \left\{x_{1},\ldots,x_{m}\right\}^{\downarrow}$, $B = \left\{y_{1},\ldots,y_{n}\right\}^{\downarrow}$. Then
$$A\vee B = \left\{x_{1},\ldots,x_{m},y_{1},\ldots,y_{n}\right\}^{\downarrow}.$$
So 
\begin{eqnarray*}
\alpha^{\sharp}(A\vee B) &=& \left\{\alpha(x_{1}),\ldots,\alpha(x_{m}),\alpha(y_{1}),\ldots,\alpha(y_{n})\right\}^{\downarrow} \\
&=& \left\{\alpha(x_{1}),\ldots,\alpha(x_{m})\right\}^{\downarrow}\vee\left\{\alpha(y_{1}),\ldots,\alpha(y_{n})\right\}^{\downarrow} = \alpha^{\sharp}(A) \vee \alpha^{\sharp}(B).
\end{eqnarray*}
\end{description}
Thus $\alpha^{\sharp}$ is a module morphism. The second part of the lemma follows immediately.
\end{proof}

We have therefore defined a functor $R$ from $\Premod_{S}$ to $\Mod_{S}$ given by $R(X) = X^{\sharp}$ and $R(\alpha) = \alpha^{\sharp}$.

\begin{proposition}
\label{Radjointforget}
The functor $R$ is left adjoint to the forgetful functor. 
\end{proposition}

\begin{proof}
First for a premodule $X$ we show that the map $\iota:X\rightarrow X^{\sharp}$ given by $\iota(x) = x^{\downarrow}$ is a premodule morphism:
\begin{description}
\item[{\rm (EM1)}] $\iota(x\cdot s) = (x\cdot s)^{\downarrow} = x^{\downarrow}\cdot s = \iota(x)\cdot s$.
\item[{\rm (EM2)}] $p^{\sharp}(\iota(x)) = p^{\sharp}(x^{\downarrow}) = p(x)$.
\item[{\rm (PM)}] $\iota(0_{X}) = O_{X}^{\downarrow} = 0_{X^{\sharp}}$.
\item[{\rm (PRMM)}] This is clear.
\end{description}

Now suppose $X$ is a premodule and let $\theta:X\rightarrow Y$ be a premodule morphism to the module $(Y,q)$. Define $\psi:X^{\sharp}\rightarrow Y$ by 
$$\psi(\left\{x_{1},\ldots,x_{m}\right\}^{\downarrow}) = \bigvee_{i=1}^{m}{\theta(x_{i})}.$$

Firstly, this is well-defined since the $\theta(x_{i})$'s are pairwise orthogonal. Let us now prove that $\psi$ is a module morphism. It is an \'{e}tale morphism since
\begin{eqnarray*}
\psi(\left\{x_{1},\ldots,x_{m}\right\}^{\downarrow}\cdot s) &=& \psi(\left\{x_{1}\cdot s,\ldots,x_{m}\cdot s\right\}^{\downarrow}) = \vee_{i=1}^{m}{\theta(x_{i}\cdot s)}\\
&=& (\vee_{i=1}^{m}{\theta(x_{i})})\cdot s = \psi(\left\{x_{1},\ldots,x_{m}\right\}^{\downarrow})\cdot s
\end{eqnarray*}
and
\begin{eqnarray*}
q(\psi(\left\{x_{1},\ldots,x_{m}\right\}^{\downarrow})) &=& q(\vee_{i=1}^{m}{\theta(x_{i})}) = \vee_{i=1}^{m}{q(\theta(x_{i}))} \\
&=& \vee_{i=1}^{m}{p(x_{i})} = p^{\sharp}(\left\{x_{1},\ldots,x_{m}\right\}^{\downarrow}).
\end{eqnarray*}
It is pointed since $\psi(0_{X^{\sharp}}) = \theta(0_{X}) = 0_{Y}$. Finally let $A = \left\{x_{1},\ldots,x_{m}\right\}^{\downarrow}$, $B = \left\{y_{1},\ldots,y_{n}\right\}^{\downarrow}$ be orthogonal. Then
$$\psi(A\vee B) = (\vee_{i=1}^{m}{\theta(x_{i})})\bigvee (\vee_{j=1}^{n}{\theta(y_{j})}) = \psi(A)\vee \psi(B).$$

We claim that $(X^{\sharp},\iota)$ is a reflection of $X$ along the forgetful functor $F:\Mod_{S}\rightarrow \Premod_{S}$, and that $\psi$ will be the unique map such that $\psi\iota = \theta$ for $\theta:X\rightarrow Y$ a premodule morphism to a module.

Let $x\in X$. Then $\psi(\iota(x)) = \psi(x^{\downarrow}) = \theta(x)$ and so $\psi\iota = \theta$.

Let $X$ be a premodule, $Y$ a module and let $\theta:X\rightarrow Y$ be a premodule morphism. Suppose that $\pi:X^{\sharp}\rightarrow Y$ is a module morphism with $\pi\iota = \theta$. We claim that $\pi = \psi$.

Let $x\in X$. Then $\pi(x^{\downarrow}) = \pi\iota(x) = \theta(x)$. Now suppose $x,y\in X$ with $x\perp y$. Then $x^{\downarrow}\perp y^{\downarrow}$ so that $x^{\downarrow}\vee y^{\downarrow} = \left\{x,y\right\}^{\downarrow}\in X^{\sharp}$. Thus
$$\pi(\left\{x,y\right\}^{\downarrow}) = \pi(x^{\downarrow}) \vee \pi(y^{\downarrow}) = \theta(x)\vee \theta(y) = \psi(\left\{x,y\right\}^{\downarrow}).$$
It therefore follows by induction that $\pi = \psi$. Thus $(X^{\sharp},\iota)$ is a reflection of $X$ along the forgetful functor.
% $F:\Mod_{S}\rightarrow \Premod_{S}$. Let $R:\Premod_{S}\rightarrow \Mod_{S}$ be given by $R(X) = X^{\sharp}$, $R(\alpha) = \alpha^{\sharp}$. 
Define a natural transformation 
$$\eta:1_{\Premod_{S}} \rightarrow F\circ R$$
by $\eta_{X}(x) = x^{\downarrow}$ for $X$ a premodule and $x\in X$.
This is a natural transformation since if $\theta:X\rightarrow Y$ is a premodule morphism, then 
$$(F\circ R)(\theta)(\eta_{X}(x)) = (F\circ R)(\theta)(x^{\downarrow}) = \left\{\theta(x)\right\}^{\downarrow} = \eta_{Y}(1_{\Premod_{S}}(\theta)(x)).$$
\end{proof}

Let $(X,p)$ be a premodule and let 
$$xS = \left\{x\cdot s| s\in S\right\}.$$
Then $(xS,p)$ naturally inherits the structure of a pointed \'{e}tale set. In fact:

\begin{lemma}
Let $(X,p)$ be a premodule. Then $(xS,p)$ is a premodule. 
\end{lemma}

\begin{proof}
Suppose that $x\cdot s, x\cdot t$ are strongly orthogonal in $xS$ with $xs\vee xt = xu$ for some $u\in S$ (note that $x\cdot s$ and $x\cdot t$ might be strongly orthogonal in $X$ without being strongly orthogonal in $xS$). Let $v\in S$. Then 
$$p(xsv)p(xtv) = v^{-1}p(xs)vv^{-1}p(xt)v = 0$$
and 
\begin{eqnarray*}
p(xuv) &=& v^{-1}p(xu)v = v^{-1}(p(xs)\vee p(xt))v \\
&=& (v^{-1}p(xs)v)\vee (v^{-1}p(xt)v) = p(xsv)\vee p(xtv).
\end{eqnarray*}
Further, $xuv\geq xsv, xtv$. Thus by Lemma \ref{usefuljoinlem}, and the fact that $X$ is a premodule, $xsv\vee xtv = xuv$ in $X$ and thus also in $xS$. Now suppose $s,t\in S$ are orthogonal. Then $x\cdot s, x\cdot t$ are strongly orthogonal in $X$ with $xs\vee xt = x(s\vee t)\in xS$. Thus $xS$ is a premodule.
\end{proof}

We will therefore call $xS$ the \emph{cyclic premodule generated by the element} $x\in X$, where $X$ is a premodule.

\begin{lemma}
\label{cyclicplem}
Let $(xS,p)$ be a cyclic premodule. Then the map $\theta:p(x)S \rightarrow xS$ given by $\theta(s) = xs$ is a surjective premodule morphism.
\end{lemma}

\begin{proof}
Let $q:p(x)S\rightarrow E(S)$ be given by $q(s) = s^{-1}s$. We prove first that $\theta$ is a pointed morphism:

\begin{description}
\item[{\rm (EM1)}] $\theta(s\cdot t) = \theta(st) = x\cdot (st) =  (x\cdot s)\cdot t = \theta(s)\cdot t$.
\item[{\rm (EM2)}] $p(\theta(s)) = p(x\cdot s) = s^{-1}p(x)s = s^{-1}s = q(s)$.
\item[{\rm (PM)}] $\theta(p(x)\cdot 0) = x\cdot 0 = 0$.
\end{description}

Next we prove surjectivity. Let $x\cdot s\in xS$. Then $\theta(p(x)s) = x\cdot s$.

Let us now check that $\theta$ is a premodule morphism. Let $s = p(x)s$, $t = p(x)t$ be strongly orthogonal in $p(x)S$. Then $s,t$ are orthogonal in $S$ and so $x\cdot s$ and $x\cdot t$ are strongly orthogonal in $xS$ with $x\cdot(s\vee t) = xs\vee xt$ and so $\theta(s\vee t) = \theta(s)\vee \theta(t)$.
\end{proof}

\begin{lemma}
\label{sharpmod}
Let $(X,p)$ be a module and let $x\in X$. Define $f_{x}:(xS)^{\sharp}\rightarrow X$ by 
$$f_{x}(\left\{xs_{1},\ldots,xs_{m}\right\}^{\downarrow}) = \bigvee_{i=1}^{m}{xs_{i}}.$$
Then $f_{x}$ is a (well-defined) module morphism. 
\end{lemma}

\begin{proof}
It is well-defined since $xs_{1},\ldots,xs_{m}$ are orthogonal and $X$ is a module, so the join exists, and if $A,B\in \overline{xS}$ with $A\equiv B$, then $f_{x}(A) = f_{x}(B)$. It is an \'{e}tale morphism since
$$f_{x}(\left\{xs_{1},\ldots,xs_{m}\right\}^{\downarrow}\cdot t) = \vee_{i=1}^{m}{xs_{i}t} = (\vee_{i=1}^{m}{xs_{i}})\cdot t = f_{x}(\left\{xs_{1},\ldots,xs_{m}\right\}^{\downarrow})\cdot t$$
and
$$p(f_{x}(\left\{xs_{1},\ldots,xs_{m}\right\}^{\downarrow})) = p(\vee_{i=1}^{m}{xs_{i}}) = \vee_{i=1}^{m}{p(xs_{i})} = p^{\sharp}(\left\{xs_{1},\ldots,xs_{m}\right\}^{\downarrow}).$$
It is obviously pointed. Let us now check that it is a module morphism. Let $A = \left\{xs_{1},\ldots,xs_{m}\right\}^{\downarrow}, B = \left\{xt_{1},\ldots,xt_{n}\right\}^{\downarrow}\in (xS)^{\sharp}$ be orthogonal. Then
$$f_{x}(A\vee B) = f_{x}(\left\{xs_{1},\ldots,xs_{m},xt_{1},\ldots,xt_{n}\right\}^{\downarrow}) = (\vee_{i=1}^{m}{xs_{i}})\bigvee (\vee_{i=1}^{n}{xt_{i}}) = f_{x}(A)\vee f_{x}(B).$$
\end{proof}

\begin{comment}
Now suppose that $(X,p)$ be a module such that $X = xS$ for some $x\in X$. We say that $X = xS$ is a \emph{cyclic module} generated $x$. 
\end{comment}

\begin{comment}
\begin{lemma}
Let $(X,p)$ be a module. Then the map $f:\overline{X}\rightarrow X$ given by
$$f(\left\{x_{1},\ldots,x_{m}\right\}^{\downarrow}) = \vee_{i=1}^{m}{x_{i}}$$
is a surjective module morphism.
\end{lemma}

\begin{proof}
(EM1): $f(\left\{x_{1},\ldots,x_{m}\right\}^{\downarrow}\cdot s) = f(\left\{x_{1}\cdot s,\ldots,x_{m}\cdot s\right\}^{\downarrow}) = \vee_{i=1}^{m}{x_{i}\cdot s} = (\vee_{i=1}^{m}{x_{i}})\cdot s = f(\left\{x_{1},\ldots,x_{m}\right\}^{\downarrow})\cdot s$.

(EM2): $p(f(\left\{x_{1},\ldots,x_{m}\right\}^{\downarrow})) = p(\vee_{i=1}^{m}{x_{i}\cdot s}) = \vee_{i=1}^{m}{p(x_{i})\cdot s} = \overline{p}(\left\{x_{1},\ldots,x_{m}\right\}^{\downarrow})$.

(PM): This is clear

(MM): Let $\left\{x_{1},\ldots,x_{m}\right\}^{\downarrow}\perp \left\{y_{1},\ldots,y_{n}\right\}^{\downarrow}$. Then
$$f(\left\{x_{1},\ldots,x_{m}\right\}^{\downarrow}\vee \left\{y_{1},\ldots,y_{n}\right\}^{\downarrow}) = (\vee_{i=1}^{m}{f(x_{i})\cdot s}) \vee (\vee_{j=1}^{n}{f(y_{j})\cdot s}) $$
$$= f(\left\{x_{1},\ldots,x_{m}\right\}^{\downarrow})\vee f(\left\{y_{1},\ldots,y_{n}\right\}^{\downarrow}).$$

Surjective: Let $x\in X$. Then $f(x^{\downarrow}) = x$.
\end{proof}
\end{comment}

Observe that $\Mod_{S}$ is a concrete category and so we will denote the underlying set of a module $X$ by $[X]$ if we want to view it as an object in $\Set$. It is clear that every injective module morphism will be monic. As for modules over rings, it turns out the converse is also true.

\begin{lemma}
In $\Mod_{S}$ every monomorphism is injective.
\end{lemma}

\begin{proof}
Let $(X,p)$ and $(Y,q)$ be modules and let $\alpha:X\rightarrow Y$ be a monomorphism. Suppose that $\alpha(x) = \alpha(y)$ where $x,y\in X$. Observe that $p(x) = p(y)$. By Lemmas \ref{cyclicplem}, \ref{alphasharp} and \ref{sharpmod} there are surjective module morphisms $\beta^{\sharp}:(p(x)S)^{\sharp}\rightarrow (xS)^{\sharp}$ and $\gamma^{\sharp}:(p(x)S)^{\sharp}\rightarrow (yS)^{\sharp}$, and module morphisms $f_{x}:(xS)^{\sharp}\rightarrow X$ and $f_{y}:(yS)^{\sharp}\rightarrow X$. We have that
\begin{eqnarray*}
(\alpha f_{x} \beta^{\sharp})(\left\{s_{1},\ldots,s_{m}\right\}^{\downarrow}) &=& 
(\alpha f_{x})(\left\{x\cdot s_{1},\ldots,x\cdot s_{m}\right\}^{\downarrow}) = \alpha(\vee_{i=1}^{m}{x\cdot s_{i}}) \\
&=& \vee_{i=1}^{m}{\alpha(x)\cdot s_{i}} = \vee_{i=1}^{m}{\alpha(y)\cdot s_{i}} = (\alpha f_{y} \gamma^{\sharp})(\left\{s_{1},\ldots,s_{m}\right\}^{\downarrow}).
\end{eqnarray*}
Thus $\alpha f_{x} \beta^{\sharp} = \alpha f_{y} \gamma^{\sharp}$. 
Since $\alpha$ is monic, $f_{x} \beta^{\sharp} = f_{y} \gamma^{\sharp}$. But 
$$(f_{x} \beta^{\sharp})(p(x)^{\downarrow}) = f_{x}(x^{\downarrow}) = x$$
and
$$(f_{y} \gamma^{\sharp})(p(x)^{\downarrow}) = f_{y}(y^{\downarrow}) = y.$$
Thus $x = y$ and so $\alpha$ is injective.
\end{proof}

The one element set $\left\{z\right\}$ is a module when we define $z\cdot s = z$ for all $s\in S$ and $p(z) = 0$. This is an initial object in $\Mod_{S}$ but not a terminal object because of condition (EM2). 

\begin{comment}
\begin{lemma}
\label{freemod}
The one element module is the only free object in $\Mod_{S}$.
\end{lemma}

\begin{proof}
Suppose $(F,p)$ is a free module over $S$. This means that there is a set $I$ and map $\sigma:I\rightarrow [F]$ such that for any module $(X,q)$ and map $f:I\rightarrow [X]$ there is a unique morphism $g:F\rightarrow X$ in $\Mod_{S}$ such that $g\sigma = f$ in $\Set$. 
Suppose first that $I$ is a non-empty set and let $\sigma:I\rightarrow [F]$ be the associated map. If $0_{F}\in \im(\sigma)$ then if $f:I\rightarrow [X]$ is any function with $0_{X}\notin \im(f)$ then it is clear that there is no morphism $g$ completing the triangle in $\Set$ since $g(0_{F})\notin \im(f)$. On the other hand if $0_{F}$ is not in $\im(\sigma)$, we can pick an $(X,p)$ and $f:I\rightarrow [X]$ with $0_{X}\in\im(f)$ and so we again meet a contradiction. Thus $I$ must be empty and $\sigma:I\rightarrow [F]$ is the empty map in $\Set$. Let $(\left\{z\right\},q)$ be the one-element module in $\Mod_{S}$ and let $f:\emptyset\rightarrow \left\{z\right\}$ be the empty map. Then the only module with a morphism to the one element module is the one element module, so $(F,p)$ must be $(\left\{z\right\},q)$. Consequently, $(\left\{z\right\},q)$ is the only free module in $\Mod_{S}$ with map $\sigma:\emptyset\rightarrow \left\{z\right\}$.
\end{proof}
\end{comment}

We will now define a coproduct in $\Mod_{S}$. Let $(X,p)$, $(Y,q)$ be modules. Define $X\bigoplus Y$ to be the subset of $X\times Y$ consisting of all those pairs $(x,y)$ such that $p(x)q(y) = 0$. If $(x,y)\in X\bigoplus Y$ then define $(p\oplus q)(x,y) = p(x)\vee q(y)$. This makes sense since $p(x)q(y) = 0$ and so the orthogonal join $p(x)\vee q(y)$ exists. We define an action $X\bigoplus Y \times S \rightarrow X\bigoplus Y$ by $(x,y)\cdot s = (x\cdot s, y\cdot s)$. This is well-defined since $p(x\cdot s)q(y\cdot s) = s^{-1}p(x)ss^{-1}q(y)s = s^{-1}p(x)q(y)s = 0$ for $(x,y)\in X\bigoplus Y$.

\begin{lemma}
$(X\bigoplus Y, p\oplus q)$ is a module.
\end{lemma}

\begin{proof}
\begin{description}
\item[{\rm (E1)}] $(x,y)\cdot (p\oplus q)(x,y) = (x,y)\cdot (p(x)\vee q(y)) = (x,y)$.
\item[{\rm (E2)}] $(p\oplus q)(x\cdot s,y\cdot s) = p(x\cdot s)\vee q(y\cdot s) = s^{-1}(p(x)\vee q(y))s = s^{-1}(p\oplus q)(x,y)s$.
\item[{\rm (P1) - (P3)}] These are clear since $p(x)\vee q(y)\geq p(x),q(y)$ (here $(0,0)$ is the zero).
\item[{\rm (M1)}] If $(p\oplus q)(x,y)(p\oplus q)(w,z) = 0$ then $p(x)p(w) = 0$ and $q(y)q(z) = 0$. Thus there exists $(x\vee w, y\vee z)$. Further
$$(x\vee w, y\vee z)\cdot (p\oplus q)(x,y) = (x,y)$$
and
$$(x\vee w, y\vee z)\cdot (p\oplus q)(w,z) = (w,z)$$
so $(x,y),(w,z)\leq (x\vee w, y\vee z)$. Now suppose that $(u,v)\in X\bigoplus Y$ is such that $(x,y),(w,z)\leq (u,v)$. Then $u\cdot (p(x)\vee q(y)) = x$, $v\cdot (p(x)\vee q(y)) = y$, $u\cdot (p(w)\vee q(z)) = w$ and $v\cdot (p(w)\vee q(z)) = z$. Thus
$$u\cdot((p\oplus q)(x\vee w, y\vee z)) = u\cdot (p(x)\vee p(w) \vee q(y)\vee q(z)) = x\vee w.$$
Similarly $v \cdot((p\oplus q)(x\vee w, y\vee z)) = y\vee z$. So $(x\vee w, y\vee z) = (x,y)\vee (w,z)$ and 
$$(p\oplus q)(x\vee w, y\vee z) = p(x)\vee p(w) \vee q(y)\vee q(z).$$
\item[{\rm (M2)}] For $(x,y)\perp (w,z)$ we have 
$$(x\vee w, y\vee z)\cdot s = ((x\vee w)\cdot s, (y\vee z)\cdot s) = (x\cdot s, y\cdot s)\vee (w\cdot s, z\cdot s).$$
\end{description}
\end{proof}

Let $(X,p),(Y,q),(Z,r)$ be modules and suppose that $f:X\rightarrow Z$ and $g:Y\rightarrow Z$ are module morphisms. Then we can define a map
$$f\oplus g:X\bigoplus Y\rightarrow Z$$
by $(f\oplus g)(x,y) = f(x)\vee g(y)$. Note that this makes sense since $r(f(x))r(g(y)) = p(x)q(y) = 0$. In fact:

\begin{lemma}
\label{mapsum}
With $X,Y,Z,f,g$ as above, $f\oplus g$ is a module morphism.
\end{lemma}

\begin{proof}
It is an \'{e}tale morphism since
$$(f\oplus g)(x\cdot s, y\cdot s) = f(x\cdot s)\vee g(y\cdot s) = (f\oplus g)(x,y)\cdot s$$
and
$$r((f\oplus g)(x,y)) = r(f(x)\vee g(y)) = r(f(x)) \vee r(g(y)) = p(x)\vee q(y) = (p\oplus q)(x,y).$$
It is pointed since $(f\oplus g)(0,0) = p(0)\vee q(0) = 0$ and for $x\neq 0, y\neq 0$, we have $(f\oplus g)(x,y) \neq 0$. 
Finally, to check that it is a module morphism, suppose $(p\oplus q)(x,y)(p\oplus q)(w,z) = 0$. Then by the above $(x,y)\vee (w,z) = (x\vee w,y\vee z)$ and
\begin{eqnarray*}
(f\oplus g)(x\vee w,y\vee z) &=& f(x\vee w)\vee g(y\vee z) = f(x)\vee f(w)\vee g(y)\vee g(z) \\
&=& (f\oplus g)(x,y)\vee (f\oplus g)(w,z).
\end{eqnarray*}
\end{proof}

Define $\iota_{1}:X\rightarrow X\bigoplus Y$ by $\iota_{1}(x) = (x,0)$ and $\iota_{2}:Y\rightarrow X\bigoplus Y$ by $\iota_{2}(y) = (0,y)$. It is easy to see that $\iota_{1}$ and $\iota_{2}$ are module morphisms.

\begin{lemma}
$(X\bigoplus Y, \iota_{1}, \iota_{2})$ is a coproduct in $\Mod_{S}$.
\end{lemma}

\begin{proof}
We need to show that if $\alpha:X\rightarrow Z$, $\beta:Y\rightarrow Z$ are module morphisms to a module $(Z,r)$ then there exists a unique module morphism $\gamma:X\bigoplus Y\rightarrow Z$ with $\alpha = \gamma\iota_{1}$ and $\beta = \gamma\iota_{2}$. We claim $\gamma = \alpha\oplus \beta$.

Firstly, 
$$(\alpha\oplus\beta)(\iota_{1}(x)) = (\alpha\oplus\beta)(x,0) = \alpha(x)$$
and
$$(\alpha\oplus\beta)(\iota_{2}(y)) = (\alpha\oplus\beta)(0,y) = \beta(y).$$

Now suppose $\delta:X\bigoplus Y\rightarrow Z$ is a module morphism with $\alpha = \delta\iota_{1}$ and $\beta = \delta\iota_{2}$. Then $\delta(\iota_{1}(x)) = \alpha(x)$ and so $\delta(x,0) = \alpha(x)$. Similarly $\delta(0,y) = \beta(y)$. For $p(x)q(y) = 0$ we have $(x,0)\perp (0,y)$ and $(x,0)\vee (0,y) = (x,y)$. Thus
$$\delta(x,y) = \delta(x,0) \vee \delta(0,y) = \alpha(x) \vee \beta(y) = (\alpha\oplus\beta)(x,y).$$
\end{proof}

We may define coproducts of an arbitrary set of modules $\left\{X_{i}:i\in I\right\}$ by considering those elements of the direct product $\times_{i\in I}{X_{i}}$ which have only a finite number of non-zero elements.

\begin{lemma}
\label{pushoutslem}
Let $(X,p)$, $(Y_{1},q_{1})$, $(Y_{2},q_{2})$ be modules, $f_{1}:X\rightarrow Y_{1}$ and $f_{2}:X\rightarrow Y_{2}$ module morphisms and suppose that $\ker(f_{1}) = \ker(f_{2})$. Then there exists a pushout of $f_{1}$ and $f_{2}$.
\end{lemma}

\begin{proof}
Let $(X,p)$, $(Y_{1},q_{1})$, $(Y_{2},q_{2})$ be modules, $f_{1}:X\rightarrow Y_{1}$ and $f_{2}:X\rightarrow Y_{2}$ module morphisms with $\ker(f_{1}) = \ker(f_{2})$. Define a binary relation $\sigma$ on $Y_{1}\bigoplus Y_{2}$ by 
$$(a_{1},b_{1})\,\sigma\, (a_{2},b_{2})$$
if there exist $x_{1},x_{2}\in X$, $y_{1}\in Y_{1}$, $y_{2}\in Y_{2}$ with $f_{i}(x_{j})\perp y_{i}$ for $i,j = 1,2$,
$$(a_{1},b_{1}) = (y_{1}\vee f_{1}(x_{1}), y_{2}\vee f_{2}(x_{2}))$$
and
$$(a_{2},b_{2}) = (y_{1}\vee f_{1}(x_{2}), y_{2}\vee f_{2}(x_{1})).$$
We prove that $\sigma$ is a congruence. It is clear that $\sigma$ is reflective and symmetric. Let us check transitivity. 
Suppose $(a_{1},b_{1}),(a_{2},b_{2}),(a_{3},b_{3})\in Y_{1}\bigoplus Y_{2}$ are elements with $(a_{1},b_{1})\,\sigma\, (a_{2},b_{2})$ and $(a_{2},b_{2})\,\sigma\, (a_{3},b_{3})$. Let $x_{1},x_{2},x_{3},x_{4}\in X$, $y_{1},z_{1}\in Y_{1}$, $y_{2},z_{2}\in Y_{2}$ be such that $f_{i}(x_{j})\perp y_{i}$ for $i,j = 1,2$, $f_{i}(x_{j})\perp z_{i}$ for $i = 1,2$, $j=3,4$ and 
$$(a_{1},b_{1}) = (y_{1}\vee f_{1}(x_{1}), y_{2}\vee f_{2}(x_{2})),$$
$$(a_{2},b_{2}) = (y_{1}\vee f_{1}(x_{2}), y_{2}\vee f_{2}(x_{1})) = (z_{1}\vee f_{1}(x_{3}), z_{2}\vee f_{2}(x_{4}))$$
and
$$(a_{3},b_{3}) = (z_{1}\vee f_{1}(x_{4}), z_{2}\vee f_{2}(x_{3})),$$
so that $y_{1}\vee f_{1}(x_{2}) = z_{1}\vee f_{1}(x_{3})$ and $y_{2}\vee f_{2}(x_{1}) = z_{2}\vee f_{2}(x_{4})$.
Define 
$$u_{1} = y_{1}\cdot q_{1}(z_{1}) \vee f_{1}(x_{1}\cdot p(x_{4})),$$
$$u_{2} = f_{2}(x_{2}\cdot p(x_{3})) \vee y_{2}\cdot q_{2}(z_{2}),$$
$$v_{1} = x_{2}\cdot q_{1}(z_{1}) \vee x_{4}\cdot q_{2}(y_{2})$$
and
$$v_{2} = x_{3}\cdot q_{1}(y_{1}) \vee x_{1}\cdot q_{2}(z_{2}),$$
where each of the joins is the join of two orthogonal elements of modules.
Then $u_{1}\in Y_{1}$, $u_{2}\in Y_{2}$, $v_{1},v_{2}\in X$ are such that $f_{i}(v_{j})\perp u_{i}$ for $i,j = 1,2$,
$$(a_{1},b_{1}) = (u_{1}\vee f_{1}(v_{2}), u_{2}\vee f_{2}(v_{1}))$$
and
$$(a_{3},b_{3}) = (u_{1}\vee f_{1}(v_{1}), u_{2}\vee f_{2}(v_{2})).$$
Thus $(a_{1},b_{1})\, \sigma\, (a_{3},b_{3})$ and so $\sigma$ is transitive.
It is clear that axioms (C1) and (C2) for a congruence hold. Let us check (C3). Suppose $(a_{1},b_{1}),(a_{2},b_{2}),(c_{1},d_{1}),(c_{2},d_{2})\in Y_{1}\bigoplus Y_{2}$ are elements with $(a_{i},b_{i})\perp (c_{i},d_{i})$, $i=1,2$, $(a_{1},b_{1})\,\sigma\, (a_{2},b_{2})$ and $(c_{1},d_{1})\,\sigma\,(c_{2},d_{2})$. Let $x_{1},x_{2},x_{3},x_{4}\in X$, $y_{1},z_{1}\in Y_{1}$, $y_{2},z_{2}\in Y_{2}$ be such that $f_{i}(x_{j})\perp y_{i}$ for $i,j = 1,2$, $f_{i}(x_{j})\perp z_{i}$ for $i = 1,2$, $j=3,4$ and
$$(a_{1},b_{1}) = (y_{1}\vee f_{1}(x_{1}), y_{2}\vee f_{2}(x_{2})),$$
$$(a_{2},b_{2}) = (y_{1}\vee f_{1}(x_{2}), y_{2}\vee f_{2}(x_{1})),$$
$$(c_{1},d_{1}) = (z_{1}\vee f_{1}(x_{3}), z_{2}\vee f_{2}(x_{4}))$$
and
$$(c_{2},d_{2}) = (z_{1}\vee f_{1}(x_{4}), z_{2}\vee f_{2}(x_{3})).$$
Let $u_{1} = y_{1}\vee z_{1}$, $u_{2} = y_{2}\vee z_{2}$, $v_{1} = x_{1}\vee x_{3}$, $v_{2} = x_{2}\vee x_{4}$. Then $u_{1}\in Y_{1}$, $u_{2}\in Y_{2}$, $v_{1},v_{2}\in X$ are such that $f_{i}(v_{j})\perp u_{i}$ for $i,j = 1,2$,
$$(a_{1}\vee c_{1},b_{1}\vee d_{1}) = (u_{1}\vee f_{1}(v_{1}), u_{2}\vee f_{2}(v_{2}))$$
and
$$(a_{2}\vee c_{2},b_{2}\vee d_{2}) = (u_{1}\vee f_{1}(v_{2}), u_{2}\vee f_{2}(v_{1})).$$
Thus $\sigma$ is a congruence. Define $Z = (Y_{1}\bigoplus Y_{2})/\sigma$, denote elements by $[y_{1},y_{2}]$, define $k_{1}:Y_{1}\rightarrow Z$ by $k_{1}(y) = [y,0]$ and $k_{2}:Y_{2}\rightarrow Z$ by $k_{2}(y) = [0,y]$. It follows from Lemma \ref{projectioncong} that $k_{1},k_{2}$ are module morphisms. We claim $(Z,k_{1},k_{2})$ is the pushout of $X$. Firstly, 
$$k_{1}(f_{1}(x)) = [f_{1}(x),0] = [0,f_{2}(x)] = k_{2}(f_{2}(x)).$$
Now suppose $(Z^{\prime},r)$ is another module and  $g_{1}:Y_{1}\rightarrow Z^{\prime}$, $g_{2}:Y_{2}\rightarrow Z^{\prime}$ are module morphisms with $g_{1}f_{1} = g_{2}f_{2}$. 
Define $g: Z\rightarrow Z^{\prime}$ by $g([y_{1},y_{2}]) = g_{1}(y_{1})\vee g_{2}(y_{2})$. Let us verify that $g$ is well-defined. 
Suppose $(a_{1},b_{1})\,\sigma\, (a_{2},b_{2})$ and $x_{1},x_{2}\in X$, $y_{1}\in Y_{1}$, $y_{2}\in Y_{2}$ are such that $f_{i}(x_{j})\perp y_{i}$ for $i,j = 1,2$,
$$(a_{1},b_{1}) = (y_{1}\vee f_{1}(x_{1}), y_{2}\vee f_{2}(x_{2}))$$
and
$$(a_{2},b_{2}) = (y_{1}\vee f_{1}(x_{2}), y_{2}\vee f_{2}(x_{1})).$$
Then
\begin{eqnarray*}
g([a_{1},b_{1}]) &=& g_{1}(y_{1})\vee g_{1}(f_{1}(x_{1}))\vee g_{2}(y_{2})\vee g_{2}(f_{2}(x_{2})) \\
&=& g_{1}(y_{1})\vee g_{2}(f_{2}(x_{1}))\vee g_{2}(y_{2})\vee g_{1}(f_{1}(x_{2})) = g([a_{2},b_{2}]).
\end{eqnarray*}
We see that $g$ is an \'{e}tale morphism since 
$$g([y_{1},y_{2}]\cdot s) = g([y_{1}\cdot s,y_{2}\cdot s]) = g_{1}(y_{1}\cdot s)\vee g_{2}(y_{2}\cdot s) = (g_{1}(y_{1})\vee g_{2}(y_{2}))\cdot s$$
and
$$(q_{1}\oplus q_{2})(y_{1},y_{2}) = q_{1}(y_{1})\vee q_{2}(y_{2}) = r(g(y_{1}))\vee r(g(y_{2})).$$
It is obviously pointed and it is a module morphism by construction.
Furthermore, it is readily verified that $gk_{1} = g_{1}$ and $gk_{2} = g_{2}$. Uniqueness follows from the fact that if $h:Z\rightarrow Z^{\prime}$ is such that $hk_{1} = g_{1}$ and $hk_{2} = g_{2}$ then 
$$h([y_{1},y_{2}]) = h([y_{1},0]) \vee h([0,y_{2}]) = g_{1}(y_{1}) \vee g_{2}(y_{2}) = g([y_{1},y_{2}]).$$
\end{proof}

\begin{lemma}
In $\Mod_{S}$ every epimorphism is a surjection.
\end{lemma}

\begin{proof}
Let $(X,p)$, $(Y,q)$ be modules, let $\theta:X\rightarrow Y$ be a module epimorphism and let $Z$ be the categorical cokernel of $\theta$, i.e. the pushout of $\theta$ with itself as described in Lemma \ref{pushoutslem}. Explicitly, $Z = (Y\bigoplus Y)/\sigma$ where $(x,y)\,\sigma\, (u,v)$ if and only if there exist $x_{1},y_{1}\in Y$, $x_{2},y_{2}\in \im(\theta)$ with $x_{1}\perp x_{2}$, $x_{1}\perp y_{2}$, $y_{1}\perp x_{2}$, $y_{1}\perp y_{2}$, $(x,y) = (x_{1}\vee x_{2},y_{1}\vee y_{2})$ and $(u,v) = (x_{1}\vee y_{2}, y_{1}\vee x_{2})$.

Now assume $\theta$ is not surjective. We will reach a contradiction. Recall $k_{1},k_{2}:Y\rightarrow Z$ are given by $k_{1}(y) = [(y,0)]$, $k_{2}(y) = [(0,y)]$ and observe that $k_{1}(\theta(x)) = k_{2}(\theta(x))$.

We say that a pair $(x,y)\in Y\oplus Y$ is \emph{odd} if $x$ belongs to the image of $\theta$ and $y$ does not. We claim that if $(x,y)\,\sigma\, (u,v)$ then $(x,y)$ is odd if and only if $(u,v)$ is odd. Suppose that $(x,y)$ is odd. Let $(x,y) = (x_{1}\vee x_{2}, y_{1}\vee y_{2})$ and $(u,v) = (x_{1}\vee y_{2},y_{1}\vee x_{2})$, where $x_{2},y_{2}\in \im(\theta)$. Now if $x$ is in the image of $\theta$ then so too are both $x_{1}$ and $x_{2}$ since the image of a module morphism is an order ideal. By assumption, $y_{2}$ is in the image of $\theta$ and so $u$ is in the image of $\theta$. If $v$ were in the image of $\theta$ then so too would $y_{1}$ and $x_{2}$. But this would imply that $y$ was in the image. It follows that $(u,v)$ is odd. The reverse direction follows by symmetry.

\begin{comment}
If follows that $\stackrel{\ast}{\leftrightarrow}$ preserves oddness and so $\equiv$ preserves oddness.
\end{comment}

Let $y$ be an element of $Y$ that is not in the image of $\theta$. Then $(0,y)$ is odd and $(y,0)$ is not. If follows that $(0,y)$ and $(y,0)$ are not $\sigma$-related. We have therefore proved that $k_{1}\neq k_{2}$, a contradiction.
\end{proof}

Let $I$ be a set and let
$$F_{I} = (I \times (S\setminus \left\{0\right\})) \cup \left\{0\right\}.$$
Define $(i,s)\cdot t = (i,st)$ if $st \neq 0$, and $0$ otherwise. Also define $0 \cdot s = 0$ for all $s\in S$. Let $p:F_{I}\rightarrow E(S)$ be defined by $p(i,s) = s^{-1}s$ and $p(0) = 0$. Then this gives $F_{I}$ the structure of a premodule via Proposition \ref{rightidealpremod}. We will say a module $X$ is \emph{free} with respect to a set $I$ if there is a premodule morphism $\sigma: F_{I}\rightarrow X$ such that for any premodule morphism $f:F_{I}\rightarrow Y$ where $Y$ is a module there is a unique module morphism $g:X \rightarrow Y$ such that $g\sigma = f$ and such that if a module $X^{\prime}$ together with a map $\sigma^{\prime}:F_{I}\rightarrow X^{\prime}$ also satisfies these conditions then $X \cong X^{\prime}/\rho$ for some congruence $\rho$.

\begin{lemma}
Let $I$ be a non-empty set. The module 
$$\bigoplus_{i\in I}{S^{\sharp}}$$
is the unique (up to isomorphism) free module with respect to the set $I$.
\end{lemma}

\begin{proof}
Let $X = \bigoplus_{i\in I}{S^{\sharp}}$.
Define $\sigma: F_{I}\rightarrow X$ by 
$$\sigma(i,s) = (0,\ldots,0,s^{\downarrow},0,\ldots)$$
where $s$ is in the $i$th position of $X$.
This will be a premodule morphism essentially for the same reason as $\iota$ is in Proposition \ref{Radjointforget}.
Now suppose $Y$ is a module and $f:F_{I}\rightarrow Y$ is a premodule morphism.
Define $g:X\rightarrow Y$ by
$$g(\left\{s_{1,1},\ldots,s_{1,m_{1}}\right\}^{\downarrow}, \left\{s_{2,1},\ldots,s_{2,m_{2}}\right\}^{\downarrow}, \ldots) =
\bigvee_{k\in I}{\bigvee_{i=1}^{m_{k}}{f(k,s_{k,i})}}.$$
It is easy to see that $g$ will be a module morphism and that $g\sigma = f$.
Suppose that $X^{\prime}$ is a module such that $\sigma^{\prime}:F_{I}\rightarrow X^{\prime}$ also satisfies the above condition. Then there exists a module morphism $g:X^{\prime} \rightarrow X$ such that $g\sigma^{\prime} = \sigma$. Let 
$$x = (\left\{s_{1,1},\ldots,s_{1,m_{1}}\right\}^{\downarrow}, \left\{s_{2,1},\ldots,s_{2,m_{2}}\right\}^{\downarrow}, \ldots)\in X$$
be arbitrary. Then
$$x = g\big(\bigvee_{k\in I}{\bigvee_{i=1}^{m_{k}}{\sigma^{\prime}(k,s_{k,i})}}\big).$$
Thus $g$ is surjective and so by Lemma \ref{FITmods} we have $X\cong X^{\prime}/\ker(g)$. It is easy to see that $X$ will then be unique up to isomorphism.
\end{proof}

\begin{lemma}
\label{modcoeq}
$\Mod_{S}$ has all coequalisers.
\end{lemma}

\begin{proof}
Let $f_{1},f_{2}:(X,p)\rightarrow (Y,q)$ be two module morphisms. We will say $a\leftrightarrow b$ in $Y$ if there exist $x_{1},x_{2}\in X$ and $y\in Y$ such that 
$$p(x_{1})p(x_{2}) = p(x_{1})q(y) = p(x_{2})q(y) = 0,$$
$$a = f_{1}(x_{1})\vee f_{2}(x_{2})\vee y$$
and
$$b = f_{1}(x_{2})\vee f_{2}(x_{1})\vee y.$$
Note that this implies that $q(a) = q(b)$. Let $\sigma$ be the transitive closure of $\leftrightarrow$. We now show that $\sigma$ is a congruence on $Y$. It is easy to see that (C1) and (C2) hold, so we just check (C3). The key observation is that if $a \leftrightarrow b$, $c\leftrightarrow d$, $a\perp c$ and $b\perp d$ then $a\perp d$ and $c\perp b$. Suppose $x_{1},x_{2},x_{3},x_{4}\in X$ and $y_{1}y_{2}\in Y$ are such that 
$$p(x_{1})p(x_{2}) = p(x_{1})q(y_{1}) = p(x_{2})q(y_{1}) = 0,$$
$$p(x_{3})p(x_{4}) = p(x_{3})q(y_{2}) = p(x_{4})q(y_{2}) = 0,$$
$$a = f_{1}(x_{1})\vee f_{2}(x_{2})\vee y_{1},$$
$$b = f_{1}(x_{2})\vee f_{2}(x_{1})\vee y_{1},$$
$$c = f_{1}(x_{3})\vee f_{2}(x_{4})\vee y_{2}$$
and
$$d = f_{1}(x_{4})\vee f_{2}(x_{3})\vee y_{2}.$$
Then
$$a\vee c = (f_{1}(x_{1})\vee f_{2}(x_{2})\vee y_{1})\vee (f_{1}(x_{3})\vee f_{2}(x_{4})\vee y_{2}) = f_{1}(x_{1}\vee x_{3}) \vee f_{2}(x_{2}\vee x_{4})\vee (y_{1}\vee y_{2})$$
and
$$b\vee d = (f_{1}(x_{2})\vee f_{2}(x_{1})\vee y_{1})\vee (f_{1}(x_{4})\vee f_{2}(x_{3})\vee y_{2}) = f_{2}(x_{1}\vee x_{3}) \vee f_{1}(x_{2}\vee x_{4})\vee (y_{1}\vee y_{2}),$$
so that $a\vee c\leftrightarrow b\vee d$. It is then easy to see that the transitive closure, $\sigma$, of $\leftrightarrow$ will be a congruence. Let $K = Y / \sigma$ and let $k:Y\rightarrow K$ be the projection map, which we know by the preceding theory is a module morphism. Then by construction $kf_{1} = kf_{2}$. Now suppose $g:(Y,q)\rightarrow (Z,r)$ is a module morphism such that $gf_{1} = gf_{2}$. Let $a,b\in Y$ be $\leftrightarrow$-related and suppose that $x_{1},x_{2}\in X$ and $y\in Y$ are such that 
$$p(x_{1})p(x_{2}) = p(x_{1})q(y) = p(x_{2})q(y) = 0,$$
$$a = f_{1}(x_{1})\vee f_{2}(x_{2})\vee y$$
and
$$b = f_{1}(x_{2})\vee f_{2}(x_{1})\vee y.$$
Then
$$g(a) = g(f_{1}(x_{1}))\vee g(f_{2}(x_{2}))\vee g(y) = g(f_{2}(x_{1}))\vee g(f_{1}(x_{2}))\vee g(y) = g(b).$$
More generally, if $a\,\sigma\, b$ then $g(a) = g(b)$. Since $k$ is surjective, for each $c\in K$, $k^{-1}(c)$ is non-empty. We therefore define $g^{\prime}:K\rightarrow Z$ by 
$$g^{\prime}(k(a)) = g(a).$$
The preceding remarks tell us that this map is well-defined. It is easy to check that $g^{\prime}$ is a pointed \'{e}tale morphism. It is in fact a module morphism since if $k(a)\perp k(b)$ then 
\begin{eqnarray*}
g^{\prime}(k(a)\vee k(b)) = g^{\prime}(k(a\vee b)) = g(a\vee b) = g(a)\vee g(b) = g^{\prime}(k(a))\vee g^{\prime}(k(b)).
\end{eqnarray*}
We have $g^{\prime}k = g$ by construction, and this is the unique map satisfying these properties. Thus $(K,k)$ is the coequaliser of $f_{1},f_{2}$.
\end{proof}

It follows from Lemma \ref{modcoeq} and the fact that $\Mod_{S}$ allows arbitrary coproducts that $\Mod_{S}$ is cocomplete (and so in fact all pushouts exist, not just those of Lemma \ref{pushoutslem}). On the other hand, it is not complete as it is not possible to define a product on modules because of axiom (EM2) for module morphisms. Furthermore, not all pullbacks exist. For example, if $(X,p)$ and $(Y,q)$ are such that $|X|,|Y|>1$ then the maps $\iota_{1}:X\rightarrow X\bigoplus Y$ and $\iota_{2}:Y\rightarrow X\bigoplus Y$ will not have a pullback. We do have the following consolatory lemma:

\begin{lemma}
$\Mod_{S}$ has all equalisers.
\end{lemma}

\begin{proof}
Let $(X,p)$, $(Y,q)$ be modules, $f,g:X\rightarrow Y$ be module morphisms and let
$$K = \left\{x\in X|f(x) = g(x)\right\}.$$
Then $K$ has the structure of a module since $x_{1},x_{2}\in K$ with $x_{1}\perp x_{2}$ implies $x_{1}\vee x_{2}\in K$. It inherits the map $q:X\rightarrow E(S)$ from $X$. Define $\iota:K\rightarrow X$ to be the embedding map. This is readily seen to be a module monomorphism. Suppose $(Z,r)$ is a module and $h:Z\rightarrow X$ is a module morphism with $fh = gh$. Then this implies $\im(h) \subseteq K$ and so there is a module morphism $h^{\prime}:Z\rightarrow K$ with $h = \iota h^{\prime}$ and this morphism is unique by construction. Thus $(K,\iota)$ is the equaliser of $(f,g)$.
\end{proof}

A module $P$ is said to be \emph{projective} if for every module morphism $\pi:P\rightarrow Y$ and module epimorphism $\alpha:X\rightarrow Y$ there exists a module morphism $\beta:P\rightarrow X$ such that $\alpha\beta = \pi$. 
%It is clear from Lemma \ref{freemod} that we cannot hope to study projective modules as direct summands of free modules.

\begin{lemma}
Let $P_{1}, P_{2}$ be projective modules. Then $P_{1}\bigoplus P_{2}$ is projective.
\end{lemma}

\begin{proof}
Let $\pi:P_{1}\bigoplus P_{2}\rightarrow Y$ be a module morphism and $\alpha:X\rightarrow Y$ a module epimorphism. Define $\iota_{1}:P_{1}\rightarrow P_{1}\bigoplus P_{2}$ by $\iota_{1}(x) = (x,0)$ and define $\iota_{2}:P_{2}\rightarrow P_{1}\bigoplus P_{2}$ by $\iota_{2}(y) = (0,y)$. Then $\pi\iota_{1}:P_{1}\rightarrow Y$ and $\pi\iota_{2}:P_{2}\rightarrow Y$ are module morphisms and so there are maps $\beta_{1}:P_{1}\rightarrow X$ and $\beta_{2}:P_{2}\rightarrow X$ such that $\pi\iota_{1} = \alpha\beta_{1}$ and $\pi\iota_{2} = \alpha\beta_{2}$. Let $\gamma = \beta_{1}\oplus\beta_{2}:P_{1}\bigoplus P_{2}\rightarrow X$ so that $\gamma$ is given by $\gamma(x,y) = \beta_{1}(x)\vee \beta_{2}(y)$. We know from Lemma \ref{mapsum} that $\gamma$ is a module morphism. Further
$$(\alpha\gamma)(x,y) = \alpha(\beta_{1}(x)\vee \beta_{2}(y)) = \alpha(\beta_{1}(x))\vee\alpha(\beta_{2}(y)) = (\pi\iota_{1}(x))\vee (\pi\iota_{2}(y)) $$
$$= \pi(x,0)\vee \pi(0,y) = \pi(x,y).$$
\end{proof}

We can extend the previous lemma: if $P_{1},\ldots,P_{n}$ are projective modules then 
$$\bigoplus_{i=1}^{n}{P_{i}}$$
is projective.
The converse is also true:

\begin{lemma}
Let $P = P_{1}\bigoplus P_{2}$ be a projective module. Then $P_{1}$ and $P_{2}$ are projective.
\end{lemma}

\begin{proof}
We will prove $P_{1}$ is projective. The proof for $P_{2}$ is similar. Let $\pi:P_{1}\rightarrow Y$ be a module morphism and let $\alpha:X\rightarrow Y$ be a module epimorphism. Define $\pi^{\prime}:P_{1}\bigoplus P_{2}\rightarrow Y\bigoplus P_{2}$ by 
$$\pi^{\prime}(p_{1},p_{2}) = (\pi(p_{1}), p_{2})$$
and define $\alpha^{\prime}:X\bigoplus P_{2}\rightarrow Y\bigoplus P_{2}$ by
$$\alpha^{\prime}(x,p_{2}) = (\alpha(x),p_{2}).$$
It is easy to check that $\pi^{\prime}$ is a module morphism and $\alpha^{\prime}$ is a module epimorphism.
There is thus a module morphism $\beta^{\prime}:P_{1}\bigoplus P_{2}\rightarrow X\bigoplus P_{2}$ such that $\alpha^{\prime}\beta^{\prime} = \pi^{\prime}$. Denote by 
$$\beta^{\prime}(p_{1},p_{2}) = (\beta^{\prime}_{1}(p_{1},p_{2}),\beta^{\prime}_{2}(p_{1},p_{2})).$$
We thus have
\begin{eqnarray*}
(\pi(p_{1}),0) &=& \pi^{\prime}(p_{1},0) = \alpha^{\prime}(\beta^{\prime}(p_{1},0)) = \alpha^{\prime}(\beta^{\prime}_{1}(p_{1},0),\beta^{\prime}_{2}(p_{1},0)) \\
&=& (\alpha(\beta^{\prime}_{1}(p_{1},0)),\beta^{\prime}_{2}(p_{1},0)).
\end{eqnarray*}
It follows that $\beta^{\prime}_{2}(p_{1},0) = 0$. So define $\beta:P_{1}\rightarrow X$ by
$$\beta(p_{1}) = \beta^{\prime}_{1}(p_{1},0).$$
By the above this is a module morphism. Further
$$\alpha(\beta(p_{1})) = \alpha(\beta^{\prime}_{1}(p_{1},0)) = \pi(p_{1})$$
and so $\alpha\beta = \pi$.
\end{proof}

\begin{lemma}
$(eS)^{\sharp}$ is a projective module for each idempotent $e\in E(S)$.
\end{lemma}

\begin{proof}
Let $\pi:(eS)^{\sharp}\rightarrow Y$ be a module morphism, $\alpha:X\rightarrow Y$ a module epimorphism and $y\in \alpha^{-1}(\pi(e^{\downarrow}))$ be a fixed element. Define $\beta:(eS)^{\sharp}\rightarrow X$ by
$$\beta(\left\{s_{1},\ldots,s_{m}\right\}^{\downarrow}) = \bigvee_{i=1}^{m}{y\cdot s_{i}}.$$
If $s_{1},s_{2}\in S$ are orthogonal in $S$ then
$$\beta(\left\{s_{1},\ldots,s_{m}\right\}^{\downarrow}) = \bigvee_{i=1}^{m}{y\cdot s_{i}} = y\cdot(s_{1}\vee s_{2})\bigvee (\vee_{i=3}^{m}{y\cdot s_{i}}) = \beta(\left\{s_{1}\vee s_{2},s_{3}\ldots,s_{m}\right\}^{\downarrow}).$$
It follows that $\beta$ is well-defined. It is easy to see that it is a pointed \'{e}tale morphism. If $\left\{s_{1},\ldots,s_{m}\right\}^{\downarrow}, \left\{t_{1},\ldots,t_{n}\right\}^{\downarrow}\in (eS)^{\sharp}$ are orthogonal then
\begin{eqnarray*}
\beta(\left\{s_{1},\ldots,s_{m}\right\}^{\downarrow}\vee \left\{t_{1},\ldots,t_{n}\right\}^{\downarrow}) &=& \beta(\left\{s_{1},\ldots,s_{m},t_{1},\ldots,t_{n}\right\}^{\downarrow})\\
&=& (\vee_{i=1}^{m}{y\cdot s_{i}})\bigvee (\vee_{i=1}^{n}{y\cdot t_{i}})\\
&=& \beta(\left\{s_{1},\ldots,s_{m}\right\}^{\downarrow}) \vee \beta(\left\{t_{1},\ldots,t_{n}\right\}^{\downarrow}).
\end{eqnarray*}
Thus $\beta$ is a module morphism.
Further
\begin{eqnarray*}
(\alpha\beta)(\left\{s_{1},\ldots,s_{m}\right\}^{\downarrow}) &=& \alpha(\vee_{i=1}^{m}{y\cdot s_{i}}) = \vee_{i=1}^{m}{\alpha(y)\cdot s_{i}} = \vee_{i=1}^{m}{\pi(e^{\downarrow})\cdot s_{i}}\\ 
&=& \vee_{i=1}^{m}{\pi(s_{i}^{\downarrow})} = \pi(\vee_{i=1}^{m}{s_{i}^{\downarrow}}) = \pi(\left\{s_{1},\ldots,s_{m}\right\}^{\downarrow}).
\end{eqnarray*}
\end{proof}

The following will be used shortly:

\begin{lemma}
\label{premodisomod}
If $X$ and $Y$ are isomorphic as premodules then $X^{\sharp}$ and $Y^{\sharp}$ are isomorphic as modules.
\end{lemma}

\begin{proof}
Suppose $\theta:(X,p)\rightarrow (Y,q)$ is a bijective premodule morphism. Then $\theta^{\sharp}:X^{\sharp}\rightarrow Y^{\sharp}$ is a surjective module morphism by Lemma \ref{alphasharp}. 
We will now prove that if $\left\{x_{1},\ldots,x_{m}\right\}^{\downarrow}\in X^{\sharp}$ is such that none of the $x_{i}$'s are strongly orthogonal to each other then $\left\{\theta(x_{1}),\ldots,\theta(x_{m})\right\}^{\downarrow}\in X^{\sharp}$ is such that none of the $\theta(x_{i})$'s are strongly orthogonal to each other. Suppose on the contrary that $x_{1}\perp x_{2}\in X$ are such that they are not strongly orthogonal but such that $\theta(x_{1})$ and $\theta(x_{2})$ are strongly orthogonal. Since $\theta(x_{1})$ and $\theta(x_{2})$ are strongly orthogonal and $\theta$ is surjective there is a $z\in X$ with 
$\theta(z) = \theta(x_{1})\vee \theta(x_{2})$
and 
$p(z) = p(x_{1})\vee p(x_{2})$.
Let $y_{1} = z\cdot p(x_{1})$ and $y_{2} = z\cdot p(x_{2})$. Then $p(y_{1}) = p(x_{1})$ and $p(y_{2}) = p(x_{2})$. Thus by Lemma \ref{usefuljoinlem}, $z = y_{1}\vee y_{2}$ and $y_{1},y_{2}$ are strongly orthogonal. We then have
$$\theta(y_{1}) = \theta(z\cdot p(y_{1})) = \theta(z)\cdot p(x_{1}) = \theta(x_{1}).$$
Similarly, $\theta(y_{2}) = \theta(x_{2})$. Since $\theta$ is injective we must have $y_{1} = x_{1}$ and $y_{2} = x_{2}$, which implies $x_{1}$ and $x_{2}$ are strongly orthogonal, contradicting our original assumption.
Now suppose 
$$\theta^{\sharp}(\left\{x_{1},\ldots,x_{m}\right\}^{\downarrow}) = \theta^{\sharp}(\left\{y_{1},\ldots,y_{n}\right\}^{\downarrow})$$
with none of the $x_{i}$'s strongly orthogonal to each other and none of the $y_{i}$'s strongly orthogonal to each other.
Then we have 
$$\left\{\theta(x_{1}),\ldots,\theta(x_{m})\right\}^{\downarrow} \equiv \left\{\theta(y_{1}),\ldots,\theta(y_{n})\right\}^{\downarrow},$$
with none of the $\theta(x_{i})$'s strongly orthogonal to each other and none of the $\theta(y_{i})$'s strongly orthogonal to each other.
Thus $n = m$ and by the injectivity of $\theta$, 
$$\left\{x_{1},\ldots,x_{m}\right\}^{\downarrow} = \left\{y_{1},\ldots,y_{n}\right\}^{\downarrow}$$
and so $\theta^{\sharp}$ is injective.
\end{proof}

\begin{lemma}
Let $e,f\in E(S)$ be idempotents in $S$. Then $eS$ and $fS$ are isomorphic as premodules if and only if $e\,\mathcal{D}\, f$.
\end{lemma}

\begin{proof}
($\Leftarrow$) Suppose $e = s^{-1}s$ and $f = ss^{-1}$. Define $\theta: eS\rightarrow fS$ by $\theta(t) = st$. This a well-defined surjective premodule morphism by Lemma \ref{cyclicplem} since $sS$ is isomorphic to $fS$. Further, if $\theta(t) = \theta(u)$ then $st = su$ which implies $s^{-1}st = s^{-1}su$ and so $t = u$. Thus $\theta$ is also injective.

($\Rightarrow$) Suppose $\alpha:(eS,p)\rightarrow (fS,q)$ is a bijective premodule morphism with inverse $\beta$ and suppose $\alpha(e) = s = fse$
and $\beta(f) = t = etf.$
Then $s^{-1}s = q(s) = p(e) = e$ and $t^{-1}t = f$. Further
$$f = \alpha(t) = \alpha(e)\cdot t = st$$
and so 
$$s^{-1} = s^{-1}f = s^{-1}st = et = t.$$
Thus $f = ss^{-1}$ and $e\,\mathcal{D}\, f$.  
\end{proof}

\begin{lemma}
\label{isoidemmod}
Let $e,f\in E(S)$ be idempotents in $S$. Then $(eS)^{\sharp}$ and $(fS)^{\sharp}$ are isomorphic as modules if and only if $e\,\mathcal{D}\, f$.
\end{lemma}

\begin{proof}
($\Leftarrow$) $e\,\mathcal{D}\, f$ implies $eS$ and $fS$ are isomorphic as premodules and so by Lemma \ref{premodisomod} $(eS)^{\sharp}$ and $(fS)^{\sharp}$ are isomorphic as modules.

($\Rightarrow$) Suppose $\theta: (eS)^{\sharp}\rightarrow (fS)^{\sharp}$ and $\phi: (fS)^{\sharp}\rightarrow (eS)^{\sharp}$ are mutually inverse module isomorphisms and suppose $\theta(e^{\downarrow}) = \left\{s_{1},\ldots,s_{m}\right\}^{\downarrow}\in (fS)^{\sharp}$ and $\phi(f^{\downarrow}) = \left\{u_{1},\ldots,u_{n}\right\}^{\downarrow}\in (eS)^{\sharp}$. Then
$$f^{\downarrow} = \bigvee_{i=1}^{n}\theta(u_{i}^{\downarrow}) = \bigvee_{i=1}^{n}\theta(e^{\downarrow}) u_{i} = \bigvee_{i=1}^{n}\bigvee_{j=1}^{m}(s_{j}u_{i})^{\downarrow}.$$
Thus 
$$f = \bigvee_{i=1}^{n}\bigvee_{j=1}^{m}(s_{j}u_{i}),$$
with $s_{j}u_{i}$, $s_{r}u_{k}$ strongly orthogonal for all $i,j,r,k$. It then follows that $s_{j}u_{i} \perp s_{r}u_{k}$ in $S$ for all $i,j,r,k$. Similarly,
$$e = \bigvee_{i=1}^{n}\bigvee_{j=1}^{m}(u_{i}s_{j}),$$
with $u_{i}s_{j} \perp u_{k} s_{r}$ in $S$ for all $i,j,r,k$. Furthermore, $e,f\in E(S)$ implies that $u_{i}s_{j}, s_{j}u_{i}\in E(S)$.
Since $\theta$ and $\phi$ are module morphisms, we must also have
$$e = \bigvee_{i=1}^{m}{s_{i}^{-1}s_{i}}$$
and 
$$f = \bigvee_{i=1}^{n}{u_{i}^{-1}u_{i}}.$$
So postmultiplying $f$ by $u_{i}^{-1}$ gives
$$u_{i}^{-1} = fu_{i}^{-1} = \bigvee_{j=1}^{m}(s_{j}u_{i}u_{i}^{-1}).$$
Since $s_{j}u_{i}\in E(S)$, $u_{i}^{-1}s_{j}^{-1} = s_{j}u_{i}$. Thus
$$u_{i}^{-1} = \bigvee_{j=1}^{m}(u_{i}^{-1}s_{j}^{-1}u_{i}^{-1}).$$
So
$$u_{i}u_{i}^{-1} = \bigvee_{j=1}^{m}(u_{i}u_{i}^{-1}s_{j}^{-1}u_{i}^{-1}) = \bigvee_{j=1}^{m}(s_{j}^{-1}u_{i}^{-1}u_{i}u_{i}^{-1}),$$
giving
$s_{j}u_{i}u_{i}^{-1} = s_{j}s_{j}^{-1}u_{i}^{-1}$
and so
$s_{j}u_{i} = s_{j}s_{j}^{-1}u_{i}^{-1}u_{i}$
for all $i,j$. Thus
$$\bigvee_{i=1}^{n}{s_{j}u_{i}} = s_{j}s_{j}^{-1}.$$
Since $\vee_{i=1}^{n}{s_{j}u_{i}}$ is orthogonal in $S$ to $\vee_{i=1}^{n}{s_{k}u_{i}}$ for $k\neq j$, we have
$$s_{j}s_{j}^{-1}s_{k}s_{k}^{-1} = 0$$
for $j \neq k$. Thus $s_{j}\perp s_{k}$ and $u_{j}\perp u_{k}$ in $S$ for $j\neq k$ and so for some $s\in fS$ and $u\in eS$ we have
$\theta(e^{\downarrow}) = s^{\downarrow}$
and
$\phi(f^{\downarrow}) = u^{\downarrow}$.
We also have $e = s^{-1}s$ and $f = u^{-1}u$.
Now 
$$f^{\downarrow} = \theta(u^{\downarrow}) = \theta(e^{\downarrow})\cdot u = s^{\downarrow}\cdot u = (su)^{\downarrow}.$$
Thus $f = su$ and so $s^{-1} = s^{-1}f = s^{-1}su = eu = u$, giving $e\,\mathcal{D}\, f$.
\end{proof}

\begin{lemma}
\label{perpidemmod}
Let $e,f\in E(S)$ be such that $ef = 0$. Then $(eS)^{\sharp}\bigoplus (fS)^{\sharp}$ is isomorphic to $((e\vee f)S)^{\sharp}$.
\end{lemma}

\begin{proof}
Define a map $h:(eS)^{\sharp}\bigoplus (fS)^{\sharp}\rightarrow ((e\vee f)S)^{\sharp}$ by
$$h(\left\{s_{1},\ldots,s_{m}\right\}^{\downarrow},\left\{t_{1},\ldots,t_{n}\right\}^{\downarrow}) 
= \left\{s_{1},\ldots,s_{m}, t_{1},\ldots,t_{n}\right\}^{\downarrow}.$$
Then since $s_{i} = es_{i}$ and $t_{i} = ft_{i}$, using the orthogonality conditions we see this is a valid element of $((e\vee f)S)^{\sharp}$.
It is easy to see that $h$ is an injective module morphism.
Let us check that $h$ is surjective. Let $\left\{s_{1},\ldots,s_{m}\right\}^{\downarrow}\in ((e\vee f)S)^{\sharp}$. Then 
$$h(\left\{es_{1},\ldots,es_{m}\right\}^{\downarrow},\left\{fs_{1},\ldots,fs_{m}\right\}^{\downarrow}) = \left\{s_{1},\ldots,s_{m}\right\}^{\downarrow}.$$
\begin{comment}
Let $x = \left\{s_{1},\ldots,s_{m}\right\}^{\downarrow},\left\{t_{1},\ldots,t_{n}\right\}^{\downarrow}, y = \left\{u_{1},\ldots,u_{k}\right\}^{\downarrow},\left\{v_{1},\ldots,v_{l}\right\}^{\downarrow}\in (eS)^{\sharp}\bigoplus (fS)^{\sharp}$ be such that $h(x) = h(y)$. Then stuff.
\end{comment}
\end{proof}

We will now prove a couple of related results which we will use later. For $X = \bigoplus_{i=1}^{m}{(e_{i}S)^{\sharp}}$, denote by 
$$\mathbf{e_{i}} = (0,\ldots,0,e_{i}^{\downarrow},0,\ldots,0)$$
where the $e_{i}$ is in the $i$th position.

\begin{lemma}
\label{projmod1}
Let 
$$\theta:\bigoplus_{i=1}^{m}{(e_{i}S)^{\sharp}}\rightarrow \bigoplus_{i=1}^{n}{(f_{i}S)^{\sharp}}$$
be a module isomorphism. Then for each $i$, there exist $a_{ik}\in f_{k}S$ with
$$\theta(\mathbf{e_{i}}) = (a_{1i}^{\downarrow},\ldots,a_{ni}^{\downarrow}).$$
\end{lemma}

\begin{proof}
Let $\phi = \theta^{-1}$. Suppose
$$\theta(\mathbf{e_{i}}) = (\left\{a_{1i1},\ldots,a_{1ir_{i1}}\right\}^{\downarrow},\ldots,\left\{a_{ni1},\ldots,a_{nir_{in}}\right\}^{\downarrow})$$
and
$$\phi(\mathbf{f_{i}}) = (\left\{b_{1i1},\ldots,b_{1is_{i1}}\right\}^{\downarrow},\ldots,\left\{b_{mi1},\ldots,b_{mis_{im}}\right\}^{\downarrow}).$$
We will prove for all $i,j$ that $a_{ji1},\ldots,a_{jir_{ij}}$ are orthogonal elements of $S$ and the lemma will then follow. By constuction, $a_{jiu}a_{jiv}^{-1} = 0$ for $u\neq v$. Thus we just need to prove $a_{jiv}^{-1}a_{jiu} = 0$ for $u\neq v$. 
\begin{comment}
Firstly, by the $p$-preserving property of the module map, we must have
$$f_{i} = \bigvee_{k=1}^{m}\bigvee_{j=1}^{s_{im}}{b_{kij}^{-1}b_{kij}}.$$
\end{comment}
Since $\theta$ is a bijection, we have
\begin{eqnarray*}
\mathbf{e_{i}}^{\downarrow} & = & \phi(\left\{a_{1i1},\ldots,a_{1ir_{i1}}\right\}^{\downarrow},\ldots,\left\{a_{ni1},\ldots,a_{nir_{in}}\right\}^{\downarrow}) \\
  & = & \phi(\vee_{j=1}^{r_{i1}}{f_{1}^{\downarrow}\cdot a_{1ij}},\ldots,\vee_{j=1}^{r_{in}}{f_{n}^{\downarrow}\cdot a_{nij}})\\
  & = & \bigvee_{k=1}^{n}\bigvee_{j=1}^{r_{ik}}{\phi(\mathbf{f_{k}})\cdot a_{kij}} \\
  & = & \bigvee_{k=1}^{n}\bigvee_{j=1}^{r_{ik}}{(\left\{b_{1k1},\ldots,b_{1ks_{k1}}\right\}^{\downarrow},\ldots,\left\{b_{mk1},\ldots,b_{mks_{km}}\right\}^{\downarrow})\cdot a_{kij}}
\end{eqnarray*}
We therefore have
$$\bigvee_{k=1}^{n}\bigvee_{j=1}^{r_{ik}}\bigvee_{u=1}^{s_{ki}}{b_{iku}a_{kij}} = e_{i}$$
and
$$\bigvee_{k=1}^{n}\bigvee_{j=1}^{r_{ik}}\bigvee_{u=1}^{s_{kv}}{b_{vku}a_{kij}} = 0$$
for $v\neq i$. Since $a_{kij}a_{viw}^{-1} = 0$ unless $k = v$ and $j = w$, postmultiplying the first equation by $a_{kij}^{-1}$ gives
$$\bigvee_{u=1}^{s_{ki}}{b_{iku}a_{kij}a_{kij}^{-1}} = a_{kij}^{-1}.$$
It then follows that $b_{iku}a_{kij}\in E(S)$ for all $i,j,k,u$ and so $b_{iku}a_{kij} = a_{kij}^{-1}b_{iku}^{-1}$. Similarly, applying this argument for $f_{i}$ instead gives $a_{ikv}b_{kiu}\in E(S)$ for all $i,k,u,v$. Thus, for $j\neq v$, we have
\begin{eqnarray*}
a_{kij}^{-1}a_{kiv} & = & \bigvee_{u=1}^{s_{ki}}{b_{iku}a_{kij}a_{kij}^{-1}a_{kiv}}
 =  \bigvee_{u=1}^{s_{ki}}{a_{kij}^{-1}b_{iku}^{-1}a_{kij}^{-1}a_{kiv}}\\
& = & \bigvee_{u=1}^{s_{ki}}{a_{kij}^{-1}a_{kij}b_{iku}a_{kiv}}
 =  \bigvee_{u=1}^{s_{ki}}{b_{iku}a_{kiv}a_{kij}^{-1}a_{kij}} = 0.
\end{eqnarray*}
\end{proof}

Since $0\in S$, we can assume in such calculations that $m = n$, by letting some of the $e_{i}$'s be equal to $0$. 

\begin{lemma}
\label{projmod2}
Let 
$$\theta:\bigoplus_{i=1}^{m}{(e_{i}S)^{\sharp}}\rightarrow \bigoplus_{i=1}^{m}{(f_{i}S)^{\sharp}}$$
be a module isomorphism with
$$\theta(\mathbf{e_{i}}) = (a_{1i}^{\downarrow},\ldots,a_{mi}^{\downarrow})$$
and 
$$\theta^{-1}(\mathbf{f_{i}}) = (b_{1i}^{\downarrow},\ldots,b_{mi}^{\downarrow}),$$
Then 
\begin{enumerate}
\item For all $i$ and for $j\neq k$, we have $a_{ij}^{-1}a_{ik} = 0$.
\item For all $i,j$, we have $a_{ij} = b_{ji}^{-1}$.
\end{enumerate}
\end{lemma}

\begin{proof}
Firstly, we can assume that $a_{ik} = f_{i}a_{ik}e_{k}$ and $b_{ik} = e_{i}b_{ik}f_{k}$.
We have, 
\begin{eqnarray*}
\mathbf{f_{i}} & = & \theta(b_{1i}^{\downarrow},\ldots,b_{mi}^{\downarrow}) =  \bigvee_{k=1}^{m}\theta(\mathbf{e_{k}})\cdot b_{ki} \\
& = & (\vee_{k=1}^{m}a_{1k}b_{ki},\ldots,\vee_{k=1}^{m}a_{mk}b_{ki}).
\end{eqnarray*}
So $\vee_{k=1}^{m}a_{jk}b_{ki} = 0$ for $i\neq j$ and $\vee_{k=1}^{m}a_{ik}b_{ki} = f_{i}$. Postmultiplying by $b_{ki}^{-1}b_{ki}$ gives
$$a_{ik}b_{ki} = b_{ki}^{-1}b_{ki}.$$
A similar argument gives 
$$b_{ik}a_{ki} = a_{ki}^{-1}a_{ki}$$
and for $i\neq j$
$$b_{jk}a_{ki} = 0.$$
Let us now prove the claims:
\begin{enumerate}
\item Using the fact that $b_{ik}a_{ki} = a_{ki}^{-1}b_{ik}^{-1}$, $a_{ik}b_{ki} = b_{ki}^{-1}a_{ik}^{-1}$ and $a_{ki}^{-1} = b_{ik}a_{ki}a_{ki}^{-1}$, we have, for all $i$ and for $j\neq k$,
$$a_{ij}^{-1}a_{ik} = b_{ji}a_{ij}a_{ij}^{-1}a_{ik} = a_{ij}^{-1}b_{ji}^{-1}a_{ij}^{-1}a_{ik} = a_{ij}^{-1}a_{ij}b_{ji}a_{ik} = b_{ji}a_{ik}a_{ij}^{-1}a_{ij} = 0.$$
\item For all $i,j$, we have $b_{ij}a_{ji}a_{ji}^{-1} = a_{ji}^{-1}$ and so $a_{ji}^{-1}\leq b_{ij}$. On the other hand, $a_{ji}b_{ij}b_{ij}^{-1} = b_{ij}^{-1}$, giving $b_{ij}^{-1}\leq a_{ji}$ and therefore $b_{ij}\leq a_{ji}^{-1}$. Thus $b_{ij} = a_{ji}^{-1}$.
\end{enumerate}
\end{proof}

We denote the full subcategory of $\Mod_{S}$ consisting of all projective modules isomorphic to $\bigoplus_{i=1}^{n}{(e_{i}S)^{\sharp}}$ for some idempotents $e_{i}\in E(S)$ by $\Proj_{S}$.
By definition, $\bigoplus$ gives $(\Proj_{S},\bigoplus)$ the structure of a commutative monoid, where the identity is the one element module.
\begin{comment}
We denote the full subcategory of $\Mod_{S}$ consisting of all modules of the form
$$\bigoplus_{i=1}^{n}{(e_{i}S)^{\sharp}}$$
where the $e_{i}$ are idempotents in $S$ by $\IdemProj_{S}$. 
\end{comment}
For $S$ an arbitrary orthogonally complete inverse semigroup we will define $K(S)$ to be the Grothendieck group of $(\Proj_{S},\bigoplus)$. It is clear by Lemmas \ref{isoidemmod} and \ref{perpidemmod} that if $S$ is a $K$-inverse semigroup, this definition agrees with our earlier definition of a $K$-group.

\begin{remark}
It may in fact be possible to rework quite a bit of this section for the more general setting of orthogonally complete ordered groupoids. Here is a suggestion for one possible approach. Let $G$ be an orthogonally complete ordered groupoid. We will say $X$ is an \emph{\'{e}tale set} if there is a map $p:X\rightarrow G_{0}$ and a partially defined function $X\times G\rightarrow X$, denoted $(x,g)\mapsto x\cdot g$ such that
\begin{itemize}
\item For each $x\in X$ we have $\exists x\cdot p(x)$ and $x\cdot p(x) = x$.
\item For $x\in X$, $e\in G_{0}$, $\exists x\cdot e$ iff $e \wedge p(x) \neq 0$ in which case $x \cdot e = x \cdot (e\wedge p(x))$.
\item For $x\in X$, $g\in G$, $\exists x\cdot g$ iff $\exists x\cdot \ran(g)$ in which case $x\cdot g = x \cdot (\ran(g) \wedge p(x)|g)$.
\item For $x\in X$, $g,h\in G$, if $\exists gh$ and $\exists x\cdot g$ then $\exists x\cdot (gh)$ and $x\cdot (gh) = (x\cdot g)\cdot h$.
\item For $x\in X$, $g\in G$ with $\exists x\cdot g$ we have $p(x\cdot g) = (\ran(g) \wedge p(x)|g)^{-1}(\ran(g) \wedge p(x)|g)$.
\end{itemize}
One then defines pointed sets, premodules and modules analogously to the case of inverse semigroups. 
\end{remark}

\section{Matrices over inverse semigroups}

In the previous section we described how to define the $K$-group of an arbitrary orthogonally complete inverse semigroup using certain finitely generated projective modules. In this section we shall show that there is another way of calculating the same group but this time using matrices over inverse semigroups. We shall generalise the rook matrices of Solomon \cite{solomon}.

Throughout this section let $S$ be an orthogonally complete inverse semigroup. An $m\times n$ matrix $A$ with entries in $S$ is said to be a \emph{rook matrix} if it satisfies the following conditions:

(RM1): If $a$ and $b$ lie in the same row of $A$ then $a^{-1}b = 0$.
  
(RM2): If $a$ and $b$ lie in the same column of $A$ then $ab^{-1} = 0$.

We denote the set of all finite-dimensional rook matrices over $S$ by $R(S)$. In what follows a matrix denoted $A$ will have $i,j$th entry given by $a_{ij}$. Let $A$ be an $m\times n$ rook matrix and $B$ an $n\times p$ rook matrix. The $m\times p$ matrix $C = AB$ has entries
$$c_{ij} = \vee_{k=1}^{n}{a_{ik}b_{kj}}.$$
That this join is well-defined is guaranteed by axioms (RM1) and (RM2).

We use the term \emph{semigroupoid} to mean a structure that is the same as a category but does not necessarily have identities.

\begin{lemma}
$R(S)$ is a semigroupoid.
\end{lemma}

\begin{proof}
We need to show this operation when defined returns a rook monoid and is associative. First we show that for $A$ an $m\times n$ rook matrix, $B$ an $n\times p$ rook matrix the product $C = AB$ is an $m\times p$ rook matrix. That is, we need to show that for all allowable $i,j,k$ with $i\neq j$, we have $c_{ki}^{-1}c_{kj} = 0$ and $c_{ik}c_{jk}^{-1} = 0$.

One easily verifies using standard properties of orthogonal joins (e.g. see \cite{LawOrthCompPoly}) that for all allowable $i,j$ we have
$$c_{ij}^{-1} = \bigvee_{k=1}^{n}{b_{kj}^{-1}a_{ik}^{-1}}.$$

Thus, for $i\neq j$, we have
$$c_{ki}^{-1}c_{kj} = \left(\bigvee_{l=1}^{n}{b_{li}^{-1}a_{kl}^{-1}}\right)\left(\bigvee_{l=1}^{n}{a_{kl}b_{lj}}\right) 
= \bigvee_{l=1}^{n}{b_{li}^{-1}a_{kl}^{-1}a_{kl}b_{lj}} = \bigvee_{l=1}^{n}{b_{li}^{-1}a_{kl}^{-1}a_{kl}b_{li}b_{li}^{-1}b_{lj}} = 0$$
and
$$c_{ik}c_{jk}^{-1} = \left(\bigvee_{l=1}^{n}{a_{il}b_{lk}}\right)\left(\bigvee_{l=1}^{n}{b_{lk}^{-1}a_{jl}^{-1}}\right) = \bigvee_{l=1}^{n}{a_{il}b_{lk}b_{lk}^{-1}a_{jl}^{-1}} = \bigvee_{l=1}^{n}{a_{il}a_{jl}^{-1}a_{jl}b_{lk}b_{lk}^{-1}a_{jl}^{-1}} = 0.$$

\begin{comment}
$$\pi_{ij} = \bigvee_{k=1}^{\infty}{b_{kj}^{-1}a_{ik}^{-1}}.$$
It can be shown that $\pi_{ij}$ is well-defined in the same manner as above.
Then, using the orthogonality conditions on $A$ and $B$, we have
$$\pi_{ij}c_{ij}\pi_{ij} = \left(\bigvee_{k=1}^{\infty}{b_{kj}^{-1}a_{ik}^{-1}}\right)\left(\bigvee_{k=1}^{\infty}{a_{ik}b_{kj}}\right)\left(\bigvee_{k=1}^{\infty}{b_{kj}^{-1}a_{ik}^{-1}}\right) = \bigvee_{k=1}^{\infty}{b_{kj}^{-1}a_{ik}^{-1}} = \pi_{ij}$$
and   
$$c_{ij}\pi_{ij}c_{ij} = \left(\bigvee_{k=1}^{\infty}{a_{ik}b_{kj}}\right)\left(\bigvee_{k=1}^{\infty}{b_{kj}^{-1}a_{ik}^{-1}}\right)\left(\bigvee_{k=1}^{\infty}{a_{ik}b_{kj}}\right) = \bigvee_{k=1}^{\infty}{a_{ik}b_{kj}} = c_{ij}.$$
Thus
$$c_{ij}^{-1} = \pi_{ij} = \bigvee_{k=1}^{\infty}{b_{kj}^{-1}a_{ik}^{-1}}.$$

So, noting that $S$ is a monoid, we have
$$c_{ki}^{-1}c_{kj} = \left(\bigvee_{l=1}^{n}{b_{li}^{-1}a_{kl}^{-1}}\right)\left(\bigvee_{l=1}^{n}{a_{kl}b_{lj}}\right) 
= \bigvee_{l=1}^{n}{b_{li}^{-1}a_{kl}^{-1}a_{kl}b_{lj}} \leq \bigvee_{l=1}^{\infty}{b_{li}^{-1}b_{lj}} = 0$$
and
$$c_{ik}c_{jk}^{-1} = \left(\bigvee_{l=1}^{\infty}{a_{il}b_{lk}}\right)\left(\bigvee_{l=1}^{\infty}{b_{lk}^{-1}a_{jl}^{-1}}\right) = \bigvee_{l=1}^{\infty}{a_{il}b_{lk}b_{lk}^{-1}a_{jl}^{-1}} \leq \bigvee_{l=1}^{\infty}{a_{il}a_{jl}^{-1}} = 0.$$
Thus this operation is well-defined.
\end{comment}

Now let us prove that $R(S)$ is associative. We want to show (when the dimensions match up appropriately) 
$$(A\cdot B)\cdot C = A\cdot (B\cdot C).$$
Let $M = A\cdot B$, $P = B\cdot C$, $N = M\cdot C$, $Q = A\cdot P$. Then
$$n_{ij} = \bigvee_{k}{m_{ik}c_{kj}} = \bigvee_{k}{\left(\left(\bigvee_{r}{a_{ir}b_{rk}}\right)c_{kj}\right)} = \bigvee_{k}{\bigvee_{r}{a_{ir}b_{rk}c_{kj}}}$$
and
$$q_{ij} = \bigvee_{k}{a_{ik}p_{kj}} = \bigvee_{k}{\left(a_{ik}\left(\bigvee_{r}{b_{kr}c_{rj}}\right)\right)} = \bigvee_{k}{\bigvee_{r}{a_{ik}b_{kr}c_{rj}}}.$$
Thus $n_{ij} = q_{ij}$ and so $(A\cdot B)\cdot C = A\cdot (B\cdot C)$.
\end{proof}

Observe that if $S$ were chosen to be the two element Boolean algebra then rook matrices over $S$ are essentially the same as the rook matrices of Solomon \cite{solomon}.

\begin{lemma}
The idempotents of $R(S)$ are square matrices whose diagonal entries are idempotents in $E(S)$ and whose off-diagonal entries are $0$. 
\end{lemma}

\begin{proof}
Let $E$ be an $n\times n$ rook matrix with $E^{2} = E$. We have for all $i,j$:
$$e_{ij} = \bigvee_{k=1}^{n}{e_{ik}e_{kj}}.$$
So 
$$e_{ij} = e_{ij}e_{ij}^{-1}e_{ij} = \bigvee_{k=1}^{n}{e_{ij}e_{ij}^{-1}e_{ik}e_{kj}} = e_{ij}e_{ij}^{-1}e_{ij}e_{jj} = e_{ij}e_{jj}.$$
Thus for all $i,j$ we have $e_{jj}\in E(S)$. %and $e_{ij}^{-1}e_{ij}\leq e_{jj}$.
\begin{comment}
Furthermore,
$$e_{ij}e_{ij}^{-1} =  \bigvee_{k=1}^{n}{e_{ik}e_{kj}e_{ij}^{-1}} = e_{ii}e_{ij}e_{ij}^{-1}$$
and so for all $i,j$ we have $e_{ij}e_{ij}^{-1}\leq e_{ii}$.
\end{comment}
For $i\neq j$, we have
$$e_{ij} = e_{ij}e_{jj} = e_{ij}e_{jj}^{-1} = 0.$$
\end{proof}

\begin{lemma}
The $n\times n$ idempotent matrices of $R(S)$ commute and their product is again idempotent.
\end{lemma}

\begin{proof}
Let $E,F\in E(R(S))$ be $n\times n$ idempotent matrices, $G = EF$ and $H = FE$. Then for $i\neq j$ we have
$$g_{ij} = \bigvee_{k=1}^{n}{e_{ik}f_{kj}} = 0$$
and
$$g_{ii} = \bigvee_{k=1}^{n}{e_{ik}f_{ki}} = e_{ii}f_{ii}.$$
On the other hand,
$$h_{ij} = \bigvee_{k=1}^{n}{f_{ik}e_{kj}} = 0$$
and
$$h_{ii} = \bigvee_{k=1}^{n}{f_{ik}e_{ki}} = f_{ii}e_{ii}.$$
\end{proof}

\begin{lemma}
$R(S)$ is an inverse semigroupoid. In particular, for an $m\times n$ matrix $A\in R(S)$, letting $B$ denote the $n\times m$ matrix with $b_{ij} = a_{ji}^{-1}$ for all $i,j$, we have $B = A^{-1}$.
\end{lemma}

\begin{proof}
First we need to check that $B$ is a rook matrix. We have for $i\neq j$, 
$$b_{ki}^{-1}b_{kj} = a_{ik}a_{jk}^{-1} = 0$$ 
and
$$b_{ik}b_{jk}^{-1} = a_{ki}^{-1}a_{kj} = 0.$$
We want to show $ABA = A$ and $BAB = B$. Let $M = AB$, $N = ABA$, $P = BAB$. For $i \neq j$,
$$m_{ij} = \bigvee_{k}{a_{ik}b_{kj}} = \bigvee_{k}{a_{ik}a_{jk}^{-1}} = 0$$
and
$$m_{ii} = \bigvee_{k}{a_{ik}b_{ki}} = \bigvee_{k}{a_{ik}a_{ik}^{-1}}.$$
So, for all $i,j$, we have
$$n_{ij} = \bigvee_{k}{m_{ik}a_{kj}} = m_{ii}a_{ij} = \bigvee_{k}{a_{ik}a_{ik}^{-1}a_{ij}} = a_{ij}$$
and
$$p_{ij} = \bigvee_{k}{b_{ik}m_{kj}} = b_{ij}m_{jj} = \bigvee_{k}{a_{ji}^{-1}a_{jk}a_{jk}^{-1}} = a_{ji}^{-1} = b_{ij}.$$
It is easy to see that a regular semigroupoid whose idempotents commute is an inverse semigroupoid. (Explicitly, suppose $C$ is another inverse for $A$. Then
$$C = CAC = CABAC = CACAB = CAB = CABAB$$
$$ = BACAB = BAB = B.)$$
\end{proof}

If $n$ is a finite non-zero natural number, define $M_{n}(S)$ to be the inverse semigroup of all $n\times n$ rook matrices over $S$. 
\begin{comment}

Let $M_{\infty}(S)$ denote the set of infinite matrices over $S$ with only finitely many non-zero entries such that if $A = (a_{ij})\in M_{\infty}(S)$, then for all $i,j,k$ with $i\neq j$ we have $a_{ki}^{-1}a_{kj} = 0$ and $a_{ik}a_{jk}^{-1} = 0$.

\begin{thm}
For $A,B\in M_{\infty}(S)$, define $C = A\cdot B$ to be matrix with entries
$$c_{ij} = \bigvee_{k=1}^{\infty}{a_{ik}b_{kj}}.$$
Then $M_{\infty}(S)$ forms a semigroup under this operation.
\end{thm}

\begin{proof}

\underline{Well-defined}: First, we need to show for all $i,j,k,l$ that
$$\exists a_{ik}b_{kj}\vee a_{il}b_{lj}.$$
To see this, note that
$$(a_{ik}b_{kj})^{-1}(a_{il}b_{lj}) = b_{kj}^{-1}a_{ik}^{-1}a_{il}b_{lj} = 0$$
and
$$(a_{ik}b_{kj})(a_{il}b_{lj})^{-1} = a_{ik}b_{kj}b_{lj}^{-1}a_{il}^{-1} = 0,$$
so by the orthogonal completeness of $S$, $\exists a_{ik}b_{kj}\vee a_{il}b_{lj}$.
Next we need to show that $C\in M_{\infty}(S)$. Clearly it will have finitely many non-zero entries, so we just have to show that for all $i,j,k$ with $i\neq j$ that
$$
\end{proof}

From now on we will simply write $AB$ instead of $A\cdot B$.

On the other hand, it is clear that any matrix of the form described in the previous lemma will be idempotent.

\begin{corollary}
$M_{\infty}(S)$ is an inverse semigroup.
\end{corollary}

\begin{proof}
$M_{\infty}(S)$ is regular and its idempotents commute and is thus inverse.
\end{proof}

\end{comment}
For $n\times n$ rook matrices $A,B\in M_{n}(S)$ we will denote the natural partial order by $\leq$.

\begin{lemma}
For $A,B\in M_{n}(S)$, we have $A\leq B$ if and only if $a_{ij}\leq b_{ij}$ for all $i,j$.
\end{lemma}

\begin{proof}
$A\leq B$ means $A = BA^{-1}A$. Let $C = A^{-1}$, $D = CA$ and $E = BD$. Then from the above, we have $d_{ij} = 0$ for $i\neq j$ and
$$d_{ii} = \bigvee_{k=1}^{n}{a_{ki}^{-1}a_{ki}}.$$
So
$$e_{ij} = \bigvee_{k=1}^{n}{b_{ik}d_{kj}} = b_{ij}d_{jj} = \bigvee_{k=1}^{n}{b_{ij}a_{kj}^{-1}a_{kj}}.$$
So, if $A = BA^{-1}A$ then
$$a_{ij} = a_{ij}a_{ij}^{-1}a_{ij} = e_{ij}a_{ij}^{-1}a_{ij} = \bigvee_{k=1}^{n}{b_{ij}a_{kj}^{-1}a_{kj}a_{ij}^{-1}a_{ij}} = b_{ij}a_{ij}^{-1}a_{ij}a_{ij}^{-1}a_{ij} = b_{ij}a_{ij}^{-1}a_{ij}.$$
Suppose now that $a_{ij} = b_{ij}a_{ij}^{-1}a_{ij}$ for all $i,j$ and let $C,D,E$ be as above. Then
\begin{eqnarray*}
e_{ij} &=& \bigvee_{k=1}^{n}{b_{ij}a_{kj}^{-1}a_{kj}} = \bigvee_{k=1}^{n}{b_{ij}a_{kj}^{-1}a_{kj}b_{kj}^{-1}b_{kj}a_{kj}^{-1}a_{kj}} = \bigvee_{k=1}^{n}{b_{ij}b_{kj}^{-1}b_{kj}a_{kj}^{-1}a_{kj}a_{kj}^{-1}a_{kj}} \\
&=& b_{ij}b_{ij}^{-1}b_{ij}a_{ij}^{-1}a_{ij}a_{ij}^{-1}a_{ij} = b_{ij}a_{ij}^{-1}a_{ij} = a_{ij}.
\end{eqnarray*}
\end{proof}

\begin{lemma}
If $A,B\in M_{n}(S)$ are orthogonal, then their join exists. Furthermore, letting $C = A\vee B$, we have
$$c_{ij} = a_{ij}\vee b_{ij}$$
for all $i,j$.
\end{lemma}

\begin{proof}
If $A,B\in M_{n}(S)$ are orthogonal, then
$$\bigvee_{k=1}^{n}{a_{ki}^{-1}b_{kj}} = 0 = \bigvee_{k=1}^{n}{a_{ik}b_{jk}^{-1}}.$$
Thus for all $i,j,k$, we have $a_{ki}^{-1}b_{kj} = a_{ik}b_{jk}^{-1} = 0$ and so for all $i,j$, $\exists a_{ij}\vee b_{ij}$ and $\exists a_{ij}^{-1}\vee b_{ij}^{-1}$. Let $C$ be the matrix with entries $c_{ij} = a_{ij}\vee b_{ij}$. We need to show that $C\in M_{n}(S)$. It will then be clear by the previous lemma that $C = A\vee B$. So we will therefore verify that for all $i,j,k$ with $i\neq j$ we have $c_{ki}^{-1}c_{kj} = 0$ and $c_{ik}c_{jk}^{-1} = 0$. First note that $(a_{ij}\vee b_{ij})^{-1} = a_{ij}^{-1}\vee b_{ij}^{-1}$.
\begin{comment}
We have
$$(a_{ij}^{-1}\vee b_{ij}^{-1})(a_{ij}\vee b_{ij})(a_{ij}^{-1}\vee b_{ij}^{-1}) = a_{ij}^{-1}\vee b_{ij}^{-1}$$
and
$$(a_{ij}\vee b_{ij})(a_{ij}^{-1}\vee b_{ij}^{-1})(a_{ij}\vee b_{ij}) = a_{ij}\vee b_{ij}.$$
\end{comment}
So
$$c_{ki}^{-1}c_{kj} = (a_{ki}^{-1}\vee b_{ki}^{-1})(a_{kj}\vee b_{kj}) = 0$$
and
$$c_{ik}c_{jk}^{-1} = (a_{ik}\vee b_{ik})(a_{jk}^{-1}\vee b_{jk}^{-1}) = 0.$$
\end{proof}

\begin{lemma}
Let $A,B \in M_{n}(S)$ be orthogonal. Then for all $D\in M_{n}(S)$ we have
$$D(A\vee B) = DA \vee DB.$$
\end{lemma}

\begin{proof}
Let $A,B,D\in M_{n}(S)$ with $A$ orthogonal to $B$, let $C = A\vee B$ be as above, and let $E = DA$, $F = DB$. First we must check that $\exists E\vee F$. To this end, let $G = EF^{-1}$ and $H = E^{-1}F$. Orthogonality of $A$ and $B$ gives $G = 0$ and
$$H = A^{-1}D^{-1}DB = A^{-1}AA^{-1}D^{-1}DB = A^{-1}D^{-1}DAA^{-1}B = 0.$$
\begin{comment}
$$h_{ij} = \bigvee_{k=1}^{n}{e_{ki}^{-1}f_{kj}} = 
\bigvee_{k=1}^{n}{\bigvee_{r=1}^{n}{\bigvee_{s=1}^{n}{\left(a_{ri}^{-1}d_{kr}^{-1}d_{ks}b_{sj}\right)}}} = 
\bigvee_{k=1}^{n}{\bigvee_{r=1}^{n}{\left(a_{ri}^{-1}d_{kr}^{-1}d_{kr}b_{rj}\right)}} $$
$$\leq \bigvee_{k=1}^{\infty}{a_{ri}^{-1}b_{rj}} = 0.$$
\end{comment}
Thus $\exists E\vee F$. Let $M = E\vee F$ and $N = DC$. Then
$$m_{ij} = e_{ij}\vee f_{ij} = \left(\bigvee_{k=1}^{n}{d_{ik}a_{kj}}\right)\vee\left(\bigvee_{k=1}^{n}{d_{ik}b_{kj}}\right) =
\bigvee_{k=1}^{n}{d_{ik}(a_{kj}\vee b_{kj})} $$
$$= \bigvee_{k=1}^{n}{d_{ik}c_{kj}} = n_{ij}.$$
\end{proof}

Combining the previous three lemmas we have

\begin{theorem}
$M_{n}(S)$ is orthogonally complete for each $n\in \mathbb{N}$.
\end{theorem}

We now define what we mean by $M_{\omega}(S)$. Its elements are $\mathbb{N}\times \mathbb{N}$ matrices whose entries are elements of $S$, such that these matrices are rook matrices in that they satisfy conditions (RM1) and (RM2), and there are only finitely many non-zero entries.

It is clear that by replacing $n$ by $\infty$ in the previous lemmas we have the following

\begin{thm}
$M_{\omega}(S)$ is an orthogonally complete inverse semigroup.
\end{thm}

Let us now determine the form of Green's $\mathcal{D}$-relation on the set of idempotents of $M_{\omega}(S)$. 
For $\mathbf{e} = (e_{1},\ldots,e_{n})$, where $e_{i}\in E(S)$ for each $i$, we will denote by $\Delta(\mathbf{e})$ the matrix $E\in M_{\omega}(S)$ with entries $e_{ii} = e_{i}$ for $i = 1,\ldots,n$ and $0$ everywhere else.

\begin{lemma}
Let $\mathbf{e} = (e_{1},\ldots,e_{n})$, $\mathbf{f} = (e_{2},e_{1},e_{3},\ldots,e_{n})$, where $n\geq 2$. Then
$$\Delta(\mathbf{e})\,\mathcal{D}\,\Delta(\mathbf{f}).$$
\end{lemma}

\begin{proof}
Let $A\in M_{\omega}(S)$ be the matrix with entries $a_{12} = e_{1}$, $a_{21} = e_{2}$, $a_{ii} = e_{i}$ for $i = 3,\ldots,n$ and $0$ everywhere else. An easy calculation shows that $AA^{-1} = \Delta(\mathbf{e})$ and $A^{-1}A = \Delta(\mathbf{f})$.
\end{proof}

The fact that we swapped the first two diagonal entries of the matrix was unimportant. Thus we can \emph{slide} entries in the diagonal and remain in the same $\mathcal{D}$-class. In particular, this tells us that $M_{\omega}(S)$ is orthogonally separating and is therefore a K-inverse semigroup.

\begin{lemma}
Let $\mathbf{e} = (e_{1},e_{2},\ldots,e_{n})$, $\mathbf{f} = (e_{1}\vee e_{2},e_{3},\ldots,e_{n})$ where $e_{1}\perp e_{2}$. Then
$$\Delta(\mathbf{e})\, \mathcal{D}\, \Delta(\mathbf{f}).$$
\end{lemma}

\begin{proof}
Let $A\in M_{\omega}(S)$ be the matrix with entries $a_{11} = e_{1}$, $a_{21} = e_{2}$, $a_{ii} = e_{i}$ for $i = 3,\ldots,n$ and $0$ everywhere else. An easy calculation shows that $AA^{-1} = \Delta(\mathbf{e})$ and $A^{-1}A = \Delta(\mathbf{f})$.
\end{proof}

Thus, we can also \emph{combine} and \emph{split} orthogonal joins.

\begin{lemma}
Let $\mathbf{e} = (e_{1},e_{2},\ldots,e_{n})$, $\mathbf{f} = (f_{1},e_{2},\ldots,e_{n})$ where $e_{1}\,\mathcal{D}\,f_{1}$. Then
$$\Delta(\mathbf{e})\, \mathcal{D}\, \Delta(\mathbf{f}).$$
\end{lemma}

\begin{proof}
Suppose $a\in S$ is such that $aa^{-1} = e_{1}$ and $a^{-1}a = f_{1}$. 
Let $A\in M_{\omega}(S)$ be the matrix with entries $a_{11} = a$, $a_{ii} = e_{i}$ for $i = 2,\ldots,n$ and $0$ everywhere else. 
An easy calculation shows that $AA^{-1} = \Delta(\mathbf{e})$ and $A^{-1}A = \Delta(\mathbf{f})$.
\end{proof}

This tells us that we can \emph{swap} entries for $\mathcal{D}$-related elements. In fact, these three types of moves completely describe the $\mathcal{D}$-classes of $E(M_{\omega}(S))$.

\begin{lemma}
\label{dclassmatrix}
Let $E, F\in M_{\omega}(S)$ be idempotent matrices in the same $\mathcal{D}$-class. Then one can go from $E$ to $F$ in a finite number of slide, combining, splitting and swap moves.
\end{lemma}

\begin{proof}
Suppose $E = AA^{-1}$, $F = A^{-1}A$ for some $A\in M_{\omega}(S)$. Then 
$$e_{ii} = \bigvee_{k=1}^{\infty}{a_{ik}a_{ik}^{-1}}$$
and
$$f_{ii} = \bigvee_{k=1}^{\infty}{a_{ki}^{-1}a_{ki}}.$$
Firstly, since $A$ only has finitely many non-zero entries, these joins are over a finite number of orthogonal elements. So, we can split the joins and slide the entries along the diagonal in $E$, so that each diagonal entry is now of the form $a_{ik}a_{ik}^{-1}$ for some $i,k$. Then we can replace each $a_{ik}a_{ik}^{-1}$ with $a_{ik}^{-1}a_{ik}$ by performing a swap move. Finally, joining enough orthogonal elements together will then give $F$.
\end{proof}

\begin{comment}
For 
$$\mathbf{e} = (e_{1},e_{2},\ldots,e_{m}),$$
let $\Delta(\mathbf{e}) = G\in M_{\omega}(S)$ be the matrix with entries $g_{ii} = e_{i}$ for $i = 1,\ldots,m$ and $0$ everywhere else and let
$$\mathbf{e_{i}} = (0,\ldots,0,e_{i}^{\downarrow},0\ldots,0) \in \bigoplus_{i=1}^{m}{(e_{i}S)^{\sharp}},$$
where the $e_{i}$ is in the $i$th position.
\end{comment}

Let $S$ be an orthogonally complete inverse semigroup and let
$$\theta:\bigoplus_{i=1}^{m}{(e_{i}S)^{\sharp}}\rightarrow \bigoplus_{i=1}^{m}{(f_{i}S)^{\sharp}}$$
be a module isomorphism with $e_{i},f_{i}\in E(S)$ for each $i$. Then we know by Lemmas \ref{projmod1} and \ref{projmod2} that there exist $a_{ij}\in f_{i}Se_{j}$ with $a_{ij}a_{kj}^{-1} = 0$ for $i\neq k$, $a_{ij}^{-1}a_{ik} = 0$ for $j\neq k$,
$$\theta(\mathbf{e_{j}}) = (a_{1j}^{\downarrow},\ldots,a_{mj}^{\downarrow})$$
and
$$\theta^{-1}(\mathbf{f_{i}}) = ((a_{i1}^{-1})^{\downarrow},\ldots,(a_{im}^{-1})^{\downarrow}).$$

Thus the matrix $A$ with entries $a_{ij}$ (and $0$'s everywhere else) is an element of $M_{\omega}(S)$, $\Delta(\mathbf{e}) = A^{-1}A$ and $\Delta(\mathbf{f}) = AA^{-1}$. In fact, the converse is also true:

\begin{lemma}
Let $\mathbf{e} = (e_{1},\ldots,e_{m})$, $\mathbf{f} = (f_{1},\ldots,f_{m})$ and let $A\in M_{\omega}(S)$ be such that $\Delta(\mathbf{e}) = A^{-1}A$ and $\Delta(\mathbf{f}) = AA^{-1}$. Then the map 
$$\theta:\bigoplus_{i=1}^{m}(e_{i}S)^{\sharp}\rightarrow \bigoplus_{i=1}^{m}(f_{i}S)^{\sharp}$$
given on generators by
$$\theta(0,\ldots,0,e_{i}^{\downarrow},0,\ldots,0) = (a_{1i}^{\downarrow},\ldots,a_{mi}^{\downarrow})$$
is a module isomorphism.
\end{lemma}

\begin{proof}
First, since $A\in M_{\omega}(S)$ we must have $a_{ki}a_{li}^{-1} = 0$ for all $i,k,l$ with $k\neq l$. Thus 
$$(a_{1i}^{\downarrow},\ldots,a_{mi}^{\downarrow}) \in \bigoplus_{i=1}^{m}(f_{i}S)^{\sharp}$$
for all $i = 1,\ldots,m$.
To see that $\theta$ is a module morphism, note that 
$$q(a_{1i}^{\downarrow},\ldots,a_{mi}^{\downarrow}) = \bigvee_{k=1}^{m}{a_{ki}^{-1}a_{ki}} = e_{i}.$$
Let us now check that $\theta$ is surjective. We claim that for all $i = 1,\ldots, m$ we have
$$(a_{i1}^{-1\downarrow},\ldots,a_{im}^{-1\downarrow})\in \bigoplus_{k=1}^{m}(e_{k}S)^{\sharp}.$$
Firstly, 
$$e_{k}a_{ik}^{-1} = \left(\bigvee_{i=1}^{m}{a_{ik}^{-1}a_{ik}}\right)a_{ik}^{-1} = a_{ik}^{-1}a_{ik}a_{ik}^{-1} = a_{ik}^{-1}.$$
Secondly, $a_{ik}a_{ik}^{-1}a_{il}a_{il}^{-1} = 0$ if $k\neq l$. Now
\begin{eqnarray*}
\theta(a_{i1}^{-1\downarrow},\ldots,a_{im}^{-1\downarrow}) & = & \bigvee_{k=1}^{m}{((a_{1k}a_{ik}^{-1})^{\downarrow},\ldots,(a_{mk}a_{ik}^{-1})^{\downarrow})} \\
& = & \left(0,\ldots,0,\left(\bigvee_{k=1}^{m}{a_{ik}a_{ik}^{-1}}\right)^{\downarrow},0,\ldots,0\right)\\
& = & (0,\ldots,0,f_{i}^{\downarrow},0,\ldots,0).
\end{eqnarray*}
Thus $\theta$ is surjective.
Finally, let us check that $\theta$ is injective. Let 
$$x = (\left\{x_{11},\ldots,x_{1r_{1}}\right\}^{\downarrow},\ldots,\left\{x_{m1},\ldots,x_{mr_{m}}\right\}^{\downarrow}) \in \bigoplus_{i=1}^{m}(e_{i}S)^{\sharp}$$
and
$$y = (\left\{y_{11},\ldots,y_{1s_{1}}\right\}^{\downarrow},\ldots,\left\{y_{m1},\ldots,y_{ms_{m}}\right\}^{\downarrow}) \in \bigoplus_{i=1}^{m}(e_{i}S)^{\sharp}$$
be such that $\theta(x) = \theta(y)$. Then for all $k = 1,\ldots,m$ we have
$$\bigvee_{i=1}^{m}\bigvee_{t=1}^{r_{i}}{(a_{ki}x_{it})^{\downarrow}} = \bigvee_{i=1}^{m}\bigvee_{t=1}^{s_{i}}{(a_{ki}y_{it})^{\downarrow}}$$
in $(f_{k}S)^{\sharp}$. So premultiplying both sides of the equation by $a_{ki}^{-1}$ gives
$$\bigvee_{t=1}^{r_{i}}{(a_{ki}^{-1}a_{ki}x_{it})^{\downarrow}} = \bigvee_{t=1}^{s_{i}}{(a_{ki}^{-1}a_{ki}y_{it})^{\downarrow}}$$
in $(e_{i}S)^{\sharp}$.
Taking the join over all $k$ gives
$$\bigvee_{t=1}^{r_{i}}{x_{it}^{\downarrow}} 
= \bigvee_{k=1}^{m}{\bigvee_{t=1}^{r_{i}}{(a_{ki}^{-1}a_{ki}x_{it})^{\downarrow}}} 
= \bigvee_{k=1}^{m}{\bigvee_{t=1}^{s_{i}}{(a_{ki}^{-1}a_{ki}y_{it})^{\downarrow}}}
= \bigvee_{t=1}^{s_{i}}{y_{it}^{\downarrow}}.$$
Thus $x = y$ and so $\theta$ is injective.
\end{proof}

It therefore follows that the objects of $\Proj_{S}$ are in one-one correspondence with the $\mathcal{D}$-classes of idempotents of $M_{\omega}(S)$. Thus, we have proved:
\begin{comment}

\begin{lemma}
Let $E\in E(M_{\infty}(S))$, let $n\in \mathbb{N}$ be such that $e_{ii} = 0$ for all $i\geq n+1$ and let $F\in M_{\infty}(S)$ be such that for $n+1 \leq i\leq 2n$, $f_{ii} = e_{kk}$ where $k = i-n$ and $f_{ij} = 0$ for all other entries. Then $E\mathcal{D} F$. 
\end{lemma}

\begin{proof}
Let $S\in E(M_{\infty}(S))$ be the matrix with entries $s_{k,i}$ for $i = n+1,\ldots, 2n$, $k = n-i$, and zero everywhere else. Then $SS^{-1} = E$ and $S^{-1}S = F$.
\end{proof}

Thus if $S$ is an orthogonally complete inverse monoid with $0$ then $M_{\infty}(S)$ is a K-inverse semigroup.

The following can easily be proved
\end{comment}

\begin{thm}
Let $S$ be an orthogonally complete inverse semigroup. Then 
$$K(S) \cong K(M_{\omega}(S)).$$
\end{thm}

Lemma \ref{dclassmatrix} tells us how to give $A(S)$ in terms of a semigroup presentation. Let $X = \left\{A_{e}|e\in E(S)\right\}$ and let $\mathcal{R}$ be the set of relations given by:
\begin{enumerate}
\item $A_{e}A_{f} = A_{f}A_{e}$ for all $e,f\in E(S)$.
\item $A_{e} = A_{f}$ if $e\, \mathcal{D} \, f$.
\item $A_{e}A_{f} = A_{e\vee f}$ if $ef = 0$.
\end{enumerate}

Then $A(S)$ has the following semigroup presentation:
$$A(S) = \langle X \quad | \quad \mathcal{R} \rangle .$$

\section{Functorial properties of $M_{\omega}$ and $K$}

A homomorphism  $\phi:S\rightarrow T$ between orthogonally complete inverse semigroups is said to be \emph{orthogonal join preserving} if $s\perp t$ implies $\phi(s\vee t) = \phi(s)\vee \phi(t)$ for all $s,t\in S$ (both $s\vee t$ and $\phi(s)\vee \phi(t)$ exist since $S$ and $T$ are orthogonally complete and if $s\perp t$ in $S$, then $\phi(s)\perp \phi(t)$ in $T$). We will always assume $\phi(0) = 0$ for any homomorphism $\phi$.

Let $\phi:S\rightarrow T$ be an orthogonal join preserving homomorphism between two orthogonally complete inverse semigroups. Define $\phi^{\ast}:M_{\omega}(S)\rightarrow M_{\omega}(T)$ by $\phi^{\ast}(A) = B$, where $b_{ij} = \phi(a_{ij})$ for all $i,j$. 

\begin{lemma}
\label{phiasthom}
$\phi^{\ast}$ is a well-defined homomorphism. In addition, $\phi^{\ast}$ is orthogonal join preserving.
\end{lemma}

\begin{proof}
Firstly, we see for $A\in M_{\omega}(S)$ that $\phi^{\ast}(A)$ will satisfy the same orthogonally conditions as for $A$, so $\phi^{\ast}(A)\in M_{\omega}(T)$.
Let $A,B\in M_{\omega}(S)$, $C = \phi^{\ast}(AB)$ and $D = \phi^{\ast}(A)\phi^{\ast}(B)$. Then for all $i,j$ we have
$$c_{ij} = \phi\left(\bigvee_{k=1}^{\infty}{a_{ik}b_{kj}}\right) = \bigvee_{k=1}^{\infty}{\phi(a_{ik})\phi(b_{kj})} = d_{ij}$$
and so $\phi^{\ast}(AB) = \phi^{\ast}(A)\phi^{\ast}(B)$.
Now let us show that $\phi^{\ast}$ preserves orthogonal joins. Let $A\perp B$, $C = \phi(A\vee B)$ and $D = \phi(A)\vee \phi(B)$ ($D$ exists by an earlier remark). Then
$$c_{ij} = \phi(a_{ij}\vee b_{ij}) = \phi(a_{ij})\vee \phi(b_{ij}) = d_{ij}.$$
\end{proof}

If $\phi$ is injective then $\phi^{\ast}$ must also be injective. Suppose $\phi$ is surjective. Then $\phi^{\ast}$ will be surjective if and only if $\phi^{-1}(0) = 0$.

\begin{comment}
\begin{lemma}
Let $\phi:S\rightarrow T$ be an orthogonal join preserving homomorphism of orthogonally complete inverse monoids such that $\phi(e)\mathcal{D}_{T} \phi(f)$ if and only if $e\mathcal{D}_{S} f$ for all $e,f\in S$. Then for all idempotent matrices $E,F\in E(M_{\infty}(S))$ we have $\phi^{\ast}(E)$ and $\phi^{\ast}(F)$ are $\mathcal{D}$-related in $M_{\infty}(T)$ if and only if $E$ and $F$ are $\mathcal{D}$-related in $M_{\infty}(S)$.
\end{lemma}

\begin{proof}
One direction is clear. Suppose $\phi^{\ast}(E)$ and $\phi^{\ast}(F)$ are $\mathcal{D}$-related in $M_{\infty}(T)$. Then there exists $B\in M_{\infty}(T)$ with $BB^{-1} = E$ and $B^{-1}B = F$. Explicitly, for each $i$, we have
$$e_{ii} = \bigvee_{k=1}^{\infty}{b_{ik}b_{ik}^{-1}}$$
and
$$f_{ii} = \bigvee_{k=1}^{\infty}{b_{ki}^{-1}b_{ki}}.$$
MORE WORK NEEDED.
\end{proof}
\end{comment}

\begin{lemma}
\label{kinvhom}
Let $S,T$ be $K$-inverse semigroups with $\phi:S\rightarrow T$ an orthogonal join preserving homomorphism.
Then there is a homomorphism $\overline{\phi}:K(S)\rightarrow K(T)$. If $\phi$ is surjective then $\overline{\phi}$ is surjective.
\end{lemma}

\begin{proof}
Let $e,f\in E(S)$ with $ef = 0$. Then $\phi(e\vee f) = \phi(e)\vee \phi(f)$. 

Define $\phi^{\dagger}:A(S)\rightarrow A(T)$ by $\phi^{\dagger}([e]) = [\phi(e)]$. 
If $e,f\in E(S)$ with $e\,\mathcal{D}\, f$ then $\phi(e)\,\mathcal{D}\, \phi(f)$ and so $\phi^{\dagger}$ is well-defined. 
Further for $e,f\in E(S)$, we have 
\begin{eqnarray*}
\phi^{\dagger}([e]+[f]) &=& \phi^{\dagger}([e^{\prime}]+[f^{\prime}]) = \phi^{\dagger}([e^{\prime}\vee f^{\prime}]) = [\phi(e^{\prime}\vee f^{\prime})] = [\phi(e^{\prime})\vee \phi(f^{\prime})]\\
&=& [\phi(e^{\prime})] + [\phi(f^{\prime})] = [\phi(e)] + [\phi(f)] = \phi^{\dagger}([e]) + \phi^{\dagger}([f]),
\end{eqnarray*}
where $e^{\prime}f^{\prime} = 0$, $e\,\mathcal{D}\, e^{\prime}$ and $f\,\mathcal{D}\, f^{\prime}$.
If $\phi$ is surjective then it is clear that $\phi^{\dagger}$ is surjective. Standard theory (c.f. \cite{Rosenberg}) then tells us that we can lift $\phi^{\dagger}:A(S)\rightarrow A(T)$ to a homomorphism $\overline{\phi}:K(S) = \mathcal{G}(A(S))\rightarrow K(T) = \mathcal{G}(A(T))$ and that if $\phi^{\dagger}$ is surjective then $\overline{\phi}$ will be surjective.
\end{proof} 

Combining Lemmas \ref{phiasthom} and \ref{kinvhom} we have

\begin{thm}
Let $S,T$ be orthogonally complete inverse monoids with $\phi:S\rightarrow T$ an orthogonal join preserving homomorphism.
Then there is a homomorphism $K(S)\rightarrow K(T)$. If $\phi$ is surjective and $\phi^{-1}(0) = 0$ then this homomorphism is surjective.
\end{thm}
\begin{comment}
\begin{proof}
We have that $\phi^{\ast}:M_{\omega}(S)\rightarrow M_{\omega}(T)$ is an orthogonal join preserving homomorphism. Thus $\overline{\phi^{\ast}}:K(M_{\omega}(S))\rightarrow K(M_{\omega}(T))$ is a homomorphism. 
Further, if $\phi$ surjective then $\overline{\phi^{\ast}}$ is surjective by the above.
\end{proof}
\end{comment}

\begin{comment}
\begin{lemma}
Let $S,T$ be orthogonally complete inverse semigroups. Then we have
\begin{enumerate}
\item $E(S\times T) = E(S)\times E(T)$.
\item $S\times T$ is inverse and $(s,t)^{-1} = (s^{-1},t^{-1})$.
\item $(s_{1},t_{1})\leq (s_{2},t_{2})$ iff $s_{1}\leq s_{2}$ and $t_{1}\leq t_{2}$, where $s_{1},s_{2}\in S$ and $t_{1},t_{2}\in T$.
\item $S\times T$ is orthogonally complete
\item $(e_{1},f_{1})\mathcal{D}(e_{2},f_{2})$ iff $e_{1}\mathcal{D} e_{2}$ and $f_{1}\mathcal{D} f_{2}$, where $e_{1},e_{2}\in E(S)$ and $f_{1},f_{2}\in E(T)$.
\item $(e_{1},f_{1})\perp(e_{2},f_{2})$ iff $e_{1}\perp e_{2}$ and $f_{1}\perp f_{2}$, where $e_{1},e_{2}\in E(S)$ and $f_{1},f_{2}\in E(T)$.
\item $(e,0)\perp (0,f)$ and $(e,0)\vee (0,f) = (e,f)$.
\end{enumerate}
\end{lemma}

\begin{proof}
\begin{enumerate}
\item Easy to check
\item Idempotents commute and easy calculation
\item Easy to check
\item Follows easily from orthogonal completeness of $S$ and $T$, and (3)
\item Easy calculation
\item Obvious
\item Simple calculation
\end{enumerate}
\end{proof}
\end{comment}

If $S,T$ are inverse semigroups then their cartesian product $S\times T$ will also be an inverse semigroup. It is easy to see that if $S$ and $T$ are both orthogonally complete then $S\times T$ will be orthogonally complete. $S\times T$ will satisfy the following properties:
\begin{itemize}
\item $E(S\times T) = E(S)\times E(T)$.
\item $(s,t)^{-1} = (s^{-1},t^{-1})$ for $s\in S$, $t\in T$.
\item $(s_{1},t_{1})\leq (s_{2},t_{2})$ if and only if $s_{1}\leq s_{2}$ and $t_{1}\leq t_{2}$, where $s_{1},s_{2}\in S$ and $t_{1},t_{2}\in T$.
\item $(e_{1},f_{1})\,\mathcal{D}\,(e_{2},f_{2})$ if and only if $e_{1}\,\mathcal{D}\, e_{2}$ and $f_{1}\,\mathcal{D}\, f_{2}$, where $e_{1},e_{2}\in E(S)$ and $f_{1},f_{2}\in E(T)$.
\item $(e_{1},f_{1})\perp(e_{2},f_{2})$ if and only if $e_{1}\perp e_{2}$ and $f_{1}\perp f_{2}$, where $e_{1},e_{2}\in E(S)$ and $f_{1},f_{2}\in E(T)$.
\item If $s_{1}\perp s_{2}\in S$ and $t_{1}\perp t_{2}\in T$ then 
$$(s_{1},t_{1})\vee(s_{2},t_{2}) = (s_{1}\vee s_{2},t_{1}\vee t_{2}).$$
\end{itemize}

\begin{lemma}
\label{cartmoniso}
For $S,T$ be orthogonally complete inverse semigroups, we have
$$\mathcal{A}(M_{\omega}(S\times T)) \cong \mathcal{A}(M_{\omega}(S))\times \mathcal{A}(M_{\omega}(T)).$$
\end{lemma}

\begin{proof}
Let $\Delta(\mathbf{e})\in E(M_{\omega}(S\times T))$ be an idempotent matrix with $\mathbf{e} = (e_{1},\ldots,e_{m})$. Then $e_{i}$ will be of the form $e_{i} = (a_{i},b_{i})$, where $a_{i}\in E(S)$ and $b_{i}\in E(T)$ are idempotents. Observe that for each $i$ we have
$$e_{i} = (a_{i},b_{i}) = (a_{i},0)\vee (0,b_{i}),$$
where this is the join of two orthogonal elements. Thus $\Delta(\mathbf{e})\,\mathcal{D}\,\Delta(\mathbf{f})$ where
$$\mathbf{f} = ((a_{1},0),\ldots,(a_{m},0),(0,b_{1}),\ldots,(0,b_{m})).$$
It follows that there is a bijection 
$$\theta:\mathcal{A}(M_{\omega}(S\times T))\rightarrow \mathcal{A}(M_{\omega}(S))\times \mathcal{A}(M_{\omega}(T))$$
given by
$$\theta([\Delta((a_{1},b_{1}),\ldots,(a_{m},b_{m}))]) = ([\Delta(a_{1},\ldots,a_{m})],[\Delta(b_{1},\ldots,b_{m})]).$$
If $a_{i},c_{i}\in S$, $b_{i},d_{i}\in T$ are such that $a_{i}\perp c_{i}$ and $b_{i}\perp d_{i}$ for each $i$ then
\begin{eqnarray*}
&\theta([\Delta((a_{1},b_{1}),\ldots,(a_{m},b_{m}))]+[\Delta((c_{1},d_{1}),\ldots,(c_{m},d_{m}))])\\
= &\theta([\Delta((a_{1},b_{1}),\ldots,(a_{m},b_{m}))\vee \Delta((c_{1},d_{1}),\ldots,(c_{m},d_{m}))])\\ 
= &\theta([\Delta((a_{1}\vee c_{1},b_{1}\vee d_{1}),\ldots,(a_{m}\vee c_{m},b_{m}\vee d_{m}))])\\
= &([\Delta(a_{1}\vee c_{1},\ldots,a_{m}\vee c_{m})],[\Delta(b_{1}\vee d_{1},\ldots,b_{m}\vee d_{m})])\\
= &([\Delta(a_{1},\ldots,a_{m})],[\Delta(b_{1},\ldots,b_{m})]) + ([\Delta(c_{1},\ldots,c_{m})],[\Delta(d_{1},\ldots,d_{m})])\\
= &\theta([\Delta((a_{1},b_{1}),\ldots,(a_{m},b_{m}))]) + \theta([\Delta((c_{1},d_{1}),\ldots,(c_{m},d_{m}))]).
\end{eqnarray*}
Thus $\theta$ is an isomorphism.
\end{proof}

\begin{lemma}
\label{cartgrotiso}
Let $S,T$ be commutative monoids. Then $\mathcal{G}(S\times T)\cong \mathcal{G}(S)\times \mathcal{G}(T)$.
\end{lemma}

\begin{proof}
Let $\phi_{1}:S\rightarrow \mathcal{G}(S)$ and $\phi_{2}:T\rightarrow \mathcal{G}(T)$ be the universal maps and let $\phi:S\times T\rightarrow \mathcal{G}(S)\times \mathcal{G}(T)$ be given by $\phi(s,t) = (\phi_{1}(s),\phi_{2}(t))$. 
Let $\theta:S\times T\rightarrow G$ be a monoid homomorphism to a commutative group $G$. 
Thus $\theta_{1}:S\rightarrow G$ and $\theta_{2}:T\rightarrow G$ given by $\theta_{1}(s) = \theta(s,0)$ and $\theta_{2}(t) = \theta(0,t)$ are homomorphisms. 
There are therefore unique maps $\pi_{1}:\mathcal{G}(S)\rightarrow G$ and $\pi_{2}:\mathcal{G}(T)\rightarrow G$ such that $\pi_{i}\phi_{i} = \theta_{i}$ for $i = 1,2$. 
Let $\pi:\mathcal{G}(S)\times \mathcal{G}(T)\rightarrow G$ be given by $\pi(s,t) = \pi_{1}(s) + \pi_{2}(t)$. 
It is easy to check $\pi$ is a homomorphism and $\pi\phi = \theta$.
On the other hand, suppose $\sigma:\mathcal{G}(S)\times \mathcal{G}(T)\rightarrow G$ is a homomorphism with $\sigma\phi = \theta$. 
By the uniqueness of the maps $\pi_{1}$ and $\pi_{2}$, we must have $\sigma(g,0) = \pi_{1}(g)$ and $\sigma(0,h) = \pi_{2}(h)$. 
Thus $\sigma(g,h) = \pi_{1}(g)+\pi_{2}(h) = \pi(g,h)$. 
\end{proof}

Combining Lemmas \ref{cartmoniso} and \ref{cartgrotiso} we see that

\begin{thm}
Let $S,T$ be orthogonally complete inverse semigroups. Then
$$K(S\times T) \cong K(S) \times K(T).$$ 
\end{thm}

\section{Commutative inverse semigroups}

It turns out one can say more about commutative orthogonally complete inverse semigroups. Suppose $S$ is such a semigroup. Then $s\perp t$ is equivalent to $s^{-1}st^{-1}t = 0$. It therefore follows that $eS$ is in fact a module for each $e\in E(S)$. 
\begin{comment}
Let $(X,p)$, $(Y,q)$ be modules. Define 
$$X\otimes Y = (X\times Y)/\sim$$
where $\sim$ is the equivalence relation defined by $(x\cdot s,y) \sim (x,y\cdot s)$. Define
$$[x,y]\cdot s = [x,y\cdot s] = [x\cdot s, y]$$
and
$$(p\otimes q)([x,y]) = p(x)q(y).$$

\begin{lemma}
$X\otimes Y$ is a module.
\end{lemma}

\begin{proof}
It is clear that $X\otimes Y$ is a pointed \'{e}tale set. One can easily check that if $[x,y]\perp [u,w]$ with $p(x)=p(y)$ and $p(u)=p(w)$, then $[x,y]\vee [u,w] = [x\vee u, y\vee w]$ and that $[x,y]\cdot s\vee [u,w]\cdot s = [x\vee u, y\vee w]\cdot s$.
\end{proof}

\begin{lemma}
Let $(X,p)$, $(Y,q)$ and $(Z,r)$ be modules. Then
$$X\otimes (Y\oplus Z)\cong (X\otimes Y)\oplus (X\otimes Z).$$
\end{lemma}

\begin{proof}
Define $\phi:X\otimes (Y\oplus Z)\rightarrow (X\otimes Y)\oplus (X\otimes Z)$ by 
$$\phi([x,(y,z)]) = ([x,y],[x,z]).$$
Then $\phi$ is a module isomorphism.
\end{proof}

\begin{lemma}
Let $e,f\in E(S)$ be idempotents in $S$. Then
$$eS\otimes fS \cong efS.$$
\end{lemma}

\begin{proof}
Define $\phi:eS\otimes fS \rightarrow efS$ by $\phi([s,t]) = st$.
\end{proof}
\end{comment}
Let us now consider matrices over such semigroups. For any idempotents $e,f\in E(S)$ we have $e\,\mathcal{D}\, f$ if and only if $e = f$. Thus when considering the $\mathcal{D}$-classes of the idempotents of $M_{\omega}(S)$, we can only slide along the diagonal or combine / split up orthogonal joins, but not swap $\mathcal{D}$-related elements of $S$. It follows that for all idempotent rook matrices $E,F\in E(M_{\omega}(S))$ we have $E\,\mathcal{D}\, F$ in $M_{\omega}(S)$ if and only if $E\,\mathcal{D}\, F$ in $M_{\omega}(E(S))$. Taking joins is independent of the non-idempotent elements and so we have just argued for the following:

\begin{thm}
\label{kidemcomm}
Let $S$ be a commutative orthogonally complete inverse semigroup. Then 
$$K(S) \cong K(E(S)).$$
\end{thm}

Let $S$ be a commutative orthogonally complete inverse semigroup. We can define a tensor / Kronecker product on $R(S)$. Let $A$ be an $n\times m$ rook matrix and let $B$ be a $p\times q$ rook matrix. Define $A\otimes B$ to be the $np\times mq$ rook matrix
$$A\otimes B = \begin{pmatrix}
  a_{11}B & a_{12}B & \cdots & a_{1m}B\\
  a_{21}B & a_{22}B & \cdots & a_{2m}B\\
  \vdots & \vdots & \ddots & \vdots \\
  a_{n1}B & a_{n2}B & \cdots & a_{nm}B 
 \end{pmatrix}.$$

It is easy to see that the tensor product of matrices over commutative orthogonally complete inverse semigroups satisfies the following properties:

\begin{lemma}
\label{proptens}
Let $A,B,C,D\in R(S)$ be finite dimensional rook matrices. Then
\begin{enumerate}
\item $(A\otimes B)\otimes C = A\otimes (B\otimes C)$.
\item If there exist $AC$ and $BD$ then $(A\otimes B)(C\otimes D) = (AC)\otimes (BD)$.
\item $(A\otimes B)^{-1} = A^{-1}\otimes B^{-1}$.
\end{enumerate}
\end{lemma}

\begin{comment}
\begin{proof}
\begin{enumerate}
\item This follows in the same way as for matrices over rings (see reference, to be included). Note that the associativity of $\otimes$ does not depend on the commutativity of $S$.
\item Exactly the same proof for matrices over commutative rings (see reference, to be included), this time depending on the commutativity of $S$.
\item Follows from part (2). Again depends on the commutativity of $S$. 
\end{enumerate}
\end{proof}
\end{comment}

We now deduce the following:

\begin{lemma}
Let $E_{1},E_{2},F_{1},F_{2}\in R(S)$ be idempotent finite dimensional rook matrices with $E_{1}\,\mathcal{D}\, F_{1}$ and $E_{2}\,\mathcal{D}\, F_{2}$. Then
$$E_{1}\otimes E_{2}\,\mathcal{D}\, F_{1}\otimes F_{2}.$$
\end{lemma}

\begin{proof}
Suppose $A_{1},A_{2}\in R(S)$ are such that $A_{i}A_{i}^{-1} = E_{i}$ and $A_{i}^{-1}A_{i} = F_{i}$ and let $B = A_{1}\otimes A_{2}$. Then by Lemma \ref{proptens} (2) and (3) we have $BB^{-1} = E_{1}\otimes E_{2}$ and $B^{-1}B = F_{1}\otimes F_{2}$.
\end{proof}

We can therefore define $E\otimes F$ up to $\mathcal{D}$-class for two idempotent matrices $E,F\in M_{\omega}(S)$ by sliding entries around. 

\begin{lemma}
Let $E,F\in M_{\omega}$ be idempotent matrices. Then
$$E\otimes F\, \mathcal{D}\, F\otimes E.$$
\end{lemma}

\begin{proof}
Use the preamble to Theorem \ref{kidemcomm}.
\end{proof}

Thus $(A(S),+,\otimes)$ is a commutative semiring. In fact, it easy to see that if $S$ were required to be the Boolean completion of a $0$-bisimple inverse semigroup instead of being commutive then $(A(S),+,\otimes)$ might be a semiring. If $S$ has an identity, then $(A(S),\otimes)$ becomes a semiring with identity. It follows that $K(S)$ can sometimes inherit the structure of a ring from $A(S)$.

\section{States and traces}

In this section we will define \emph{states} and \emph{traces} for orthogonally complete inverse monoids by analogy to the definitions in $C^{\ast}$-algebra theory (for states, see \cite{Landsman} \S 2.8 and for traces, see \cite{KellendonkPutnam} \S 7). 
%We will hopefully later be able to generalise this for orthogonally complete inverse semigroups.

We will define a \emph{state} on an orthogonally complete inverse monoid $S$ to be a map $\tau: S \rightarrow \mathbb{C}$ that is
\begin{enumerate}
\item \emph{Positive}: $\tau(e)$ is a non-negative real number for all idempotents $e\in E(S)$
\item \emph{Normalised}: $\tau(1) = 1$
\item \emph{Linear}: If $s,t\in S$  are orthogonal elements of $S$ then $\tau(s\vee t) = \tau(s) + \tau(t)$.
\end{enumerate}
A \emph{trace} will be a state $\tau: S \rightarrow \mathbb{C}$ such that $\tau(st) = \tau(ts)$ for all $s,t\in S$.

It is of course possible that a given semigroup $S$ may have no states or traces which can be defined on it. For example, the Cuntz monoid $C_{n}$ will only have traces defined on it if $n = 1$.
\begin{comment}
(Mark: a quick check of the definition shows that the states of a Cuntz monoid are in one-one correspondence with the states of the corresponding Cuntz $C^{\ast}$-algebra)
\end{comment}

Note that the linearity condition on states implies that $\tau(0) = 0$ and that for any trace $\tau:S\rightarrow \mathbb{C}$ if $e\,\mathcal{D}\, f$ are idempotents then $\tau(e) = \tau(f)$ since $\tau(ss^{-1}) = \tau(s^{-1}s)$. We now use this to connect the idea of traces with our notion of a $K$-group for $S$.

\begin{lemma}
Let $S$ be an orthogonally complete inverse monoid and let $\tau:S\rightarrow \mathbb{C}$ be a trace on $S$. Then there is an induced group homomorphism
$$\tilde{\tau}:K(S)\rightarrow \mathbb{R}.$$
\end{lemma}

\begin{proof}
Define $\overline{\tau}:A(M_{\omega}(S))\rightarrow\mathbb{R}$ by 
$$\overline{\tau}([E]) = \sum_{i=1}^{\infty}(\tau(e_{ii})).$$
We check that $\tau$ is well-defined by noting the fact that $\mathcal{D}$-related idempotents of $S$ are sent to the same number. It is easy to see that $\overline{\tau}$ is a monoid homomorphism and so this induces a group homomorphism $\tilde{\tau}:K(S)\rightarrow \mathbb{R}$.
\end{proof}

\section{Examples}

We will now calculate $K(S)$ for a number of examples.

\subsection{Symmetric inverse monoids}

We saw earlier (Section 4.2) that if $S$ were the set of bijections on the natural numbers with finite support that $K(S) \cong \mathbb{Z}$.

\par Now let $S = I_{n}$ be the symmetric inverse monoid on a set of size $n$, where $n<\infty$. Then $S$ is an orthogonally complete inverse semigroup. Again for $e,f\in E(S)$ we have $e\,\mathcal{D}\, f$ if and only if $|\supp(e)| = |\supp(f)|$ and if $e,f\in E(S)$ are such that $ef = 0$ then $|\supp(e\vee f)| = |\supp(e)| + |\supp(f)|$. We therefore again have:

$$K(S) \cong \mathbb{Z}.$$

\subsection{Groups with adjoined zero}

Let $G$ be a group and let $S$ be $G$ with a $0$ adjoined. Then $S$ is an inverse $\wedge$-semigroup (in fact it is $E^{\ast}$-unitary) and it is orthogonally complete. We see that $E(S) = \left\{0,1\right\}$ and so $K(S) \cong \mathbb{Z}$. 

\subsection{Boolean algebras}

Suppose $S$ is an arbitrary (possibly infinite) unital Boolean algebra, viewed as an inverse semigroup by defining $ab = a\wedge b$, $B(S)$ is the associated Boolean space and for each element $a\in S$ denote by $\mathcal{V}_{a}$ the set of ultrafilters of $S$ containing $a$. 

We have the following facts which follow from results in \cite{LawsonLongNCSD} and \cite{LawsonLenz}, but we prove here for completeness.

\begin{lemma}
\label{Boollem}
\begin{enumerate}
\item For all $a,b\in S$ we have $\mathcal{V}_{a\vee b} = \mathcal{V}_{a}\cup \mathcal{V}_{b}$.
\item For all $a,b\in S$ we have $\mathcal{V}_{a\wedge b} = \mathcal{V}_{a}\cap \mathcal{V}_{b}$.
\item $\mathcal{V}_{a} = \mathcal{V}_{b}$ implies $a = b$.
\end{enumerate}
\end{lemma}

\begin{proof}
\begin{enumerate}
\item It is clear that $\mathcal{V}_{a}\cup \mathcal{V}_{b} \subseteq \mathcal{V}_{a\vee b}$, so we just prove the other inclusion. Let $F\in \mathcal{V}_{a\vee b}$. Suppose first that $ac \neq 0$ for all $c\in F$. Then $F\cup \left\{a\right\}$ will generate a proper filter. Since $F$ is an ultrafilter, it follows that $a\in F$. Now suppose $a,b\notin F$. There must be $c,d\in F$ with $ac = 0 = bd$. Since $c,d\in F$, we must have $cd\in F$ and $cd(a\vee b)\in F$. But
$$cd(a\vee b) = cda \vee cdb = 0 \vee 0 = 0,$$
a contradiction. Thus either $a\in F$ or $b\in F$.  
\item Let $F\in \mathcal{V}_{a}\cap \mathcal{V}_{b}$. Then $a,b\in F$ and thus $ab\in F$. On the other hand, if $F \in \mathcal{V}_{a\wedge b}$ then $a,b\in F$ and so $F\in \mathcal{V}_{a}\cap \mathcal{V}_{b}$.
\item Suppose $\mathcal{V}_{a} = \mathcal{V}_{b}$. Then 
$$\mathcal{V}_{a\wedge b} = \mathcal{V}_{a}\cap \mathcal{V}_{b} = \mathcal{V}_{a}.$$
Let $c = (1\setminus ab)a$. Then $abc = 0$ and $ab\vee c = a$. Let $F\in \mathcal{V}_{c}$. Then $c\in F$ implies $a\in F$ and so $F\in \mathcal{V}_{a} = \mathcal{V}_{ab}$. But then $ab\in F$ which implies $0 = abc\in F$, a contradiction.                                
\end{enumerate}
\end{proof}

Since $B(S)$ has the sets $\mathcal{V}_{a}$, $a\in S$, as a basis, for each open set $U$ there exist a collection of elements $a_{i}$, $i\in I$, with $U = \cup_{i\in I}{\mathcal{V}_{a_{i}}}$. Further, $\mathcal{V}_{1} = B(S)$. Let $N(B(S))$ denote the set of continuous functions from $B(S)$ to $\mathbb{N}\cup \left\{0\right\}$. It is easy to see that $N(B(S))$ forms a ring under pointwise multiplication.

For $0\neq a\in S$, define $f_{a}:B(S)\rightarrow \mathbb{N}$ by $f_{a}(x) = 1$ if $x\in \mathcal{V}_{a}$ and $0$ otherwise. Then $f_{a}$ is a continuous function since $\mathcal{V}_{a}$ is open and $B(S)\setminus \mathcal{V}_{a} = V_{1\setminus a}$ is open.

For 
$$\mathbf{a} = (a_{1},a_{2},\ldots,a_{m}),$$
let $\Delta(\mathbf{a}) = E\in M_{\omega}(S)$ be the matrix with entries $e_{ii} = a_{i}$ for $i = 1,\ldots,m$ and $0$ everywhere else.
Define $f_{\mathbf{a}}:B(S)\rightarrow \mathbb{N}$ by $f_{\mathbf{a}} = f_{a_{1}} + \ldots + f_{a_{m}}$. This again will be a continuous function.

Suppose $a_{1},a_{2}\in S$ are such that $a_{1}a_{2} = 0$. Then we have $\mathcal{V}_{a_{1}}\cap \mathcal{V}_{a_{2}} = \emptyset$ and so $f_{(a_{1},a_{2})} = f_{a_{1}\vee a_{2}}$. Thus $\Delta(\mathbf{a})\, \mathcal{D}\, \Delta(\mathbf{b})$ implies $f_{\mathbf{a}} = f_{\mathbf{b}}$. On the other hand Lemma \ref{Boollem} (3) tells us that if $f_{\mathbf{a}} = f_{\mathbf{b}}$ then $\Delta(\mathbf{a})\, \mathcal{D}\, \Delta(\mathbf{b})$.

It follows that we have a well-defined semigroup monomorphism $\theta: A(M_{\omega}(S))\rightarrow N(B(S))$ given by
$$\theta([\Delta(\mathbf{a})]) = f_{\mathbf{a}}.$$ 

Now let $f\in N(B(S))$ be an arbitrary continuous function. Then since $f$ is continuous and $B(S)$ is compact, $\im(f)$ is compact and therefore $|\im(f)|$ is finite. Further, for all $x\in \mathbb{N}\cup\left\{0\right\}$ we have $f^{-1}(x)$ is clopen (and therefore compact) and so $f^{-1}(x) = U_{a}$ for some $a$. Thus $\theta$ is an isomorphism.

Let $Z(B(S))$ denote the set of continuous functions from $B(S)\rightarrow \mathbb{Z}$. It follows from the remarks of the preceding paragraph that

$$K(S) \cong Z(B(S)).$$

Since $S$ is commutative, we know that $A(S)$ will be a semiring. In fact, we see that $\theta(E\otimes F) = \theta(E)\theta(F)$, where $(\theta(E)\theta(F))(x) = \theta(E)(x)\theta(F)(x)$. Thus $K(S)$ has the structure of a ring. 

We can actually view the Boolean algebra $S$ as a ring by defining $+$ to be symmetric difference: 
$$e + f = (e\setminus f) \vee (f\setminus e).$$
In the case where $ef = 0$, $e+f = e\vee f$. It follows from \cite{Magid} that (algebraic) 
$$K_{0}(S) \cong Z(B(S)).$$  

\begin{comment}
We begin first with a lemma which will be used shortly.

\begin{lemma}
\label{setminusopen}
Let $U,V\subseteq X$ be open subsets. Then $V\setminus U$ is also open. 
\end{lemma}

\begin{proof}
Since $\emptyset$ is open, we assume $U\cap V \neq U$ and $U\cap V\neq \emptyset$. Let $e_{i},f_{j}\in S$ be such that
$$U = \bigcup_{i\in I}{\mathcal{V}_{e_{i}}}$$
and
$$V = \bigcup_{j\in J}{\mathcal{V}_{f_{j}}}$$
for some index sets $I,J$. Let 
$$U^{\prime} = \bigcup_{i\in I}{\mathcal{V}_{1\setminus e_{i}}}.$$
Then $U\cap U^{\prime} = \emptyset$ and $U\cup U^{\prime} = X$. Thus $V\setminus U = V\cup U^{\prime}$ and so $V\setminus U$ is open.
\end{proof}
\end{comment}

Now let us consider the topological $K$-theory of the space $X = B(S)$. Let $p:E\rightarrow X$ be a locally-trivial finite-dimensional vector bundle over $\mathbb{C}$. For each $n\in \mathbb{N}\cup\left\{0\right\}$ we define
$$U_{n} = \left\{x\in X| \rank_{E}(x) = n\right\}.$$
Since the function $\rank_{E}:X\rightarrow \mathbb{N}\cup\left\{0\right\}$ is continuous, there are only finitely many $n$ with $U_{n}$ non-zero. Furthermore each $U_{n}$ is a compact open subset of $X$. Thus $U_{n} = \mathcal{V}_{e_{n}}$ for some $e_{n}\in S$. Since $p:E\rightarrow X$ is locally-trivial, for each $x\in U_{n}$ there is an open set $U_{x}$ containing $x$ with 
$$p|_{p^{-1}(U_{x})}:p^{-1}(U_{x})\rightarrow U_{x}$$
vector bundle isomorphic to the trivial bundle $U_{x}\times \mathbb{C}^{n}\rightarrow U_{x}$. Since open sets are unions of compact open sets it follows that we can pick $U_{x}$ to be $\mathcal{V}_{e}$ for some $e\in S$. Lemma \ref{Boollem} then tells us that we may assume that $U_{x} \cap U_{y}$ is either empty or $U_{x} = U_{y}$ for each $x,y\in U_{n}$. It then follows that 
$$p|_{p^{-1}(U_{n})}:p^{-1}(U_{n}) \rightarrow U_{n}$$
is isomorphic to a trivial vector bundle for each $n$. Thus $p:E\rightarrow U$ is isomorphic to the disjoint union of a finite number of trivial vector bundles, so we may assume 
$$E = \coprod_{k=1}^{m}{\mathcal{V}_{e_{k}}\times \mathbb{C}^{n_{k}}}$$
and $p:E\rightarrow X$ is given by $p(x,v) = x$. 

Let $f:X\rightarrow \mathbb{N}\cup \left\{0\right\}$ be an arbitrary continuous function with $\im(f)=\left\{x_{1},\ldots,x_{m}\right\}$ and let $a_{1},\ldots,a_{m}\in S$ be such that $f^{-1}(x_{k}) = \mathcal{V}_{a_{k}}$. Then define a vector bundle on $B(S)$ by 
$$E_{f} = \coprod_{k=1}^{m}{(\mathcal{V}_{a_{k}}\times \mathbb{C}^{x_{k}})}.$$

It is now not hard to see that such vector bundles are in $1-1$ correspondence with continuous functions on $X$.
Further, the vector bundle associated to $f + g$ will be isomorphic to $E_{f}\oplus E_{g}$. Thus
$$K^{0}_{\mathbb{C}}(X)\cong K(S).$$      

\subsection{Cuntz-Krieger semigroups}

We will now compute the $K$-groups of the Boolean completions of graph inverse semigroups whose underlying graph is finite.
Before going further let us recall the Lenz arrow relation $\rightarrow$. Let $S$ be an inverse $\wedge$-semigroup with $0$ and let $s,s^{\prime}\in S$. We will write $s\rightarrow s^{\prime}$ if for all non-zero $t\leq s$ we have $t\wedge s^{\prime} \neq 0$. If $s,s_{1},\ldots,s_{m}$ are elements of $S$ then we will write
$$s\rightarrow \left\{s_{1},\ldots,s_{m}\right\}$$
if for every non-zero $t\leq s$ we have $t\wedge s_{i} \neq 0$ for some $1\leq i\leq m$.
We write $\left\{s_{1},\ldots,s_{m}\right\} \rightarrow \left\{t_{1},\ldots,t_{n}\right\}$ if $s_{i}\rightarrow \left\{t_{1},\ldots,t_{n}\right\}$ for each $1\leq i\leq m$. We say $s\leftrightarrow t$ if $s\rightarrow t$ and $t\rightarrow s$, and 
$$\left\{s_{1},\ldots,s_{m}\right\} \leftrightarrow \left\{t_{1},\ldots,t_{n}\right\}$$
if $\left\{s_{1},\ldots,s_{m}\right\} \rightarrow \left\{t_{1},\ldots,t_{n}\right\}$ and $\left\{t_{1},\ldots,t_{n}\right\}\rightarrow \left\{s_{1},\ldots,s_{m}\right\}$.

Let $\mathcal{G}$ be a finite directed graph and $P_{\mathcal{G}}$ be the associated graph inverse semigroup (see Section 3.9 for the construction).
% We can generalise the results of the previous subsection to certain inverse semigroups defined on graphs.
We will denote the orthogonal completion of $P_{\mathcal{G}}$ by $D_{\mathcal{G}}$. Elements of $D_{\mathcal{G}}$ are of the form $A^{0}$ where $A$ is a finite, possibly empty, set of mutually orthogonal non-zero elements of $P_{\mathcal{G}}$ and $A^{0}$ is $A\cup \left\{0\right\}$. 
Under elementwise multiplication $D_{\mathcal{G}}$ forms an orthogonally complete inverse semigroup (for details see \cite{LawOrthCompPoly}). 
An element $A^{0}\in D_{\mathcal{G}}$ is idempotent if and only if every element of $A$ is an idemptotent in $P_{\mathcal{G}}$.

We can define a congruence on $D_{\mathcal{G}}$ by $A^{0}\equiv B^{0}$ iff $A\leftrightarrow B$ as sets of elements of $P_{\mathcal{G}}$ (recall that graph inverse semigroups are $E^{\ast}$-unitary and are therefore examples of inverse $\wedge$-semigroups and so we can consider $\rightarrow$ on $P_{\mathcal{G}}$). We denote $D_{\mathcal{G}}/\equiv$ by $CK_{\mathcal{G}}$ and call it the \emph{Cuntz-Krieger semigroup} of $\mathcal{G}$. 
These semigroups are studied in detail in \cite{JonesLawsonGraph}, as a generalisation of the Cuntz monoids introduced in \cite{LawPolyThomp}. 
It was shown in \cite{JonesLawsonGraph} that $CK_{\mathcal{G}}$ is a Boolean inverse monoid and therefore in particular an orthogonally complete inverse semigroup. 
Denote elements of $CK_{\mathcal{G}}$ by $[A^{0}]$ where $A^{0}\in D_{\mathcal{G}}$.
Clearly if $A^{0}\in D_{\mathcal{G}}$ is an idempotent then $[A^{0}]$ is an idempotent in $CK_{\mathcal{G}}$.
It was shown in \cite{JonesLawsonGraph} that $\equiv$ is an idempotent pure congruence. In fact:

\begin{lemma}
If $[A^{0}]$ is an idempotent element of $CK_{\mathcal{G}}$ then $A^{0}$ is an idempotent in $D_{\mathcal{G}}$
\end{lemma}

\begin{proof}
Let $s = \left\{[x_{1},y_{1}],\ldots,[x_{m},y_{m}]\right\}^{0} \in D_{\mathcal{G}}$, $m\geq 1$, and suppose $s^{2} \equiv s$ in $D_{\mathcal{G}}$.
Since $m\geq 1$ we must have $[x_{i},y_{i}][x_{j},y_{j}]\neq 0$ for some $1\leq i,j\leq m$. We have two cases: either $x_{j}$ is a prefix of $y_{i}$ or $y_{i}$ is a prefix of $x_{j}$. First suppose $x_{j} = y_{i}p$ for some element $p\in \mathcal{G}^{\ast}$. Then
$$[x_{i},y_{i}][x_{j},y_{j}] = [x_{i}p,y_{j}]$$
Since $s^{2}\equiv s$ we must have $[x_{i}p,y_{j}]\wedge [x_{k},y_{k}]$ for some $1\leq k\leq m$. By Lemma \ref{curlemasssem} and the fact that elements of $s$ are orthogonal we have $y_{j} = y_{k}$, $x_{i} = x_{k}$ and $p$ is empty. Thus $x_{i} = y_{i} = x_{j} = y_{j}$. 
A similar argument shows that if $x_{j}p = y_{i}$ then we again have $x_{i} = y_{i} = x_{j} = y_{j}$.
It follows that 
$$s^{2} = \left\{[z_{1},z_{1}],\ldots,[z_{n},z_{n}]\right\}^{0}$$ 
for some $z_{i}$'s in $\mathcal{G}^{\ast}$. Since $\leftrightarrow$ is idempotent pure, we must have
$$s = \left\{[z_{1},z_{1}],\ldots,[z_{n},z_{n}]\right\}^{0}.$$
\end{proof}

A couple of remarks:

\begin{remark}
\label{usefulremarksckis}
\begin{enumerate}
\item We have 
$$\left\{[x_{1},x_{1}],\ldots,[x_{n},x_{n}]\right\}^{0}\, \mathcal{D}\, \left\{[y_{1},y_{1}],\ldots,[y_{n},y_{n}]\right\}^{0}$$
in $D_{\mathcal{G}}$ if $\dom(x_{i}) = \dom(y_{i})$ for each $i$, and up to re-ordering of elements this describes $\mathcal{D}$ completely for idempotent elements of $D_{\mathcal{G}}$.
\item If $y$ is a route in $\mathcal{G}$ and $x_{1},\ldots,x_{n}$ are all the edges of $\mathcal{G}_{1}$ with $\ran(x_{i}) = \dom(y)$ then
$$\left\{[yx_{1},yx_{1}],\ldots,[yx_{n},yx_{n}]\right\} \leftrightarrow \left\{[y,y]\right\}$$
in $P_{\mathcal{G}}$. In fact, $\equiv$ on $E(CK_{\mathcal{G}})$ is the equivalence relation generated by
$$\left\{[y_{1}x_{1},y_{1}x_{1}],\ldots,[y_{1}x_{n},y_{1}x_{n}],[y_{2},y_{2}],\ldots,[y_{m},y_{m}]\right\}^{0} \equiv \left\{[y_{1},y_{1}],\ldots,[y_{m},y_{m}]\right\}^{0},$$
where $[y_{1},y_{1}],\ldots,[y_{m},y_{m}]$ are mutually orthogonal elements of $E(P_{\mathcal{G}})$ and \\
$x_{1},\ldots,x_{n}\in \mathcal{G}_{1}$ are all the edges with $\ran(x_{i}) = \dom(y_{1})$.
\end{enumerate}
\end{remark}

Since $[x,x]\,\mathcal{D}\,[\dom(x),\dom(x)]$ in $P_{\mathcal{G}}$ for every $x\in \mathcal{G}^{\ast}$ and since 
$$\left\{[x_{1},x_{1}],\ldots,[x_{n},x_{n}]\right\}^{0} = \bigvee_{i=1}^{n}{\left\{[x_{i},x_{i}]\right\}^{0}}$$
in $CK_{\mathcal{G}}$ it follows that the group $K(CK_{\mathcal{G}})$ can be generated by the elements $\left\{[a,a]\right\}^{0}$ where $a\in \mathcal{G}_{0}$. 
For brevity we will denote the element $[\left\{[a,a]\right\}^{0}]$ in $K(CK_{\mathcal{G}})$ by $a$.
Remark \ref{usefulremarksckis} (2) tells us that
$$\left\{[a,a]\right\}^{0} = \bigvee_{\stackrel{x\in \mathcal{G}_{1}}{\ran(x) = a}}{\left\{[x,x]\right\}^{0}}$$
in $CK_{\mathcal{G}}$. By splitting up this join in $M_{\omega}(CK_{\mathcal{G}})$ and replacing $[x,x]$ by $[\dom(x),\dom(x)]$ using the $\mathcal{D}$-relation for $P_{\mathcal{G}}$ we obtain the relation
$$a = \sum_{\stackrel{x\in \mathcal{G}_{1}}{\ran(x) = a}}{\dom(x)}$$
in $K(CK_{\mathcal{G}})$. 
More generally, consider the relation
$$\left\{[y_{1}x_{1},y_{1}x_{1}],\ldots,[y_{1}x_{n},y_{1}x_{n}],[y_{2},y_{2}],\ldots,[y_{m},y_{m}]\right\}^{0} \equiv \left\{[y_{1},y_{1}],\ldots,[y_{m},y_{m}]\right\}^{0}$$
in $D_{\mathcal{G}}$ where $[y_{1},y_{1}],\ldots,[y_{m},y_{m}]$ are mutually orthogonal elements of $E(P_{\mathcal{G}})$ and $x_{1},\ldots,x_{n}\in \mathcal{G}_{1}$ are all the edges with $\ran(x_{i}) = \dom(y_{1})$. Then splitting up the joins in $M_{\omega}(CK_{\mathcal{G}})$ and replacing $[x,x]$ by $[\dom(x),\dom(x)]$ for each route $x\in \mathcal{G}^{\ast}$ using the $\mathcal{D}$-relation for $P_{\mathcal{G}}$ gives the relation 
$$\sum_{i=1}^{n}{\dom(x_{i})} + \sum_{j=2}^{m}{\dom(y_{j})} = \sum_{j=1}^{m}{\dom(y_{j})}$$
in $K(CK_{\mathcal{G}})$. Since $K(CK_{\mathcal{G}})$ is cancellative this gives
$$\sum_{i=1}^{n}{\dom(x_{i})} = \dom(y_{1}),$$
which we knew already.
% Using Remark \ref{usefulremarksckis} we see this relation generates all other relations in $K(CK_{\mathcal{G}})$.
Thus 
$$K(CK_{\mathcal{G}}) \cong FCG(\mathcal{G}_{0}) / N$$
where $FCG$ denotes taking the free commutative group (written additively and with $0$ as identity) and $N$ is the normal subgroup generated by the relations
$$a = \sum_{\stackrel{x\in \mathcal{G}_{1}}{\ran(x) = a}}{\dom(x)}.$$

\begin{comment}
Using these facts one can show that $K(CK_{\mathcal{G}})$ is the free commutative group (written additively with identity $0$) on $\mathcal{G}_{0}$ modulo the relations 
$$a = \sum_{\ran(x) = a}{\dom(x)}$$
for each $a\in \mathcal{G}_{0}$. 
\end{comment}

This agrees with $K^{0}(\mathcal{O}_{\mathcal{G}})$ for $\mathcal{O}_{\mathcal{G}}$ the Cuntz-Krieger algebra on the graph $\mathcal{G}$ (see e.g. Remark 4.6 of \cite{EphremSpielberg}).

For example, suppose $\mathcal{G}$ is a graph with a single vertex and $n$ edges where $n\geq 2$. Then $P_{\mathcal{G}}$ is simply the polycyclic monoid on $n$ generators and $CK_{\mathcal{G}} \cong C_{n}$, the Cuntz monoid on n generators. In this case $K(CK_{\mathcal{G}})$ will be generated by a single element $a$ (corresponding to the one vertex) and subject to the relation
$$a = \sum_{\stackrel{x\in \mathcal{G}_{1}}{\ran(x) = a}}{\dom(x)} = \sum_{i=1}^{n}{a}$$
and so
$$K(CK_{\mathcal{G}}) \cong \langle a | a = a^{n} \rangle \cong \mathbb{Z}_{n-1},$$
which agrees with $K^{0}(\mathcal{O}_{n})$ (see e.g. Example V.I.3.4 of \cite{BlackadarOA}).
As $C_{n}$ is the Boolean completion of a 0-bisimple inverse semigroup the natural ring structure of $\mathbb{Z}_{n-1}$ arises because of the natural semiring structure on $A(C_{n})$ (since the tensor product described in Section 4.7 makes sense). 

\chapter{Discussion and Further Directions}

We have seen in this thesis that self-similar group actions and left Rees monoids appear in a number of different places, with the underlying theme being self-similarity. This self-similarity can be seen in the similarity transformations of attractors of iterated function systems, recursion in automata and in the normal form of HNN-extensions. One might hope that it may be possible to describe further ideas from self-similar group actions and fundamental groups of graphs of groups in terms of the structure of some underlying semigroups.
It seems that although the fractals which appear in this theory can be geometrically very different that at least some properties of certain classes of fractals will be incorporated in the associated Rees monoid. One may also be able to study the representation theory of the inverse semigroups associated to left Rees monoids in a similar manner to the representation theory of polycyclic monoids.

In Chapter 3 we saw that left Rees categories have a number of different characterisations. It was indicated in Section 3.8 that there exist connections with the representation theory of algebras. The author believes that there may be some fruitful future work in pursuing this further. The automata in Section 3.6 are similar to ones appearing in theoretical computer science. It may therefore be possible to apply ideas about left Rees categories to understand ideas there better. The theory of graph iterated function systems is not as well-developed as that for iterated function systems and so it may be discovered in the future that Rees categories have a r\^{o}le to play in this area.

In Chapter 4 I gave a possible definition of a $K$-group of an orthogonally complete inverse semigroup $S$, by analogy with algebraic $K$-theory. This definition was given in terms of an appropriate notion of projective modules and in terms of idempotent matrices over $S$, and these definitions were shown to be equivalent. It was found that for several examples that the group one calculates is isomorphic to the $K_{0}$-group of an associated $C^{\ast}$-algebra. The next step would be to characterise the classes of semigroups for which this is true. It may also be possible to prove a result along the lines of 
$$K_{0}(D(S)\otimes K)\cong K_{0}(C(S))$$
where $S$ is a particular kind of inverse semigroup with $0$ (for example, strongly $E^{\ast}$-unitary or $F$-inverse), $D(S)$ is its distributive completion, $K$ is a semigroup analogue of operators on a compact space, $\otimes$ is some form of tensor product of inverse semigroups and $C(S)$ is some form of $C^{\ast}$-algebra constructed from $S$ via $D(S)$.

One motivation for the theory that has been developed comes from tilings. Given a tiling, one can define a tiling semigroup $S$ and from that a tiling $C^{\ast}$-algebra $C(S)$. These $C^{\ast}$-algebras are used to model observables in certain quantum systems (\cite{KellendonkInteger}). It was proposed by Bellissard (\cite{Bellissard}) that one can use trace functions defined on $C(S)$ and thus also on $K_{0}(C(S))$ as part of a gap-labelling theory, giving information about quantum mechanics on certain tilings appearing in solid state physics. Tiling semigroups are an example of semigroups to which one should be able to apply the above the theory. It was shown in Section 4.8 that one can define trace functions on orthogonally complete inverse semigroups and by extension on their $K$-groups, and this suggests that one might be able to describe this gap-labelling theory in terms of inverse semigroups.

We saw in Section 4.4 that these $K$-groups could be defined in terms of modules. It was found that the category of modules $\Mod_{S}$ over an orthogonally complete inverse semigroup $S$ is in fact a cocomplete concrete category and so I believe one should be able to study the representation theory of such inverse semigroups via this category, and by extension the representation theory of the corresponding $C^{\ast}$-algebras. In addition, it might be possible to make use of the fact that right ideals of orthogonally complete inverse semigroups are premodules in this representation theory. 

Lawson and I are in the process of studying more about the rook matrices and their properties. This may yield additional insight into how to take this theory further. In particular, it may be possible to define higher $K$-groups for inverse semigroups in terms of these matrices.

In the introduction it was mentioned that Morita equivalence has recently been found to work in a nice way for inverse semigroups, and that the different definitions one might want to use to describe Morita equivalence are in fact equivalent (\cite{FunkLawStein}). One might hope to relate $K(S)$ to morita equivalence, in particular by studying the underlying inductive groupoid of the inverse semigroup as in Section 4.3.

\appendix

\chapter{Scala Implementation}

It is possible to describe in Scala a self-similar action, which allows one to easily perform calculations with the associated left Rees monoid. What follows is an implementation of the similarity monoid of the Sierpinski gasket. The procedures restrictionGX, actionGX and productXGYH will be the same whatever the left Rees monoid is.

\begin{verbatim}
def genAction(x:String, g:String) = (x,g) match {
    case ("L","s") => ("R", "s")
    case ("R","s") => ("L", "s")
    case ("T","s") => ("T", "s")
    case ("L","r") => ("T", "r")
    case ("R","r") => ("L", "r")
    case ("T","r") => ("R", "r")
}
\end{verbatim}

\begin{verbatim}
def restrictionGX(x:String,g:String):String=(x,g) match{
    case ("","") => ""
    case ("", h) => h
    case (y, "") => ""
    case (y,h) if (y+h).size == 2 => genAction(y,h)._2
    case (y,h) => restrictionGX(y.tail,restrictionGX(
        actionGX(Character.toString(y.head),
        Character.toString(h.last)),h.init)
        + restrictionGX(Character.toString(y.head),
        Character.toString(h.last)))
}
\end{verbatim}

\begin{verbatim}

def actionGX(x:String, g:String): String = (x,g) match {
    case ("","") => ""
    case ("", h) => ""
    case (y, "") => y
    case (y,h) if (y+h).size == 2 => genAction(y,h)._1
    case (y,h) => actionGX(actionGX(Character.toString(
        y.head),Character.toString(h.last)),h.init) 
        + actionGX(actionGX(y.tail,restrictionGX(
        Character.toString(y.head), 
        Character.toString(h.last))),restrictionGX(
        actionGX(Character.toString(y.head),
        Character.toString(h.last)),h.init)) 

} 

\end{verbatim}

\begin{verbatim}
def productXGYH(x:String,g:String,y:String,h:String): 
    (String,String) = {
    (x + actionGX(y,g), restrictionGX(y,g) + h)
}

println(productXGYH("L","rsr","RTL","s"))
\end{verbatim}

After this program has been run, the output is (LLTR,rsrs)

\bibliographystyle{plain}

\end{document}